\documentclass[12pt,a4paper,twoside,english,titlepage,bibtotoc,tocindent,idxtotoc,makeidx,BCOR5mm,DIV11]{scrbook}

\usepackage[latin1]{inputenc}
\usepackage[T1]{fontenc}
\usepackage{amsmath}
\usepackage{amsfonts}
\usepackage{amssymb}
\usepackage{amsthm}
\usepackage{makeidx}
\usepackage{multind}
\usepackage[breaklinks]{hyperref}
\usepackage[all]{xy}
\usepackage{enumitem}
\usepackage{lscape} 
\usepackage{myalign}
\usepackage{textcomp}
\usepackage{multicol}
\usepackage{fancyhdr}
\usepackage{afterpage}

\makeatletter
\def\printindex#1#2{
  \chapter*{#2}
  \addcontentsline{toc}{chapter}{#2}
  \begin{multicols}{2}
  \@input{#1.ind}
  \end{multicols}
}
\makeatother

\makeindex{notation}
\makeindex{subject}

\setenumerate{topsep=0pt, itemsep=0pt}
\setitemize{topsep=0pt, itemsep=0pt}
\setdescription{topsep=0pt, itemsep=0pt}
\setlist{noitemsep}
\newcommand{\td}{\,\mathrm{d}}
\newcommand{\Lie}{\textup{Lie}}

\newcommand{\const}{\textup{const}}

\newcommand{\Ad}{\textup{Ad}}
\newcommand{\ad}{\textup{ad}}
\renewcommand{\Re}{\textup{Re}}
\renewcommand{\det}{\textup{det}}
\newcommand{\id}{\textup{id}}
\newcommand{\tr}{\textup{tr}}
\newcommand{\rk}{\textup{rk}}
\newcommand{\Tr}{\textup{Tr}}
\newcommand{\Sym}{\textup{Sym}}
\newcommand{\Skew}{\textup{Skew}}
\newcommand{\Herm}{\textup{Herm}}
\newcommand{\Str}{\textup{Str}}
\newcommand{\str}{\mathfrak{str}}
\newcommand{\Aut}{\textup{Aut}}

\newcommand{\der}{\mathfrak{der}}
\newcommand{\Co}{\textup{Co}}
\newcommand{\co}{\mathfrak{co}}
\newcommand{\GL}{\textup{GL}}
\newcommand{\gl}{\mathfrak{gl}}
\newcommand{\SL}{\textup{SL}}
\renewcommand{\sl}{\mathfrak{sl}}

\newcommand{\Sp}{\textup{Sp}}
\renewcommand{\sp}{\mathfrak{sp}}
\newcommand{\Mp}{\textup{Mp}}
\newcommand{\upO}{\textup{O}}
\newcommand{\SO}{\textup{SO}}
\newcommand{\so}{\mathfrak{so}}

\newcommand{\su}{\mathfrak{su}}
\newcommand{\End}{\textup{End}}
\newcommand{\RR}{\mathbb{R}}
\newcommand{\KK}{\mathbb{K}}
\newcommand{\CC}{\mathbb{C}}
\newcommand{\ZZ}{\mathbb{Z}}
\newcommand{\NN}{\mathbb{N}}
\renewcommand{\SS}{\mathbb{S}}
\newcommand{\HH}{\mathbb{H}}
\newcommand{\OO}{\mathbb{O}}

\newcommand{\lf}{\mathfrak{l}}

\newcommand{\nf}{\mathfrak{n}}
\newcommand{\nfo}{\overline{\mathfrak{n}}}

\renewcommand{\1}{\mathbf{1}}
\newcommand{\Ind}{\textup{Ind}}

\newcommand{\calF}{\mathcal{F}}
\newcommand{\calO}{\mathcal{O}}
\newcommand{\calS}{\mathcal{S}}
\newcommand{\calW}{\mathcal{W}}
\newcommand{\calB}{\mathcal{B}}
\newcommand{\calH}{\mathcal{H}}
\newcommand{\calP}{\mathcal{P}}
\newcommand{\calU}{\mathcal{U}}

\newcommand{\calD}{\mathcal{D}}
\newcommand{\calT}{\mathcal{T}}
\newcommand{\calV}{\mathcal{V}}
\newcommand{\calZ}{\mathcal{Z}}

\newcommand{\frakg}{\mathfrak{g}}

\newcommand{\frakk}{\mathfrak{k}}
\newcommand{\frakp}{\mathfrak{p}}

\newcommand{\frakn}{\mathfrak{n}}
\newcommand{\fraka}{\mathfrak{a}}
\newcommand{\frakq}{\mathfrak{q}}
\newcommand{\frakm}{\mathfrak{m}}
\newcommand{\frakl}{\mathfrak{l}}
\newcommand{\frakt}{\mathfrak{t}}

\newcommand{\frakh}{\mathfrak{h}}
\newcommand{\Det}{\textup{Det}}

\renewcommand{\mod}{\textup{mod}}
\newcommand{\rad}{\textup{rad}}
\newcommand{\res}{\textup{res}}
\newcommand{\Int}{\textup{Int}}
\newcommand{\vol}{\textup{vol}}
\newcommand{\pr}{\textup{pr}}

\newcommand{\Ann}{\textup{Ann}}

\theoremstyle{plain}
\newtheorem{theorem}{Theorem}[section]

\newtheorem{proposition}[theorem]{Proposition}
\newtheorem{lemma}[theorem]{Lemma}
\newtheorem{corollary}[theorem]{Corollary}

\newtheorem{fact}[theorem]{Fact}
\newtheorem*{questions}{Questions}

\theoremstyle{definition}
\newtheorem{definition}[theorem]{Definition}
\newtheorem*{definitionIntro}{Definition}
\newtheorem{example}[theorem]{Example}

\newtheorem{remark}[theorem]{Remark}

\begin{document}

\titlehead{\small Universit\"at Paderborn\\Fakult\"at f\"ur Elektrotechnik, Informatik und Mathematik\\Institut f\"ur Mathematik}
\subject{Dissertation}
\title{Minimal representations of\\conformal groups and\\generalized Laguerre functions}
\author{\large{Jan M\"ollers}}
\date{\small{\textit{13. September 2010}}}
\publishers{
 \small{1. Gutachter: Prof. Dr. Joachim Hilgert} \\
 \small{2. Gutachter: Prof. Dr. T. Kobayashi}\\\vspace{.5cm}
 \small{Betreuer: Prof. Dr. Joachim Hilgert}
}

\maketitle

\cleardoublepage

\frontmatter

\section*{Abstract}
We give a unified construction of the minimal representation of a finite cover $G$ of the conformal group of a (non necessarily euclidean) Jordan algebra $V$. This representation is realized on the $L^2$-space of the minimal orbit $\calO$ of the structure group $L$ of $V$. We construct its corresponding $(\frakg,\frakk)$-module and show that it can be integrated to a unitary irreducible representation of $G$ on $L^2(\calO)$.

In particular, we obtain a unified approach to the two most prominent minimal representations, namely the Segal--Shale--Weil representation of the metaplectic group $\Mp(n,\RR)$ and the minimal representation of $\upO(p+1,q+1)$ which was recently studied by T. Kobayashi, G. Mano and B. {\O}rsted.

In the second part we investigate special functions which give rise to $\frakk$-finite vectors in the representation. Various properties of these special functions such as differential equations, recurrence relations and integral formulas connect to the representation theory involved.

Finally, we define the conformal inversion operator $\calF_\calO$ by the action of the longest Weyl group element. $\calF_\calO$ is a unitary operator on $L^2(\calO)$ of order $2$. We show that the action of $\calF_\calO$ on radial functions is given by a special case of Meijer's $G$-transform.

\section*{Zusammenfassung}
Wir konstruieren einheitlich die minimale Darstellung einer endlichen \"{U}berlagerung $G$ der konformen Gruppe einer (nicht notwendigerweise euklidischen) Jordanalgebra $V$. Diese Darstellung l\"{a}sst sich auf dem $L^2$-Raum der minimalen Bahn der Strukturgruppe $L$ von $V$ realisieren. Wir konstruieren den zugeh\"{o}rigen $(\frakg,\frakk)$-Modul und zeigen, dass er sich zu einer unit\"{a}ren irreduziblen Darstellung von $G$ auf $L^2(\calO)$ integrieren l\"{a}sst.

Insbesondere liefert dies eine einheitliche Sichtweise auf die beiden bekanntesten minimalen Darstellungen: Die Segal--Shale--Weil Darstellung der metaplektischen Gruppe $\Mp(n,\RR)$ und die minimale Darstellung von $\upO(p+1,q+1)$, die k\"{u}rzlich von T. Kobayashi, G. Mano und B. {\O}rsted studiert wurde.

Im zweiten Teil untersuchen wir spezielle Funktionen, die explizite $\frakk$-endliche Vektoren in der Darstellung liefern. Verschiedene Eigenschaften dieser speziellen Funktionen wie Differentialgleichungen, Rekursions- und Integralformeln werden bewiesen und in Bezug zur Darstellungstheorie gesetzt.

Zuletzt definieren wir den konformen Inversionsoperator $\calF_\calO$ durch die Wirkung des l\"{a}ngsten Weylgruppenelements. $\calF_\calO$ ist ein unit\"{a}rer Operator der Ordnung $2$ auf $L^2(\calO)$. Wir zeigen, dass die Wirkung von $\calF_\calO$ auf radialen Funktionen durch eine spezielle Form von Meijer's $G$-Transformation gegeben ist.

\cleardoublepage
\section*{Acknowledgements}

First and foremost I would like to thank my advisor Prof. Dr. Joachim Hilgert. He had the right intuition concerning the topic of my thesis and his office door was always open for me.

I am also happy to have been invited to Tokyo three times by Prof. Dr. T. Kobayashi. During my stays I always received great scientific support and I am grateful for the warm hospitality.

Not to forget is Prof. Dr. S\"{o}nke Hansen who helped me see things clearer from time to time.

For the financial support of my dissertation project I thank the IRTG ``Geometry and Analysis of Symmetries'' and the GCOE program of the University of Tokyo.

Finally I would like to thank my family, friends and collegues for supporting me during my whole studies.

\cleardoublepage

\tableofcontents
\cleardoublepage

\mainmatter

\afterpage{\fancyhead[LE,RO]{Introduction}}
\addcontentsline{toc}{chapter}{Introduction}
\chapter*{Introduction}

We explain the results of this thesis from two different points of view:\vspace{.2cm}
\begin{itemize}
\item The study of minimal representations is motivated from unitary representation theory. Minimal representations are thought to correspond to the minimal nilpotent coadjoint orbit via the orbit philosophy.\vspace{.2cm}
\item On the other hand, the \lq smallness\rq\ of the minimal representation results in large symmetries in its geometric realizations. We investigate an interesting relashionship between minimal representations and certain special functions that solve a fourth order ordinary differential equation. The special functions appear as $\frakk$-finite vectors in the $L^2$-model of the minimal representation.
\end{itemize}

\section*{Minimal representations}

In the theory of unitary representations it is an unsolved problem to determine the unitary irreducible representations of all simple real Lie groups. For simply-connected nilpotent groups $G$, Kirillov's orbit method establishes a correspondence between the unitary irreducible representations of $G$ and its coadjoint orbits. Unfortunately, this methods does not work for arbitrary simple real Lie groups. Nevertheless, Kirillov's method suggests an intimate relation between coadjoint orbits and unitary irreducible representations.

On the one hand, to every unitary irreducible representation $\pi$ of a simple real Lie group $G$ one can associate the annihilator $\Ann(\pi)$ of the derived representation $\td\pi$ in the universal enveloping algebra $\calU(\frakg)$ of $\frakg=\Lie(G)$. Its associated variety $\calV(\Ann(\pi))\subseteq\frakg_\CC^*$ is the closure of a nilpotent coadjoint orbit. On the other hand, there are quantization procedures which associate unitary representations to certain coadjoint orbits. For the nilpotent coadjoint orbits such a quantization procedure is least understood. To gain a better understanding of the relation between unitary representations and nilpotent coadjoint orbits one studies representations which correspond to the minimal nilpotent coadjoint orbit.

\begin{definitionIntro}[{\cite[Definition 4.6]{GS05}}]
A unitary irreducible representation $\pi$ of a simple real Lie group $G$ is called \textit{minimal}\index{subject}{minimal representation} if its annihilator $\Ann(\pi)$ is equal to the Joseph ideal.
\end{definitionIntro}

The Joseph ideal is the unique completely prime two-sided ideal in $\calU(\frakg)$ whose associated variety is the closure of the minimal nilpotent coadjoint orbit (see \cite[Section 4.4]{GS05}). Therefore, minimal representations are thought to correspond to the minimal nilpotent coadjoint orbit. For simple real Lie groups the number of isomorphism classes of minimal representations is always finite. In many cases there is either one or no minimal representation. A self-contained exposition of minimal representations can be found in \cite{GS05}.\\

The most prominent minimal representation is probably the (holomorphic part of the) metaplectic representation. The metaplectic representation (also called oscillator representation or Segal--Shale--Weil representation) is a unitary representation of the metaplectic group $\Mp(n,\RR)$ (the double cover of the symplectic group $\Sp(n,\RR)$) on $L^2(\RR^n)$. There are various connections between the metaplectic representation and other fields of mathematics such as symplectic geometry or number theory. An overview of the metaplectic representation can be found in \cite[Chapter 4]{Fol89} whereas the original papers of I. E. Segal, D. Shale and A. Weil are \cite{Seg63,Sha62,Wei64}.

Another example for a minimal representation has recently attracted more and more attention: the minimal representation of the indefinite orthogonal group $\upO(p+1,q+1)$ with $p+q$ even. A realization on $L^2(C)$, where $C\subseteq\RR^{p+q}$ is an isotropic cone, was constructed by T. Kobayashi and B. {\O}rsted in \cite{KO03c}.\\

There are several results about the construction of minimal representations:\vspace{.2cm}
\begin{itemize}
\item In \cite{BK94} and \cite{Bry98} R. Brylinski and B. Kostant construct the minimal representation of a certain class of simple real Lie groups $G$ on the space of sections of a particular half-form bundle. However, in their construction the case where the corresponding symmetric space $G/K$ is hermitean is excluded.\vspace{.2cm}
\item For the hermitean case, S. Sahi constructs the minimal representation in \cite{Sah92}. Together with A. Dvorsky he also gives a construction for another class of groups in \cite{DS99}. The same is done in \cite{BSZ06}. They all exclude the case $\frakg=\so(p+1,q+1)$ from their considerations.\vspace{.2cm}
\item In the case of the group $\upO(p+1,q+1)$, $p+q$ even, there are several results. For the group $\upO(4,4)$ the minimal representation was first constructed by B. Kostant in \cite{Kos90}. Later B. Binegar and R. Zierau generalized the construction to arbitrary parameters $p$ and $q$ with $p+q$ even (see \cite{BZ91}). Two different geometric models including the $L^2$-model are constructed in \cite{KO03a,KO03b,KO03c}.\vspace{.2cm}
\end{itemize}

However, what is missing is a unified construction of an $L^2$-model of the minimal representation. The right framework for this construction seems to be the framework of Jordan algebras. In the two examples $G=\Mp(n,\RR)$ and $G=\SO(p+1,q+1)_0$ the group $G$ is a finite cover of the identity component $\Co(V)_0$ of the conformal group $\Co(V)$ of a certain simple real Jordan algebra $V$. Therefore, one may ask the following questions:

\begin{questions}
\begin{enumerate}
\item[\textup{(1)}] For which simple real Jordan algebra $V$ does a finite cover $G$ of $\Co(V)_0$ admit a minimal representation?
\item[\textup{(2)}] Is there a natural realization of the minimal representation of $G$ on a certain $L^2$-space?
\end{enumerate}
\end{questions}

By a result of D. A. Vogan, no covering group of $SO(p+1,q+1)_0$ admits a minimal representation if $p+q$ is odd and $p,q\geq3$ (see \cite[Theorem 2.13]{Vog81}). In all other cases one can show that there is a finite cover of the conformal group which admits a minimal representation. The minimal representation can be realized on the $L^2$-space of the minimal non-zero orbit $\calO$ of the structure group $\Str(V)$ of $V$. This is proved in \cite{Sah92} for euclidean Jordan algebras, in \cite{DS99}, \cite{DS03} and \cite[Section 8]{BSZ06} for non-euclidean Jordan algebras of rank $\geq3$, and in \cite{KO03c} for the remaining case $G=\upO(p+1,q+1)$. In Section \ref{sec:MinRepConstruction} we give a unified construction which works for the most general class of Jordan algebras. Our construction can be described as follows:\\

We start with a simple real Jordan algebra $V$ of split rank $r_0\geq2$ with simple maximal euclidean subalgebra $V^+$. Its structure group $\Str(V)$ acts linearly on $V$ and has finitely many orbits. The minimal non-zero orbit $\calO$ of the identity component $\Str(V)_0$ carries a unique $\Str(V)_0$-equivariant measure $\td\mu$. This gives the representation space $L^2(\calO,\td\mu)$.

Let $\frakg$ be the Lie algebra of the conformal group $\Co(V)$ of $V$. First, we construct a Lie algebra representation $\td\pi$ of $\frakg$ on $C^\infty(\calO)$ (see Section \ref{sec:InfinitesimalRep}). Further, we define a function $\psi_0\in C^\infty(\calO)$ by
\begin{align*}
 \psi_0(x) &:= \widetilde{K}_{\frac{\nu}{2}}(|x|), & x\in\calO,
\end{align*}
where $\widetilde{K}_\alpha(z)$ denotes the normalized $K$-Bessel function (see Appendix \ref{app:BesselFunctions}), $|\!\!-\!\!|$ is a certain norm on $V$ and $\nu$ is a structure constant of the Jordan algebra $V$. The subrepresentation of $C^\infty(\calO)$ generated by $\psi_0$ is a $(\frakg,\frakk)$-module if and only if $\frakg\ncong\so(p+1,q+1)$ with $p+q$ odd (see Proposition \ref{prop:Kfinite}). (This is exactly the case for which no minimal representation exists.) Excluding this case, the $(\frakg,\frakk)$-module integrates to a unitary irreducible representation $\pi$ of a finite cover $G$ of $\Co(V)_0$ on $L^2(\calO,\td\mu)$ (see Theorem \ref{thm:IntgkModule}). This representation is in fact a minimal representation (see Remark \ref{rem:MinRep}).

The Jordan algebras corresponding to the groups $\upO(p+1,q+1)$ are those of rank $2$. From a representation theoretic point of view, this case is most difficult to handle, because the corresponding minimal representation is in general neither a highest weight representation nor spherical. In all other cases the representation theory is simpler:
\begin{itemize}
\item For a euclidean Jordan algebra the minimal representation is a highest weight representation.
\item For a non-euclidean Jordan algebra of rank $\geq3$ the function $\psi_0$ is a $K$-spherical vector and hence the minimal representation is spherical.
\end{itemize}
Therefore, in the case of rank $2$ Jordan algebras the calculations are more involved and are treated separately in Appendix \ref{app:Rank2}.

\section*{Generalized Laguerre functions}

The constructed $L^2$-model of the minimal representation allows a wide range of applications, in particular to the theory of special functions.\\

In general, it is quite hard to find explicit expressions for $\frakk$-finite vectors in unitary representations. However, for the minimal representation we determine an explicit $\frakk$-finite vector in every $\frakk$-type. In order to do so, we first compute the action of the $\frakk$-Casimir on radial functions in Section \ref{sec:CasimirAction}. It turns out that the $\frakk$-Casimir essentially acts on the radial parameter $x\in\RR_+$ by the ordinary fourth order differential operator
\begin{align*}
 \calD_{\mu,\nu} &= \frac{1}{x^2}\left((\theta+\mu+\nu)(\theta+\mu)-x^2\right)\left(\theta(\theta+\nu)-x^2\right),
\end{align*}
where $\theta=x\frac{\td}{\td x}$ and $\mu$ and $\nu$ are certain structure constants of the Jordan algebra. We show that this operator extends to a self-adjoint operator on the Hilbert space $L^2(\RR_+,x^{\mu+\nu+1}\td x)$ and compute its spectrum (see Corollary \ref{cor:DmunuSASpectrum}). The $L^2$-eigenfunctions are constructed in terms of their generating function
\begin{align*}
 G_2^{\mu,\nu}(t,x) &= \frac{1}{(1-t)^{\frac{\mu+\nu+2}2}}\widetilde{I}_{\frac{\mu}{2}}\left(\frac{tx}{1-t}\right)\widetilde{K}_{\frac{\nu}{2}}\left(\frac{x}{1-t}\right).
\end{align*}
The generating function $G_2^{\mu,\nu}(t,x)$ is analytic near $t=0$ and hence defines a sequence $(\Lambda_{2,j}^{\mu,\nu}(x))_j$ of functions on $\RR_+$ by
\begin{align*}
 G_2^{\mu,\nu}(t,x) &= \sum_{j=0}^\infty{\Lambda_{2,j}^{\mu,\nu}(x)t^j}.
\end{align*}
We show that for every $j$ the function $\Lambda_{2,j}^{\mu,\nu}(x)$ is an $L^2$-eigenfunction of $\calD_{\mu,\nu}$ for the eigenvalue $4j(j+\mu+1)$ (see Theorem \ref{thm:EigFct}). This implies that the radial functions
\begin{align*}
 \psi_j(x) &:= \Lambda_{2,j}^{\mu,\nu}(|x|), & x\in\calO,
\end{align*}
are explicit $\frakk$-finite vectors in the minimal representation.\\

The parameters $\mu$ and $\nu$ are structure constants of the Jordan algebra. However, the formula for the operator $\calD_{\mu,\nu}$ as well as the construction of the eigenfunctions $\Lambda_{2,j}^{\mu,\nu}(x)$ makes sense also for general complex parameters $\mu,\nu\in\CC$. In Section \ref{ch:GenLagFct} we study properties of $\calD_{\mu,\nu}$ for arbitrary $\mu$ and $\nu$. We further construct a generic fundamental system $\Lambda_{i,j}^{\mu,\nu}(x)$, $i=1,2,3,4$, of the differential equation
\begin{align*}
 \calD_{\mu,\nu}u &= 4j(j+\mu+1)u.
\end{align*}
For $i=2$ the function $\Lambda_{2,j}^{\mu,\nu}(x)$ is an $L^2$-eigenfunction of $\calD_{\mu,\nu}$. Various properties of the special functions $\Lambda_{i,j}^{\mu,\nu}(x)$, $i=1,2,3,4$, such as recurrence relations and integral formulas are derived.\\

Now, if one assumes that the parameters $\mu$ and $\nu$ appear as structure constants of a simple real Jordan algebra for which the minimal representation exists, then representation theory can be used to give short proofs for statements on the $L^2$-eigenfunctions $\Lambda_{2,j}^{\mu,\nu}(x)$:\vspace{.2cm}
\begin{enumerate}
\item[\textup{(1)}] The\ \ functions\ \ $\Lambda_{2,j}^{\mu,\nu}(x)$\ \ ($j=0,1,2,\ldots$)\ \ form\ \ an\ \ orthogonal\ \ basis\ \ of $L^2(\RR_+,x^{\mu+\nu+1}\td x)$ (see Corollary \ref{cor:completeness}). A closed expression for their norms is given in Corollary \ref{cor:Norms}.\vspace{.2cm}
\item[\textup{(2)}] The Lie algebra action predicts various recurrence relations (see Section \ref{sec:RepTh}). These recurrence relations are stated in Section \ref{sec:RecRel}.\vspace{.2cm}
\end{enumerate}
In the case that $\frakg=\so(p+1,q+1)$, these results are already proved in \cite{HKMM09a}. There, only the minimal representation of $\upO(p+1,q+1)$ is used. Hence, the set of parameters $(\mu,\nu)$ which appear in \cite{HKMM09a} is strictly smaller than the set of parameters for which the statements are proved in Chapter \ref{ch:GenLagFct}.

If the Jordan algebra $V$ with which we start is euclidean, then the parameter $\nu$ is equal to $-1$. In this case the functions $\Lambda_{2,j}^{\mu,\nu}(x)$ simplify to Laguerre functions (see Corollary \ref{cor:SpecialValue}):
\begin{align*}
 \Lambda_{2,j}^{\mu,-1}(x) &= \const\cdot e^{-x}L_j^\mu(2x),
\end{align*}
where $L_n^\alpha(z)$ denote the Laguerre polynomials as introduced in Appendix \ref{app:Laguerre}. For this case, the differential equation and the recurrence relations are a reformulation of \cite[Theorem 6.3]{ADO07}.\\

Another type of special functions occurs if one studies the \textit{unitary inversion operator} $\calF_\calO$. This operator is defined using the group action of the minimal representation $\pi$ (see Section \ref{sec:UnitInvOp}). $\calF_\calO$ is a unitary involutive operator on $L^2(\calO,\td\mu)$ which resembles the euclidean Fourier transform. Various properties of $\calF_\calO$ are proved in Theorem \ref{prop:FOProperties}. Together with the action of a maximal parabolic subgroup of $G$ (which can be written down explicitly) the operator $\calF_\calO$ determines the whole representation $\pi$. Therefore, to gain a better understanding of the minimal representation, it might help to find an explicit formula for the action of $\calF_\calO$. For the case $\frakg=\so(p+1,q+1)$ the full integral kernel of $\calF_\calO$ was computed by T. Kobayashi and G. Mano in \cite{KM07a,KM08}. To generalize this result, we determine, as a first step into this direction, the action of $\calF_\calO$ on radial functions. It turns out that $\calF_\calO$ preserves the space of radial functions and acts on a radial function $\psi(x)=f(|x|)$ as the integral transform (see Theorem \ref{thm:RadialUnitaryInversion})
\begin{align*}
 \calT^{\mu,\nu}f(x) &= \int_0^\infty{K^{\mu,\nu}(xy)f(y)y^{\mu+\nu+1}\td y}
\end{align*}
with integral kernel
\begin{align*}
 K^{\mu,\nu}(x) = \frac{1}{2^{\mu+\nu+1}}G^{20}_{04}\left(\left(\frac{x}{4}\right)^2\left|0,-\frac{\nu}{2},-\frac{\mu}{2},-\frac{\mu+\nu}{2}\right.\right)
\end{align*}
given in terms of Meijer's $G$-function (see Appendix \ref{app:GFct}). The operator $\calT^{\mu,\nu}$ is a special case of the more general $G$-transform as studied in \cite{Fox61}. As a corollary of these observations we obtain that the functions $\Lambda_{2,j}^{\mu,\nu}(x)$ are eigefunctions of the $G$-transform $\calT^{\mu,\nu}$ for the eigenvalues $(-1)^j$.\\

All in all, we observe an intimate relation between the special functions $\Lambda_{2,j}^{\mu,\nu}(x)$ and minimal representations. On the one hand, results about the special functions $\Lambda_{2,j}^{\mu,\nu}(x)$ are used to obtain explicit expressions for $\frakk$-finite vectors in the minimal representation. But on the other hand, representation theory also provides proofs of statements on the special functions $\Lambda_{2,j}^{\mu,\nu}(x)$ such as orthogonality relations, completeness or integral formulas.

\section*{Outline of the thesis}

In the first chapter we introduce the concept of Jordan algebras. The basic structure theory is explained and the structure constants $\mu$ and $\nu$ are defined. We further describe the structure group and its orbits as well as equivariant measures on the orbits. The conformal group and its Lie algebra are discussed in detail. Finally, we define the Bessel operators which are needed for the Lie algebra action of the minimal representation.

Chapter \ref{ch:MinRep} is concerned with the minimal representation. We first give a detailed contruction of the representation. In the second part we explain the relation of the minimal representation to generalized principal series representations. We further show that the Casimir element $C_\frakk$ acts on radial functions as the fourth order differential operator $\calD_{\mu,\nu}$. In the fourth section we define the unitary inversion operator $\calF_\calO$ and prove several properties for it. We also prove that $\calF_\calO$ acts on radial functions by the $G$-transform $\calT^{\mu,\nu}$.

The third chapter deals with the differential operator $\calD_{\mu,\nu}$ and its eigenfunctions. Here we do in general not assume that $\mu$ and $\nu$ are the structure constants of a certain Jordan algebra $V$. We construct eigenfunctions $\Lambda_{i,j}^{\mu,\nu}(x)$, $i=1,2,3,4$, of $\calD_{\mu,\nu}$ in terms of their generating functions and investigate main properties such as asymptotic behavior, recurrence relations or integral representations. Now suppose, $\mu$ and $\nu$ are the structure constants of a Jordan algebra $V$ for which the minimal representation exists. For the $L^2$-eigenfunctions $\Lambda_{2,j}^{\mu,\nu}(x)$ we derive orthogonality relations, expressions for the norms, a completeness statement and simplification formulas. In the last section we interpret the functions $\Lambda_{2,j}^{\mu,\nu}(x)$ as radial parts of $\frakk$-finite vectors in the minimal representation associated to the Jordan algebra $V$.

\section*{Outlook}

\begin{enumerate}
\item[\textup{(1)}] Our construction of the minimal representation uses the rich structure of Jordan algebras. A generalization of the concept of Jordan algebras leads to the notion of Jordan triple systems. Many objects that are needed in our construction still exist in the theory of Jordan triple systems. Therefore, it is an interesting question whether the construction of the minimal representation can also be carried out in the more general framework of Jordan triple systems.\vspace{.2cm}
\item[\textup{(2)}] In the special case where $\frakg=\so(p+1,q+1)$, T. Kobayashi and G. Mano computed the action of the unitary inversion operator $\calF_\calO$ not only on radial functions, but on every $K_L$-isotypic component, where $K_L=K\cap\Str(V)=\SO(p)\times\SO(q)$ (see \cite[Theorem 4.1.1]{KM08}). Using these results they determined the full integral kernel $K(x,y)\in\calD'(\calO\times\calO)$ of $\calF_\calO$ (see \cite[Theorem 5.1.1]{KM08}). The same method might work also in the general case.\vspace{.2cm}
\item[\textup{(3)}] A big advantage of the $L^2$-realization of the minimal representation is that it is well-suited for tensor product computations. The decomposition of tensor powers of the minimal representation is studied in \cite[Theorem 0.2]{DS99} for non-euclidean Jordan algebras of rank $\geq3$ and in \cite{Dvo07} for the case $\frakg=\so(p+1,q+1)$. It should be possible to prove these results in the general framework.\vspace{.2cm}
\item[\textup{(4)}] The same might be possible for branching laws for the restriction to a symmetric subgroup. In \cite{Sep07a,Sep07b,Sep08,MS10} the branching laws for restriction to the structure group $\Str(V)$ are studied in the case of euclidean Jordan algebras. Some ideas might also apply in the general case.\\
\end{enumerate}

Notation: $\NN=\{1,2,3,\ldots\}$.

\cleardoublepage
\fancyhead[LE]{\nouppercase{\leftmark}}
\fancyhead[RO]{\nouppercase{\rightmark}}
\chapter{Jordan theory}

In this chapter we introduce the main concepts in the theory of Jordan algebras. We first define Jordan algebras and analyze their algebraic structure. To every Jordan algebra we associate two important groups which are needed in Chapter \ref{ch:MinRep} to construct representations:
\begin{itemize}
\item The structure group which acts linearly on the Jordan algebra. Its orbits provide the geometry of the representation space.
\item The conformal group which acts on the Jordan algebra by rational transformations. The minimal representation will be a unitary irreducible representation of a finite cover of it.
\end{itemize}
The stated results are either known or simple computations which are needed in the subsequent chapters. The notation is mostly as in \cite{FK94} where most results of this chapter can be found, although only for the special case of euclidean Jordan algebras.

\section{Jordan algebras}

The algebraic framework for the construction of the minimal representation will be the framework of Jordan algebras. Jordan algebras can be defined over general fields, but for our purpose it suffices to consider either $\KK=\RR$ or $\KK=\CC$.

\begin{definition}
A vectorspace $V$ together with a bilinear multiplication $V\times V\rightarrow V,\,(x,y)\mapsto x\cdot y=xy,$ and a unit element $e$ (i.e. $x\cdot e=x=e\cdot x$ for every $x\in V$) is called \textit{Jordan algebra}\index{subject}{Jordan algebra} if the following two properties hold for any $x,y\in V$:
\begin{align}
 x\cdot y &= y\cdot x,\label{eq:J1}\tag{J1}\\
 x\cdot(x^2\cdot y) &= x^2\cdot(x\cdot y).\label{eq:J2}\tag{J2}
\end{align}
\end{definition}

Let us fix some notation:

\begin{itemize}
\item We denote by $L(x)\in\End(V)$\index{notation}{Lx@$L(x)$} the multiplication by $x\in V$. With this the axiom \eqref{eq:J2} can be written as
\begin{align*}
 [L(x),L(x^2)] &= 0 & \forall\,x\in V.
\end{align*}
\item Write
\begin{align*}
 P(x) &= 2L(x)^2-L(x^2)\index{notation}{Px@$P(x)$}
\end{align*}
for the quadratic representation and
\begin{align*}
 P(x,y) &= L(x)L(y)+L(y)L(x)-L(xy)\index{notation}{Pxy@$P(x,y)$}
\end{align*}
for its polarized version.
\item Define
\begin{align}
 x\Box y &:= L(xy)+[L(x),L(y)].\label{eq:DefBoxOp}\index{notation}{xBoxy@$x\Box y$}
\end{align}
Then $(x\Box y)z=P(x,z)y$.
\end{itemize}

Let $V$ be a finite-dimensional Jordan algebra of dimension $n$\index{notation}{n@$n$}. To $x\in V$ one can associate a generic minimal polynomial (see e.g. \cite[Section II.2]{FK94})
\begin{align}
 f_x(\lambda) &= \lambda^r-a_1(x)\lambda^{r-1}+\ldots+(-1)^ra_r(x).\label{eq:MinPoly}
\end{align}
Its degree $r$\index{notation}{r@$r$} is called the \textit{rank of $V$}\index{subject}{Jordan algebra!rank}. For $1\leq j\leq r$ the function $a_j(x)$ is a homogeneous polynomial on $V$ of degree $j$. Every such polynomial $a_j(x)$ is invariant under automorphisms of $V$, i.e. invertible linear transformations $g\in\GL(V)$ which preserve the Jordan product:
\begin{align*}
 g(x\cdot y) &= gx\cdot gy &\forall\, x,y\in V.
\end{align*}
In particular, the \textit{Jordan trace}\index{subject}{Jordan trace}
\begin{align*}
 \tr(x) &:= a_1(x)\index{notation}{trx@$\tr(x)$}
\end{align*}
and the \textit{Jordan determinant}\index{subject}{Jordan determinant}
\begin{align*}
 \det(x) &:= \Delta(x) := a_r(x)\index{notation}{detx@$\det(x)$}\index{notation}{Deltax@$\Delta(x)$}
\end{align*}
are invariant under automorphisms. (To avoid confusion, we write $\Tr$\index{notation}{Tr@$\Tr$} and $\Det$\index{notation}{Det@$\Det$} for the usual trace and determinant of an endomorphism.) For $x=e$ the identity element we have (cf. \cite[Proposition II.2.2]{FK94}):
\begin{align}
 \tr(e) &= r, & \det(e) &= 1.\label{eq:trdetofe}
\end{align}
An element $x\in V$ is called \textit{invertible}\index{subject}{invertible element} if there exists $y\in\KK[x]$ such that $xy=e=yx$. The inverse $y$ is unique and we write $x^{-1}:=y$\index{notation}{xinverse@$x^{-1}$}. An element $x$ is invertible if and only if $\Delta(x)\neq0$, and in this case $\Delta(x)x^{-1}$ is polynomial in $x$ (see \cite[Proposition II.2.4]{FK94}). The differential of the map $x\mapsto x^{-1}$ is given in terms of the quadratic representation:
\begin{align*}
 D_u(x^{-1}) &= -P(x)^{-1}u.
\end{align*}

The symmetric bilinear form
\begin{align*}
 \tau(x,y) &:= \tr(xy), & x,y\in V,\index{notation}{tauxy@$\tau(x,y)$}
\end{align*}
is called the \textit{trace form}\index{subject}{trace form} of $V$. The trace form is associative, i.e.
\begin{align*}
 \tau(xy,z) &= \tau(x,yz) & \forall\, x,y,z\in V.
\end{align*}
If $\tau$ is non-degenerate, we call $V$ \textit{semisimple}\index{subject}{Jordan algebra!semisimple}, and if $\KK=\RR$ and $\tau$ is positive definite, we call $V$ \textit{euclidean}\index{subject}{Jordan algebra!euclidean}. Further, $V$ is called \textit{simple}\index{subject}{Jordan algebra!simple} if $V$ is semisimple and has no non-trivial ideal. From now on we assume that $\KK=\RR$.

An involutive automorphism $\alpha$ of $V$ such that
\begin{align*}
 (x|y) &:= \tau(x,\alpha y)\index{notation}{1bracket@$(-"|-)$}
\end{align*}
is positive definite, is called \textit{Cartan involution of $V$}\index{subject}{Cartan involution}\index{notation}{alpha@$\alpha$}. Such a Cartan involution always exists and two Cartan involutions are conjugate by an automorphism of $V$ (see \cite[Satz 4.1, Satz 5.2]{Hel69}). We have the decomposition
\begin{align*}
 V=V^+\oplus V^-\index{notation}{V1@$V^+$}\index{notation}{V2@$V^-$}
\end{align*}
into $\pm1$ eigenspaces of $V$. It is further easy to see that
\begin{align*}
 V^\pm\cdot V^\pm &\subseteq V^+,\\
 V^+\cdot V^- &\subseteq V^-.
\end{align*}
Hence, the $+1$ eigenspace $V^+$ is a euclidean Jordan subalgebra of $V$ with the same identity element $e$. Note that if $V$ itself is already euclidean, then the identity $\alpha=\id_V$ is a Cartan involution and since two Cartan involutions are conjugate, it is also the only Cartan involution. In this case clearly $V^+=V$ and $V^-=0$.

We denote by $n_0$\index{notation}{nnull@$n_0$} and $r_0$\index{notation}{rnull@$r_0$} dimension and rank of $V^+$ and call $r_0$ the \textit{split rank}\index{subject}{Jordan algebra!split rank} of $V$. The constants $n_0$ and $r_0$ only depend on the isomophism class of the Jordan algebra $V$, not on the choice of $\alpha$. In fact, if $\beta$ is another Cartan involution, then $\beta=g\alpha g^{-1}$ for an automorphism $g$. Hence, $gV^+$ is the $+1$ eigenspace of $\beta$ which is clearly isomorphic to $V^+$ as Jordan algebra. Therefore, dimension and rank of $V^+$ and $gV^+$ have to coincide.

One can use the Cartan involution to show that the Jordan trace can be written as the trace of an endomorphism on $V$.

\begin{lemma}\label{lem:Trtr}
Let $V$ be a simple real Jordan algebra such that $V^+$ is also simple. Then
\begin{align}
 \Tr(L(x)) &= \frac{n}{r}\tr(x), & x&\in V.\label{eq:Trtr}
\end{align}
\end{lemma}

\begin{proof}
By\ \ \cite[Proposition II.4.3]{FK94}\ \ the\ \ symmetric\ \ bilinear\ \ forms\ \ $\tr(xy)$\ \ and $\Tr(L(xy))$ are associative. Since $V^+$ is simple, by \cite[Proposition III.4.1]{FK94} every two symmetric associative bilinear forms on $V^+$ are scalar multiples of each other. Hence, there has to be a constant $\lambda\in\RR$ such that $\Tr(L(x))=\lambda\,\tr(x)$ for all $x\in V^+$. Putting $x=e$ we find with \eqref{eq:trdetofe} that $\lambda=\frac{n}{r}$. It remains to show \eqref{eq:Trtr} for $x\in V^-$. But in this case $\tr(x)=\tr(\alpha x)=-\tr(x)$ and hence, $\tr(x)=0$. On the other hand, $\Tr(L(x))=\Tr(\alpha L(x)\alpha)=\Tr(L(\alpha x))=-\Tr(L(x))$ and therefore also $\Tr(L(x))=0$ which shows the claim.
\end{proof}

\begin{example}\label{ex:JordanAlgebras}
\begin{enumerate}
 \item[\textup{(1)}] Let $V=\Sym(n,\RR)$\index{notation}{SymnR@$\Sym(n,\RR)$} be the space of symmetric $n\times n$ matrices with real entries. Endowed with the multiplication
  \begin{equation*}
   x\cdot y :=\frac{1}{2}(xy+yx)
  \end{equation*}
  $V$ becomes a simple euclidean Jordan algebra of dimension $\frac{n(n-1)}{2}$ and rank $n$. Trace and determinant are the usual ones for matrices:
  \begin{align*}
   \tr(x) &= \Tr(x), & \det(x) &= \Det(x).
  \end{align*}
  Hence, the trace form is given by $\tau(x,y)=\Tr(xy)$. The inverse $x^{-1}$ of $x\in V$ exists if and only if $\Det(x)\neq0$ and in this case $x^{-1}$ is the usual inverse of the matrix $x$.
 \item[\textup{(2)}] Let $V=\RR\times W$ where $W$ is a real vector space of dimension $n-1$ with a symmetric bilinear form $\beta:W\times W\rightarrow\RR$. Then $V$ turns into a Jordan algebra with multiplication given by
  \begin{equation*}
   (\lambda,u)\cdot(\mu,v) := (\lambda\mu+\beta(u,v),\lambda v+\mu u).
  \end{equation*}
  $V$ is of dimension $n$ and rank $2$. Trace and determinant are given by
  \begin{align*}
   \tr(\lambda,u) &= 2\lambda, & \det(\lambda,u) &= \lambda^2-\beta(u,u),
  \end{align*}
  and an element $(\lambda,u)\in V$ is invertible if and only if $\det(\lambda,u)=\lambda^2-\beta(u,u)\neq0$. In this case the inverse is given by $(\lambda,u)^{-1}=\frac{1}{\det(\lambda,u)}(\lambda,-u)$. The trace form can be written as
  \begin{equation*}
   \tau((\lambda,u),(\mu,v)) = 2(\lambda\mu+\beta(u,v)).
  \end{equation*}
  Hence, $V$ is semisimple if and only if $\beta$ is non-degenerate and $V$ is euclidean if and only if $\beta$ is positive definite. For $W=\RR^{p-1,q}=\RR^{p+q-1}$ with bilinear form $\beta$ given by the matrix
  \begin{align*}
   \left(\begin{array}{cc}-\1_{p-1}&\\&\1_q\end{array}\right)
  \end{align*}
  we put $\RR^{p,q}:=\RR\times\RR^{p-1,q}$\index{notation}{Rpq@$\RR^{p,q}$}, $p\geq1$, $q\geq0$. Then
  \begin{align*}
   \tau(x,y) &= 2(x_1y_1-x_2y_2-\ldots-x_py_p+x_{p+1}y_{p+1}+\ldots+x_{p+q}y_{p+q}),\\
   \Delta(x) &= x_1^2+\ldots+x_p^2-x_{p+1}^2-\ldots-x_{p+q}^2.
  \end{align*}
  Thus, $\RR^{p,q}$ is euclidean if and only if $p=1$. In any case, a Cartan involution of $\RR^{p,q}$ is given by
  \begin{align}
   \alpha &= \left(\begin{array}{ccc}1&&\\&-\1_{p-1}&\\&&\1_q\end{array}\right).\label{eq:VpqCartanInv}
  \end{align}
  With this choice the euclidean subalgebra $(\RR^{p,q})^+$ amounts to
  \begin{align*}
   (\RR^{p,q})^+ &= \RR e_1\oplus\RR e_{p+1}\oplus\ldots\RR e_n \cong \RR^{1,q}.
  \end{align*}
\end{enumerate}
\end{example}

\section{Peirce decomposition}

From now on let $V$ be a simple real Jordan algebra, $\alpha$ a Cartan involution and assume that $V^+$ is also simple. We introduce the Peirce decomposition of $V$ which describes the structure with the use of idempotents. From the Peirce decomposition we derive some formulas that are needed later.

\subsection{Peirce decomposition for one idempotent}

Let $c\in V^+$ be any \textit{idempotent}\index{subject}{idempotent}, i.e. $c^2=c$. By \cite[Chapter VI.1]{FK94} the only possible eigenvalues of the operator $L(c)$ are $0$, $\frac{1}{2}$ and $1$. Since $L(c)$ is symmetric with respect to the inner product $(-|-)$, this gives the following orthogonal decomposition:
\begin{align}
 V &= \textstyle V(c,1) \oplus V(c,\frac{1}{2}) \oplus V(c,0),\label{eq:PeirceDecompSingle}
\end{align}
where
\begin{align*}
 V(c,\lambda) &= \{x\in V:L(c)x=\lambda x\}.\index{notation}{Vclambda@$V(c,\lambda)$}
\end{align*}
\eqref{eq:PeirceDecompSingle} is called \textit{Peirce decomposition}\index{subject}{Peirce decomposition} corresponding to $c$. Since $L(c)$ is also symmetric with respect to the trace form $\tau$, the decomposition in \eqref{eq:PeirceDecompSingle} is also orthogonal with respect to $\tau$. The subspaces $V(c,1)$ and $V(c,0)$ are subalgebras of $V$ with unit elements $c$ and $e-c$, respectively. Hence $V(c,1)\cdot V(c,1)\subseteq V(c,1)$ and similarly for $V(c,0)$. We have the following additional inclusions (cf. \cite[Proposition IV.1.1]{FK94}):
\begin{align*}
 V(c,1)\cdot V(c,0) &= 0,\\
 \textstyle (V(c,1)+V(c,0))\cdot V(c,\frac{1}{2}) &\subseteq \textstyle V(c,\frac{1}{2}),\\
 \textstyle V(c,\frac{1}{2})\cdot V(c,\frac{1}{2}) &\subseteq V(c,1)+V(c,0).
\end{align*}
The projection onto $V(c,1)$ in the Peirce decomposition \eqref{eq:PeirceDecompSingle} is given by $P(c)$ (see \cite[Chapter IV, Section 1]{FK94}).

\subsection{Peirce decomposition for a Jordan frame}\label{sec:PeirceJordanFrame}

An idempotent is called \textit{primitive}\index{subject}{idempotent!primitive} if it is non-zero and cannot be written as the sum of two non-zero idempotents. Further, two idempotents $c_1$ and $c_2$ are called \textit{orthogonal}\index{subject}{idempotent!orthogonal} if $c_1c_2=0$. A collection $c_1,\ldots,c_k$ of orthogonal primitive idempotents in $V^+$ with $c_1+\ldots+c_k=e$ is called a \textit{Jordan frame}\index{subject}{Jordan frame}. By \cite[Theorem III.1.2]{FK94} the number $k$ of idempotents in a Jordan frame is always equal to the rank $r_0$ of $V^+$. For every two Jordan frames $c_1,\ldots,c_{r_0}$ and $d_1,\ldots,d_{r_0}$ there exists an automorphism $g$ of $V$ such that $gc_i=d_i$, $1\leq i\leq r_0$ (see \cite[Satz 8.3]{Hel69}).

Now choose a Jordan frame $c_1,\ldots,c_{r_0}$ in $V^+$. Then the operators $L(c_1),\ldots,L(c_{r_0})$ commute by \cite[Proposition II.1.1 (1)]{FK94} and hence are simultaneously diagonalizable. Since each $L(c_i)$ has possible eigenvalues $0$, $\frac{1}{2}$ and $1$ and $\sum_{i=1}^{r_0}{L(c_i)}=L(e)=\id_V$, this yields the \textit{Peirce decomposition}\index{subject}{Peirce decomposition}
\begin{align}
 V &= \bigoplus_{1\leq i\leq j\leq r_0}{V_{ij}},\label{eq:PeirceDecomp}\\
\intertext{where}
 V_{ii} &= V(c_i,1) && \mbox{for }1\leq i\leq r_0,\notag\index{notation}{Vij@$V_{ij}$}\\
 V_{ij} &= \textstyle V(c_i,\frac{1}{2})\cap V(c_j,\frac{1}{2}) && \mbox{for $1\leq i\neq j\leq r_0$.}\notag
\end{align}
Since the endomorphisms $L(c_i)$, $1\leq i\leq r_0$, are all symmetric with respect to the inner product $(-|-)$, the direct sum in \eqref{eq:PeirceDecomp} is orthogonal. As previously remarked, the group of automorphisms contains all possible permutations of the idempotents $c_1,\ldots,c_{r_0}$. Therefore, the subalgebras $V_{ii}$ have a common dimension $e+1$\index{notation}{e@$e$} and the subspaces $V_{ij}$ ($i<j$) have a common dimension $d$\index{notation}{d@$d$}. Then clearly
\begin{align}
 \frac{n}{r_0}=e+1+(r_0-1)\frac{d}{2}.\label{eq:noverr0}
\end{align}
We call a Jordan algebra $V$ \textit{reduced}\index{subject}{Jordan algebra!reduced} if $V_{ii}=\RR c_i$ for every $i=1,\ldots,r_0$, or equivalently if $e=0$. From \cite[\S 8, Korollar 2]{Hel69} it follows that if $V$ is reduced, then $r=r_0$, and if $V$ is non-reduced, then $r=2r_0$. We can write $\rk(V_{ii})=\frac{r}{r_0}$. Euclidean Jordan algebras are always reduced (see \cite[Theorem III.1.1]{FK94}). Hence $V^+_{ii}:=V_{ii}\cap V^+=\RR c_i$\index{notation}{Vijplus@$V_{ij}^+$}. If we denote by $d_0$\index{notation}{dnull@$d_0$} the dimension of $V^+_{ij}:=V_{ij}\cap V^+$ ($i<j$), then equation \eqref{eq:noverr0} for the euclidean subalgebra $V^+$ reads
\begin{align*}
 \frac{n_0}{r_0} &= 1+(r_0-1)\frac{d_0}{2}.
\end{align*}
Tables \ref{tb:ConstantsV} and \ref{tb:ConstantsVplus} list all simple real Jordan algebras with simple $V^+$ and their corresponding structure constants. A close look at the table allows the following observation: If $V$ is non-euclidean, then $d=2d_0$ except in the case where $V=\RR^{p,q}$ with $p\neq q$. We state and prove this observation without using a classification result.

\begin{proposition}\label{prop:ClassificationEuclSph}
Let $V$ be a simple real Jordan algebra of split rank $r_0\geq2$, $\alpha$ a Cartan involution and assume that $V^+$ is also simple. Then exactly one of the following three statements holds:
\begin{enumerate}
\item[\textup{(1)}] $V$ is euclidean and in particular $d=d_0$,
\item[\textup{(2)}] $V$ is non-euclidean of rank $r\geq3$ and $d=2d_0$,
\item[\textup{(3)}] $V\cong\RR^{p,q}$, $p,q\geq2$.
\end{enumerate}
\end{proposition}

\begin{proof}
Since $r\geq r_0$, we have $r\geq2$. If $V$ is of rank $r=r_0=2$, then $V\cong\RR^{p,q}$, $p,q\geq1$, by \cite[Chapter VI, Satz 7.1]{BK66}. In the case where $p=1$ the algebra $V$ is euclidean. The case $q=1$ cannot occur, because then $V^+\cong\RR^{1,1}$ which is not simple.\\
Now, if $V$ is non-euclidean of rank $r\geq3$, then either $r=r_0$ or $r=2r_0$. If $r=r_0$, then $r_0\geq3$ and by \cite[end of \S 6]{Hel69} we have $d_0=\dim V_{ij}^+=\dim V_{ij}^-=d-d_0$ for $i<j$ and hence $d=2d_0$. If $r=2r_0$, then $V_{ii}\neq\RR c_i$ for $i=1,\ldots,r_0$ and by \cite[Lemma 6.3]{Hel69} we obtain the same conclusion. This finishes the proof.
\end{proof}

\begin{example}
\begin{enumerate}
\item[\textup{(1)}] For $V=\Sym(n,\RR)$ the matrices $c_i:=E_{ii}$, $1\leq i\leq n$, form a Jordan frame. The Peirce spaces are given by
\begin{align*}
 V_{ii} &= \RR c_i && \mbox{for }1\leq i\leq n,\\
 V_{ij} &= \RR(E_{ij}+E_{ji}) && \mbox{for }1\leq i<j\leq n.
\end{align*}
Hence, $d=1$.
\item[\textup{(2)}] For $V=\RR^{p,q}$, $p,q\geq1$, a Jordan frame is given by $c_1=\frac{1}{2}(e_1+e_n)$, $c_2=\frac{1}{2}(e_1-e_n)$, $n=\dim(V)=p+q$. The Peirce spaces are
\begin{align*}
 V_{11} &= \RR c_1,\\
 V_{12} &= \RR e_2\oplus\ldots\oplus\RR e_{n-1},\\
 V_{22} &= \RR c_2.
\end{align*}
Therefore $V$ is reduced, i.e. $e=0$, and $d=p+q-2$, $d_0=q-1$.
\end{enumerate}
\end{example}

\subsection{Applications}

Using the Peirce decomposition we do some calculations that we need later on. First, to use inductive arguments, we calculate the trace of the lower rank subalgebras $V(c,1)$ for $c\in V$ an idempotent, and also on $V^+$

\begin{lemma}\label{lem:TrOnSubspace}
Let $V$ be a simple Jordan algebra such that $V^+$ is also simple.
\begin{enumerate}
\item[\textup{(1)}] For $x\in V^+$ we have
\begin{align*}
 \tr_{V^+}(x) &= \frac{r_0}{r}\tr_V(x).
\end{align*}
\item[\textup{(2)}] Let $c$ be an idempotent in $V^+$. Then for $x\in V(c,1)$:
\begin{align}
 \tr_V(x) &= \tr_{V(c,1)}(x).\label{eq:TrOnSubspace}
\end{align}
\end{enumerate}
\end{lemma}

\begin{proof}
We only prove the second statement. The first statement follows by the same arguments.\\
Both $\tr_V(xy)$ and $\tr_{V(c,1)}(xy)$ are associative symmetric bilinear forms on the euclidean simple Jordan algebra $V^+(c,1)$. By \cite[Proposition III.4.1]{FK94} they are scalar multiples of each other and hence $\tr_V(x)=\lambda\,\tr_{V(c,1)}(x)$ for $x\in V^+(c,1)$. We claim that $\lambda=1$. In fact, write $c=c_1+\ldots+c_k$, where $c_i\in V^+(c,1)$ are orthogonal idempotents which are primitive in $V^+(c,1)$. Then $k=\rk(V^+(c,1))=\frac{r_0}{r}\rk(V(c,1))$. The idempotents $c_i$ are also primitive in $V^+$ and we can extend the system to a Jordan frame $c_1,\ldots,c_{r_0}$ in $V^+$. Since the group of automorphisms contains all possible permutations of $c_1,\ldots,c_{r_0}$ and leaves the trace invariant, we find with \eqref{eq:trdetofe} that
\begin{align*}
 \tr_V(c) &= k\,\tr_V(c_1) = \frac{k}{r_0}\tr_V(c_1+\ldots+c_{r_0}) = \frac{k}{r_0}\tr_V(e) = k\frac{r}{r_0}.
\end{align*}
On the other hand,
\begin{align*}
 \tr_{V(c,1)}(c) &= \rk(V(c,1)) = k\frac{r}{r_0}
\end{align*}
and hence $\lambda=1$. It remains to show \eqref{eq:TrOnSubspace} for $x\in V^-(c,1)$. In this case $\tr_V(x)=\tr_V(\alpha x)=-\tr_V(x)$ and therefore $\tr_V(x)=0$. The same argument works for $\tr_{V(c,1)}(x)$ since $\alpha c=c$ and hence $\alpha$ restricts to an automorphism of $V(c,1)$, leaving $\tr_{V(c,1)}$ invariant. This finishes the proof.
\end{proof}

The next statements are needed to calculate the action of the Bessel operator in Section \ref{sec:BesselOp} and the Casimir operator in Section \ref{sec:CasimirAction}.

\begin{lemma}\label{lem:SquareSums}
Let $(e_\alpha)_\alpha$ be an orthonormal basis of $V$ with respect to $(-|-)$ and $(\overline{e}_\alpha)_\alpha$ the dual basis with respect to the trace form $\tau(-,-)$, i.e. $\overline{e}_\alpha=\alpha(e_\alpha)$. Then
\begin{align*}
 \sum_\alpha{e_\alpha^2} &= \frac{2n_0-n}{r}e &\mbox{and} && \sum_\alpha{e_\alpha\cdot\overline{e}_\alpha} &= \frac{n}{r}e.
\end{align*}
\end{lemma}

\begin{proof}
It is easily seen that the elements $\sum_\alpha{e_\alpha^2}$ and $\sum_\alpha{e_\alpha\cdot\overline{e}_\alpha}$ are independent of the choice of the orthonormal basis. Since the Peirce decomposition $V=\oplus_{i\leq j}{V_{ij}}$ is orthogonal with respect to $(-|-)$, we can choose an orthonormal basis $(e_\alpha)_\alpha$ such that each $e_\alpha$ is contained in one of the $V_{ij}$, $i\leq j$. Furthermore, since the Cartan involution $\alpha$ leaves each $V_{ij}$, $i\leq j$, invariant, we can even choose the $e_\alpha$ to be either in $V_{ij}^+:=V_{ij}\cap V^+$ or in $V_{ij}^-:=V_{ij}\cap V^-$.
\begin{enumerate}
 \item[\textup{(a)}] Let $e_\alpha\in V_{ii}^+$. Then $e_\alpha=\lambda c_i$ with $\lambda=\|c_i\|^{-1}=\left(\frac{r_0}{r}\right)^{\frac{1}{2}}$ by Lemma \ref{lem:TrOnSubspace}~(2). Hence
 \begin{align*}
  e_\alpha^2 &= \frac{r_0}{r}c_i & \mbox{and } & & \dim V_{ii}^+ &= 1.
 \end{align*}
 \item[\textup{(b)}] Let $e_\alpha\in V_{ii}^-$. Then $e_\alpha^2\in V_{ii}^+=\RR c_i$ and therefore $e_\alpha^2=\lambda c_i$. Since
 \begin{align*}
  1 &= \|e_\alpha\|^2 = (e_\alpha|e_\alpha) = -\tau(e_\alpha,e_\alpha)\\
  &= -\tau(e_\alpha^2,e) = -\lambda\tau(c_i,e) = -\lambda\|c_i\|^2 = -\lambda\frac{r}{r_0},
 \end{align*}
 we obtain
 \begin{align*}
  e_\alpha^2 &= -\frac{r_0}{r}c_i & \mbox{and } & & \dim V_{ii}^- &= e.
 \end{align*}
 \item[\textup{(c)}] Let $e_\alpha\in V_{ij}^+$, $i<j$. Then $e_\alpha^2\in V_{ii}^++V_{jj}^+$ and hence $e_\alpha^2=\lambda c_i+\mu c_j$. Similar to the calculation in (b) one obtains $\lambda=\mu=\sqrt{\frac{r_0}{2r}}$ and hence
 \begin{align*}
  e_\alpha^2 &= \frac{r_0}{2r}(c_i+c_j) & \mbox{and } & & \dim V_{ij}^+ &= d_0.
 \end{align*}
 \item[\textup{(d)}] Let $e_\alpha\in V_{ij}^-$, $i<j$. Applying the same arguments as in (b) and (c) yields
 \begin{align*}
  e_\alpha^2 &= -\frac{r_0}{2r}(c_i+c_j) & \mbox{and } & & \dim V_{ij}^- &= d-d_0.
 \end{align*}
\end{enumerate}
Putting everything together gives
\begin{align*}
 \sum_\alpha{e_\alpha^2} &= \sum_{i=1}^{r_0}{\sum_{e_\alpha\in V_{ii}}{e_\alpha^2}} + \sum_{1\leq i<j\leq r_0}{\sum_{e_\alpha\in V_{ij}}{e_\alpha^2}}\\
 &= \frac{r_0}{r}\sum_{i=1}^{r_0}{(1-e)c_i} + \frac{r_0}{r}\sum_{1\leq i<j\leq r_0}{\frac{1}{2}(d_0-(d-d_0))(c_i+c_j)}\\
 &= \frac{r_0}{r}\left(2\left(1+(r_0-1)\frac{d_0}{2}\right)-\left(e+1+(r_0-1)\frac{d}{2}\right)\right)\sum_{i=1}^{r_0}{c_i}\\
 &= \frac{2n_0-n}{r}e.
\end{align*}
The second formula follows from the first as follows. Choose an orthonormal basis $(e_\alpha)_\alpha$ of $V$ with $e_\alpha\in V^+\cup V^-$. By Lemma \ref{lem:TrOnSubspace}~(1) the elements
$\sqrt{r/r_0}\,e_\alpha$ with $e_\alpha\in V^+$ form an orthonormal basis of $V^+$. Further, $\overline{e}_\alpha=\alpha e_\alpha$ and we calculate, using the first formula:
\begin{align*}
 \sum_\alpha{e_\alpha\cdot\overline{e}_\alpha} &= \sum_{\begin{subarray}{c}\alpha\\e_\alpha\in V^+\end{subarray}}{e_\alpha^2}-\sum_{\begin{subarray}{c}\alpha\\e_\alpha\in V^-\end{subarray}}{e_\alpha^2} = 2\sum_{\begin{subarray}{c}\alpha\\e_\alpha\in V^+\end{subarray}}{e_\alpha^2}-\sum_{\begin{subarray}{c}\alpha\\e_\alpha\in V\end{subarray}}{e_\alpha^2}\\
 &= \left(2\frac{r_0}{r}\cdot\frac{n_0}{r_0} - \frac{2n_0-n}{r}\right)e = \frac{n}{r}e.\qedhere
\end{align*}
\end{proof}

\begin{lemma}\label{lem:QRepC}
Let $c\in V^+$ be a primitive idempotent and $(e_\alpha)_\alpha$ an orthonormal basis of $V$ with respect to $(-|-)$. Then
\begin{align*}
 \sum_\alpha{P(e_\alpha)c} &= \frac{r_0}{r}\left(\frac{d}{2}-d_0-e+1\right)c + \frac{r_0}{r}\left(d_0-\frac{d}{2}\right)e.
\end{align*}
\end{lemma}

\begin{proof}
Again it is easily seen that the expression $\sum_\alpha{P(e_\alpha)c}$ is independent of the chosen orthonormal basis. Since the Peirce decomposition $V=V(c,1)\oplus V(c,\frac{1}{2})\oplus V(c,0)$ is orthogonal, we can choose an orthonormal basis $(e_\alpha)_\alpha$ such that $e_\alpha\in V(c,1)\cup V(c,\frac{1}{2})\cup V(c,0)$ for every $\alpha$. Then, by Lemma \ref{lem:TrOnSubspace}~(2), the $e_\alpha$ in $V(c,1)$ form an orthonormal basis of $V(c,1)$ and those $e_\alpha$ in $V(c,0)=V(e-c,1)$ form an orthonormal basis of $V(c,0)$. Now let us determine the action of $P(e_\alpha)$ on $c$ for the following three cases:
\begin{enumerate}
 \item[\textup{(a)}] $e_\alpha\in V(c,1)$. Since $V(c,1)$ is a subalgebra, also $e_\alpha^2\in V(c,1)$. Hence
 \begin{align*}
  P(e_\alpha)c &= 2e_\alpha(e_\alpha c)-e_\alpha^2c = 2e_\alpha^2-e_\alpha^2=e_\alpha^2.
 \end{align*}
 \item[\textup{(b)}] $e_\alpha\in V(c,\frac{1}{2})$. In this case
 \begin{align*}
  P(e_\alpha)c &= 2e_\alpha(e_\alpha c)-e_\alpha^2c = e_\alpha^2 - e_\alpha^2c = (e-c)e_\alpha^2.
 \end{align*}
 \item[\textup{(c)}] $e_\alpha\in V(c,0)$. Also $V(c,0)$ is a subalgebra and hence $e_\alpha^2\in V(c,0)$. Then clearly
 \begin{align*}
  P(e_\alpha)c &= 2e_\alpha(e_\alpha c)-e_\alpha^2c = 0.
 \end{align*}
\end{enumerate}
Altogether we obtain
\begin{align*}
 \sum_\alpha{P(e_\alpha)c} &= \sum_{\begin{subarray}{c}\alpha\\e_\alpha\in V(c,1)\end{subarray}}{e_\alpha^2} + (e-c)\sum_{\begin{subarray}{c}\alpha\\e_\alpha\in V(c,\frac{1}{2})\end{subarray}}{e_\alpha^2}\\
 &= (e-c)\sum_\alpha{e_\alpha^2} + \sum_{\begin{subarray}{c}\alpha\\e_\alpha\in V(c,1)\end{subarray}}{e_\alpha^2} - \sum_{\begin{subarray}{c}\alpha\\e_\alpha\in V(c,0)\end{subarray}}{e_\alpha^2}.
\end{align*}
Put $n_c:=\dim\,V(c,1)$, $n_{c,0}:=\dim\,V^+(c,1)$, $r_c:=\rk\,V(c,1)$ and similarly for $e-c$. We have
\begin{align*}
 n_c &= e+1, & n_{e-c} &= (r_0-1)(e+1)+(r_0-1)(r_0-2)\frac{d}{2},\\
 n_{c,0} &= 1, & n_{e-c,0} &= (r_0-1)+(r_0-1)(r_0-2)\frac{d_0}{2},\\
 r_c &= \frac{r}{r_0}, & r_{e-c} &= (r_0-1)\frac{r}{r_0},
\end{align*}
and hence, using Lemma \ref{lem:SquareSums}:
\begin{align*}
 \sum_\alpha{P(e_\alpha)c} &= \frac{2n_0-n}{r}(e-c) + \frac{2n_{c,0}-n_c}{r_c}c - \frac{2n_{e-c,0}-n_{e-c}}{r_{e-c}}(e-c)\\
 &= \frac{r_0}{r}\left(\frac{d}{2}-d_0-e+1\right)c + \frac{r_0}{r}\left(d_0-\frac{d}{2}\right)e.\qedhere
\end{align*}
\end{proof}

\begin{lemma}
Let $(e_j)_j\subseteq V$ be an orthonormal basis of $V$ with respect to the inner product $(-|-)$. Then for $x\in V$:
\begin{align}
 \sum_{j=1}^n{\tau(P(\alpha x,\alpha e_j)x,e_j)} &= \frac{n}{r}(x|x).\label{eq:SumQuadRep}
\end{align}
\end{lemma}

\begin{proof}
Using Lemma \ref{lem:Trtr} we obtain
\begin{align*}
 \sum_{j=1}^n{\tau(P(\alpha x,\alpha e_j)x,e_j)} &= \sum_{j=1}^n{\tau((\alpha x\Box x)(\alpha e_j),e_j)}\\
 &= \sum_{j=1}^n{((x\Box\alpha x)e_j|e_j)} = \Tr(x\Box\alpha x)\\
 &= \Tr(L(x\cdot\alpha x)+[L(x),L(\alpha x)]) = \Tr(L(x\cdot\alpha x))\\
 &= \frac{n}{r}\tr(x\cdot\alpha x) = \frac{n}{r}(x|x).\qedhere
\end{align*}
\end{proof}

\section{The constants $\mu$ and $\nu$}\label{sec:MuNu}

For every Jordan algebra $V$ we introduce another two constants $\mu$ and $\nu$ by
\begin{align}
 \mu &= \mu(V) := \frac{n}{r_0}+\left|d_0-\frac{d}{2}\right|-2, & \nu &= \nu(V) := \frac{d}{2}-\left|d_0-\frac{d}{2}\right|-e-1.\label{eq:DefMuNu}\index{notation}{mu@$\mu$}\index{notation}{muV@$\mu(V)$}\index{notation}{nu@$\nu$}\index{notation}{nuV@$\nu(V)$}
\end{align}
These constants will appear as parameters of certain special functions in the minimal representation. Using Proposition \ref{prop:ClassificationEuclSph} we can calculate $\mu$ and $\nu$ explicitly:
\begin{align*}
 (\mu,\nu) &= \begin{cases}(\frac{rd}{2}-1,-1) & \mbox{if $V$ is euclidean,}\\(\frac{n}{r_0}-2,\frac{d}{2}-e-1) & \mbox{if $V$ is non-euclidean of rank $r\geq3$}\\(\max(p,q)-2,\min(p,q)-2) & \mbox{if $V\cong\RR^{p,q}$, $p,q\geq2$.}\end{cases}
\end{align*}

Let us collect some basic inequalities for $\mu$ and $\nu$ here.

\begin{lemma}\label{lem:MuNuProperties}
If $V$ is a simple Jordan algebra of split rank $r_0\geq2$, then
\begin{enumerate}
\item[\textup{(1)}] $\mu+\nu\geq-1$,
\item[\textup{(2)}] $\mu-\nu\geq0$,
\item[\textup{(3)}] $\mu\geq-\frac{1}{2}$.
\end{enumerate}
\end{lemma}

\begin{proof}
First note that $r_0\geq2$ by assumption and $d\geq1$. (If $d=0$, then $V$ is be the direct sum of the ideals $V(c_i,1)$, $1\leq i\leq r_0$, and hence not simple.) Together with \eqref{eq:noverr0} we obtain:
\begin{align*}
 \mu+\nu &= \frac{n}{r_0}+\frac{d}{2}-e-3\\
 &= \frac{r_0d}{2}-2 \geq -1,\\
 \mu-\nu &= \frac{n}{r_0}-\frac{d}{2}+2\left|d_0-\frac{d}{2}\right|+e-1\\
 &= (r_0-2)\frac{d}{2}+2\left|d_0-\frac{d}{2}\right|+2e \geq 0.
\end{align*}
Finally (3) is a direct consequence of (1) and (2) which finishes the proof.
\end{proof}

Denote by $\Xi$ the set of all possible values of $(\mu,\nu)$, excluding the cases for which it will turn out that there is no minimal representation:
\begin{multline*}
 \Xi := \{(\mu(V),\nu(V)):\mbox{$V$ is a simple real Jordan algebra}\\
 \mbox{of split rank $r_0\geq2$, $V^+$ is simple and if $r=2$, then $n$ is even}\}.\index{notation}{Xi@$\Xi$}
\end{multline*}
The classification of all simple real Jordan algebras (see Table \ref{tb:Parameters}) allows us to compute the set $\Xi$ explicitly:
\begin{multline*}
 \Xi = \textstyle\{(\mu,-1):\mu\in\frac{1}{2}\NN_0\}\cup\{(\mu,0):\mu\in\NN_0\}\cup\{(\mu,\nu):\mu,\nu\in\NN_0,\mu+\nu\in2\ZZ\}.
\end{multline*}
Note that $\nu$ is always an integer since
\begin{align*}
 \nu &= \min(d,2d_0)-d_0-e-1,
\end{align*}
whereas $\mu$ is in general only a half-integer (e.g. for $V=\Sym(n,\RR)$).

\section{The structure group and its Lie algebra}

We define the structure group of a Jordan algebra. Further, we give a root space decomposition of its Lie algebra which is adapted to the structure of the Jordan algebra.

\subsection{The structure group}

Denote  by $g^\#$\index{notation}{ghash@$g^\#$} the adjoint of $g\in\GL(V)$ with respect to the trace form $\tau$.

\subsubsection{Definition of the structure group}

The \textit{structure group}\index{subject}{structure group} $\Str(V)$\index{notation}{StrV@$\Str(V)$} of $V$ is the group of invertible linear transformations $g\in\GL(V)$ such that for every invertible $x\in V$ the element $gx$ is also invertible and
\begin{align*}
 (gx)^{-1} &= g^{-\#}x^{-1}.
\end{align*}
$\Str(V)$ is a real reductive group (see \cite[Corollary 8.8]{Loo77}). By \cite[Proposition VIII.2.5]{FK94} an equivalent description of the structure group is given in terms of the quadratic representation: $g\in\GL(V)$ is in the structure group if and only if
\begin{align}
 P(gx) &= gP(x)g^\#, & \forall\, x\in V.\label{eq:QRepStrGrp}
\end{align}
There is yet another equivalent description of $\Str(V)$ in terms of the Jordan determinant. Namely, it is easy to see (cf. \cite[Chapter VIII, Exercise 5]{FK94}) that $g\in\GL(V)$ belongs to the structure group if and only if there exists a constant $\chi(g)\in\KK^\times$ with
\begin{align}
 \det(gx) &= \chi(g)\det(x) & \forall\,x\in V.\label{eq:DetEquiv}\index{notation}{chig@$\chi(g)$}
\end{align}
The map $\chi:\Str(V)\rightarrow\RR^\times$ is given by $\chi(g)=\Delta(ge)$ and defines a character of $\Str(V)$. Using this equivariance property we can now calculate derivatives of the Jordan determinant $\Delta(x)$. For the proof we denote by $\ell$ the left-regular representation of the structure group $\Str(V)$ on functions $f$ which are defined on $V$:
\begin{align}
 (\ell(g)f)(x) &:= f(g^{-1}x).\label{eq:DefEll}\index{notation}{lg@$\ell(g)$}
\end{align}

\begin{lemma}\label{lem:DerivDelta}
The derivative of $\Delta$ in a point $x\in V$ in direction $u\in V$ is given by
\begin{align}
 D_u\Delta(x) &= \Delta(x)\tau(x^{-1},u).\label{eq:DDelta}
\end{align}
\end{lemma}

\begin{proof}
By \cite[Section II.2]{FK94} the generic minimal polynomial of $x$ is given by $f_x(\lambda)=\Delta(\lambda e-x)$. Hence
\begin{align*}
 D_u\Delta(e) &= \left.\frac{\td}{\td t}\right|_{t=0}\Delta(e+tu) = \left.\frac{\td}{\td t}\right|_{t=0}(1+\tr(u)t+\textup{higher order terms}) = \tr(u).
\end{align*}
Now let $x=ge$ with $g\in\Str(V)$. Then $\ell(g^{-1})\Delta=\Delta(ge)\Delta=\Delta(x)\Delta$ and hence, by the chain rule,
\begin{align*}
 D_u\Delta(x) &= D_{g^{-1}u}(\ell(g^{-1})\Delta)(e) = \Delta(x)\tr(g^{-1}u)\\
 &= \Delta(x)\tau(g^{-\#}e,u) = \Delta(x)\tau(x^{-1},u).
\end{align*}
Now, the orbit of $\Str(V)$ containing $e$ is open and both sides of \eqref{eq:DDelta} are polynomials in $x$. Therefore \eqref{eq:DDelta} must hold for every $x\in V$.
\end{proof}

\subsubsection{The automorphism group}

The group $\Aut(V)$\index{notation}{AutV@$\Aut(V)$} of \textit{automorphisms}\index{subject}{automorphism group} of $V$ is a subgroup of $\Str(V)$. In fact, it is exactly the subgroup of $\Str(V)$ stabilizing the identity element $e$ (see \cite[Proposition VIII.2.4~(ii)]{FK94}). Moreover, $(\Str(V),\Aut(V))$ is a symmetric pair: The map
\begin{align*}
\sigma:\Str(V)\rightarrow\Str(V),\,g\mapsto g^{-\#}:=(g^{-1})^\#=(g^\#)^{-1},\index{notation}{sigma@$\sigma$}
\end{align*}
defines an involution of the structure group and with \cite[Proposition VIII.2.6]{FK94} it is easy to see that
\begin{align}
 \Str(V)_0^\sigma &\subseteq \Aut(V) \subseteq \Str(V)^\sigma.\label{eq:StrAutSymmPair}
\end{align}
If $V$ is euclidean, then $\sigma$ is a Cartan involution and hence, $\Aut(V)$ is compact. However, this is not true in general. Corresponding to the involution $\sigma$, the Lie algebra $\mathfrak{str}(V)=\Lie(\Str(V))$\index{notation}{strV@$\str(V)$} splits into the direct sum of the $\pm1$-eigenspaces of $\sigma$ (see \cite[Proposition VIII.2.6]{FK94}):
\begin{equation}
 \str(V) = \frakh + \frakq,\label{eq:strderPlusL}
\end{equation}
where
\begin{align}
 \frakh :={}& \{X\in\str(V):\sigma(X)=X\}\notag\index{notation}{h@$\frakh$}\\
 ={}& \mathfrak{der}(V) := \{D\in\End(V):D(x\cdot y)=Dx\cdot y+x\cdot Dy\ \forall\, x,y\in V\},\label{eq:LieAlgH}\index{notation}{derV@$\der(V)$}\\
 \frakq :={}& \{X\in\str(V):\sigma(X)=-X\}\notag\index{notation}{q@$\frakq$}\\
 ={}& L(V) = \{L(x):x\in V\}.\label{eq:LieAlgQ}\index{notation}{LV@$L(V)$}
\end{align}
The Lie algebra $\mathfrak{der}(V)$ of \textit{derivations}\index{subject}{derivation} is the Lie algebra of $\Aut(V)$. The defining property for a derivation can be equivalently written as $[D,L(x)]=L(Dx)$ for all $x\in V$. Hence, in the decomposition \eqref{eq:strderPlusL} the Lie bracket is given by
\begin{align}
 [L(x)+D,L(x')+D'] &= L(Dx'-D'x) + \left([L(x),L(x')]+[D,D']\right)\label{eq:LieStrBracket}
\end{align}
for $x,x'\in V$ and $D,D'\in\der(V)$. Note that for $x,y\in V$ the commutator $[L(x),L(y)]$ is a derivation. Finite sums of derivations of this type are called inner derivations. Since $V$ is semisimple, every derivation is inner (see \cite[Theorem 2]{Jac49}). A direct consequence of this fact is that the trace of every derivation vanishes, since the trace of every commutator vanishes:
\begin{align}
 \Tr(D) &= 0 & \forall\, D\in\mathfrak{der}(V).\label{eq:TrOfDerivation}
\end{align}
Using Lemma \ref{lem:Trtr} we also obtain
\begin{align}
 \tr(Dx) &= \frac{r}{n}\Tr(L(Dx)) = \frac{r}{n}\Tr([D,L(x)]) = 0 & \forall\, D\in\mathfrak{der}(V),x\in V.\label{eq:JordanTrOfDerivation}
\end{align}

In Section \ref{sec:PeirceJordanFrame} we remarked that the automorphism group acts transitively on the set of Jordan frames in $V^+$. We can now give a more precise version of this statement.

\begin{lemma}\label{lem:AutTransitiveJordanFrames}
Let $V$ be a simple Jordan algebra such that $V^+$ is also simple. Then for any two Jordan frames $c_1,\ldots,c_{r_0}$ and $d_1,\ldots,d_{r_0}$ in $V^+$ there exists a derivation $D\in\mathfrak{der}(V)$ with $\alpha D=D\alpha$ such that
\begin{align*}
 e^Dc_i &= d_i & \forall\, i=1,\ldots,r_0.
\end{align*}
\end{lemma}

\begin{proof}
By \cite[Satz 8.3]{Hel69}, applied to $V^+$, there exists an element $h\in\Aut(V^+)_0$ such that $hc_i=d_i$ for all $i=1,\ldots,r_0$. Since $\Aut(V^+)_0$ is compact, it is the image of its Lie algebra under the exponential map. Therefore, there exists a derivation $D\in\mathfrak{der}(V^+)$ such that $h=e^D$. All derivations in $\mathfrak{der}(V^+)$ are inner and hence $D$ extends to $V$ with the property that $\alpha D=D\alpha$ which shows the claim.
\end{proof}

\subsubsection{A Cartan involution}

The involution $\sigma$ is in general not a Cartan involution of $\str(V)$ (only if $V$ is euclidean). To obtain a Cartan involution we have to conjugate with the Cartan involution $\alpha$ of $V$. In fact, the involution
\begin{align*}
 \theta:\Str(V)\rightarrow\Str(V),\,g\mapsto g^{-*}=\alpha g^{-\#}\alpha,\index{notation}{theta@$\theta$}
\end{align*}
where $^*$\index{notation}{gstar@$g^*$} denotes the adjoint with respect to the inner product $(-|-)$, is a Cartan involution of $\Str(V)$. The fixed point set $\Str(V)^\theta$ of $\theta$ is therefore a maximal compact subgroup of $\Str(V)$. By definition $\Str(V)^\theta$ is also the intersection of $\Str(V)$ with the orthogonal group of the inner product $(-|-)$. Note that if $V$ is euclidean, then $\alpha=\1$ and hence $\theta=\sigma$. As previously remarked, in this case the automorphism group $\Aut(V)$ is compact.

Returning to the general case, it is easy to see that $\sigma$ commutes with the Cartan involution $\theta$. Hence, $\theta(\frakh)=\frakh$ and also $\theta(\Aut(V))=\Aut(V)$. Then by \cite[Corollary 1.1.5.4]{War72} $\Aut(V)$ is a real reductive group.

The Cartan decomposition of the Lie algebra $\str(V)$ with respect to $\theta$ is given by
\begin{align*}
 \str(V) &= \frakk_\frakl+\frakp_\frakl,
\end{align*}
where
\begin{align}
 \frakk_\frakl :={}& \{X\in\str(V):\theta(X)=X\}\notag\index{notation}{kl@$\frakk_\frakl$}\\
 ={}& \{L(x)+D:x\in V^-,\alpha D=D\alpha\},\label{eq:LieAlgKL}\\
 \frakp_\frakl ={}& \{X\in\str(V):\theta(X)=-X\}\notag\index{notation}{pl@$\frakp_\frakl$}\\
 ={}& \{L(x)+D:x\in V^+,\alpha D=-D\alpha\}.\label{eq:LieAlgPL}
\end{align}

Now, let $L$\index{notation}{L@$L$} be the subgroup of $\GL(V)$ generated by the identity component $\Str(V)_0$ of the structure group and the Cartan involution $\alpha$. Clearly
\begin{align*}
 \Str(V)_0 &\subseteq L \subseteq \Str(V).
\end{align*}
$L$ has at most two connected components, namely $\Str(V)_0$ and $\alpha\Str(V)_0$. Denote by $\frakl=\str(V)$\index{notation}{l@$\frakl$} its Lie algebra. The involutions $\theta$ and $\sigma$ leave $L$ invariant since $\theta(\alpha)=\sigma(\alpha)=\alpha$. Then $K_L:=L^\theta$\index{notation}{KL@$K_L$} is a maximal compact subgroup of $L$.

\begin{example}\label{ex:StrGrp}
\begin{enumerate}
\item[\textup{(1)}] The structure algebra of $V=\Sym(n,\RR)$ is easily seen to be $\str(V)=\gl(n,\RR)=\sl(n,\RR)\oplus\RR$, acting by
\begin{align*}
 X\cdot a &= Xa+aX^t & \mbox{for }X\in\gl(n,\RR),a\in V.
\end{align*}
Integrating this action to the universal cover and factoring out the elements that act trivially shows that the identity component of $\Str(V)$ is isomorphic to $\RR_+\SL(n,\RR)$, acting by
\begin{align*}
 g\cdot x &= gxg^t & \mbox{for }g\in\RR_+\SL(n,\RR),x\in V.
\end{align*}
Since $V$ is euclidean, $L=\Str(V)_0$. The maximal compact subgroup of $L$ is $K_L=\SO(n)$, acting by conjugation.
\item[\textup{(2)}] For $V=\RR^{p,q}$ the characterization \eqref{eq:DetEquiv} can be used to show that the structure group is given by
\begin{align*}
 \Str(V) &= \begin{cases}\RR_+\upO(p,q) & \mbox{if }p\neq q,\\\RR_+\upO(p,q)\cup g_0\cdot\RR_+\upO(p,q) & \mbox{if }p=q,\end{cases}
\end{align*}
where
\begin{align*}
 g_0 &= \left(\begin{array}{cc}0&\1_p\\\1_p&0\end{array}\right).
\end{align*}
Then clearly $\Str(V)_0=\RR_+\SO(p,q)_0$. By \eqref{eq:VpqCartanInv} the Cartan involution $\alpha$ is contained in $\Str(V)_0$ if and only if $p$ is odd. In this case $L=\Str(V)_0$. If $p$ is even, $L=\Str(V)_0\cup\alpha\Str(V)_0$. The maximal compact subgroup of $L_0$ is $(K_L)_0=\SO(p)\times\SO(q)$.
\end{enumerate}
\end{example}

\subsection{Root space decomposition}

As a preparation for the root space decomposition we prove a simple lemma concerning derivations of the form $[L(c),L(u)]$ with $c$ an idempotent.

\begin{lemma}\label{lem:BracketZero}
\begin{enumerate}
\item[\textup{(1)}] Let $c\in V$ be an idempotent. Then for any $u\in V(c,1)+V(c,0)$:
\begin{align*}
 [L(c),L(u)] &= 0.
\end{align*}
\item[\textup{(2)}] Let $c_1,c_2\in V$ be orthogonal idempotents. Then for any $u\in V(c_1,\frac{1}{2})\cap V(c_2,\frac{1}{2})$:
\begin{align*}
 [L(c_1),L(u)] &= -[L(c_2),L(u)].
\end{align*}
\end{enumerate}
\end{lemma}

\begin{proof}
(2) follows directly from (1) with $c:=c_1+c_2$. For (1) we let $u_1\in V(c,1)$ and $u_2\in V(c,0)$. Using \cite[Proposition II.1.1~(i)]{FK94} we find that
\begin{align*}
 [L(u_1),L(c)] &= [L(u_1),L(c^2)] = -2[L(c),L(u_1c)]\\
 &= -2[L(c),L(u_1)] = 2[L(u_1),L(c)]
\end{align*}
and
\begin{align*}
 [L(u_2),L(c)] &= [L(u_2),L(c^2)] = -2[L(c),L(u_2c)] = 0.
\end{align*}
This finishes the proof.
\end{proof}

Now, the subalgebra
\begin{align}
 \fraka &:= \sum_{i=1}^{r_0}{\RR L(c_i)} \subseteq \frakl\label{eq:DefLiea}\index{notation}{a@$\fraka$}
\end{align}
is abelian by \cite[Proposition II.1.1~(1)]{FK94}. We even have the following lemma:

\begin{lemma}\label{lem:aMaxAbelian}
$\fraka$ is maximal abelian in $L(V^+)$.
\end{lemma}

\begin{proof}
By the Peirce decomposition of $V^+$ it suffices to show that there exists no non-zero element $x=\sum_{i<j}{x_{ij}}$, $x_{ij}\in V^+_{ij}$, with $[L(x),L(c_i)]=0$ for all $i=1,\ldots,r_0$. If $x$ is such an element, then for $i<j$ we have
\begin{align*}
 0 &= [L(x),L(c_i)]c_j = c_i(c_jx) = \frac{1}{4}x_{ij}
\end{align*}
and hence $x=0$.
\end{proof}

A basis of the dual space $\fraka^*$ of $\fraka$ is given by the functionals $\varepsilon_1,\ldots,\varepsilon_{r_0}$, where for $i=1,\ldots,r_0$:
\begin{align}
 \varepsilon_i\left(\sum_{j=1}^{r_0}{t_jL(c_j)}\right) &:= t_i.\label{eq:DefEpsilonj}\index{notation}{epsiloni@$\varepsilon_i$}
\end{align}
Denote by $\Sigma(\frakl,\fraka)\subseteq\fraka^*$ the set of non-zero weights of $\frakl$ with respect to $\fraka$.

\begin{proposition}\label{prop:StrRootDecomp}
The set $\Sigma(\frakl,\fraka)$ is a root system of type $A_{r_0-1}$ and given by
\begin{align}
 \Sigma(\frakl,\fraka) &= \left\{\frac{\varepsilon_i-\varepsilon_j}{2}:i\neq j\right\}.\label{eq:LRoots}
\end{align}
The corresponding root spaces amount to
\begin{align}
 \frakl_{ij} &:= \frakl_{\frac{\varepsilon_i-\varepsilon_j}{2}} = \textstyle\{c_i\Box x=\frac{1}{2}L(x)+[L(c_i),L(x)]:x\in V_{ij}\}\label{eq:LRootSpaces}\index{notation}{lij@$\frakl_{ij}$}
\end{align}
for $i\neq j$ and
\begin{align}
 \frakl_0 &= \left\{L(x)+D:x\in\bigoplus_{i=1}^{r_0}{V_{ii}},Dc_i=0\,\forall\, i=1,\ldots,r_0\right\}.\label{eq:LRootSpace0}\index{notation}{lnull@$\frakl_0$}
\end{align}
\end{proposition}

\begin{proof}
Let $\gamma=\sum_{i=1}^{r_0}{\gamma_i\varepsilon_i}\in\Sigma(\frakl,\fraka)$ and $0\neq X=L(x)+D\in\frakl_\gamma$. Since $\gamma\neq0$, there is $1\leq i\leq r_0$ such that $\gamma_i\neq0$. Hence, by \eqref{eq:LieStrBracket}:
\begin{align*}
 \gamma_i(L(x)+D) &= \gamma_iX = [L(c_i),X] = L(-Dc_i) + [L(c_i),L(x)].
\end{align*}
Therefore, $D=\gamma_i^{-1}[L(c_i),L(x)]$ and
\begin{align*}
 x &= -\gamma_i^{-1}Dx = -\gamma_i^{-2}[L(c_i),L(x)]c_i = -\gamma_i^{-2}(c_i(c_ix)-c_ix).
\end{align*}
Write $x=x_1+x_{\frac{1}{2}}+x_0$ in the Peirce decomposition relative to $c_i$, i.e. $x_\lambda\in V(c_i,\lambda)$, $\lambda=0,\frac{1}{2},1$. Then
\begin{align*}
 c_i(c_ix)-c_ix &= -\frac{1}{4}x_{\frac{1}{2}}
\end{align*}
and hence $x=(2\gamma_i)^{-2}x_{\frac{1}{2}}\in V(c_i,\frac{1}{2})$. This implies $x=x_{\frac{1}{2}}$ and $\gamma_i=\pm\frac{1}{2}$. Altogether we obtain
\begin{align*}
 X &= L(x)\pm2[L(c_i),L(x)], & \mbox{with }x\in V(c_i,\textstyle\frac{1}{2}).
\end{align*}
Now, if $\gamma_j=0$ for every $j\neq i$, then for every $j\neq i$ we have
\begin{align*}
 0 &= [L(c_i),L(x)]c_j=c_i(c_jx)=c_j(c_ix)=\frac{1}{2}c_jx.
\end{align*}
But this is only possible if $x=0$ since $x\in V(c_i,\frac{1}{2})=\sum_{j\neq i}{V_{ij}}$. But $x=0$ implies $X=0$ which contradicts our assumption. Therefore, there has to be $j\neq i$ such that $\gamma_j\neq0$. The same argument as above shows that $x\in V(c_j,\frac{1}{2})$ and hence $x\in V_{ij}$. Thus, by Lemma \ref{lem:BracketZero}~(2):
\begin{align*}
 X &= L(x)\pm2[L(c_i),L(x)] = L(x)\mp2[L(c_j),L(x)],
\end{align*}
and a direct computation shows that $\gamma=\pm\frac{\varepsilon_i-\varepsilon_j}{2}$. It remains to compute $\frakl_0$.
An element $X=L(x)+D\in\frakl$ is in $\frakl_0$ if and only if for every $i=1,\ldots,r_0$ we have
\begin{align*}
 0 &= [L(c_i),X] = L(-Dc_i)+[L(c_i),L(x)].
\end{align*}
This is equivalent to $Dc_i=0$ and $[L(c_i),L(x)]=0$ for all $i=1,\ldots,r_0$. Clearly $x\in\sum_{i=1}^{r_0}{V_{ii}}$ has this property by Lemma \ref{lem:BracketZero}~(1). Conversely, let $[L(c_i),L(x)]=0$ for every $i=1,\ldots,r_0$. Write $x=\sum_{1\leq i\leq j\leq r_0}{x_{ij}}$ in the Peirce decomposition, i.e. $x_{ij}\in V_{ij}$, $1\leq i\leq j\leq r_0$. Then by Lemma \ref{lem:BracketZero}~(1) we find that for all $i<j$:
\begin{align*}
 0 = [L(c_i),L(x)]c_j &= \frac{1}{4}x_{ij}
\end{align*}
and hence $x=\sum_{i=1}^{r_0}{x_{ii}}\in\sum_{i=1}^{r_0}{V_{ii}}$ which finishes the proof.
\end{proof}

We choose the positive system
\begin{align}
 \Sigma^+(\frakl,\fraka) &= \left\{\frac{\varepsilon_i-\varepsilon_j}{2}:1\leq i<j\leq r_0\right\}.\label{eq:PositiveLRoots}
\end{align}

Later we also need some information about the center of $\frakl$.

\begin{lemma}
For the center $Z(\frakl)$\index{notation}{Zl@$Z(\frakl)$} of $\frakl$ we have the following inclusions
\begin{align}
 \RR L(e) &\subseteq Z(\frakl) \subseteq \sum_{i=1}^{r_0}{L(V_{ii})}.\label{eq:CenterL}
\end{align}
\end{lemma}

\begin{proof}
The first inclusion of \eqref{eq:CenterL} is clear since $L(e)=\id_V$. For the second inclusion let $L(x)+D\in Z(\frakl)$. Then by \eqref{eq:LieStrBracket} for every $x'\in V$:
\begin{align*}
 0 &= [L(x)+D,L(x')] = L(Dx') + [L(x),L(x')]
\end{align*}
Hence, $D=0$ and $[L(x),L(x')]=0$ for every $x'\in V$. Write $x=\sum_{i\leq j}{x_{ij}}$ in its Peirce decomposition, i.e. $x_{ij}\in V_{ij}$. Then for $i<j$:
\begin{align*}
 0 &= [L(c_i),L(x)]c_j = \frac{x_{ij}}{4}.
\end{align*}
Therefore, $x\in\sum_{i=1}^{r_0}{V_{ii}}$ and the proof is complete.
\end{proof}

\begin{remark}
In general one does not have $Z(\frakl)=\RR L(e)$. For instance, let $V$ is a simple real euclidean Jordan algebra and view its complexification $V_\CC$ also as real Jordan algebra. Then $V_\CC$ is also simple and the center $Z(\str(V_\CC))$ contains at least $\CC L(e)$.
\end{remark}

\section{Orbits of the structure group and equivariant measures}\label{sec:OrbitsMeasures}

There are only finitely many orbits under the action of $\Str(V)_0$ on $V$. An explicit description of these orbits can be found in \cite{Kan98}. We will merely be interested in the open orbit of $\Str(V)_0$ containing the unit element and the orbits which are contained in its boundary. On these orbits we construct equivariant measures. This yields $L^2$-spaces on which we later construct unitary representations.

\subsection{The open cone $\Omega$}

Let $\Omega=\Str(V)_0\cdot e$\index{notation}{Omega@$\Omega$} be the open orbit of $\Str(V)_0$ containing the identity element of the Jordan algebra. $\Omega$ is an open cone in $V$. Since $\alpha e=e$, we also have $\Omega=L\cdot e$. If we denote by $H$\index{notation}{H@$H$} the stabilizer subgroup of $e$ in $L$, then the cone $\Omega\cong L/H$ is a reductive symmetric space. In fact, $H=L\cap\Aut(V)$ and therefore, by \eqref{eq:StrAutSymmPair} we have $L^\sigma_0\subseteq H\subseteq L^\sigma$ which shows that $(L,H)$ is a symmetric pair.

\begin{example}
\begin{enumerate}
\item[\textup{(1)}] For $V=\Sym(n,\RR)$ the cone $\Omega$ is the convex cone of symmetric positive definite matrices:
\begin{align*}
 \Omega &= \{x\in\Sym(n,\RR):\mbox{all eigenvalues of $x$ are positive}\}.
\end{align*}
\item[\textup{(2)}] For $V=\RR^{p,q}$ we have to distinguish between two cases. If $p=1$, $q\geq2$, then $\Omega$ is the convex cone given by
\begin{align*}
 \Omega &= \{x\in\RR^{1,q}:x_1>0,x_1^2-x_2^2-\ldots-x_n^2>0\}.
\end{align*}
For $p,q\geq2$ we have
\begin{align*}
 \Omega &= \{x\in\RR^{p,q}:x_1^2+\ldots+x_p^2-x_{p+1}^2-\ldots-x_n^2>0\},
\end{align*}
which is not convex.
\end{enumerate}
\end{example}

In the case where $V$ is euclidean, $H=K_L$ and hence $\Omega$ is a Riemannian symmetric space. In this case a polar decomposition for $\Omega$ is given in \cite[Chapter VI, Section 2]{FK94}.

To derive a polar decomposition also in the general case, we follow \cite[Chapter 3]{vdB05}. Consider the involutions $\sigma$ and $\theta$ on the Lie algebra level. In the decomposition \eqref{eq:strderPlusL} they are given by
\begin{align*}
 \sigma(L(x)+D) &= -(L(x)+D)^\# = -L(x)+D,\\
 \theta(L(x)+D) &= -(L(x)+D)^* = -L(\alpha x)+\alpha D\alpha
\end{align*}
for $x\in V$ and $D\in\mathfrak{der}(V)$. The decomposition of $\frakl$ into $\pm1$-eigenspaces of $\sigma$ and $\theta$ is written as
\begin{align*}
 \frakl &= \frakk_\frakl+\frakp_\frakl = \frakh+\frakq
\end{align*}
with the notation of \eqref{eq:LieAlgH}, \eqref{eq:LieAlgQ}, \eqref{eq:LieAlgKL} and \eqref{eq:LieAlgPL}. Since $\sigma$ and $\theta$ commute, we obtain the decomposition in simultaneous eigenspaces
\begin{align*}
 \frakl &= (\frakk_\frakl\cap\frakh)+(\frakk_\frakl\cap\frakq)+(\frakp_\frakl\cap\frakh)+(\frakp_\frakl\cap\frakq).
\end{align*}
Put
\begin{align}
 \frakl^+ &:= (\frakk_\frakl\cap\frakh)+(\frakp_\frakl\cap\frakq) = \{L(x)+D:x\in V^+,\alpha D=D\alpha\}\cong\str(V^+),\label{eq:LieAlgL+}\index{notation}{l1@$\frakl^+$}\\
 \frakl^- &:= (\frakk_\frakl\cap\frakq)+(\frakp_\frakl\cap\frakh) = \{L(x)+D:x\in V^-,\alpha D=-D\alpha\}.\label{eq:LieAlgL-}\index{notation}{l2@$\frakl^-$}
\end{align}
Since $\frakp_\frakl\cap\frakq=L(V^+)$, the subspace $\fraka\subseteq\frakp_\frakl\cap\frakq$ as defined in \eqref{eq:DefLiea} is maximal abelian by Lemma \ref{lem:aMaxAbelian}. Let $A:=\exp(\fraka)$ be the corresponding analytic subgroup of $L$ and put
\begin{align}
 \fraka^+ &:= \left\{\sum_{i=1}^{r_0}{t_iL(c_i)}:t_1>\ldots>t_{r_0}\right\}.\label{eq:aPositiveChamber}\index{notation}{aplus@$\fraka^+$}
\end{align}

\begin{proposition}\label{prop:KAH}
\begin{enumerate}
\item[\textup{(1)}] The group $L$ decomposes as $L=K_LAH$ and the orbit $\Omega$ has the polar decomposition $\Omega=K_LA\cdot e$.
\item[\textup{(2)}] Up to scalar multiples there is precisely one invariant measure on $\Omega\cong L/H$ which is given by
\begin{align*}
 \int_{\Omega}{f(x)\td\nu(x)} &= \int_{K_L}{\int_{\fraka^+}{f(k\exp(X)\cdot e)J(X)\td X}\td k},\index{notation}{dnux@$\td\nu(x)$}
\end{align*}
where $\td k$ is a Haar measure on $K_L$, $\td X$ a Lebesgue measure on $\fraka$ and
\begin{align}
 J\left(\sum_{i=1}^{r_0}{t_iL(c_i)}\right) &= \prod_{1\leq i<j\leq r_0}{\sinh^{d_0}\left(\frac{t_i-t_j}{2}\right)\cosh^{d-d_0}\left(\frac{t_i-t_j}{2}\right)}.\label{eq:JacobiDetInvMeasure}\index{notation}{JX@$J(X)$}
\end{align}
\end{enumerate}
\end{proposition}

\begin{proof}
\begin{enumerate}
\item[\textup{(1)}] This is \cite[Lemma 3.6]{vdB05}.
\item[\textup{(2)}] By Proposition \ref{prop:StrRootDecomp} the root system $\Sigma(\frakl,\fraka)$ is of type $A_{r_0-1}$. Hence, the positive Weyl chamber corresponding to the positive system \eqref{eq:PositiveLRoots} is precisely $\fraka^+$. Let $W:=N_{K_L}(\fraka)/Z_{K_L}(\fraka)$ be the Weyl group of the root system $\Sigma(\frakl,\fraka)$ and $W_{K_L\cap H}:=N_{K_L\cap H}(\fraka)/Z_{K_L\cap H}(\fraka)$, viewed as a subgroup of $W$. Then, by \cite[Theorem 3.9]{vdB05} the unique invariant measure on $L/H$ (up to scalar multiples) is given by
 \begin{align*}
  \int_{\Omega}{f(x)\td\nu(x)} &= \sum_{[w]\in W/W_{K_L\cap H}}{\int_{K_L}{\int_{\fraka^+}{f(k\exp(wX)\cdot e)\widetilde{J}(X)\td X}\td k}},
 \end{align*}
 where
 \begin{align*}
  \widetilde{J}(X) &= \prod_{\alpha\in\Sigma^+(\frakl,\fraka)}{\sinh^{m_\alpha^+}\alpha(X)\cosh^{m_\alpha^-}\alpha(X)}
 \end{align*}
 with\ \ $m_\alpha^\pm=\dim\,\frakl_\alpha^\pm$.\ \ Observe\ \ that\ \ $W$\ \ is\ \ the\ \ symmetric\ \ group\ \ on $\{L(c_1), \ldots, L(c_{r_0})\}$, as it is the Weyl group of the root system $\Sigma(\frakl,\fraka)$ of type $A_{r_0-1}$. We claim that $W_{K_L\cap H}=W$. In fact, let $\pi$ be any permutation on $\{1,\ldots,r_0\}$. By Lemma \ref{lem:AutTransitiveJordanFrames} there is a derivation $D\in\der(V)$ such that $\alpha D=D\alpha$ and $e^Dc_i=c_{\pi(i)}$. Thus, by \eqref{eq:LieAlgKL} we have $D\in\frakk_\frakl$. Hence, $e^D\in K_L\cap H$ and it follows that $W_{K_L\cap H}=W$.\\
 It remains to show that $\widetilde{J}(X)=J(X)$. Using \eqref{eq:LRootSpaces}, \eqref{eq:LieAlgL+} and \eqref{eq:LieAlgL-} we find that for $i\neq j$:
 \begin{align*}
  \frakl_{\frac{\varepsilon_i-\varepsilon_j}{2}}^+ &= \{c_i\Box x:x\in V_{ij}^+\} & \mbox{and} && \frakl_{\frac{\varepsilon_i-\varepsilon_j}{2}}^- &= \{c_i\Box x:x\in V_{ij}^-\}.\\
  \intertext{It follows that}
  m_{\frac{\varepsilon_i-\varepsilon_j}{2}}^+ &= \dim\,V_{ij}^+=d_0 & \mbox{and} && m_{\frac{\varepsilon_i-\varepsilon_j}{2}}^- &= \dim\,V_{ij}^-=d-d_0
 \end{align*}
 and hence $\widetilde{J}(X)=J(X)$. This finishes the proof.\qedhere
\end{enumerate}
\end{proof}

As a consequence of the decomposition $L=K_LAH$ we can now calculate the character $\chi$ of $\Str(V)$ on $L$:

\begin{proposition}\label{prop:CharLambda}
Let $V$ be a simple real Jordan algebra with Cartan involution $\alpha$ such that $V^+$ is also simple. Then
\begin{align}
 \chi(g) &= |\Det\,g|^{\frac{r}{n}} & \forall\, g\in L\label{eq:ChiDet}
\end{align}
and
\begin{align}
 \Delta\left(\sum_{i=1}^{r_0}{e^{t_i}c_i}\right) &= e^{\frac{r}{r_0}\sum_{i=1}^{r_0}{t_i}}.\label{eq:DeltaOnCi}
\end{align}
\end{proposition}

\begin{proof}
We first show \eqref{eq:ChiDet} for $g\in\Str(V)_0=L_0=(K_L)_0A(H\cap L_0)$. Both left and right side of \eqref{eq:ChiDet} define a positive Character $L_0\rightarrow\RR_+$ since $L_0$ is connected. On $(K_L)_0$ both sides are $\equiv1$, because $(K_L)_0$ is compact. Further, the left side is $\equiv1$ on $H\cap L_0$ since $\chi(g)=\Delta(ge)$ and $ge=e$ for $g\in H$. But also the determinant is $\equiv1$ on $H$: $g\in H$ implies $gg^\#=\id$ and $\Det(g^\#)=\Det(g)$. Therefore, it remains to show that \eqref{eq:ChiDet} holds on $A=\exp(\fraka)$.\\
By Lemma \ref{lem:AutTransitiveJordanFrames} the group $H_0\subseteq H\cap L_0$ contains all possible permutations of the elements $L(c_i)\in\fraka$, $1\leq i\leq r_0$ (acting by the adjoint representation). Therefore, each character of $L_0$ takes the same values on the elements $\exp(L(c_i))$, $1\leq i\leq r_0$. Thus, it suffices to show \eqref{eq:ChiDet} for $g=\exp(X)$ with $X=t(L(c_1)+\ldots+L(c_{r_0}))=tL(e)=t\,\id_V\in\fraka$, $t\in\RR$. For this we have
\begin{align*}
 \chi(\exp(X)) &= \chi(e^t\id_V) = \Delta(e^t\cdot e) = e^{rt} = (\Det(e^t\id_V))^{\frac{r}{n}} = (\Det(\exp(X)))^{\frac{r}{n}}.
\end{align*}
Hence, \eqref{eq:ChiDet} holds for $g\in L_0$. Now, $L=L_0\cup\alpha L_0$. Both sides of \eqref{eq:ChiDet} define characters of $L$ which agree on $L_0$. Further, $\chi(\alpha)=\Delta(\alpha e)=\Delta(e)=1$ and $\Det(\alpha)=\pm1$ since $\alpha\in K_L$ and $K_L$ is compact. Thus, \eqref{eq:ChiDet} follows and it remains to prove \eqref{eq:DeltaOnCi}.\\
The left side of \eqref{eq:DeltaOnCi} can be written as
\begin{align*}
 \Delta\left(\sum_{i=1}^{r_0}{e^{t_i}c_i}\right)=\Delta(\exp(X)\cdot e)=\chi(\exp(X))
\end{align*}
with $X=\sum_{i=1}^{r_0}{t_iL(c_i)}$. But since $\chi$ takes the same values on $\exp(L(c_i))$, $1\leq i\leq r_0$, we obtain
\begin{align*}
 \chi(\exp(X)) &= \chi\left(\exp\left(\frac{1}{r_0}\left(\sum_{i=1}^{r_0}{t_i}\right)L(e)\right)\right) = \Delta\left(e^{\frac{1}{r_0}\sum_{i=1}^{r_0}{t_i}}\cdot e\right) = e^{\frac{r}{r_0}\sum_{i=1}^{r_0}{t_i}}
\end{align*}
which proves \eqref{eq:DeltaOnCi}.
\end{proof}

Using the previous lemma together with \eqref{eq:DetEquiv} we find that $\Delta(x)^{-\frac{n}{r}}\td x$, $\td x$ denoting a Lebesgue measure on $\Omega\subseteq V$, is an $L$-invariant measure on $\Omega$. Therefore, by Proposition \ref{prop:KAH} it follows that the measure $\td\nu$ is absolutely continuous with respect to the Lebesgue measure $\td x$ on $\Omega\subseteq V$ and
\begin{align*}
 \td\nu(x) &= \const\cdot\Delta(x)^{-\frac{n}{r}}\td x.
\end{align*}

\subsection{Orbits in the boundary of $\Omega$}

The boundary $\partial\Omega$ is the union of orbits of lower rank. We have the following stratification:
\begin{align*}
 \overline{\Omega} &= \calO_0\cup\ldots\cup\calO_{r_0},
\end{align*}
where $\calO_k=L_0\cdot e_k$\index{notation}{Ok@$\calO_k$} with
\begin{align*}
 e_k &:= c_1+\ldots+c_k, & 0\leq k\leq r_0.\index{notation}{ek@$e_k$}
\end{align*}
Since $\alpha e_k=e_k$, we also have $\calO_k=L\cdot e_k$. Every orbit is a homogeneous space $\calO_k=L/H_k$, where $H_k$\index{notation}{Hk@$H_k$} denotes the stabilizer of $e_k$ in $L$. In general these homogeneous spaces are not symmetric.

\begin{example}\label{ex:MinimalOrbit}
We will mostly be interested in the non-zero orbit $\calO_1$ of minimal rank. Let us compute this orbit for our two main examples.
\begin{enumerate}
\item[\textup{(1)}] For $V=\Sym(n,\RR)$ we have
\begin{align*}
 \calO_1 &= \{xx^t:x\in\RR^n\setminus\{0\}\}.
\end{align*}
Moreover, the map
\begin{align}
 \RR^n\setminus\{0\}\rightarrow\calO_1,\,x\mapsto xx^t,\label{eq:SymTwoFoldCovering}
\end{align}
is a surjective two-fold covering.
\item[\textup{(2)}] For $V=\RR^{p,q}$ we again have to distinguish between two cases. If $p=1$, $q\geq2$, then
\begin{align*}
 \calO_1 &= \{x\in\RR^{1,q}:x_1>0,x_1^2-x_2^2-\ldots-x_n^2=0\}
\end{align*}
is the forward light cone in $\RR^{1,q}$. For $p,q\geq2$ we have
\begin{align*}
 \calO_1 &= \{x\in\RR^{p,q}:x_1^2+\ldots+x_p^2-x_{p+1}^2-\ldots-x_n^2=0\}\setminus\{0\}.
\end{align*}
In both cases, $\calO_1$ can be parameterized by bipolar coordinates:
\begin{align}
 \RR_+\times\SS^{p-1}_0\times\SS^{q-1}\stackrel{\sim}{\rightarrow}\calO_1,\,(t,\omega,\eta)\mapsto(t\omega,s\eta),\label{eq:VpqPolarCoordinates}
\end{align}
where $\SS^{n-1}$\index{notation}{Sn@$\SS^{n-1}$} denotes the unit sphere in $\RR^n$ and $\SS^{n-1}_0$\index{notation}{Snnull@$\SS^{n-1}_0$} is the identity component of $\SS^{n-1}$:
\begin{align*}
 \SS^{n-1}_0 &= \begin{cases}\{1\} & \mbox{if }n=1,\\\SS^{n-1} & \mbox{if }n>1.\end{cases}
\end{align*}
\end{enumerate}
\end{example}

To construct equivariant measures on the orbits $\calO_k$ we need to compute the modular functions of the stabilizers $H_k$. The crucial point for this is to show that $H_k$ is contained in a certain parabolic subgroup. The following results are basically the statements in \cite[Lemma 3.6 and Corollary 3.7]{BSZ06}:

\begin{proposition}\label{prop:QkParabolic}
The group $Q_k:=\{g\in L:g V(e_k,1)\subseteq V(e_k,1)\}$\index{notation}{Qk@$Q_k$} is a parabolic subgroup of $L$. The Langlands decomposition of its Lie algebra $\frakq_k$\index{notation}{qk@$\frakq_k$} is given by $\frakq_k=\frakm_k\oplus\frakn_k$, where
\begin{align*}
 \frakm_k &= \frakl_0\oplus\bigoplus_{1\leq i,j\leq k}{\frakl_{ij}}\oplus\bigoplus_{k<i,j\leq r_0}{\frakl_{ij}},\index{notation}{mk@$\frakm_k$}\\
 \frakn_k &= \bigoplus_{1\leq i\leq k<j\leq r_0}{\frakl_{ij}},\index{notation}{nk@$\frakn_k$}
\end{align*}
and the corresponding Langlands decomposition of $Q_k$ is given by $Q_k=M_kN_k$, where
\begin{align*}
 M_k &= \{g\in L:gL(e_k)=L(e_k)g\},\index{notation}{Mk@$M_k$}\\
 N_k &= \exp(\frakn_k).\index{notation}{Nk@$N_k$}
\end{align*}
\end{proposition}

\begin{proof}
In view of Theorem \ref{thm:Parabolics}~(1) we claim that $Q_k=P_{F_k}$ (in the notation of Appendix \ref{app:Parabolics}), where
\begin{align*}
 F_k :={}& \left\{\frac{\varepsilon_1-\varepsilon_2}{2},\ldots,\frac{\varepsilon_{k-1}-\varepsilon_k}{2},\frac{\varepsilon_{k+2}-\varepsilon_{k+1}}{2},\ldots,\frac{\varepsilon_{r_0-1}-\varepsilon_{r_0}}{2}\right\}\\
 ={}& \Pi\setminus\left\{\frac{\varepsilon_k-\varepsilon_{k+1}}{2}\right\}\subseteq\Pi.
\end{align*}
In fact, we have
\begin{align*}
 \fraka_{F_k} &= \RR L(e_k)+\RR L(e-e_k),\\
\intertext{and hence}
 M_k &:= M_{F_k} = \{g\in L:gL(e_k)=L(e_k)g\},\\
 \frakm_k &:= \frakm_{F_k} = \{X\in\frakl:[L(e_k),X]=0\}.
\end{align*}
Since for $X\in\frakl_{ij}$ we have $[L(e_k),X]=\frac{\delta_{ik}-\delta_{jk}}{2}X$, it follows that
\begin{align*}
 \frakm_k &= \frakl_0\oplus\bigoplus_{1\leq i,j\leq k}{\frakl_{ij}}\oplus\bigoplus_{k<i,j\leq r_0}{\frakl_{ij}}.
\end{align*}
Further,
\begin{align*}
 \Sigma_k^+ &:= \Sigma_{F_k}^+(\frakl,\fraka_{F_k}) = \left\{\frac{\varepsilon_i-\varepsilon_j}{2}:1\leq i\leq k<j\leq r_0\right\},\\
 \frakn_k &:= \frakn_{F_k} = \bigoplus_{1\leq i\leq k<j\leq r_0}{\frakl_{ij}},\\
 N_k &:= N_{F_k} = \exp(\frakn_k),
\end{align*}
and $P_k:=P_{F_k}=M_kN_k$. Clearly $P_k\subseteq Q_k$ since both $M_k$ and $N_k$ are contained in $Q_k$. Since $P_k$ is, as a parabolic subgroup, equal to the normalizer of its Lie algebra, the inclusion $Q_k\subseteq P_k$ follows if we show that $Q_k$ has the same Lie algebra as $P_k$.\\
Let $X\in\frakl$ be an element of the Lie algebra of $Q_k$. Then $XV(e_k,1)\subseteq V(e_k,1)$. Write $X=X_0+\sum_{i\neq j}{X_{ij}}$, where $X_0\in\frakl_0$ and $X_{ij}\in\frakl_{ij}$. In particular $e_k\in V(e_k,1)$ and hence $Xe_k\in V(e_k,1)$. It is easy to see that $X_{ij}e_k\in V_{ij}$ for $i\neq j$ and $X_0e_k\in\bigoplus_{i=1}^k{V_{ii}}$. Therefore, we must have $(X_{ij}+X_{ji})e_k=0$ if $1\leq i\leq k<j\leq r_0$. Write $X_{ij}=c_i\Box x_{ij}$ with $x_{ij}\in V_{ij}$. Then for $1\leq i\leq k<j\leq r_0$ we obtain
\begin{align*}
 0 &= (X_{ij}+X_{ji})e_k = \left(\frac{x_{ij}}{4}+\frac{x_{ij}}{4}-\frac{x_{ij}}{2}\right)+\left(\frac{x_{ji}}{4}+\frac{x_{ji}}{4}\right) = \frac{x_{ji}}{2}.
\end{align*}
This yields $X_{ji}=0$. It is further easily seen that if $X_{ji}=0$ for $1\leq i\leq k<j\leq r_0$, then $XV(e_k,1)\subseteq V(e_k,1)$ and therefore the Lie algebra of $Q_k$ coincides with $\frakm_k+\frakn_k$. This finishes the proof.
\end{proof}

\begin{proposition}\label{prop:VariousIntFormulas}
\begin{enumerate}
\item[\textup{(1)}] $L=K_LQ_k$ and for a suitably normalized Haar measure $\td g$ on $L$ we have the following integral formula:
\begin{align*}
 \int_L{f(g)\td g} &= \int_{K_L}{\int_{M_k}{\int_{N_k}{f(kmn)e^{2\rho_k}(m)\td n}\td m}\td k},
\end{align*}
where
\begin{align*}
 \rho_k &:= (r_0-k)\frac{d}{4}\sum_{i=1}^k{\varepsilon_i}-k\frac{d}{4}\sum_{i=k+1}^{r_0}{\varepsilon_i}.\index{notation}{rhok@$\rho_k$}
\end{align*}
\item[\textup{(2)}] $H_k\subseteq Q_k$ and $H_k=(M_k\cap H_k)N_k$. Moreover we have the following integral formula for the Haar measure $\td h$ on $H_k$:
\begin{align}
 \int_{H_k}{f(h)\td h} &= \int_{M_k\cap H_k}{\int_{N_k}{f(mn)\td n}\td m}.\label{eq:IntFormulaHk}
\end{align}
\item[\textup{(3)}] $M_k=(M_k\cap K_L)\exp(\fraka_k)(M_k\cap H_k)$, where
\begin{align*}
 \fraka_k &= \sum_{i=1}^k{\RR L(c_i)}.\index{notation}{ak@$\fraka_k$}
\end{align*}
Further, $M_k/(M_k\cap H_k)$ is a symmetric space with invariant measure $\td\nu_k$ given by
\begin{align*}
 \int_{M_k/(M_k\cap H_k)}{f(x)\td\nu_k(x)} &= \int_{M_k\cap K_L}{\int_{\fraka_k^+}{f(k\exp(X)\cdot e_k)J_k(X)\td X}\td k},\index{notation}{dnuxk@$\td\nu_k(x)$}
\end{align*}
where $\td k$ is a Haar measure on $M_k\cap K_L$, $\td X$ is a Lebesgue measure on
\begin{align}
 \fraka_k^+ &= \left\{\sum_{i=1}^k{t_iL(c_i)}:t_1>\ldots>t_k\right\}\label{eq:akPositiveChamber}\index{notation}{akplus@$\fraka_k^+$}\\
\intertext{and}
 J_k\left(\sum_{i=1}^k{t_iL(c_i)}\right) &= \prod_{1\leq i<j\leq k}{\sinh^{d_0}\left(\frac{t_i-t_j}{2}\right)\cosh^{d-d_0}\left(\frac{t_i-t_j}{2}\right)}.\label{eq:JacobiDetEquivMeasure}\index{notation}{JkX@$J_k(X)$}
\end{align}
\end{enumerate}
\end{proposition}

\begin{proof}
\begin{enumerate}
\item[\textup{(1)}] The decomposition $L=K_LQ_k$ holds by Theorem \ref{thm:Parabolics}~(2). Further, by Theorem \ref{thm:Parabolics}~(3) we have the integral formula
\begin{align*}
 \int_L{f(x)\td g} &= \int_{K_L}{\int_{M_k}{\int_{N_k}{f(kmn)e^{2\rho_{F_k}}(m)\td n}\td m}\td k},
\end{align*}
with $F_k$ as in the proof of Proposition \ref{prop:QkParabolic} and $\rho_{F_k}$ as in Appendix \ref{app:Parabolics}. It remains to show $\rho_{F_k}=\rho_k$. In fact,
\begin{align*}
 \rho_{F_k} &= \frac{1}{2}\sum_{\alpha\in\Sigma_k^+}{\dim(\frakl_\alpha)\alpha} = \frac{d}{2}\sum_{1\leq i\leq k<j\leq r_0}{\frac{\varepsilon_i-\varepsilon_j}{2}} = (r_0-k)\frac{d}{4}\sum_{i=1}^k{\varepsilon_i}-k\frac{d}{4}\sum_{i=k+1}^{r_0}{\varepsilon_i}.
\end{align*}
\item[\textup{(2)}] If $ge_k=e_k$, then by \eqref{eq:QRepStrGrp} we have $gP(e_k)g^\#=P(ge_k)=P(e_k)$. Since $P(e_k)$ is the orthogonal projection onto $V(e_k,1)$ we obtain for $x\in V(e_k,1)$:
\begin{align*}
 gx &= gP(e_k)x = P(e_k)g^{-\#}x \in V(e_k,1)
\end{align*}
and hence $gV(e_k,1)\subseteq V(e_k,1)$. Therefore, $H_k\subseteq Q_k$. Since $N_k\subseteq H_k$, we clearly have $H_k=(M_k\cap H_k)N_k$. It remains to show the integral formula.\\
By \cite[Chapter I, Proposition 1.12]{Hel84}:
\begin{align*}
 \int_{H_k}{f(h)\td h} &= \int_{M_k\cap H_k}{\int_{N_k}{f(mn)\frac{\Det(\Ad_{N_k}(n))}{\Det(\Ad_{H_k}(n))}\td n}\td m}.
\end{align*}
For $n=\exp(X)\in\frakn_k$ we have $\Det(\Ad(n))=e^{\Tr(\ad(X))}$. But $\frakn_k$ consists of root spaces $\frakl_\alpha$ with $\alpha\neq0$. Since $[\frakl_\alpha,\frakl_\beta]\subseteq\frakl_{\alpha+\beta}$, both $\Tr(\ad_{\frakn_k}(X))$ and $\Tr(\ad_{\frakh_k}(X))$ vanish and the determinants in the integral formula are $\equiv1$. Hence, \eqref{eq:IntFormulaHk} follows.
\item[\textup{(3)}] $(M_k\cap H_k)$ is the stabilizer of $e_k$ in $M_k$. Therefore, $M_k/(M_k\cap H_k)\cong M_k\cdot e_k$. The restriction $g\mapsto g|_{V(e_k,1)}$ defines a group homomorphism $M_k\mapsto\Str(V(e_k,1))$. Since the Lie algebra $\frakm_k$ of $M_k$ contains the Lie algebra $\str(V(e_k,1))$ of $\Str(V(e_k,1))$, the image of this map is the union of connected components of $\Str(V(e_k,1))$. The integral formula then follows essentially from Proposition \ref{prop:KAH}.\qedhere
\end{enumerate}
\end{proof}

Now we can finally compute the modular function of the stabilizer $H_k$.

\begin{corollary}
The modular function $\chi_{H_k}$\index{notation}{chiHk@$\chi_{H_k}$} of $H_k$ is on the Lie algebra $\frakh_k$ of $H_k$ given by
\begin{align}
 \td\chi_{H_k} &= \frac{1}{2}kd\sum_{i=k+1}^{r_0}{\varepsilon_i}.\label{eq:HkModFct}
\end{align}
\end{corollary}

\begin{proof}
By Proposition \ref{prop:VariousIntFormulas} the group $H_k$ decomposes as $H_k=(M_k\cap H_k)\ltimes N_k$. We first show that $M_k\cap H_k$ is reductive and hence unimodular. This will follow from \cite[Corollary 1.1.5.4]{War72} if we prove that $\theta(\frakm_k\cap\frakh_k)=\frakm_k\cap\frakh_k$. For this let $X\in\frakm_k\cap\frakh_k$. Then $[L(e_k),X]=0$ and $Xe_k=0$. Taking adjoints it follows that $[L(e_k),X^*]=[X,L(e_k)]^*=0$ and hence $X^*\in\frakm_k$. It remains to show that $X^*e_k=0$. Note that since $X$ preserves $V(e_k,1)$, we have $X|_{V(e_k,1)}\in\str(V(e_k,1))$. The condition $Xe_k=0$ even implies that $X|_{V(e_k,1)}\in\der(V(e_k,1))$, because $e_k$ is the unit element in $V(e_k,1)$. Therefore, by Lemma \ref{lem:TrOnSubspace}~(2) and \eqref{eq:JordanTrOfDerivation} we obtain for $a\in V(e_k,1)$:
\begin{align*}
 (X^*e_k|a) &= (e_k|Xa) = \tr_V(Xa) = \tr_{V(e_k,1)}(Xa) = 0.
\end{align*}
Hence, $X\in\frakh_k$ and it follows that $\theta$ preserves $\frakm_k\cap\frakh_k$. Thus, $\frakm_k\cap\frakh_k$ is reductive and $M_k\cap H_k$ is unimodular.\\
Since $H_k=(M_k\cap H_k)\ltimes N_k$ with $M_k\cap H_k$ unimodular, the modular function $\chi_{H_k}$ can be calculated on the Lie algebra level as the trace of the adjoint action on $\frakn_k$:
\begin{align*}
 \td\chi_{H_k}(X) &= -\Tr(\ad(X)|_{\frakn_k}).
\end{align*}
Since $[\frakl_\alpha,\frakl_\beta]\subseteq\frakl_{\alpha+\beta}$ for $\alpha,\beta\in\Sigma(\frakl,\fraka)$, this trace can only be non-zero if $X\in\frakl_0\cap\frakh_k$. By \eqref{eq:LRootSpace0} we have
\begin{align*}
 \frakl_0\cap\frakh_k &= \left\{L(x)+D:x\in\bigoplus_{i=k+1}^{r_0}{V_{ii}},Dc_j=0\ \forall\,1\leq j\leq r_0\right\}.
\end{align*}
Let $X=D\in\mathfrak{der}(V)$ with $Dc_i=0$ for all $1\leq i\leq r_0$. Then for $x\in V(c_i,\lambda)$ we have $c_i\cdot Dx=D(c_i\cdot x)-Dc_i\cdot x=\lambda Dx$ and hence $DV_{ij}\subseteq V_{ij}$ for all $1\leq i,j\leq r_0$. Further, it is easy to show that for $x\in V_{ij}$ we have $\ad(X)(c_i\Box x)=c_i\Box Dx$. Hence, for $i\neq j$ we obtain $\Tr(\ad(D)|_{\frakl_{ij}})=\Tr(D|_{\frakl_{ij}})$ and 
\begin{align*}
 \Tr(\ad(D)|_{\frakn_k}) &= \Tr(D|_{V(e_k,\frac{1}{2})}).
\end{align*}
Now by \eqref{eq:TrOfDerivation} the trace of a derivation vanishes. Since $D|_{V(e_k,1)}\in\mathfrak{der}(V(e_k,1))$ and $D|_{V(e_k,0)}\in\mathfrak{der}(V(e_k,0))$ we have
\begin{align*}
 \Tr(D) &= \Tr(D|_{V(e_k,1)}) = \Tr(D|_{V(e_k,0)}) = 0.
\end{align*}
Since $V=V(e_k,1)\oplus V(e_k,\frac{1}{2})\oplus V(e_k,0)$, we obtain $\Tr(D|_{V(e_k,\frac{1}{2})})=0$ and hence $\Tr(\ad(X))=0$.\\
It remains to consider the case $X=L(a)$ for some $a\in V_{\ell\ell}$, $k<\ell\leq r_0$. Let us first assume that $a\in V_{\ell\ell}^-$. It we denote by $\frakn_k^\pm=c_i\Box V_{ij}^\pm$, then it is easily checked that $\ad(X)\frakn_k^\pm\subseteq\frakn_k^\mp$ and hence $\Tr(\ad(X)|_{\frakn_k})=0$. It remains to calculate the modular function on $L(V_{\ell\ell}^+)=\RR L(c_\ell)$. For $a=c_\ell$ and $c_i\Box x\in\frakn_k$, $x\in V_{ij}$, $1\leq i\leq k<j\leq r_0$, we obtain with \eqref{eq:LieStrBracket} and Lemma \ref{lem:BracketZero}~(2):
\begin{align*}
 \ad(X)(c_i\Box x) &= \left[L(c_\ell),\frac{1}{2}L(x)+[L(c_i),L(x)]\right]\\
 &= L(-[L(c_i),L(x)]c_\ell) + \frac{1}{2}[L(c_\ell),L(x)]\\
 &= \begin{cases}-\frac{1}{2}c_i\Box x & \mbox{if $\ell=j$,}\\0 & \mbox{else.}\end{cases}
\end{align*}
Since $\dim(c_i\Box V_{ij})=d$, this yields $\Tr(\ad(X)|_{\frakn_k})=-\frac{1}{2}kd$ and \eqref{eq:HkModFct} holds.
\end{proof}

As another corollary from the previous decomposition theorems we obtain a polar decomposition for the orbits $\calO_k$, $0\leq k\leq r_0-1$.

\begin{corollary}\label{cor:PolarOk}
For every $0\leq k\leq r_0-1$ the group $L$ decomposes as $L=K_L\exp(\fraka_k)H_k$ and the orbit $\calO_k$ has the polar decomposition $\calO_k=K_L\exp(\fraka_k)\cdot e_k$.
\end{corollary}

\begin{proof}
The polar decomposition for the orbit $\calO_k$ clearly follows from the decomposition $L=K_L\exp(\fraka_k)H_k$ on the group level. To show the group decomposition we use the decompositions of Propositions \ref{prop:QkParabolic} and \ref{prop:VariousIntFormulas} to calculate:
\begin{align*}
 L &= K_LQ_k = K_LM_kN_k = K_L(K_L\cap M_k)\exp(\fraka_k)(H_k\cap M_k)N_k\\
 &= K_L\exp(\fraka_k)H_k.\qedhere
\end{align*}
\end{proof}

\subsection{Equivariant measures}\label{sec:EquivMeasures}

A measure $\td\mu$ on a $G$-space $X$ is called \textit{$\delta$-equivariant}\index{subject}{equivariant measure}, $\delta$ a positive character of $G$, if $\td\mu(gx)=\delta(g)\td\mu(x)$ for $g\in G$, i.e. for every $f\in L^1(X,\td\mu)$ and $g\in G$:
\begin{align*}
 \int_X{f(g^{-1}x)\td\mu(x)} &= \delta(g)\int_X{f(x)\td\mu(x)}.
\end{align*}
$\delta$ is called the \textit{modular function}\index{subject}{equivariant measure!modular function} of the equivariant measure $\td\mu$. For the existence and uniqueness of equivariant measures we have the following fact:

\begin{fact}[{\cite[Section 33D]{Loo53}}]
In order that a real character $\delta$ be the modular function for an equivariant measure $\td(gH)$ on the quotient space $G/H$, where $H$ is a closed subgroup of the locally compact group $G$, it is necessary and sufficient that $\delta(h)=\frac{\chi_H(h)}{\chi_G(h)}$ for all $h\in H$, where $\chi_G$ and $\chi_H$ denote the modular functions of $G$ and $H$, repectively. In this case, the measure is uniquely determined by the following formula:
\begin{align}
 \int_{G/H}{\int_H{f(gh)\td h}\td(gH)} &= \int_G{\delta(g)f(g)\td g}.\label{eq:IdentityEquivMeasures}
\end{align}
\end{fact}

We apply this to the orbits $\calO_k$, $0\leq k\leq r_0$.

\begin{theorem}\label{thm:EquivMeasures}
\begin{enumerate}
\item[\textup{(1)}] On $\calO_{r_0}=\Omega$ the $L$-equivariant measures which are locally finite near $0$ are (up to positive scalars) exactly the measures\index{notation}{dmulambdax@$\td\mu_\lambda(x)$}
 \begin{align*}
  \int_\Omega{f(x)\td\mu_\lambda(x)} &:= \int_{K_L}{\int_{\fraka^+}{f(k\exp(X)e)e^{\lambda\frac{r}{r_0}\sum_{i=1}^{r_0}{t_i}}J(X)\td X}\td k},
 \end{align*}
 where $X=\sum_{i=1}^{r_0}{t_iL(c_i)}$ and $\fraka^+$ and $J(X)$ as in \eqref{eq:aPositiveChamber} and \eqref{eq:JacobiDetInvMeasure}. $\td\mu_\lambda$ transforms by
 \begin{align*}
  \td\mu_\lambda(gx) &= \chi(g)^\lambda\td\mu(x) & \mbox{for $g\in L$.}
 \end{align*}
 Moreover, $\td\mu_\lambda$ is absolutely continuous with respect to the Lebesgue measure $\td x$ on $\Omega$ and we have
 \begin{align*}
  \td\mu_\lambda(x) &= \const\cdot\Delta(x)^{\lambda-\frac{n}{r}}(x)\td x & \mbox{for $\lambda>(r_0-1)\frac{r_0d}{2r}$.}
 \end{align*}
\item[\textup{(2)}] For $k=0,\ldots,r_0-1$ there is (up to positive scalars) exactly one $L$-equivariant measure $\td\mu_k$ on $\calO_k$ given by
 \begin{align*}
  \int_{\calO_k}{f(x)\td\mu_k(x)} &:= \int_{K_L}{\int_{\fraka_k^+}{f(k\exp(X)e)e^{\frac{r_0d}{2}\sum_{i=1}^k{t_i}}J_k(X)\td X}\td k},\index{notation}{dmukx@$\td\mu_k(x)$}
 \end{align*}
 where $X=\sum_{i=1}^{r_0}{t_iL(c_i)}$ and $\fraka_k^+$ and $J_k(X)$ as in \eqref{eq:akPositiveChamber} and \eqref{eq:JacobiDetEquivMeasure}. $\td\mu_k$ transforms by
 \begin{align*}
  \td\mu_k(gx) &= \chi(g)^{k\frac{r_0d}{2r}}\td\mu_k(x) & \mbox{for $g\in L$.}
 \end{align*}
\end{enumerate}
\end{theorem}

\begin{proof}
\begin{enumerate}
\item[\textup{(1)}] $\Omega\cong L/H$ as homogeneous spaces. Both $L$ and $H$ are reductive and hence unimodular (see e.g. \cite[Corollary 8.31~(d)]{Kna02}). Therefore, a character $\delta$ is the modular function for an equivariant measure on $L/H$ if and only if $\delta|_H\equiv1$. Let $\delta$ be such a positive character of $L$.\\
According to Proposition \ref{prop:KAH}~(1) the group $L$ decomposes as $L=K_LAH$. Since $K_L$ is compact, $\delta|_{K_L}\equiv1$. Therefore, we only have to determine the values of $\delta$ on $A$. Since $A=\exp(\fraka)$, it suffices to calculate the possible derived homomorphisms $\td\delta:\frakl\rightarrow\RR$. By Lemma \ref{lem:AutTransitiveJordanFrames} the group $H$ contains all possible permutations of the elements $L(c_i)\in\fraka$, $1\leq i\leq r_0$. Therefore,
\begin{align*}
 \td\delta\left(\sum_{i=1}^{r_0}{t_iL(c_i)}\right) &= \lambda\frac{r}{r_0}\sum_{i=1}^{r_0}{t_i}
\end{align*}
for some $\lambda\in\RR$. Then by \eqref{eq:DeltaOnCi} we obtain $\delta=\chi^\lambda$. Hence, the functions $\chi^\lambda$ are the only possible modular functions for $L$-equivariant measures on $\Omega$.\\
The measures $\Delta(x)^{\lambda-\frac{n}{r}}\td x$ are clearly $\chi^\lambda$-equivariant by Proposition \ref{prop:CharLambda}. Since $\Delta(x)^{-\frac{n}{r}}\td x$ is an invariant measure on $\Omega$, the stated integral formula follows from Proposition \ref{prop:KAH} (with $f(x)\Delta(x)^\lambda$ instead of $f(x)$) and \eqref{eq:DeltaOnCi}. From the integral formula one can also easily see that $\td\mu_\lambda$ is locally finite near $0$ if and only if $\lambda>(r_0-1)\frac{r_0d}{2r}$.
\item[\textup{(2)}] $\calO_k\cong L/H_k$ as homogeneous spaces. As remarked in (1) the modular function $\chi_L$ is equal to $1$ and by \eqref{eq:HkModFct} the derived modular function on $\frakh_k$ is given by
\begin{align*}
 \td\chi_{H_k} &= \frac{1}{2}kd\sum_{i=k+1}^{r_0}{\varepsilon_i}.
\end{align*}
Let $\delta$ be a positive character of $L$ with $\td\delta|_{\frakh_k}=\td\chi_{H_k}$. The Lie algebra $\frakl$ is reductive and hence $\frakl=Z(\frakl)+[\frakl,\frakl]$. On the semisimple part $[\frakl,\frakl]$ the character $\td\delta$ clearly vanishes. Therefore it suffices to compute the values of $\td\delta$ on $\bigoplus_{i=1}^{r_0}{L(V_{ii})}$ by \eqref{eq:CenterL}. On $\bigoplus_{j=k+1}^{r_0}{L(V_{ii})}\subseteq\frakh_k$ the character $\td\delta$ is given by $\td\chi_{H_k}$. Since $k<r_0$ this subspace of $\frakh_k$ is non-zero and using the $\Ad$-invariance of $\td\delta$ and Lemma \ref{lem:AutTransitiveJordanFrames} one shows that
\begin{align*}
 \td\delta &= \frac{1}{2}kd\sum_{i=1}^{r_0}{\varepsilon_i}.
\end{align*}
By Proposition \ref{prop:CharLambda} we obtain $\delta=\chi^{k\frac{r_0d}{2r}}$. For the integral formula we calculate:
\begin{align*}
 & \int_L{f(g)\chi^{k\frac{r_0d}{2r}}(g)\td g}\\
 ={}& \int_{K_L}{\int_{M_k}{\int_{N_k}{f(kmn)e^{k\frac{d}{2}\sum_{i=1}^{r_0}{\varepsilon_i}}(m)e^{2\rho_k}(m)\td n}\td m}\td k}\\
 ={}& \int_{K_L}{\int_{M_k/(M_k\cap H_k)}{\int_{M_k\cap H_k}{\int_{N_k}{f(kmhn)e^{\frac{r_0d}{2}\sum_{i=1}^k{\varepsilon_i}}(m)}}}}\\
 & \ \ \ \ \ \ \ \ \ \ \ \ \ \ \ \ \ \ \ \ \ \ \ \ \ \ \ \ \ \ \ \ \ \ \ \ \ \ \ \ \ \ \ \ \ \ \ \ \ \ \ \ \ \ \ \ \ \td n\td h\td\nu_k(m(M_k\cap H_k))\td k\\
 ={}& \int_{K_L}{\int_{M_k/(M_k\cap H_k)}{\int_{H_k}{f(kmh)e^{\frac{r_0d}{2}\sum_{i=1}^k{\varepsilon_i}}(m)\td h}\nu_k(m(M_k\cap H_k))}\td k}\\
 ={}& \int_{K_L}{\int_{M_k\cap K_L}{\int_{\fraka_k^+}{\int_{H_k}{f(kk'\exp(X)h)e^{\frac{r_0d}{2}\sum_{i=1}^k{\varepsilon_i}(X)}J_k(X)\td h}\td X}\td k'}\td k}\\
 ={}& \int_{K_L}{\int_{\fraka_k^+}{\int_{H_k}{f(k\exp(X)h)e^{\frac{r_0d}{2}\sum_{i=1}^k{t_i}}J_k(X)\td h}\td X}\td k}.
\end{align*}
Now the desired integral formula follows from \eqref{eq:IdentityEquivMeasures}.\qedhere
\end{enumerate}
\end{proof}

For convenience we denote for $\lambda>(r_0-1)\frac{r_0d}{2r}$ the open orbit $\calO_{r_0}=\Omega$ by $\calO_\lambda$\index{notation}{Olambda@$\calO_\lambda$}. Similarly, for $\lambda=k\frac{r_0d}{2r}$, $k=0,\ldots,r_0-1$, we put $\calO_\lambda:=\calO_k$\index{notation}{Olambda@$\calO_\lambda$} and $\td\mu_\lambda:=\td\mu_k$\index{notation}{dmulambdax@$\td\mu_\lambda(x)$}. This yields Hilbert spaces $L^2(\calO_\lambda,\td\mu_\lambda)$ exactly for $\lambda$ in the \textit{Wallach set}\index{subject}{Wallach set}
\begin{align*}
 \calW &:= \left\{0,\frac{r_0d}{2r},\ldots,(r_0-1)\frac{r_0d}{2r}\right\}\cup\left((r_0-1)\frac{r_0d}{2r},\infty\right),\index{notation}{W@$\calW$}
\end{align*}
and the measures transform by
\begin{align}
 \td\mu_\lambda(gx) &= \chi(g)^\lambda\td\mu_\lambda(x) & \mbox{for }g\in L.\label{eq:dmulambdaEquivariance}
\end{align}

For the minimal non-trivial orbit $\calO_1$ the polar decomposition of Corollary \ref{cor:PolarOk} simplifies to $\calO_1=K_L\RR_+c_1$. Further, the integral formula in Theorem \ref{thm:EquivMeasures}~(2) amounts to
\begin{align}
 \int_{\calO_1}{f(x)\td\mu_1(x)} &= \int_{K_L}{\int_0^\infty{f(ktc_1)t^{\mu+\nu+1}\td t}\td k},\label{eq:dmuIntFormula}
\end{align}
since $\mu+\nu+1=\frac{r_0d}{2}-1$. Hence, the space of radial (or equivalently $K_L$-invariant) functions in $L^2(\calO_1,\td\mu_1)$ is given by $L^2(\calO_1,\td\mu_1)_{\textup{rad}}\cong L^2(\RR_+,t^{\mu+\nu+1}\td t)$\index{notation}{L2O1dmu1rad@$L^2(\calO_1,\td\mu_1)_{\textup{rad}}$}.

\begin{example}\label{ex:EquivMeasures}
\begin{enumerate}
\item[\textup{(1)}] For $V=\Sym(n,\RR)$ the two-fold covering \eqref{eq:SymTwoFoldCovering} induces a unitary (up to a scalar) isomorphism
\begin{align}
 \calU:L^2(\calO_1,\td\mu_1)\rightarrow L^2_{\textup{even}}(\RR^n),\ \calU\psi(x) := \psi(xx^t),\label{eq:DefCalU}\index{notation}{U@$\calU$}
\end{align}
where $L^2_{\textup{even}}(\RR^n)$\index{notation}{L2evenRn@$L^2_{\textup{even}}(\RR^n)$} denotes the space of even $L^2$-functions on $\RR^n$. In fact, for $\psi\in L^2(\calO_1,\td\mu_1)$:
\begin{align*}
 \int_{\RR^n}{|\calU\psi(x)|^2\td x} &= \vol(\SS^{n-1})\int_{\SO(n)}{\int_0^\infty{|\calU\psi(kte_1)|^2t^{n-1}\td t}\td k}\\
 &= \vol(\SS^{n-1})\int_{\SO(n)}{\int_0^\infty{|\psi(t^2kc_1k^t)|^2t^{n-1}\td t}\td k}\\
 &= \frac{\vol(\SS^{n-1})}{2}\int_{\SO(n)}{\int_0^\infty{|\psi(k\cdot sc_1)|^2s^{\frac{n}{2}-1}\td s}\td k}\\
 &= \frac{\vol(\SS^{n-1})}{2}\int_\calO{|\psi(x)|^2\td\mu(x)},
\end{align*}
where $\td k$ is the normalized Haar measure on $\SO(n)$. Hence $\calU$ is unitary (up to a scalar).
\item[\textup{(2)}] For $V=\RR^{p,q}$ the measure $\td\mu_1$ can be expressed in polar coordinates \eqref{eq:VpqPolarCoordinates}. Using \eqref{eq:dmuIntFormula} we obtain
\begin{align*}
 \td\mu_1 &= \const\cdot t^{p+q-3}\td t\td\omega\td\eta,
\end{align*}
where $\td\omega$ and $\td\eta$ denote the normalized euclidean measures on $\SS^{p-1}_0$ and $\SS^{q-1}$, respectively.
\end{enumerate}
\end{example}

To be able to deal with the measures $\td\mu_k$, $0\leq k\leq r_0-1$, we interpret them as residues of zeta functions.

\begin{proposition}\label{prop:ZetaFunctions}
\begin{enumerate}
\item[\textup{(1)}] Let $V$ be euclidean. Then the measure $\td\mu_k$, $0\leq k\leq r_0-1$, is a constant multiple of the residue of the \textit{zeta function}\index{subject}{zeta function}
 \begin{align*}
  Z(f,\lambda) &:= \int_\Omega{f(x)\Delta(x)^{\lambda-\frac{n}{r}}\td x}\index{notation}{Zflambda@$Z(f,\lambda)$}
 \end{align*}
 at the value $\lambda=k\frac{r_0d}{2r}=k\frac{d}{2}$.
\item[\textup{(2)}] Let $V$ be non-euclidean and $\ncong\RR^{p,q}$. Then the measure $\td\mu_k$, $0\leq k\leq r_0-1$, is a constant multiple of the residue of the \textit{zeta function}\index{subject}{zeta function}
 \begin{align*}
  Z(f,\lambda) &:= \int_V{f(x)|\Delta(x)|^{\lambda-\frac{n}{r}}\td x}\index{notation}{Zflambda@$Z(f,\lambda)$}
 \end{align*}
 at the value $\lambda=k\frac{r_0d}{2r}$.
\item[\textup{(3)}] Let $V=\RR^{p,q}$. Then $r_0=2$ and the orbits $\calO_k$ are given by $\calO_0=\{0\}$, $\calO_1=\{x\in\RR^{p,q}\setminus\{0\}:\Delta(x)=0\}$ and $\calO_2=\Omega=\{x\in\RR^{p,q}\setminus\{0\}:\Delta(x)>0\}$ where $\Delta(x)$ is the standard quadratic form of signature $(p,q)$. In this case the measure $\td\mu_0$ is just a scalar multiple of the Dirac delta distribution at $0$ and the measure $\td\mu_1$ is again a constant multiple of the residue of the \textit{zeta function}\index{subject}{zeta function}
 \begin{align*}
  Z(f,\lambda) &:= \int_\Omega{f(x)\Delta(x)^{\lambda-\frac{n}{r}}\td x}\index{notation}{Zflambda@$Z(f,\lambda)$}
 \end{align*}
 at the value $\lambda=\frac{r_0d}{2r}=\frac{p+q-2}{2}$.
\end{enumerate}
\end{proposition}

For details on the meromorphic extension of the zeta functions involved here see \cite[Chapter VII, Section 2]{FK94} for the euclidean case, \cite[Theorem 6.2~(2)]{BSZ06} for the non-euclidean case $\ncong\RR^{p,q}$ and \cite[Chapter III.2]{GS64} for $V=\RR^{p,q}$.

\begin{proof}
Part (1) is \cite[Proposition VII.2.3]{FK94}, part (2) is proved in \cite[Theorem 6.2]{BSZ06} and part (3) can be found in \cite[Section III.2.2]{GS64}.
\end{proof}

\begin{remark}
The case differentiation in Proposition \ref{prop:ZetaFunctions} is necessary. Firstly, the zeta function $Z(f,\lambda)$ of $V=\RR^{p,q}$ vanishes at the value $\lambda=0$ if $p$ is even (see \cite[Section III.2.2, Equation (21)]{GS64}). And secondly, the only two results on the positivity of zeta functions for Jordan algebras the author could find, are stated in \cite[Proposition VII.2.3]{FK94} and \cite[Theorem 6.2~(2)]{BSZ06}, where different definitions for the zeta functions are used. It is an interesting question whether there exists a theory of zeta functions which works for arbitrary simple real Jordan algebras and gives the equivariant measures on the orbits $\calO_k$ as residues of the zeta functions. Nevertheless, for the purpose of this article it will only be important, that the measures $\td\mu_k$ appear as residues of zeta functions which are for $\lambda>(r_0-1)\frac{r_0d}{2r}$ supported on the union of open $L$-orbits and given by $|\Delta(x)|^{\lambda-\frac{n}{r}}$.
\end{remark}

\section{The conformal group}

For a semisimple Jordan algebra $V$ we construct the conformal group $\Co(V)$ which acts on $V$ by rational transformations. We define a certain open subgroup $G$ of $\Co(V)$ and construct its universal covering group. The Lie algebra $\frakg=\co(V)$ of $\Co(V)$, also known as Kantor--Koecher--Tits algebra, is given by quadratic vector fields on $V$. We describe $\frakg$ in detail. For a maximal compact subalgebra $\frakk$ of $\frakg$ we also characterize the highest weights of $\frakk_\frakl$-spherical $\frakk$-representations via the Cartan--Helgason theorem. These representations will appear as $\frakk$-types in the minimal representation.

\subsection{The Kantor--Koecher--Tits construction}\label{sec:KKT}

Let $V$ be a semisimple real Jordan algebra.

\subsubsection{Definition of the conformal group}

The conformal group will be built up from three different rational transformations.
\begin{enumerate}
\item[\textup{(1)}] First, $V$ acts on itself by translations
\begin{align*}
 n_a(x) &:= x+a & \forall\,x\in V\index{notation}{na@$n_a$}
\end{align*}
with $a\in V$. Denote by $N:=\{n_a:a\in V\}$\index{notation}{N@$N$} the abelian group of translations which is isomorphic to $V$.
\item[\textup{(2)}] The structure group $\Str(V)$ of $V$ acts on $V$ by linear transformations.
\item[\textup{(3)}] Finally, we define the \textit{conformal inversion element}\index{subject}{conformal inversion element} $j$ by
\begin{align*}
 j(x) &= -x^{-1} & \forall\,x\in V^\times=\{y\in V:y\mbox{ invertible}\}.\index{notation}{j@$j$}
\end{align*}
In view of the minimal polynomial \eqref{eq:MinPoly} $j$ is a rational transformation of $V$.
\end{enumerate}
The \textit{conformal group}\index{subject}{conformal group} $\Co(V)$\index{notation}{CoV@$\Co(V)$} is defined as the subgroup of the group of rational transformations of $V$ which is spanned by $N$, $\Str(V)$ and $j$:
\begin{align*}
 \Co(V) := \langle N,\Str(V),j\rangle_{\textup{grp}}.
\end{align*}
$\Co(V)$ is a semisimple Lie group which is simple if and only if $V$ is simple (see \cite[Chapter VIII, Section 6]{Jac68}). The center of $\Co(V)$ and even of its identity component $\Co(V)_0$ is trivial (see \cite[Theorem VIII.1.3]{Ber00} and its proof). The semidirect product $\Str(V)\ltimes N$ is a maximal parabolic subgroup of $\Co(V)$ (see e.g. \cite[Section X.6.3]{Ber00}).\\

We let $G$\index{notation}{G@$G$} be the group generated by $\Co(V)_0$ and the Cartan involution $\alpha$. The group $G$ has at most two connected components, namely $\Co(V)_0$ and $\alpha\Co(V)_0$. If $V$ is euclidean, then clearly $G=\Co(V)_0$, but in general this is not true (e.g. for $V=\RR^{p,q}$ with $p$ even). We also have the inclusions
\begin{align*}
 \Co(V)_0 &\subseteq G \subseteq \Co(V).
\end{align*}
Further, we have $L\subseteq G\cap\Str(V)$ and we put $P:=L\ltimes N$\index{notation}{P@$P$}. Here similar inclusions hold. In general, $P$ is not maximal parabolic in $G$, but an open subgroup of the maximal parabolic subgroup $P^{\textup{max}}:=G\cap(\Str(V)\ltimes N)$\index{notation}{Pmax@$P^{\textup{max}}$}. The parabolic $P^{\textup{max}}$ has a Langlands decomposition $P^{\textup{max}}=L^{\textup{max}}\ltimes N$ with $L^{\textup{max}}:=G\cap\Str(V)$\index{notation}{Lmax@$L^{\textup{max}}$}.

\subsubsection{The conformal algebra}

Now let us examine the structure of the Lie algebra $\frakg=\co(V)$\index{notation}{g@$\frakg$}\index{notation}{coV@$\co(V)$} of $\Co(V)$, the so-called \textit{Kantor--Koecher--Tits algebra}\index{subject}{Kantor--Koecher--Tits algebra}. An element $X\in\frakg$ corresponds to a quadratic vector field on $V$ of the form
\begin{align*}
 X(z) &= u + Tz - P(z)v, & z\in V
\end{align*}
with $u,v\in V$ and $T\in\frakl$. We use the notation $X=(u,T,v)$ for short. In view of this, we have the decomposition
\begin{equation}
 \frakg = \nf + \frakl + \nfo,\label{eq:GelfandNaimark}
\end{equation}
where\index{notation}{l@$\frakl$}
\begin{alignat*}{3}
 \frakn &= \{(u,0,0):u\in V\} &&\cong V,\index{notation}{n@$\frakn$}\\
 \frakl &= \{(0,T,0):T\in\str(V)\} &&\cong \str(V),\\
 \overline{\frakn} &= \{(0,0,v):v\in V\} &&\cong V.\index{notation}{nbar@$\overline{\frakn}$}
\end{alignat*}
In this decomposition the Lie algebra $\frakp^{\textup{max}}$\index{notation}{pmax@$\frakp^{\textup{max}}$} of $P^{\textup{max}}$ (and $P$) is given by
\begin{align*}
 \frakp^{\textup{max}} &= \frakn + \frakl.
\end{align*}
If $X_j=(u_j,T_j,v_j)$, ($j=1,2$), then the Lie bracket is given by 
\begin{equation}
 [X_1,X_2] = (T_1u_2-T_2u_1,[T_1,T_2]+2(u_1\Box v_2)-2(u_2\Box v_1),-T_1^\# v_2+T_2^\# v_1),\label{eq:LieBracket}
\end{equation}
with the box operator $\Box$ as in \eqref{eq:DefBoxOp}. From this formula it is easy to see that the decomposition \eqref{eq:GelfandNaimark} actually defines a grading on $\frakg$:
\begin{align*}
 \frakg=\frakg_{-1}+\frakg_0+\frakg_1,
\end{align*}
where $\frakg_{-1}=\frakn$, $\frakg_0=\frakl$ and $\frakg_1=\overline{\frakn}$. Further, since the box operators $u\Box v$, $u,v\in V$, generate the structure algebra $\frakl$, the conformal algebra $\frakg$ is generated by $\frakn$ and $\overline{\frakn}$.

\begin{example}\label{ex:ConfGrp}
Since $G_0$ has trivial center we can calculate it by factoring out the center from the universal covering: $G_0=\widetilde{G_0}/Z(\widetilde{G_0})$. Here the universal covering $\widetilde{G_0}$ of $G_0$ is uniquely determined by the Lie algebra $\frakg$ which was described in detail.
\begin{enumerate}
\item[\textup{(1)}] Let $V=\Sym(n,\RR)$. Then $\frakg\cong\sp(n,\RR)$ via the isomorphism
\begin{align*}
 \frakg \rightarrow \sp(n,\RR),\,(u,T,v)\mapsto\left(\begin{array}{cc}T&u\\v&-T^t\end{array}\right).
\end{align*}
Then $G=G_0\cong\Sp(n,\RR)/\{\pm\1\}$, where $\Sp(n,\RR)/\{\pm\1\}$ acts on $x\in V$ by fractional linear transformations:
\begin{align*}
 \left(\begin{array}{cc}A&B\\C&D\end{array}\right)\cdot x &= (Ax+B)(Cx+D)^{-1}.
\end{align*}
\item[\textup{(2)}] Let $V=\RR^{p,q}$. Then an explicit isomorphism $\frakg\stackrel{\sim}{\rightarrow}\so(p+1,q+1)$ is given by
\begin{align*}
 (u,0,0) &\mapsto \left(\begin{array}{c|cc|c}&-(u')^t&(u'')^t&\\\hline u'&&&u'\\u''&&&u''\\\hline&(u')^t&-(u'')^t&\end{array}\right), && u\in V,\\
 (0,sT,0) &\mapsto \left(\begin{array}{c|c|c}&&-s\\\hline&T&\\\hline-s&&\end{array}\right), && T\in\so(p,q), s\in\RR,\\
 (0,0,\alpha v) &\mapsto \left(\begin{array}{c|cc|c}&(v')^t&(v'')^t&\\\hline-v'&&&v'\\v''&&&-v''\\\hline&(v')^t&(v'')^t&\end{array}\right), && v\in V.
\end{align*}
Hence, $G_0\cong\SO(p+1,q+1)_0/Z(\SO(p+1,q+1)_0)$. The center $Z(\SO(p+1,q+1)_0)$ is equal to $\{\pm\1\}$ if $p$ and $q$ are both even, and it is trivial else. By \eqref{eq:VpqCartanInv} we have $G=G_0$ if $p$ is odd and $G=G_0\cup\alpha G_0$ if $p$ is even. 
\end{enumerate}
\end{example}

\subsubsection{A Cartan involution}

The involution $\theta$ of $\Str(V)$ extends to an involution of $\Co(V)$ by
\begin{align*}
 \theta:\Co(V)\rightarrow\Co(V),\,g\mapsto w_0\circ g\circ w_0^{-1},\index{notation}{theta@$\theta$}
\end{align*}
where $w_0:=\alpha\circ j\in\Co(V)$. The map $\theta$ is a Cartan involution of $\Co(V)$ which restricts to Cartan involutions of $G$ and $G_0$. The corresponding involution $\theta$ of the Lie algebra $\frakg$ is given by (see \cite[Proposition 1.1]{Pev02})
\begin{align}
 \theta(u,T,v) &:= (-\alpha(v),-T^*,-\alpha(u)), & (u,T,v)\in\frakg.\label{eq:ThetaOng}\index{notation}{theta@$\theta$}
\end{align}
In the above notation $\nfo=\theta(\nf)$. Let $\frakg=\frakk+\frakp$\index{notation}{k@$\frakk$}\index{notation}{p@$\frakp$} be the corresponding Cartan decomposition of $\frakg$. Then
\begin{align}
 \frakk &= \{(u,T,-\alpha(u)):u\in V,T\in\lf,T+T^*=0\}.\label{eq:frakk}
\end{align}
The fixed point group $K:=G^\theta$\index{notation}{K@$K$} of $\theta$ is a maximal compact subgroup of $G$ with Lie algebra $\frakk$. Then clearly $K_L=K\cap L$. The subgroup $K_L\subseteq K$ is symmetric, the corresponding involution being $g\mapsto(-\1)\circ g\circ(-\1)$.

\begin{lemma}\label{lem:CenterK}
Assume that $V$ and $V^+$ are simple. Then the center $Z(\frakk)$\index{notation}{Zk@$Z(\frakk)$} of $\frakk$ is non-trivial only if $V$ is euclidean. In this case it is given by $Z(\frakk)=\RR(e,0,-e)$.
\end{lemma}

\begin{proof}
First, let $X=(u,T,-\alpha u)\in Z(\frakk)$. Then by \eqref{eq:LieBracket} we have for any $v\in V$ and $S\in\frakk_\frakl$:
\begin{align*}
 (Tv-Su,[T,S]-2u\Box\alpha v+2v\Box\alpha u,T^\#\alpha v-S^\#\alpha u) &= 0.
\end{align*}
Putting $S=0$ we obtain $T=0$ and this simplifies to
\begin{align}
 Su &= 0 & \mbox{and} && L(u\cdot\alpha v)+[L(u),L(\alpha v)] &= L(\alpha u\cdot v)+[L(v),L(\alpha u)].\label{eq:CenterK1}
\end{align}
In particular, one has
\begin{align*}
 u\cdot\alpha v &= \alpha u\cdot v & \forall\, v\in V.
\end{align*}
For $v=e$ this gives $\alpha u=u$ and hence $u\in V^+$. Then, from \eqref{eq:CenterK1} it also follows that
\begin{align}
 [L(u),L(\alpha v)] &= [L(v),L(u)]. & \forall\, v\in V.\label{eq:CenterK2}
\end{align}
Write $u$ in the Peirce decomposition $u=\sum_{i\leq j}{u_{ij}}$, $u_{ij}\in V_{ij}^+$. For $i<j$ we put $v=c_i$ in \eqref{eq:CenterK2} and apply both sides to $c_j$ which gives $-\frac{1}{4}u_{ij}=\frac{1}{4}u_{ij}$ and hence $u_{ij}=0$. Since $V_{ii}^+=\RR c_i$, we then have $u=\sum_{i=1}^{r_0}{\lambda_ic_i}$ with $\lambda_i\in\RR$. Using \eqref{eq:CenterK1} once more, we know that $Su=0$ for $S\in\frakk_\frakl$. For every permutation $\pi$ of $\{1,\ldots,r_0\}$ there is by Lemma \ref{lem:AutTransitiveJordanFrames} a derivation $D\in\mathfrak{der}(V)$ with $\alpha D=D\alpha$ such that $e^Dc_i=c_{\pi(i)}$ for every $i=1,\ldots,r_0$. By \eqref{eq:LieAlgKL} we have $D\in\frakk_\frakl$ and hence $Du=0$. Altogether this gives
\begin{align*}
 \sum_{i=1}^{r_0}{\lambda_ic_i} &= u = e^Du = \sum_{i=1}^{r_0}{\lambda_ic_{\pi(i)}}
\end{align*}
for all permutations $\pi$. Therefore, $\lambda_i=\lambda_j$ for all $i,j=1,\ldots,r_0$ and $u=\lambda e$, $\lambda\in\RR$.\\
It remains to show that $X=(e,0,-e)$ is actually contained in the center if and only if $V$ is euclidean. With \eqref{eq:LieBracket} we find that
\begin{align*}
 [(e,0,-e),(u,T,-\alpha u)] &= (-Te,2L(u-\alpha u),-T^\# e).
\end{align*}
If $V$ is euclidean, then $\frakk_\frakl=\mathfrak{der}(V)$ and hence $Te=T^\#e=0$ for any $T\in\frakk_\frakl$. Further, $\alpha=\1$ and $L(u-\alpha u)=0$ for all $u\in V$. Therefore, $(e,0,-e)\in Z(\frakk)$. If $V$ is non-euclidean, then there exists a non-zero element $u\in V^-$ and $L(u-\alpha u)=L(2u)\neq0$. Hence, $(e,0,-e)\notin Z(\frakk)$ which finishes the proof.
\end{proof}

\subsubsection{Exponential function and adjoint action}

The exponential function $\exp:\frakg\longrightarrow G_0$ is on the subspaces $\frakn$, $\frakl$ and $\overline{\frakn}$ given as follows:
\begin{align*}
 \exp(X) &= n_u & \mbox{for }X &= (u,0,0),\\
 \exp(X) &= e^T & \mbox{for }X &= (0,T,0),\\
 \exp(X) &= \overline{n}_v & \mbox{for }X &= (0,0,-v),
\end{align*}
where $\overline{n}_a:=jn_aj^{-1}\in\overline{N}=\theta(N)=\exp(\overline{\frakn})$\index{notation}{Nbar@$\overline{N}$}, $\overline{n}_a(x)=(x^{-1}-a)^{-1}$\index{notation}{nbara@$\overline{n}_a$}. Furthermore the adjoint action of $g\in L$ on $N$, $L$ and $\overline{N}$ is given by
\begin{align}
 \Ad(g)n_u &= n_{gu} & \mbox{for }u\in V,\label{eq:Adgn}\\
 \Ad(g)h &= ghg^{-1} &\mbox{for }h\in L,\label{eq:Adgl}\\
 \Ad(g)\overline{n}_v &= \overline{n}_{g^{-\#}v} & \mbox{for }v\in V.\label{eq:Adgnbar}
\end{align}
It follows that the adjoint action of $g\in L$ on the Lie algebra $\frakg$ writes as
\begin{align}
 \Ad(g)(u,T,v) &= (gu,gTg^{-1},g^{-\#}v) & \mbox{for }(u,T,v)\in\frakg.\label{eq:AdLieg}
\end{align}
We also need the adjoint action of $\exp(u,0,0)\in N$ on $(0,0,v)\in\overline{\frakn}$, $u,v\in V$:
\begin{align}
 \Ad(\exp(u,0,0))(0,0,v) &= e^{\ad(u,0,0)}(0,0,v)\notag\\
 &= (0,0,v)+(0,2u\Box v,0)+\frac{1}{2}(-2(u\Box v)u,0,0)\notag\\
 &= (-P(u)v,2u\Box v,v).\label{eq:AdNnbar}
\end{align}

\subsubsection{An $\sl(2)$-triple}

There is a natural homomorphism of $\SL(2,\RR)$ into the conformal group given by (cf. \cite[Proposition XI.2.1]{Ber00})
\begin{align}
 \phi: \SL(2,\RR)\rightarrow G_0,\ \phi\left(\begin{array}{cc}a&b\\c&d\end{array}\right)(x):=(ax+be)(cx+de)^{-1}.\label{eq:HomSL2Co}
\end{align}
Denote the corresponding homomorphism of Lie algebras by $\td\phi:\sl(2,\RR)\rightarrow\co(V)$. As usual, let
\begin{align}
 e :={}& \left(\begin{array}{cc}0&1\\0&0\end{array}\right), & f :={}& \left(\begin{array}{cc}0&0\\1&0\end{array}\right), & h :={}& \left(\begin{array}{cc}1&0\\0&-1\end{array}\right).\notag
\intertext{Put $E:=\td\phi(e)$, $F:=\td\phi(f)$ and $H:=\td\phi(h)$. Then}
 E ={}& (e,0,0), & F ={}& (0,0,e), & H ={}& (0,2\,\id,0)\label{eq:DefEF}\index{notation}{E@$E$}\index{notation}{F@$F$}
\end{align}
forms an $\sl(2)$-triple in $\frakg$. Further, we have
\begin{align}
\begin{split}
 e^{\frac{\pi}{2}(E-F)} &= e^{\frac{\pi}{2}\td\phi(e-f)} = \phi(e^{\frac{\pi}{2}(e-f)})\\
 &= \phi\left(\begin{array}{cc}0&1\\-1&0\end{array}\right) = j.
\end{split}\label{eq:jasexp}\\
\intertext{In particular, $j\in G_0$. Similarly}
 e^{\pi(E-F)} &= \phi\left(\begin{array}{cc}-1&0\\0&-1\end{array}\right)=\1.\label{eq:1asexp}
\end{align}

\subsubsection{The Killing form}

The Killing form on $\frakg$ is for $X_i=(u_i,T_i,v_i)$, $i=1,2$, given by (see e.g. \cite[Chapter I, proof of Proposition 7.1]{Sat80}):
\begin{align*}
 B(X_1,X_2) &= B_\frakl(T_1,T_2)+2\Tr(T_1T_2)+4\Tr\left(L(u_1v_2)\right)+4\Tr\left(L(u_2v_1)\right),
\end{align*}
where $B_\frakl$ denotes the Killing form on $\frakl$. Since $B$ is negative definite on the maximal compact subalgebra $\frakk$, we define an $\Ad$-invariant inner product on $\frakk$ by
\begin{align*}
 \langle X_1,X_2\rangle &:= -B(X_1,X_2).\index{notation}{1langlerangle@\textlangle$-,-$\textrangle}
\end{align*}
for $X_i=(u_i,T_i,-\alpha u_i)\in\frakk$, $i=1,2$. By Lemma \ref{lem:Trtr}:
\begin{align}
 \langle X_1,X_2\rangle &= B_\frakl(T_1,T_2^*)+2\Tr(T_1T_2^*)+\frac{8n}{r}(u_1|u_2).\label{eq:KillingForm}
\end{align}

\subsection{The universal covering}\label{sec:Guniversalcover}

For the connected group $G_0$ the universal covering $\widetilde{G_0}$ carries a natural Lie group structure which turns the covering map $\widetilde{G_0}\rightarrow G_0$ into a homomorphism. For the universal covering $\widetilde{G}$ of the (in general disconnected) group $G$ this is not clear. We now construct a group structure on $\widetilde{G}$.\\

If $\alpha\in G_0$, then $G=G_0$ and clearly $\widetilde{G}:=\widetilde{G_0}$\index{notation}{Gtilde@$\widetilde{G}$} is the universal cover of $G$. Denote by $\widetilde{\pr}:\widetilde{G}\rightarrow G$\index{notation}{prtilde@$\widetilde{\pr}$} the covering map. For later use we choose any $\widetilde{\alpha}\in\widetilde{G}$\index{notation}{alphatilde@$\widetilde{\alpha}$} which projects onto $\alpha$ under the covering map $\widetilde{\pr}$.

Now assume that $\alpha\in G\setminus G_0$. Then $\Ad(\alpha)$ defines isomorphisms $G_0\rightarrow G_0$ and $\frakg\rightarrow\frakg$ and by integration also $\widetilde{G_0}\rightarrow\widetilde{G_0}$. We put $\widetilde{G}:=\widetilde{G_0}\cup\widetilde{\alpha}\widetilde{G_0}$\index{notation}{Gtilde@$\widetilde{G}$}, where $\widetilde{\alpha}\widetilde{G_0}$ denotes the set of all formal products $\widetilde{\alpha}g$ with $g\in\widetilde{G_0}$. Then a product $\otimes$ on $\widetilde{G}$ can be defined as follows. For $g_1,g_2\in\widetilde{G_0}$ put
\begin{alignat*}{3}
 g_1\otimes g_2 &:= g_1g_2 &&\in \widetilde{G_0},\\
 \widetilde{\alpha}g_1\otimes g_2 &:= \widetilde{\alpha}(g_1g_2) &&\in \widetilde{\alpha}\widetilde{G_0},\\
 g_1\otimes\widetilde{\alpha}g_2 &:= \widetilde{\alpha}((\Ad(\alpha)g_1)g_2) &&\in \widetilde{\alpha}\widetilde{G_0},\\
 \widetilde{\alpha}g_1\otimes\widetilde{\alpha}g_2 &:= (\Ad(\alpha)g_1)g_2 &&\in \widetilde{G_0}.
\end{alignat*}
It is easy to see that $\widetilde{G}$ turns into a Lie group with two connected components. We also use the notation $\widetilde{\alpha}$ for the formal product $\widetilde{\alpha}\1\in\widetilde{\alpha}\widetilde{G_0}$\index{notation}{alphatilde@$\widetilde{\alpha}$}. The element $\widetilde{\alpha}$ is contained in the connected component of $\widetilde{G}$ which does not contain the identity. If $\widetilde{\pr_0}:\widetilde{G_0}\rightarrow G_0$ denotes the universal covering map of $G_0$, then the homomorphism\index{notation}{prtilde@$\widetilde{\pr}$}
\begin{align*}
 \widetilde{\pr}:\widetilde{G} &\rightarrow G,\\
 g &\mapsto \widetilde{\pr_0}(g), & \forall\, g\in\widetilde{G_0},\\
 \widetilde{\alpha}g &\mapsto \alpha\cdot\widetilde{\pr_0}(g), & \forall\, g\in\widetilde{G_0},
\end{align*}
is a universal covering of $G$.

In both cases we obtain a universal covering group $\widetilde{G}$ of $G$ with covering map $\widetilde{\pr}:\widetilde{G}\rightarrow G$ and an element $\widetilde{\alpha}\in\widetilde{G}$ which projects onto the Cartan involution $\alpha\in G$. Note that if one identifies the Lie algebra of $\widetilde{G}$ with $\frakg$, then
\begin{align}
 \Ad_\frakg(\widetilde{\alpha})=\Ad_\frakg(\alpha).\label{eq:Adgtildealpha}
\end{align}
We further observe that the group $\widetilde{K}:=\widetilde{\pr}^{-1}(K)$\index{notation}{Ktilde@$\widetilde{K}$} is a universal cover of $K$ since $K$ is a maximal compact subgroup of $G$. Note that in the euclidean case, the Lie algebra $\frakk$ of $K$ has non-trivial center by Lemma \ref{lem:CenterK} and hence $\widetilde{K}$ is not compact.

We further define\index{notation}{jtilde@$\widetilde{j}$}
\begin{align*}
 \widetilde{j} &:= \exp_{\widetilde{G}}\left(\frac{\pi}{2}(e,0,-e)\right)\in\widetilde{K}_0.
\end{align*}
By \eqref{eq:jasexp} we have $\widetilde{\pr}(\widetilde{j})=j$. Then the element define $\widetilde{w_0}:=\widetilde{\alpha}\widetilde{j}=\widetilde{j}\widetilde{\alpha}\in\widetilde{K}$\index{notation}{w0tilde@$\widetilde{w_0}$} projects onto $w_0=\alpha j$. (That $\widetilde{\alpha}$ and $\widetilde{j}$ commute follows from the definition of the multiplication on $\widetilde{G}$.)

\begin{lemma}\label{lem:Tildew0Central}
$\widetilde{w_0}\in Z(\widetilde{K})$.
\end{lemma}

\begin{proof}
First note that by \eqref{eq:Adgtildealpha} we have $\Ad_\frakk(\widetilde{w_0})=\Ad_\frakk(w_0)=\1$ and hence $w_0$ commutes with all elements in the identity component $\widetilde{K}_0$. If $\widetilde{K}=\widetilde{K}_0$ we are done. In the remaining case where $\widetilde{\alpha}\in\widetilde{K}\setminus\widetilde{K}_0$ we have
\begin{align*}
 \widetilde{w_0}(\widetilde{\alpha}k) &= \widetilde{\alpha}\widetilde{w_0}k = (\widetilde{\alpha}k)\widetilde{w_0} & \forall\, k\in\widetilde{K}_0,
\end{align*}
since both $\widetilde{w_0}$ and $\widetilde{\alpha}$ commute with all elements in $\widetilde{K}_0$. Hence $\widetilde{w_0}$ also commutes with all elements in $\widetilde{\alpha}\widetilde{K}_0$. Altogether, $\widetilde{w_0}$ commutes with every $k\in\widetilde{K}$ and is therefore central.
\end{proof}

\subsection{Root space decomposition}\label{subsec:ConfGrpRoots}

From now on assume that $V$ and $V^+$ are simple. Recall the definition \eqref{eq:DefLiea} of the abelian subalgebra
\begin{align*}
 \fraka = \left\{\sum_{i=1}^{r_0}{t_iL(c_i)}:t_i\in\RR\right\}\subseteq\frakl\subseteq\frakg.
\end{align*}
The set $\Sigma(\frakg,\fraka)$ of non-zero $\fraka$-weights of $\frakg$ forms a root system of type $C_{r_0}$. In fact,
\begin{align*}
 \Sigma(\frakg,\fraka) &= \left\{\frac{\pm\varepsilon_i\pm\varepsilon_j}{2}\right\},
\end{align*}
with $\varepsilon_j$ as in \eqref{eq:DefEpsilonj}. The root spaces are given by
\begin{alignat*}{2}
 \frakg_{\frac{\varepsilon_i+\varepsilon_j}{2}} &= (V_{ij},0,0) &&\subseteq \frakn,\\
 \frakg_\alpha &= (0,\frakl_\alpha ,0) &&\subseteq \frakl\ \ \ \ \ \ \ \mbox{for $\alpha\in\left\{\frac{\varepsilon_i-\varepsilon_j}{2}:i\neq j\right\}\cup\{0\}$,}\\
 \frakg_{-\frac{\varepsilon_i+\varepsilon_j}{2}} &= (0,0,V_{ij}) &&\subseteq \overline{\frakn}.
\end{alignat*}
We see that the constants $d$ and $e+1$ are exactly the multiplicities of the short and the long roots, respectively.

We also choose the maximal toral subalgebra
\begin{align*}
 \frakt &:= \left\{\left(\sum_{i=1}^{r_0}{t_ic_i},0,-\sum_{i=1}^{r_0}{t_ic_i}\right):t_i\in\RR\right\}\subseteq\frakk_\frakl^\perp\subseteq\frakk\index{notation}{t@$\frakt$}
\end{align*}
in the orthogonal complement of $\frakk_\frakl$ in $\frakk$. The corresponding root system of $(\frakg_\CC,\frakt_\CC)$ is again of type $C_{r_0}$ and given by
\begin{align*}
 \Sigma(\frakg_\CC,\frakt_\CC) &= \left\{\frac{\pm\gamma_i\pm\gamma_j}{2}\right\},
\end{align*}
where
\begin{align*}
 \gamma_j\left(\sum_{k=1}^{r_0}{t_kc_k},0,-\sum_{k=1}^{r_0}{t_kc_k}\right) &:= 2it_j.\index{notation}{gammai@$\gamma_i$}
\end{align*}
In fact, $i\frakt$ is the image of $\fraka$ under the Cayley transform $c=\exp(i\frac{\pi}{4}\ad(e,0,-e))$ of $\frakg$ (see e.g. \cite[\S0]{Sah93} for the euclidean case and \cite[Section 1.2]{DS99} for the non-euclidean case). Using Lemma \ref{lem:BracketZero}~(2) we find the root spaces
\begin{align*}
 (\frakg_\CC)_{\pm\frac{\gamma_i+\gamma_j}{2}} &= \{(u,\mp2iL(u),u):u\in(V_{ij})_\CC\},\\
 (\frakg_\CC)_{\pm\frac{\gamma_i-\gamma_j}{2}} &= \{(u,\pm4i[L(c_i),L(u)],-u):u\in(V_{ij})_\CC\}.
\end{align*}
Therefore the root spaces of $\frakt_\CC$ in $\frakk_\CC$ are given by
\begin{align*}
 (\frakk_\CC)_{\pm\frac{\gamma_i+\gamma_j}{2}} &= \{(u,\mp2iL(u),u):u\in (V^-_{ij})_\CC\},\\
 (\frakk_\CC)_{\pm\frac{\gamma_i-\gamma_j}{2}} &= \{(u,\pm4i[L(c_i),L(u)],-u):u\in (V^+_{ij})_\CC\}.
\end{align*}
From this one immediately obtains that the root system $\Sigma(\frakk_\CC,\frakt_\CC)$ is of type
\begin{align*}
 \begin{cases}
  A_{r_0-1} & \mbox{ if $V$ is euclidean,}\\
  C_{r_0} & \mbox{ if $V$ is non-euclidean non-reduced,}\\
  D_{r_0} & \mbox{ if $V$ is non-euclidean reduced.}
 \end{cases}
\end{align*}
We refer to these cases as case $A$, $C$ and $D$. For the half sum of all positive roots (with multiplicities $m_\gamma$) we find
\begin{align}
 \rho &= \frac{1}{2}\sum_{i<j}{m_{\frac{\gamma_i-\gamma_j}{2}}\frac{\gamma_i-\gamma_j}{2}}+\frac{1}{2}\sum_{i<j}{m_{\frac{\gamma_i+\gamma_j}{2}}\frac{\gamma_i+\gamma_j}{2}}+\frac{1}{2}\sum_i{m_{\gamma_i}\gamma_i}\notag\\
 &= \frac{d_0}{4}\sum_{i=1}^{r_0}{(r_0-2i+1)\gamma_i}+\frac{d-d_0}{4}(r_0-1)\sum_{i=1}^{r_0}{\gamma_i}+\frac{e}{2}\sum_{i=1}^{r_0}{\gamma_i}\notag\\
 &= \sum_{i=1}^{r_0}{\rho_i\gamma_i},\index{notation}{rho@$\rho$}\label{eq:DefRho}
\end{align}
where
\begin{align}
 \rho_i &= \frac{d_0}{4}(r_0-2i+1)+\frac{d-d_0}{4}(r_0-1)+\frac{e}{2} = \frac{1}{2}\left(\frac{n}{r_0}-1\right)-\frac{d_0}{2}(i-1).\label{eq:Rhoi}\index{notation}{rhoi@$\rho_i$}
\end{align}

\subsection{$\frakk$-representations with a $\frakk_\frakl$-spherical vector}

As previously remarked, $(\frakk,\frakk_\frakl)$ is a symmetric pair. Using the Cartan--Helgason theorem we can describe the highest weights of all unitary irreducible $\frakk$-representations which have a $\frakk_\frakl$-spherical vector.

\begin{proposition}\label{prop:klSphericalsReps}
The highest weight of an irreducible $\frakk$-representation with a $\frakk_\frakl$-spherical vector vanishes on the orthogonal complement of $\frakt$ in any maximal torus of $\frakk$ containing $\frakt$. The possible highest weights which give unitary irreducible $\frakk_\frakl$-spherical representations are precisely given by
\begin{align*}
 \Lambda_{\frakk_\frakl}^+(\frakk) = \begin{cases}\displaystyle\left\{\sum_{i=1}^{r_0}{t_i\gamma_i}:t_i\in\RR,\,t_i-t_j\in\ZZ,\,t_1\geq\ldots\geq t_{r_0}\right\} & \mbox{in case $A$,}\\\displaystyle\left\{\sum_{i=1}^{r_0}{t_i\gamma_i}:t_i\in\ZZ,\,t_1\geq\ldots\geq t_{r_0}\geq0\right\} & \mbox{in case $C$,}\\\displaystyle\left\{\sum_{i=1}^{r_0}{t_i\gamma_i}:t_i\in\frac{1}{2}\ZZ,\,t_i-t_j\in\ZZ,\,t_1\geq\ldots\geq t_{r_0-1}\geq|t_{r_0}|\right\} & \mbox{in case $D$.}\end{cases}\index{notation}{Lambdakklplus@$\Lambda_{\frakk_\frakl}^+(\frakk)$}
\end{align*}
Further, in each irreducible $\frakk_\frakl$-spherical $\frakk$-representation the space of $\frakk_\frakl$-spherical vectors is one-dimensional.
\end{proposition}

\begin{proof}
\begin{enumerate}
\item[\textup{(a)}] First, let $V$ be non-euclidean. Then $\frakk$ is semisimple by Lemma \ref{lem:CenterK}. By the Cartan--Helgason theorem (see e.g. \cite[Chapter V, Theorem 4.1]{Hel84}) the highest weights of all irreducible $\frakk$-representations with a $\frakk_\frakl$-spherical vector vanish on the orthogonal complement of $\frakt$ in any maximal torus of $\frakk$ containing $\frakt$ and are given by
\begin{align*}
 \Lambda_{\frakk_\frakl}^+(\frakk) = \left\{\gamma\in i\frakt^*:\frac{\langle\gamma,\alpha\rangle}{\langle\alpha,\alpha\rangle}\in\NN_0\ \forall\,\alpha\in\Sigma^+(\frakk_\CC,\frakt_\CC)\right\}.
\end{align*}
Here we have identified $\frakt_\CC^*$ with $\frakt_\CC$ via the bilinear form $\langle-,-\rangle$. Under this identification $\gamma_j$ corresponds to $\frac{r_0}{4n}i(c_j,0,-c_j)$. By Lemma \ref{lem:TrOnSubspace}~(2) we have
\begin{align*}
 (c_i|c_j) &= \delta_{ij}\tr(c_i) = \delta_{ij}\tr_{V_{ii}}(c_i)\\
 &= \delta_{ij}\rk(V_{ii}) = \delta_{ij}\frac{r}{r_0}
\end{align*}
and with \eqref{eq:KillingForm} we obtain
\begin{align}
 \langle\gamma_i,\gamma_j\rangle &= -\frac{r_0}{2n}\delta_{ij}.\label{eq:KappaGamma}
\end{align}
Thus, the previous observations imply the claim for the cases $C$ and $D$.
\item[\textup{(b)}] Now suppose $V$ is euclidean. We have $\frakk=Z(\frakk)+[\frakk,\frakk]$ with $Z(\frakk)=\RR(e,0,-e)$ by Lemma \ref{lem:CenterK} and $[\frakk,\frakk]$ semisimple. Clearly $\frakk_\frakl\subseteq[\frakk,\frakk]$ and the torus
\begin{align*}
 \left\{\sum_{i=1}^{r_0}{t_i(c_i,0,-c_i)}:\sum_{i=1}^{r_0}{t_i}=0\right\} \subseteq \frakk_\frakl^\perp \subseteq [\frakk,\frakk]
\end{align*}
is maximal in the orthogonal complement of $\frakk_\frakl$ in $[\frakk,\frakk]$. As in (a) it follows from the Cartan--Helgason theorem that
\begin{align*}
 \Lambda_{\frakk_\frakl}^+([\frakk,\frakk]) &= \left\{\sum_{i=1}^{r_0}{t_i\gamma_i}:\sum_{i=1}^{r_0}{t_i}=0,t_i-t_j\in\ZZ,t_1\geq\ldots\geq t_{r_0}\right\}
\end{align*}
Now, by Schur's Lemma, the irreducible representations of $\frakk=Z(\frakk)+[\frakk,\frakk]$ are irreducible $[\frakk,\frakk]$-representations where the center $Z(\frakk)$ acts by scalars. Therefore,
\begin{align*}
 \Lambda_{\frakk_\frakl}^+(\frakk) &= \Lambda_{\frakk_\frakl}^+([\frakk,\frakk])+\RR(\gamma_1+\ldots+\gamma_{r_0})
\end{align*}
which shows the claim for the case $A$.
\end{enumerate}
That in each irreducible $\frakk$-representation the space of $\frakk_\frakl$-spherical vectors is at most one-dimensional follows from \cite[remark at the beginning of Chapter V, \S 4.2]{Hel84}). This finishes the proof.
\end{proof}

For $\alpha\in\Lambda_{\frakk_\frakl}^+(\frakk)$ we denote by $E^\alpha$\index{notation}{Ealpha@$E^\alpha$} the irreducible $\frakk_\frakl$-spherical representation of $\frakk$ with highest weight $\alpha$.

\section{The Bessel operators}\label{sec:BesselOp}

In this section we introduce second order differential operators $\calB_\lambda$ ($\lambda\in\CC$) on $V$. These operators are needed later to describe the Lie algebra action of the minimal representation. We show that for $\lambda\in\calW$ the operator $\calB_\lambda$ is tangential to the orbit $\calO_\lambda$ and defines a symmetric operator on $L^2(\calO_\lambda,\td\mu_\lambda)$.

\subsection{Definition and Properties}

For any complex parameter $\lambda\in\CC$ we define a second order differential operator $\calB_\lambda:C^\infty(V)\longrightarrow C^\infty(V)\otimes V$ called the \textit{Bessel operator}\index{subject}{Bessel operator}, mapping complex-valued functions to vector-valued functions, by
\begin{align}
 \calB_\lambda := P\left(\frac{\partial}{\partial x}\right)x+\lambda\frac{\partial}{\partial x}.\label{eq:BesselOp}\index{notation}{Blambda@$\calB_\lambda$}
\end{align}
Here $\frac{\partial}{\partial x}:C^\infty(V)\longrightarrow C^\infty(V)\otimes V$\index{notation}{ddx@$\frac{\partial}{\partial x}$} denotes the gradient with respect to the non-degenerate trace form $\tau$ on $V$. This means that
\begin{align*}
 \tau\left(u,\frac{\partial f}{\partial x}\right) &= D_uf(x) = \left.\frac{\td}{\td t}\right|_{t=0}f(x+tu) & \forall\, u\in V.
\end{align*}
Therefore, if $(e_\alpha)_\alpha$ is a basis of $V$ with dual basis $(\overline{e}_\alpha)_\alpha$ with respect to the trace form $\tau$, then for $f\in C^\infty(V)$ we have
\begin{align*}
 \frac{\partial f}{\partial x} &= \sum_\alpha{\frac{\partial f}{\partial x_\alpha}\overline{e}_\alpha}.
\end{align*}
Inserting this in \eqref{eq:BesselOp} yields the following expression of $\calB_\lambda$ in coordinates:
\begin{align*}
 \calB_\lambda f(x) &= \sum_{\alpha,\beta}{\frac{\partial^2f}{\partial x_\alpha\partial x_\beta}P(\overline{e}_\alpha,\overline{e}_\beta)x}+\lambda\sum_\alpha{\frac{\partial f}{\partial x_\alpha}\overline{e}_\alpha}, & x\in V.
\end{align*}

First, we prove the following product rule for the Bessel operators which is an easy consequence of the definition.

\begin{lemma}\label{lem:BnuProdRule}
For $f,g\in C^\infty(V)$ we have
\begin{align}
 \calB_\lambda(f\cdot g) &= \calB_\lambda f\cdot g+2P\left(\frac{\partial f}{\partial x},\frac{\partial g}{\partial x}\right)x+f\cdot\calB_\lambda g.\label{eq:BnuProdRule}
\end{align}
\end{lemma}

The Bessel operator $\calB_\lambda$ satisfies an equivariance property with respect to the action of $\Str(V)$. Recall that $\ell$ denotes the left-regular representation of $\Str(V)$ on functions which are defined on $V$ (see \eqref{eq:DefEll}).

\begin{lemma}[{\cite[Proposition XV.2.3~(i)]{FK94}}]\label{lem:BnuEquiv}
For any $g\in\Str(V)$:
\begin{align*}
 \ell(g^{-1})\calB_\lambda\ell(g) &= g^{-\#}\calB_\lambda.
\end{align*}
\end{lemma}

\begin{proof}
If $F=\ell(g)f$, then by the chain rule
\begin{align*}
 \frac{\partial F}{\partial x} &= g^{-\#}\frac{\partial f}{\partial x}(g^{-1}x).
\end{align*}
Therefore,
\begin{align*}
 \calB_\lambda F(x) &= \left(P\left(g^{-\#}\frac{\partial}{\partial x}\right)f\right)(g^{-1}x)x+\lambda g^{-\#}\frac{\partial f}{\partial x}(g^{-1}x).
\end{align*}
By \eqref{eq:QRepStrGrp},
\begin{align*}
 P\left(g^{-\#}\frac{\partial}{\partial x}\right)=g^{-\#}P\left(\frac{\partial}{\partial x}\right)g^{-1},
\end{align*}
and the result follows.
\end{proof}

We also need the action of $\calB_\lambda$ on powers of the Jordan determinant.

\begin{lemma}[{\cite[Proposition XV.2.4]{FK94}}]\label{lem:BnuDelta}
\begin{align}
 \calB_\lambda\Delta(x)^\mu &= \mu\left(\mu+\lambda-\frac{n}{r}\right)\Delta(x)^\mu x^{-1}.\label{eq:BnuDelta}
\end{align}
\end{lemma}

\begin{proof}
The first and second derivatives of $\Delta(x)^\mu$ are given by
\begin{align*}
 \frac{\partial\Delta(x)^\mu}{\partial x_\alpha}(x) &= \mu\Delta(x)^\mu\tau(x^{-1},e_\alpha),\\
 \frac{\partial^2\Delta(x)^\mu}{\partial x_\alpha\partial x_\beta}(x) &= \mu^2\Delta(x)^\mu\tau(x^{-1},e_\alpha)\tau(x^{-1},e_\beta)-\mu\Delta(x)^\mu\tau(P(x)^{-1}e_\alpha,e_\beta).
\end{align*}
Since
\begin{align*}
 &\sum_\alpha{\tr(e_\alpha)\overline{e}_\alpha} = e,\\
 &\sum_\alpha{e_\alpha\cdot\overline{e}_\alpha} = \frac{n}{r}e,
\end{align*}
by Lemma \ref{lem:SquareSums}, it follows that
\begin{align*}
 \calB_\lambda\Delta^\mu(e) &= \mu\left(\mu+\lambda-\frac{n}{r}\right)e.
\end{align*}
In order to obtain the value of $\calB_\lambda\Delta^\mu$ at $x=ge$, we use the equivariance property of $\Delta$:
\begin{align*}
 \ell(g^{-1})\Delta(x) &= \Delta(ge)\Delta(x),
\end{align*}
and the equivariance property of $\calB_\lambda$ in Lemma \ref{lem:BnuEquiv}:
\begin{align*}
 \calB_\lambda\Delta^\mu(ge) &= \left(\ell(g^{-1})\calB_\lambda\Delta^\mu\right)(e)\\
 &= g^{-\#}\left(\calB_\lambda\ell(g^{-1})\Delta^\mu\right)(e)\\
 &= \Delta(ge)^\mu g^{-\#}\calB_\lambda\Delta^\mu(e)\\
 &= \mu\left(\mu+\lambda-\frac{n}{r}\right)\Delta(ge)^\mu(ge)^{-1}.
\end{align*}
This proves \eqref{eq:BnuDelta} for every $x=ge$ in the open orbit of the structure group containing the identity $e$. Since both sides of \eqref{eq:BnuDelta} are polynomials in $x$ the claim follows.
\end{proof}

\subsection{Symmetric operators}

The crucial part in proving that $\calB_\lambda$ is tangential to $\calO_\lambda$ and defines a symmetric operator on $L^2(\calO_\lambda,\td\mu_\lambda)$, is the following proposition for the corresponding zeta functions:

\begin{proposition}[{\cite[Proposition XV.2.4]{FK94}}]\label{prop:BnuSA}
Let $Z(f,\lambda)$ denote the zeta function corresponding to $V$ as in Proposition \ref{prop:ZetaFunctions}. Then, for $f,g\in\calS(V)$ and $\lambda\in\CC$ we have
\begin{align*}
 Z((\calB_\lambda f)\cdot g,\lambda) &= Z(f\cdot(\calB_\lambda g),\lambda),
\end{align*}
as identity of meromorphic functions in $\lambda$.
\end{proposition}

\begin{proof}
It suffices to prove the statement for $\lambda>\frac{n}{r}+2>(r_0-1)\frac{r_0d}{2r}$, then the general statement follows by meromorphic continuation. In this case
\begin{align*}
 Z(f,\lambda) &= \sum_{j=1}^s{\int_{\Omega_j}{f(x)|\Delta(x)|^{\lambda-\frac{n}{r}}\td x}},
\end{align*}
where $\Omega_j$ denote certain open orbits of $L_0$. (If $V$ is euclidean or $V=\RR^{p,q}$, then $s=1$ and $\Omega_1=\Omega$, and if $V$ is non-euclidean $\ncong\RR^{p,q}$, then $(\Omega_j)_j$ is the set of all open $L_0$-orbits.) Therefore, it is enough to show that for any open orbit $\Omega_j$ of $L_0$ we have
\begin{align*}
 \int_{\Omega_j}{(\calB_\lambda f(x))g(x)|\Delta(x)|^{\lambda-\frac{n}{r}}\td x} &= \int_{\Omega_j}{f(x)(\calB_\lambda g(x))|\Delta(x)|^{\lambda-\frac{n}{r}}\td x}
\end{align*}
On every orbit $\Omega_j$ the Jordan determinant $\Delta(x)$ is either positive or negative. Since $\lambda>\frac{n}{r}+2$, we have $|\Delta|^{\lambda-\frac{n}{r}}\in C^2(\overline{\Omega_j})$ and all derivatives of $|\Delta|^{\lambda-\frac{n}{r}}$ up to second order vanish on $\partial\Omega_j$ (use Lemma \ref{lem:DerivDelta}). This means that all boundary terms, which occur when integrating by parts twice, vanish.
\begin{enumerate}
\item[\textup{(a)}] Using integration by parts, we first prove that if all derivatives of $g$ up to second order vanish on $\partial\Omega_j$, then
 \begin{align*}
  \int_{\Omega_j}{(\calB_\lambda f(x))g(x)\td x} &= \int_{\Omega_j}{f(x)(\calB_{\frac{2n}{r}-\lambda}g(x))\td x}.
 \end{align*}
 For this we choose an orthonormal basis $(e_\alpha)_\alpha$ with $e_\alpha\in V^+\cup V^-$. Observe that
 \begin{align*}
  \frac{\partial^2}{\partial x_\alpha\partial x_\beta}(xg(x)) &= x\frac{\partial^2g}{\partial x_\alpha\partial x_\beta}(x) + e_\alpha\frac{\partial g}{\partial x_\beta}(x) + e_\beta\frac{\partial g}{\partial x_\alpha}(x)
 \end{align*}
 and hence, by Lemma \ref{lem:SquareSums}:
 \begin{align*}
  & \sum_{\alpha,\beta}{P(\overline{e}_\alpha,\overline{e}_\beta)\frac{\partial^2}{\partial x_\alpha\partial x_\beta}(xg(x))}\\
  ={}& \calB_0g(x) + 2\sum_{\alpha,\beta}{P(\overline{e}_\alpha,\overline{e}_\beta)e_\alpha\frac{\partial g}{\partial x_\beta}(x)}\\
  ={}& \calB_0g(x) + 2\sum_\beta{\left(\sum_{e_\alpha\in V^+}{P(e_\alpha,\overline{e}_\beta)e_\alpha}-\sum_{e_\alpha\in V^-}{P(e_\alpha,\overline{e}_\beta)e_\alpha}\right)\frac{\partial g}{\partial x_\beta}(x)}\\
  ={}& \calB_0g(x) + 2\sum_\beta{\left(\sum_{e_\alpha\in V^+}{e_\alpha^2}-\sum_{e_\alpha\in V^-}{e_\alpha^2}\right)\overline{e}_\beta\frac{\partial g}{\partial x_\beta}(x)}\\
  ={}& \calB_0g(x) + 2\sum_\beta{\left(\sum_\alpha{e_\alpha\cdot\overline{e}_\alpha}\right)\overline{e}_\beta\frac{\partial g}{\partial x_\beta}(x)}\\
  ={}& \calB_0g(x) + \frac{2n}{r}\frac{\partial g}{\partial x}(x).
 \end{align*}
 Therefore, integration by parts gives
 \begin{align*}
  & \int_{\Omega_j}{\calB_\lambda f(x)\cdot g(x)\td x}\\
  ={}& \int_{\Omega_j}{f(x)\cdot\left(\sum_{\alpha,\beta}{P(\overline{e}_\alpha,\overline{e}_\beta)\frac{\partial^2}{\partial x_\alpha\partial x_\beta}(xg(x))}-\lambda\sum_{\alpha}{\frac{\partial g}{\partial x_\alpha}\overline{e}_\alpha}\right)\td x}\\
  ={}& \int_{\Omega_j}{f(x)\cdot\left(\calB_0g(x)+\frac{2n}{r}\frac{\partial g}{\partial x}(x)-\lambda\frac{\partial g}{\partial x}(x)\right)\td x}\\
  ={}& \int_{\Omega_j}{f(x)\cdot\calB_{\frac{2n}{r}-\lambda}g(x)\td x}.
 \end{align*}
\item[\textup{(b)}] Now we prove that
 \begin{align*}
  \calB_\lambda(|\Delta(x)|^\mu f(x)) &= |\Delta(x)|^\mu\left(\calB_{\lambda+2\mu}f(x)+\mu\left(\mu+\lambda-\frac{n}{r}\right)x^{-1}f(x)\right).
 \end{align*}
  First assume that $\Omega_j$ is an orbit with $\Delta(x)>0$ for all $x\in\Omega_j$. Then $|\Delta(x)|^\mu=\Delta(x)^\mu$ and with the Lemmas \ref{lem:BnuProdRule}, \ref{lem:DerivDelta} and \ref{lem:BnuDelta} we have
 \begin{align*}
  \calB_\lambda(\Delta(x)^\mu f(x)) &= \Delta(x)^\mu\calB_\lambda f(x) + 2P\left(\frac{\partial\Delta^\mu}{\partial x}(x),\frac{\partial f}{\partial x}(x)\right)x\\
  & \ \ \ \ \ \ \ \ \ \ \ \ \ \ \ \ \ \ \ \ \ \ \ \ \ \ \ \ \ \ \ \ \ \ \ \ \ \ \ \ \ \ \ \ \ \ \ \ \ \ \ \ \ + \calB_\lambda(\Delta(x)^\mu)f(x)\\
  &= \Delta(x)^\mu\calB_\lambda f(x) + 2\mu\Delta(x)^\mu P\left(x^{-1},\frac{\partial f}{\partial x}(x)\right)x\\
  & \ \ \ \ \ \ \ \ \ \ \ \ \ \ \ \ \ \ \ \ \ \ \ \ \ \ \ \ \ \ \ \ \ \ +\mu\left(\mu+\lambda-\frac{n}{r}\right)\Delta(x)^\mu x^{-1}f(x)\\
  &= \Delta(x)^\mu\left(\calB_{\lambda+2\mu}f(x)+\mu\left(\mu+\lambda-\frac{n}{r}\right)x^{-1}f(x)\right).
 \end{align*}
 Now, if $\Delta(x)<0$ for all $x\in\Omega_j$, then $|\Delta(x)|^\mu=(-\Delta(x))^\mu$ and the same calculation can be carried out.
\item[\textup{(c)}] We now prove the main statement. By (a)
 \begin{align*}
  \int_{\Omega_j}{(\calB_\lambda f(x))g(x)|\Delta(x)|^{\lambda-\frac{n}{r}}\td x} &= \int_{\Omega_j}{f(x)\calB_{\frac{2n}{r}-\lambda}\left(g(x)|\Delta(x)|^{\lambda-\frac{n}{r}}(x)\right)\td x}\\
 \intertext{and by (b)}
  &= \int_{\Omega_j}{f(x)(\calB_\lambda g(x))|\Delta(x)|^{\lambda-\frac{n}{r}}(x)\td x}.
 \end{align*}
\end{enumerate}
Summing over $j=1,\ldots,s$ shows the claim.
\end{proof}

Using the previous proposition we can now prove the main result of this section:

\begin{theorem}\label{thm:BlambdaTangential}
For every $\lambda\in\calW$ the differential operator $\calB_\lambda$ is tangential to the orbit $\calO_\lambda$ and defines a symmetric operator on $L^2(\calO_\lambda,\td\mu_\lambda)$.
\end{theorem}

\begin{proof}
If $\lambda>(r_0-1)\frac{r_0d}{2r}$, then the orbit $\calO_\lambda=\Omega$ is open and every differential operator is tangential. Symmetry follows immediately from Proposition \ref{prop:BnuSA}.\\
Now assume that $\lambda=k\frac{r_0d}{2r}$, $0\leq k\leq r_0-1$. Let $\varphi\in C_c^\infty(\calO_\lambda)$ and let $\widetilde{\varphi}_1,\widetilde{\varphi}_2\in C_c^\infty(U)$ be any extensions of $\varphi$ to an open neighborhood $U$ of $\calO_\lambda$. To show that $\calB_\lambda$ is tangential to $\calO_\lambda$ we need to show that $\calB_\lambda\varphi_1=\calB_\lambda\varphi_2$ on $\calO_\lambda$. By definition $\widetilde{\varphi}:=\widetilde{\varphi}_1-\widetilde{\varphi}_2$ vanishes on $\calO_\lambda$. For any $\psi\in C_c^\infty(U)$ we obtain with Proposition \eqref{prop:BnuSA}:
\begin{align*}
 \int_{\calO_\lambda}{\calB_\lambda\widetilde{\varphi}\cdot\psi\td\mu_\lambda} &= \const\cdot\res_{\mu=\lambda}{Z\left(\calB_\mu\widetilde{\varphi}\cdot\psi,\mu\right)}\\
 &= \const\cdot\res_{\mu=\lambda}{Z\left(\widetilde{\varphi}\cdot\calB_\mu\psi,\mu\right)}\\
 &= \int_{\calO_\lambda}{\widetilde{\varphi}\cdot\calB_\lambda\psi\td\mu_\lambda} = 0.
\end{align*}
Hence $\calB_\lambda\widetilde{\varphi}=0$ in $L^2(\calO_\lambda,\td\mu_\lambda)$ which implies $\calB_\lambda\widetilde{\varphi}(x)=0$ for every $x\in\calO_\lambda$. But this means that $\calB_\lambda\widetilde{\varphi}_1=\calB_\lambda\widetilde{\varphi}_2$ on $\calO_\lambda$ and therefore $\calB_\lambda$ is tangential to $\calO_\lambda$. Symmetry now again follows from Proposition \ref{prop:BnuSA}. This finishes the proof.
\end{proof}

\subsection{Action for the minimal orbit}

Now let $\lambda=\lambda_1=\frac{r_0d}{2r}$ be the minimal discrete non-zero Wallach point. Let us compute the action of $\calB_{\lambda_1}$ on radial functions on the minimal orbit $\calO_1$, i.e. functions depending only on $\|x\|:=\sqrt{(x|x)}$\index{notation}{1absvala@{"|}\hspace{-1pt}{"|}$-${"|}\hspace{-1pt}{"|}}. For convenience we use the following normalization:
\begin{align*}
 |x| := \sqrt{\frac{r}{r_0}}\|x\|.\index{notation}{1absvalb@{"|}$-${"|}}
\end{align*}
In view of Lemma \ref{lem:TrOnSubspace}~(2) we then have for $i=1,\ldots,r_0$:
\begin{align}
 |c_i| &= \sqrt{\frac{r}{r_0}\tr(c_i)} = \sqrt{\frac{r}{r_0}\tr_{V_{ii}}(c_i)} = \sqrt{\frac{r}{r_0}\rk(V_{ii})} = \frac{r}{r_0}.\label{eq:NormCi}
\end{align}
Further, if $\psi(x)=f(|x|)$, $x\in V$, is a radial function, then
\begin{align}
 \frac{\partial\psi}{\partial x}(x) &= \frac{r}{r_0}\frac{f'(|x|)}{|x|}\alpha x.\label{eq:ddxradial}
\end{align}

\begin{proposition}\label{prop:BnuRadial}
If $\psi(x)=f(|x|)$, $x\in\calO_1$, is a radial function on $\calO_1$, $f\in C^\infty(\RR_+)$, then for $x=ktc_1$ we have
\begin{align*}
 \calB_{\lambda_1}\psi(x) &= \left(f''(|x|)+\left(d-d_0-e\right)\frac{1}{|x|}f'(|x|)\right)\alpha x+\frac{r_0}{r}\left(d_0-\frac{d}{2}\right)f'(|x|)\alpha(ke).
\end{align*}
\end{proposition}

\begin{proof}
We extend $\psi$ to $V\setminus\{0\}$ by
\begin{align*}
 \psi(x) &:= f(|x|), & x\neq0.
\end{align*}
Now, for an orthonormal basis $(e_\alpha)_\alpha$ with respect to $(-|-)$ we put $\overline{e}_\alpha:=\alpha(e_\alpha)$. Then $(\overline{e}_\alpha)_\alpha$ is dual to $(e_\alpha)_\alpha$ with respect to the trace form $\tau$ and for $x=\sum_\alpha{x_\alpha e_\alpha}$ we have $\|x\|^2=\sum_\alpha{x_\alpha^2}$. Thus we can calculate
\begin{align*}
 \calB_{\lambda_1}\psi(x) &= \sum_{\alpha,\beta}{\frac{\partial^2\psi}{\partial x_\alpha\partial x_\beta}(x)P(\overline{e}_\alpha,\overline{e}_\beta)x} + \lambda_1\sum_\alpha{\frac{\partial\psi}{\partial x_\alpha}(x)\overline{e}_\alpha}\\
 &= \frac{r}{r_0}\sum_\alpha{\frac{\partial}{\partial x_\alpha}\left[\frac{x_\beta}{|x|}f'(|x|)\right]P(\overline{e}_\alpha,\overline{e}_\beta)x} + \frac{r}{r_0}\lambda_1\sum_\alpha{\frac{x_\alpha}{|x|}f'(|x|)\overline{e}_\alpha}\\
 &= \frac{r}{r_0}\sum_\alpha{\left[\frac{\delta_{\alpha,\beta}}{|x|}f'(|x|)-\frac{r}{r_0}\frac{x_\alpha x_\beta}{|x|^3}f'(|x|)+\frac{r}{r_0}\frac{x_\alpha x_\beta}{|x|^2}f''(|x|)\right]P(\overline{e}_\alpha,\overline{e}_\beta)x}\\
 & \ \ \ \ \ \ \ \ \ \ \ \ \ \ \ \ \ \ \ \ \ \ \ \ \ \ \ \ \ \ \ \ \ \ \ \ \ \ \ \ \ \ \ \ \ \ \ \ \ \ \ \ \ \ \ \ \ \ \ \ \ \ \ \ \ \ \ \  + \frac{r}{r_0}\lambda_1\frac{\alpha(x)}{|x|}f'(|x|)\\
 &= \left(\frac{r}{r_0}\right)^2\frac{P(\alpha(x))x}{|x|^2}f''(|x|)+\frac{r}{r_0}\left(\sum_\alpha{\frac{P(\overline{e}_\alpha)x}{|x|}}-\frac{r}{r_0}\frac{P(\alpha(x))x}{|x|^3}\right.\\
 & \ \ \ \ \ \ \ \ \ \ \ \ \ \ \ \ \ \ \ \ \ \ \ \ \ \ \ \ \ \ \ \ \ \ \ \ \ \ \ \ \ \ \ \ \ \ \ \ \ \ \ \ \ \ \ \ \ \ \ \ \ \ \ \ \ \ \ \ +\left.\lambda_1\frac{\alpha(x)}{|x|}\right)f'(|x|).
\end{align*}
Now $\calO_1=K_L\RR_+c_1$, and since $\psi$ is $K_L$-invariant, we obtain, using the equivariance property of Lemma \ref{lem:BnuEquiv}:
\begin{align*}
\calB_{\lambda_1}\psi(ktc_1) &= \ell(k^{-1})\calB_{\lambda_1}\ell(k)\psi(tc_1) = k^{-\#}\calB_{\lambda_1}\psi(tc_1) = \alpha k\alpha\calB_{\lambda_1}\psi(tc_1).
\end{align*}
Therefore it suffices to compute $\calB_{\lambda_1}\psi(tc_1)$ for $t>0$. With Lemma \ref{lem:QRepC} and \eqref{eq:NormCi} we calculate:
\begin{align*}
 \calB_{\lambda_1}\psi(tc_1) &= tc_1f''(t|c_1|)+\left(\sum_\alpha{P(\overline{e}_\alpha)c_1}+\left(\lambda_1-\frac{r_0}{r}\right)c_1\right)f'(t|c_1|)\\
 &= f''(t|c_1|)tc_1+\frac{r_0}{r}\left(\frac{d}{2}-d_0-e+\frac{r}{r_0}\lambda_1\right)f'(t|c_1|)c_1\\
 & \ \ \ \ \ \ \ \ \ \ \ \ \ \ \ \ \ \ \ \ \ \ \ \ \ \ \ \ \ \ \ \ \ \ \ \ \ \ \ \ \ \ \ \ \ \ \ \ \ \ \ \ \ \ \ +\frac{r_0}{r}\left(d_0-\frac{d}{2}\right)f'(t|c_1|)e.
\end{align*}
Finally, for $x=ktc_1$ we have $|x|=t|c_1|$ and we obtain with $\lambda_1=\frac{r_0d}{2r}$
\begin{align*}
 \calB_{\lambda_1}\psi(x) &= \alpha k\alpha\calB_\lambda\psi(tc_1)\\
 &= \left(f''(|x|)+\left(d-d_0-e\right)\frac{1}{|x|}f'(|x|)\right)\alpha x+\frac{r_0}{r}\left(d_0-\frac{d}{2}\right)f'(|x|)\alpha(ke).
\end{align*}
This is the stated formula.
\end{proof}

The formula in Proposition \ref{prop:BnuRadial} can be simplified if one assumes that $V$ is either euclidean or non-euclidean of rank $\geq3$. (The remaining case is by Proposition \ref{prop:ClassificationEuclSph} $V=\RR^{p,q}$ which is treated separately in Appendix \ref{app:Rank2}.) For this we introduce the ordinary differential operator $B_\alpha$ on $\RR_+$ which is defined by
\begin{align}
 B_\alpha f(t) := f''(t)+(2\alpha+1)\frac{1}{t}f'(t)-f(t).\label{eq:OrdinaryBesselOp}\index{notation}{Balpha@$B_\alpha$}
\end{align}

\begin{corollary}\label{cor:BnuRadial}
Let $\psi(x)=f(|x|)$, $x\in\calO_1$, be a radial function on $\calO_1$.
\begin{enumerate}
\item[\textup{(1)}] If $V$ is euclidean, then
 \begin{align*}
  (\calB_{\lambda_1}-\alpha x)\psi(x) &= B_{\frac{\nu}{2}}f(|x|)\alpha x+\frac{d}{2}f'(|x|)e.
 \end{align*}
\item[\textup{(2)}] If $V$ is non-euclidean of rank $\geq3$, then
 \begin{align*}
  (\calB_{\lambda_1}-\alpha x)\psi(x) &= B_{\frac{\nu}{2}}f(|x|)\alpha x.
 \end{align*}
\end{enumerate}
\end{corollary}

\begin{proof}
By Proposition \ref{prop:BnuRadial} we have for $x=ktc_1$:
\begin{align*}
 (\calB_{\lambda_1}-\alpha x)\psi(x) &= B_{\frac{d-d_0-e-1}{2}}f(|x|)\alpha x+\frac{r_0}{r}\left(d_0-\frac{d}{2}\right)f'(|x|)\alpha(ke).
\end{align*}
\begin{enumerate}
\item[\textup{(1)}] If $V$ is euclidean, then $\nu=-1=d-d_0-e-1$. Further, $K_L=H\subseteq\Aut(V)$ and hence $ke=e$ for $k\in K_L$. Therefore,
 \begin{align*}
  (\calB_{\lambda_1}-\alpha x)\psi(x) &= B_{\frac{\nu}{2}}f(|x|)\alpha x+\frac{d}{2}f'(|x|)e.
 \end{align*}
\item[\textup{(2)}] If $V$ is non-euclidean of rank $\geq3$, then $d=2d_0$ and $\nu=d-d_0-e-1$. Thus, we obtain
 \begin{align*}
  (\calB_{\lambda_1}-\alpha x)\psi(x) &= B_{\frac{\nu}{2}}f(|x|)\alpha x.\qedhere
 \end{align*}
\end{enumerate}
\end{proof}

\begin{remark}
The normalized $I$- and $K$-Bessel functions $\widetilde{I}_\alpha(t)$ and $\widetilde{K}_\alpha(t)$ solve the differential equation $B_\alpha u=0$ (see Appendix \ref{app:BesselFunctions}). This is why we call $\calB_\lambda$ the Bessel operators. Since $\widetilde{K}_\alpha(t)$ decays exponentially as $t\rightarrow\infty$, it is used in the next chapter to construct an the $L^2$-model of the minimal representation of a finite cover of the group $G$.
\end{remark}
\chapter{Minimal representations of conformal groups}\label{ch:MinRep}

In this chapter we construct the minimal representation of a finite cover of the conformal group $G$. First, we construct its underlying $(\frakg,\frakk)$-module. We then show that it can be integrated to a unitary irreducible representation of a finite cover of $G$ on $L^2(\calO,\td\mu)$. To motivate the definition of the Lie algebra action we show that it arises by taking the Fourier transform of the action of a certain principal series representation.

Further, we prove that the $\frakk$-Casimir acts on the subspace of radial functions as a fourth order differential operator which will be studied in detail in Chapter \ref{ch:GenLagFct}. We also introduce a unitary operator $\calF_\calO$ on $L^2(\calO,\td\mu)$ which resembles the euclidean Fourier transform.

Throughout this chapter $V$ will always denote a simple real Jordan algebra, $\alpha$ a Cartan involution on $V$ and we further assume that $V^+$ is simple.

\section{Construction of the minimal representation}\label{sec:MinRepConstruction}

We first construct a representation of $\frakg$ on $C^\infty(\calO_\lambda)$ for any $\lambda\in\calW$. For the minimal non-zero discrete Wallach point $\lambda=\lambda_1=\frac{r_0d}{2r}$ we then define a subrepresentation $W$ of $C^\infty(\calO_1)$ which is contained in $L^2(\calO_1,\td\mu_1)$. Finally we show that $W$ can be integrated to a unitary irreducible representation of a finite cover of $G$ on the Hilbert space $L^2(\calO_1,\td\mu_1)$. For the special cases $V=\Sym(n,\RR)$ and $V=\RR^{p,q}$ we identify this representation with known representations.

\subsection{Infinitesimal representations on $C^\infty(\calO_\lambda)$}\label{sec:InfinitesimalRep}

On each Hilbert space $L^2(\calO_\lambda,\td\mu_\lambda)$, $\lambda\in\calW$, we define a representation $\rho_\lambda$\index{notation}{rholambda@$\rho_\lambda$} of the parabolic subgroup $P$ by
\begin{align}
 \rho_\lambda(n_a)\psi(x) &:= e^{i(x|a)}\psi(x) & n_a &\in N,\label{eq:L2Rep1}\\
 \rho_\lambda(g)\psi(x) &:= \chi(g^*)^{\frac{\lambda}{2}}\psi(g^*x) & g &\in L\label{eq:L2Rep2}
\end{align}
for $\psi\in L^2(\calO_\lambda,\td\mu_\lambda)$.

\begin{proposition}\label{prop:MackeyRep}
For $\lambda\in\calW$ the representation $\rho_\lambda$ of $P$ on $L^2(\calO_\lambda,\td\mu_\lambda)$ is unitary and irreducible, even if restricted to the identity component $P_0$ of $P$.
\end{proposition}

\begin{proof}
Clearly, the operators $\rho_\lambda(n_a)$, $a\in V$, are unitary on $L^2(\calO_\lambda,\td\mu_\lambda)$. Unitarity of the $L$-action follows from \eqref{eq:dmulambdaEquivariance} and hence $\rho_\lambda$ defines a unitary representation of $P$ on $L^2(\calO_\lambda,\td\mu_\lambda)$. It remains to show irreducibility.\\
For this we use Schur's Lemma. Suppose $A$ is a unitary operator on $L^2(\calO_\lambda,\td\mu_\lambda)$ which intertwines the $P_0$-action. Since $A$ intertwines the $N$-action, we have
\begin{align*}
 \int_{\calO_\lambda}{e^{-i(x|a)}A\phi(x)\overline{\psi(x)}\td\mu_\lambda(x)} &= \int_{\calO_\lambda}{A(e^{-i(-|a)}\phi)(x)\overline{\psi(x)}\td\mu_\lambda(x)}\\
 &= \int_{\calO_\lambda}{e^{-i(x|a)}\phi(x)\overline{A\psi(x)}\td\mu_\lambda(x)}
\end{align*}
for all $\phi,\psi\in L^2(\calO_\lambda,\td\mu_\lambda)$ and every $a\in V$. This means, that the euclidean Fourier transforms of the tempered distributions $A\phi(x)\overline{\psi(x)}\td\mu_\lambda(x)$ and $\phi(x)\overline{A\psi(x)}\td\mu_\lambda(x)$ agree in $\calS'(V)$. The Fourier transform is an isomorphism of $\calS'(V)$ and hence,
\begin{align*}
 A\phi(x)\overline{\psi(x)} &= \phi(x)\overline{A\psi(x)} & \mbox{$\mu_\lambda$-almost everywhere.}
\end{align*}
The function $\psi(x):=e^{-|x|^2}$ is clearly an $L^2$-function and we obtain
\begin{align*}
 A\phi(x) &= (\psi(x)^{-1}A\psi(x))\cdot\phi(x).
\end{align*}
Therefore, $A$ is given by multiplication with the measurable function $u(x):=\psi(x)^{-1}A\psi(x)$. Now, $A$ also commutes with the $L_0$-action. This implies that $u$ is $L_0$-invariant. Since $L_0$ acts transitively on $\calO_\lambda$, $u$ has to be constant. This means that $A$ is a scalar multiple of the identity. By Schur's Lemma the $P_0$-representation $\rho$ has to be irreducible and the proof is complete.
\end{proof}

It is a natural question to ask whether $\rho_\lambda$ extends to a unitary irreducible representation of $G$ (or some finite cover) on $L^2(\calO_\lambda,\td\mu_\lambda)$. One possible way to extend $\rho_\lambda$ is to extend the derived representation $\td\rho_\lambda$ of $\frakp^{\textup{max}}$ to $\frakg$ and integrate it to a group representation.

We define a Lie algebra representation $\td\pi_\lambda$\index{notation}{dpilambda@$\td\pi_\lambda$} of $\frakg$ on $C^\infty(\calO_\lambda)$ which extends the derived action of $\rho_\lambda$. On $\frakp^{\textup{max}}=\frakn+\frakl$ we let
\begin{align*}
 \td\pi_\lambda(X) &:= \left.\frac{\td}{\td t}\right|_{t=0}\rho_\lambda(e^{tX}) & \forall\,X\in\frakp^{\textup{max}}.
\end{align*}
For $\psi\in C^\infty(\calO_\lambda)$, the representation $\td\pi_\lambda$ is given by
\begin{align}
 \td\pi_\lambda(X)\psi(x) &= i(x\psi(x)|u) & \mbox{for }X &= (u,0,0),\label{eq:L2DerRep1}\\
 \td\pi_\lambda(X)\psi(x) &= D_{T^*x}\psi(x)+\frac{r\lambda}{2n}\Tr(T^*)\psi(x) & \mbox{for }X &= (0,T,0),\label{eq:L2DerRep2}\\
\intertext{where we have used Proposition \ref{prop:CharLambda} for the $\frakl$-action. In view of the Gelfand-Naimark decomposition \eqref{eq:GelfandNaimark} it remains to define $\td\pi_\lambda$ on $\overline{\frakn}$ in order to extend it to a representation of the whole Lie algebra $\frakg$. For this we use the Bessel operator $\calB_\lambda$ (see Section \ref{sec:BesselOp} for the definition). By Theorem \ref{thm:BlambdaTangential} the operator $\calB_\lambda$ is tangential to $\calO_\lambda$ and hence, for $\psi\in C^\infty(\calO_\lambda)$ the formula}
 \td\pi_\lambda(X)\psi(x) &= \frac{1}{i}(\calB_\lambda\psi(x)|v) & \mbox{for }X &= (0,0,-v),\label{eq:L2DerRep3}
\end{align}
defines a function $\td\pi_\lambda(X)\psi\in C^\infty(\calO_\lambda)$.

\begin{proposition}\label{prop:LieAlgRep}
For $\lambda\in\calW$ the formulas \eqref{eq:L2DerRep1}, \eqref{eq:L2DerRep2} and \eqref{eq:L2DerRep3} define a representation $\td\pi_\lambda$ of $\frakg$ on $C^\infty(\calO_\lambda)$. This representation is compatible with $\rho_\lambda$, i.e. for $p\in P$ and $X\in\frakg$ we have
\begin{align}
 \rho_\lambda(p)\td\pi_\lambda(X) &= \td\pi_\lambda(\Ad(p)X)\rho_\lambda(p).\label{eq:CompatibilityRhoPi}
\end{align}
\end{proposition}

\begin{proof}
We first show the compatibility condition \eqref{eq:CompatibilityRhoPi}. For $X\in\frakp^{\textup{max}}=\frakn+\frakl$ this condition is immediate since $\td\pi_\lambda(X)$ is just the derived action of the representation $\rho_\lambda$. It remains to show \eqref{eq:CompatibilityRhoPi} for $p\in P$ and $X=(0,0,-v)\in\overline{\frakn}$, $v\in V$.
\begin{enumerate}
\item[\textup{(a)}] Let $p=n_u\in N$, $u\in V$. Then $\Ad(p)X=(P(u)v,-2u\Box v,-v)$ by \eqref{eq:AdNnbar}. We calculate separately for $(P(u)v,0,0)$, $(0,-2u\Box v,0)$ and $(0,0,-v)$.
\begin{enumerate}
\item[\textup{(1)}] First, we have
\begin{align*}
 \td\pi_\lambda(P(u)v,0,0)\rho_\lambda(p)\psi(x) &= ie^{i(x|u)}(x|P(u)v)\psi(x).
\end{align*}
\item[\textup{(2)}] The adjoint of $u\Box v$ is $(u\Box v)^*=(\alpha v)\Box(\alpha u)$ and its trace is (using Lemma \ref{lem:Trtr})
\begin{align*}
 \Tr((u\Box v)^*) &= \Tr(u\Box v) = \Tr(L(uv)) = \frac{n}{r}\tau(u,v) = \frac{n}{r}(\alpha u|v).
\end{align*}
Hence,
\begin{align*}
 & \td\pi_\lambda(0,-2u\Box v,0)\rho_\lambda(p)\psi(x)\\
 ={}& -2D_{(u\Box v)^*x}\left[e^{i(-|u)}\psi\right](x)-\frac{r\lambda}{n}\Tr((u\Box v)^*)e^{i(x|u)}\psi(x)\\
 ={}& -2i(x|(u\Box v)u)e^{i(x|u)}\psi(x)-2e^{i(x|u)}D_{((\alpha v)\Box(\alpha u))x}\psi(x)\\
 & \ \ \ \ \ \ \ \ \ \ \ \ \ \ \ \ \ \ \ \ \ \ \ \ \ \ \ \ \ \ \ \ \ \ \ \ \ \ \ \ \ \ \ \ \ \ \ \ \ \ \ \ \ \ \ \ \ \ \ \ -\lambda e^{i(x|u)}(\alpha u|v)\psi(x).
\end{align*}
\item[\textup{(3)}] Finally, with Lemma \ref{lem:BnuProdRule} we obtain
\begin{align*}
 & \td\pi_\lambda(0,0,-v)\rho_\lambda(p)\psi(x) = \frac{1}{i}\left(\left.\calB_\lambda\left[e^{i(-|u)}\psi\right](x)\right|v\right)\\
 ={}& \frac{1}{i}\left(\left.\calB_\lambda\psi(x)\cdot e^{i(x|u)}+2P\left(\frac{\partial\psi}{\partial x}(x),\frac{\partial}{\partial x}e^{i(x|u)}\right)x+\psi(x)\cdot\calB_\lambda e^{i(x|u)}\right|v\right)\\
 ={}& \frac{1}{i}e^{i(x|u)}\left(\left.\calB_\lambda\psi(x)\right|v\right)+2e^{i(x|u)}\left(\left.P\left(\frac{\partial\psi}{\partial x}(x),\alpha u\right)x\right|v\right)\\
 & \ \ \ \ \ \ \ \ \ \ \ \ \ \ \ \ \ \ \ \ \ \ \ \ \ \,\ \ \ \ \ \ \ \ \ \ \ \ \ \ \ \ \ \ \ +\frac{1}{i}e^{i(x|u)}(P(i\alpha u)x+i\lambda\alpha u|v)\psi(x)
\end{align*}
\begin{align*}
 ={}& \frac{1}{i}e^{i(x|u)}\left(\left.\calB_\lambda\psi(x)\right|v\right)+2e^{i(x|u)}D_{((\alpha v)\Box(\alpha u))x}\psi(x)\\
 & \ \ \ \ \ \ \ \ \ \ \ \ \ \ \ \ \ \ \ \ \ \ \ \ \ \ \ \ \ \ \ \ +ie^{i(x|u)}(x|P(u)v)\psi(x)+\lambda e^{i(x|u)}(\alpha u|v)\psi(x)
\end{align*}
since
\begin{align*}
 \left(\left.P\left(\frac{\partial\psi}{\partial x},\alpha u\right)x\right|v\right) &= \tau\left(((\alpha v)\Box(\alpha u))x,\frac{\partial\psi}{\partial x}\right) = D_{((\alpha v)\Box(\alpha u))x}\psi(x).
\end{align*}
\end{enumerate}
Putting the three summands together gives
\begin{align*}
 \td\pi_\lambda(\Ad(p)X)\rho_\lambda(p)\psi(x) &= \frac{1}{i}e^{i(x|u)}\left(\left.\calB_\lambda\psi(x)\right|v\right) = \rho_\lambda(p)\td\pi_\lambda(X)\psi(x).
\end{align*}
\item[\textup{(b)}] Now, let $p=g\in L$. Then $(\Ad(p)X)=(0,0,-g^{-\#}v)$ by \eqref{eq:AdLieg}. In this case the compatibility condition \eqref{eq:CompatibilityRhoPi} is exactly the statement of Lemma \ref{lem:BnuEquiv}.
\end{enumerate}
Now we show that $\td\pi_\lambda$ is a Lie algebra representation, i.e. for $X,Y\in\frakg$ we have
\begin{align}
 \td\pi_\lambda(X)\td\pi_\lambda(Y)-\td\pi_\lambda(Y)\td\pi_\lambda(X) &= \td\pi_\lambda([X,Y]).\label{eq:AbstractFormulaLieAlgRep}
\end{align}
For $Y\in\frakp^{\textup{max}}$ we have $\td\pi_\lambda(Y)=\left.\frac{\td}{\td t}\right|_{t=0}\rho_\lambda(e^{tY})$. Therefore, for $Y\in\frakp^{\textup{max}}$ the identity \eqref{eq:AbstractFormulaLieAlgRep} follows from the compatibility condition \eqref{eq:CompatibilityRhoPi} by putting $p:=e^{tY}$ and differentiating with respect to $t$ at $t=0$. By the symmetry of \eqref{eq:AbstractFormulaLieAlgRep} in $X$ and $Y$ the remaining case is $X,Y\in\overline{\frakn}$. So let $X=(0,0,-u)$, $Y=(0,0,-v)$, $u,v\in V$. Then $[X,Y]=0$ by \eqref{eq:LieBracket} and we find that
\begin{align*}
 & \td\pi_\lambda(X)\td\pi_\lambda(Y)\psi(x) = -(\calB_\lambda(\calB_\lambda\psi|v)(x)|u)\\
 ={}& -\left[\sum_{\alpha,\beta}{\frac{\partial^2(\calB_\lambda\psi|v)}{\partial x_\alpha\partial x_\beta}(x)(P(\overline{e}_\alpha,\overline{e}_\beta)x|u)}+\sum_\alpha{\frac{\partial(\calB_\lambda\psi|v)}{\partial x_\alpha}(x)(\overline{e}_\alpha|u)}\right]\\
 ={}& -\left[ \sum_{\alpha,\beta,\gamma,\delta}{\frac{\partial^2}{\partial x_\alpha\partial x_\beta}\left[\frac{\partial^2\psi}{\partial x_\gamma\partial x_\delta}(x)(P(\overline{e}_\gamma,\overline{e}_\delta)x|v)\right](P(\overline{e}_\alpha,\overline{e}_\beta)x|u)} \right.\\
 & \ \ \ \ \ \ \ \ \ \ \ \ \ \ \ \ \ \ \ \ +\sum_{\alpha,\beta,\gamma}{\frac{\partial^2}{\partial x_\alpha\partial x_\beta}\left[\frac{\partial\psi}{\partial x_\gamma}(x)(\overline{e}_\gamma|v)\right](P(\overline{e}_\alpha,\overline{e}_\beta)x|u)}\\
 & \ \ \ \ \ \ \ \ \ \ \ \ \ \ \ \ \ \ \ \ \ \ \ \ \ \ \ \ \ \ \ \ \ \ \ \ \ +\sum_{\alpha,\gamma,\delta}{\frac{\partial}{\partial x_\alpha}\left[\frac{\partial^2\psi}{\partial x_\gamma\partial x_\delta}(x)(P(\overline{e}_\gamma,\overline{e}_\delta)x|v)\right](\overline{e}_\alpha|u)}\\
 & \ \ \ \ \ \ \ \ \ \ \ \ \ \ \ \ \ \ \ \ \ \ \ \ \ \ \ \ \ \ \ \ \ \ \ \ \ \ \ \ \ \ \ \ \ \ \ \ \ \ \ \ \ \ \ \left.+\sum_{\alpha,\gamma}{\frac{\partial}{\partial x_\alpha}\left[\frac{\partial\psi}{\partial x_\gamma}(x)(\overline{e}_\gamma|v)\right](\overline{e}_\alpha|u)}\right]\\
 ={}& -\left[ \sum_{\alpha,\beta,\gamma,\delta}{\frac{\partial^4\psi}{\partial x_\alpha\partial x_\beta\partial x_\gamma\partial x_\delta}(x)(P(\overline{e}_\gamma,\overline{e}_\delta)x|v)(P(\overline{e}_\alpha,\overline{e}_\beta)x|u)} \right.\\
 & \ \ \ +2\sum_{\alpha,\beta,\gamma,\delta}{\frac{\partial^3\psi}{\partial x_\alpha\partial x_\gamma\partial x_\delta}(x)(P(\overline{e}_\gamma,\overline{e}_\delta)e_\beta|v)(P(\overline{e}_\alpha,\overline{e}_\beta)x|u)}
\end{align*}
\begin{align*}
 & \ \ \ \ \ \ \ \ \ \ \ \ \ \ \ \ \ \ \ \ +\sum_{\alpha,\beta,\gamma}{\frac{\partial^3\psi}{\partial x_\alpha\partial x_\beta\partial x_\gamma}(x)(\overline{e}_\gamma|v)(P(\overline{e}_\alpha,\overline{e}_\beta)x|u)}\\
 & \ \ \ \ \ \ \ \ \ \ \ \ \ \ \ \ \ \ \ \ \ \ \ \ \ \ \ \ \ \ \ \ \ \ \ \ \ +\sum_{\alpha,\gamma,\delta}{\frac{\partial^3\psi}{\partial x_\alpha\partial x_\gamma\partial x_\delta}(x)(P(\overline{e}_\gamma,\overline{e}_\delta)x|v)(\overline{e}_\alpha|u)}\\
 & \ \ \ \ \ \ \ \ \ \ \ \ \ \ \ \ \ \ \ \ \ \ \ \ \ \ \ \ \ \ \ \ \ \ \ \ \ +\sum_{\alpha,\gamma,\delta}{\frac{\partial^2\psi}{\partial x_\gamma\partial x_\delta}(x)(P(\overline{e}_\gamma,\overline{e}_\delta)e_\alpha|v)(\overline{e}_\alpha|u)}\\
 & \ \ \ \ \ \ \ \ \ \ \ \ \ \ \ \ \ \ \ \ \ \ \ \ \ \ \ \ \ \ \ \ \ \ \ \ \ \ \ \ \ \ \ \ \ \ \ \ \ \ \ \ \ \ \ \ \ \ \ \ \left.+\sum_{\alpha,\gamma}{\frac{\partial^2\psi}{\partial x_\alpha\partial x_\gamma}(x)(\overline{e}_\gamma|v)(\overline{e}_\alpha|u)}\right].
\end{align*}
The first and the last summand are clearly symmetric in $u$ and $v$. The same holds for the sum of third and fourth summand. The fifth summand can be written as
\begin{multline*}
 \sum_{\gamma,\delta}{\frac{\partial^2\psi}{\partial x_\gamma\partial x_\delta}(x)\left(\left.\sum_\alpha{\tau(P(\overline{e}_\gamma,\overline{e}_\delta)\alpha v,e_\alpha)\overline{e}_\alpha}\right|u\right)}\\
 = \sum_{\gamma,\delta}{\frac{\partial^2\psi}{\partial x_\gamma\partial x_\delta}(x)(P(\overline{e}_\gamma,\overline{e}_\delta)\alpha v|u)} = \sum_{\gamma,\delta}{\frac{\partial^2\psi}{\partial x_\gamma\partial x_\delta}(x)\tau(P(\alpha\overline{e}_\gamma,\alpha\overline{e}_\delta)v|u)}
\end{multline*}
which is also symmetric in $u$ and $v$ since $P(\alpha\overline{e}_\gamma,\alpha\overline{e}_\delta)$ is a symmetric operator with respect to the trace form $\tau$. The same method applies for the second summand. Together we obtain that $\td\pi_\lambda(X)\td\pi_\lambda(Y)\psi(x)$ is symmetric in $X,Y\in\overline{\frakn}$ which means that
\begin{align*}
 \td\pi_\lambda(X)\td\pi_\lambda(Y)\psi(x) - \td\pi_\lambda(Y)\td\pi_\lambda(X)\psi(x) = 0.
\end{align*}
Hence, \eqref{eq:AbstractFormulaLieAlgRep} holds for all $X,Y\in\frakg$ and $\td\pi_\lambda$ is a Lie algebra representation. This finishes the proof.
\end{proof}

\begin{remark}
In Section \ref{sec:PrincipalSeries} we show that $\td\pi_\lambda$ is the Fourier transformed picture of a principal series representation in the non-compact picture. The definition \eqref{eq:L2DerRep3} of the $\overline{\frakn}$-action is motivated by these considerations. This also gives an alternative proof that $\td\pi_\lambda$ is indeed a Lie algebra representation.
\end{remark}

\subsection{Construction of the $(\frakg,\frakk)$-module}\label{sec:ConstructiongkModule}

From now on we assume that the split rank $r_0\geq2$ and consider only the minimal orbit $\calO_1$. For convenience, put $\lambda=\lambda_1:=\frac{r_0d}{2r}$\index{notation}{lambda1@$\lambda_1$}, $\td\pi:=\td\pi_\lambda$\index{notation}{dpi@$\td\pi$}, $\calO:=\calO_1$\index{notation}{O@$\calO$}, $\td\mu:=\td\mu_1$\index{notation}{dmu@$\td\mu$} and $\calB:=\calB_\lambda$\index{notation}{B@$\calB$}. We use the notation
\begin{align*}
 X\cdot\psi &:= \td\pi(X)\psi,
\end{align*}
for the action of $X\in\frakg$ on a function $\psi\in C^\infty(\calO)$. The representation $\td\pi$ clearly extends to a representation of the \textit{universal enveloping algebra}\index{subject}{universal enveloping algebra} $\calU(\frakg)$\index{notation}{Ug@$\calU(\frakg)$} on $C^\infty(\calO)$ whose action will be denoted similarly.

Let $\psi_0$ be the radial function on $\calO$ defined by
\begin{align}
 \psi_0(x) &:= \widetilde{K}_{\frac{\nu}{2}}(|x|),& x\in\calO,\label{def:Psi0}\index{notation}{psi0@$\psi_0$}
\end{align}
where $\widetilde{K}_\alpha(z)$ denotes the normalized $K$-Bessel function as introduced in Appendix \ref{app:BesselFunctions} and $\nu$ is the parameter defined in \eqref{eq:DefMuNu}. By \eqref{eq:DiffEqModBessel} the function $\widetilde{K}_\alpha$ is a solution of the second order equation $B_\alpha u=0$, where $B_\alpha$ is the operator defined in \eqref{eq:OrdinaryBesselOp}.

We put
\begin{align*}
 W_0 &:= \calU(\frakk)\psi_0 & \mbox{and} && W &:= \calU(\frakg)\psi_0.\index{notation}{W0@$W_0$}\index{notation}{W@$W$}
\end{align*}
$W$ is clearly a $\frakg$-subrepresentation of $C^\infty(\calO)$ and $W_0$ is a $\frakk$-subrepresentation of $W$. To show that $W$ is actually a $(\frakg,\frakk)$-module, we have to show that it is $\frakk$-finite. The first step is to show that the generator $\psi_0$ is $\frakk$-finite. This can be done by direct computation. For the precise statement we fix the following notation: Let $\calP$ be any space of polynomials on $V$. Then we denote by $\widetilde{K}_\alpha\otimes\calP$ the space of functions
\begin{align*}
 \widetilde{K}_\alpha\otimes\varphi:\calO\rightarrow\CC,\,x\mapsto\widetilde{K}_\alpha(|x|)\varphi(x)
\end{align*}
with $\varphi\in\calP$. For $\calP$ we use
\begin{align*}
 \CC[V]_{\geq k} &:= \{p\in\CC[V]:p\mbox{ is a sum of homogeneous polynomials of degree }\geq k\},\index{notation}{CVgeqk@$\CC[V]_{\geq k}$}
\end{align*}
or the space of \textit{spherical harmonics}\index{subject}{spherical harmonics}
\begin{align*}
 \calH^k(\RR^n) &:= \{p\in\CC[x_1,\ldots,x_n]:p\mbox{ is homogeneous of degree $k$ and harmonic}\}.\index{notation}{HkRn@$\calH^k(\RR^n)$}
\end{align*}

\begin{proposition}\label{prop:Kfinite}
Let $V$ be a simple Jordan algebra with simple $V^+$. Then the $\frakk$-module $W_0$ is finite-dimensional if and only if $V\ncong\RR^{p,q}$ with $p+q$ odd, $p,q\geq2$. If this is the case, $W_0\cong E^{\alpha_0}$ with
\begin{align}
 \alpha_0 &:= \begin{cases}\frac{d}{4}\sum_{i=1}^{r_0}{\gamma_i} & \mbox{if $V$ is euclidean,}\\0 & \mbox{if $V$ is non-euclidean of rank $\geq3$,}\\\frac{1}{2}\left|d_0-\frac{d}{2}\right|\gamma_1+\frac{1}{2}\left(d_0-\frac{d}{2}\right)\gamma_2 & \mbox{if $V\cong\RR^{p,q}$, $p,q\geq2$.}\end{cases}\label{eq:DefGamma0}\index{notation}{alpha0@$\alpha_0$}
\end{align}
More precisely:
\begin{enumerate}
\item[\textup{(a)}] If $V$ is euclidean, then
 \begin{align*}
  W_0 &= \CC\psi_0
 \end{align*}
 and the center $Z(\frakk)=\RR(e,0,-e)$ acts by
 \begin{align*}
  \td\pi(e,0,-e)\psi_0 &= \frac{rd}{2}i\psi_0.
 \end{align*}
\item[\textup{(b)}] If $V$ is non-euclidean of rank $r\geq3$, then
 \begin{align*}
  W_0 &= \CC\psi_0
 \end{align*}
 and $\psi_0$ is a $\frakk$-spherical vector.
\item[\textup{(c)}] If $V=\RR^{p,q}$ with $p+q$ even, $p,q\geq2$, then
 \begin{align}
  W_0 &= \begin{cases}\displaystyle\ \bigoplus_{k=0}^{\frac{q-p}{2}}{\widetilde{K}_{\frac{\nu}{2}+k}\otimes\calH^k(\RR^p)}\cong\calH^{\frac{q-p}{2}}(\RR^{p+1}) & \mbox{if $p\leq q$,}\\\displaystyle\ \bigoplus_{k=0}^{\frac{p-q}{2}}{\widetilde{K}_{\frac{\nu}{2}+k}\otimes\calH^k(\RR^q)}\cong\calH^{\frac{p-q}{2}}(\RR^{q+1}) & \mbox{if $p\geq q$,}\end{cases}\label{eq:MinKtypeRk2finite}
 \end{align}
\item[\textup{(d)}] If $V=\RR^{p,q}$ with $p+q$ odd, $p,q\geq2$, then
 \begin{align}
  W_0 &= \begin{cases}\displaystyle\ \bigoplus_{k=0}^\infty{\widetilde{K}_{\frac{\nu}{2}+k}\otimes\calH^k(\RR^p)} & \mbox{if $p\leq q$,}\\\displaystyle\ \bigoplus_{k=0}^\infty{\widetilde{K}_{\frac{\nu}{2}+k}\otimes\calH^k(\RR^q)} & \mbox{if $p\geq q$.}\end{cases}\label{eq:MinKtypeRk2infinite}
 \end{align}
\end{enumerate}
\end{proposition}

\begin{proof}
Since $\psi_0$ is $K_L$-invariant, clearly $\td\pi(\frakk_\frakl)\psi_0=0$. Therefore it suffices to apply elements of the form $(u,0,-\alpha(u))\in\frakk$, $u\in V$, to $\psi_0$. By \eqref{eq:L2DerRep1} and \eqref{eq:L2DerRep3} we have
\begin{align}
 \td\pi(u,0,-\alpha(u))\psi(x) &= \frac{1}{i}\tau((\calB-\alpha x)\psi(x),u).\label{eq:KActionInTermsOfBesselOp}
\end{align}
Now we have to distinguish between three different cases.
\begin{enumerate}
 \item[\textup{(1)}] If $V$ is euclidean, then by Corollary \ref{cor:BnuRadial}~(1)
  \begin{align*}
   \td\pi(u,0,-\alpha(u))\psi_0(x) &= \frac{1}{i}\underbrace{B_{\frac{\nu}{2}}\widetilde{K}_{\frac{\nu}{2}}(|x|)}_{=0}(x|u)+\frac{1}{i}\frac{d}{2}\widetilde{K}_{\frac{\nu}{2}}'(|x|)\tr(u)
  \end{align*}
  Since $\nu=-1$, we have by \eqref{eq:IKBesselMinusHalf}
  \begin{align*}
   \widetilde{K}_{\frac{\nu}{2}}(|x|) &= \frac{\sqrt{\pi}}{2}e^{-|x|}
  \end{align*}
  and therefore $\widetilde{K}_{\frac{\nu}{2}}'(|x|)=-\widetilde{K}_{\frac{\nu}{2}}(|x|)$. Altogether this gives
  \begin{align}
   \td\pi(u,0,-\alpha(u))\psi_0(x) = i\frac{d}{2}\tr(u)\psi_0(x).\label{eq:KactionEucl}
  \end{align}
  Hence, $W_0=\CC\psi_0$. Putting $u=e$ gives the action of the center $Z(\frakk)=\RR(e,0,-e)$. Further, for $u=c_i$, $1\leq i\leq r_0$, we find that $W_0$ is of highest weight $\frac{d}{4}\sum_{i=1}^{r_0}{\gamma_i}$.
 \item[\textup{(2)}] If $V$ is non-euclidean of rank $r\geq3$, then $d=2d_0$ (see Proposition \ref{prop:ClassificationEuclSph}) and with Corollary \ref{cor:BnuRadial}~(2) we obtain
  \begin{align}
   \td\pi(u,0,-\alpha(u))\psi_0(x) &= \frac{1}{i}\underbrace{B_{\frac{\nu}{2}}\widetilde{K}_{\frac{\nu}{2}}(|x|)}_{=0}(x|u).\label{eq:KactionNonEuclSph}
  \end{align}
  This implies that $W_0=\CC\psi_0$ is the trivial representation.
 \item[\textup{(3)}] Now, for the case $V=\RR^{p,q}$, $p,q\geq2$, the calculations can be found in Appendix \ref{app:Rank2MinKtype}.\qedhere
\end{enumerate}
\end{proof}

\begin{remark}
The fact that $\psi_0$ is not $\frakk$-finite if $V=\RR^{p,q}$ with $p+q$ odd corresponds to the result by D. Vogan that no covering group of $\SO(p+1,q+1)_0$ has a minimal representation if $p+q$ is odd and $p,q\geq3$ (see \cite[Theorem 2.13]{Vog81}).
\end{remark}

Now a standard argument shows that the fact that $\psi_0$ is $\frakk$-finite already implies that $W$ is $\frakk$-finite. Since we could not find a reference for this basic fact, we include a short proof.

\begin{lemma}
If $W_0=\calU(\frakk)\psi_0$ is finite-dimensional, then $W=\calU(\frakg)\psi_0$ is a $(\frakg,\frakk)$-module.
\end{lemma}

\begin{proof}
Let $\frakg_1:=\frakg_\CC\oplus\CC\subseteq\calU(\frakg)$ and define $W_{n+1}:=\frakg_1W_n$ for $n\geq0$. We claim that
\begin{enumerate}
 \item[\textup{(1)}] $W_n$ is finite-dimensional for every $n$,
 \item[\textup{(2)}] $W_n$ is $\frakk$-invariant for every $n$,
 \item[\textup{(3)}] $W=\bigcup_n{W_n}$.
\end{enumerate}
The first statement follows easily by induction on $n$, since $W_0$ and $\frakg_1$ are finite-dimensional. The third statement is also clear by the definition of $\calU(\frakg)$. For the second statement we give a proof by induction on $n$:\\
For $n=0$ the statement is clear by the definition of $W_0$. For the induction step let $w\in W_{n+1}$ and $X\in\frakk$. Then $w=\sum_j{Y_jv_j}$ with $Y_j\in\frakg_1$ and $v_j\in W_n$. We have
\begin{align*}
 Xw &= \sum_j{X(Y_jv_j)} = \sum_j{\left([X,Y_j]v_j + Y_j(Xv_j)\right)}.
\end{align*}
Here $[X,Y_j]\in\frakg_1$ and hence $[X,Y_j]v_j\in W_{n+1}$ for each $j$. Furthermore $Xv_j\in W_n$ by the induction assumption and hence $Y_j(Xv_j)\in W_{n+1}$ for every $j$. Together this gives $Xw\in W_{n+1}$ which shows that $W_{n+1}$ is $\frakk$-invariant.\\
Now the $\frakk$-finiteness of every vector $w\in W$ follows.
\end{proof}

To see that $W$ integrates to a representation on $L^2(\calO,\td\mu)$ we have to show that $W$ is contained in $L^2(\calO,\td\mu)$.

\begin{proposition}\label{prop:WinL2}
\begin{enumerate}
 \item[\textup{(a)}] If $V$ is a simple Jordan algebra of rank $r\geq3$, then
  \begin{align}
   W &\subseteq \bigoplus_{\ell=0}^\infty{\widetilde{K}_{\frac{\nu}{2}+\ell}\otimes\CC[V]_{\geq2\ell}} \subseteq L^2(\calO,\td\mu).\label{eq:L2spherical}
  \end{align}
 \item[\textup{(b)}] If $V=\RR^{p,q}$ with $p+q$ even, then
  \begin{align}
   W &\subseteq \bigoplus_{\ell=0}^\infty{\bigoplus_{k=0}^{|\frac{p-q}{2}|}{\widetilde{K}_{\frac{\nu}{2}+k+\ell}\otimes\CC[V]_{\geq k+2\ell}}} \subseteq L^2(\calO,\td\mu).\label{eq:L2nonspherical}
  \end{align}
 \item[\textup{(c)}] If $V=\RR^{p,q}$ with $p+q$ odd, then $W\nsubseteq L^2(\calO,\td\mu)$.
\end{enumerate}
\end{proposition}

\begin{proof}
$\frakg=\frakk+\frakp^{\textup{max}}$ implies $\calU(\frakg)=\calU(\frakp^{\textup{max}})\calU(\frakk)$ by the Poincar\'e--Birkhoff--Witt Theorem. Therefore
\begin{align*}
 W &= \calU(\frakg)\psi_0 = \calU(\frakp^{\textup{max}})W_0.
\end{align*}
Now if $V\ncong\RR^{p,q}$ with $p+q$ odd, then by Proposition \ref{prop:Kfinite} the $\frakk$-module $W_0$ is already contained in the direct sum in \eqref{eq:L2spherical} or \eqref{eq:L2nonspherical}, respectively. (In fact, $W_0$ is contained in the direct summand for $\ell=0$.) Therefore, it suffices to check that these direct sums are stable under the action of $\frakp^{\textup{max}}=\frakn+\frakl$. We check the actions of $\frakn$ and $\frakl$ separately.
\begin{enumerate}
 \item[\textup{(1)}] \textit{Action of $\frakn$.} By \eqref{eq:L2DerRep1} the action of $\frakn$ is given by multiplication by polynomials. This clearly leaves the direct sums invariant.
 \item[\textup{(2)}] \textit{Action of $\frakl$.} In view of \eqref{eq:L2DerRep2} it suffices to show that the operators $D_{Tx}$ for $T\in\frakl=\str(V)$ leave the direct sums invariant. But this is a consequence of the following calculation which uses \eqref{eq:ddxradial} and \eqref{eq:BesselDiffFormulas}:
 \begin{align*}
  & D_{Tx}\left[\widetilde{K}_{\frac{\nu}{2}+k}(|x|)\varphi(x)\right]\\
  ={}& \frac{r}{r_0}(Tx|x)\frac{\widetilde{K}_{\frac{\nu}{2}+k}'(|x|)}{|x|}\varphi(x) + \widetilde{K}_{\frac{\nu}{2}+k}(|x|)D_{Tx}\varphi(x)\\
  ={}& -\frac{r}{2r_0}\widetilde{K}_{\frac{\nu}{2}+k+1}(|x|)\cdot(Tx|x)\varphi(x) + \widetilde{K}_{\frac{\nu}{2}+k}(|x|)\cdot D_{Tx}\varphi(x).
 \end{align*}
\end{enumerate}
This proves the first inclusion in \eqref{eq:L2spherical} and \eqref{eq:L2nonspherical}. It remains to show the $L^2$-statements.\\
Note that by \eqref{eq:dmuIntFormula} a function $\widetilde{K}_\alpha\otimes\varphi$ ($\varphi\in\CC[V]$ homogeneous of degree $m$) is contained in $L^2(\calO,\td\mu)$ if and only if $\widetilde{K}_\alpha(t)t^m\in L^2(\RR_+,t^{\mu+\nu+1}\td t)$. By \eqref{eq:BesselAsymptAtInfty} the $K$-Bessel function $\widetilde{K}_\alpha(t)$ decays exponentially as $t\rightarrow\infty$ and hence $\widetilde{K}_\alpha(t)t^m\in L^2((1,\infty),t^{\mu+\nu+1}\td t)$ for any $\alpha\in\RR$ and $m\geq0$. Therefore it suffices to check the asymptotic behavior of $\widetilde{K}_\alpha(t)t^m$ as $t\rightarrow0$. For $\alpha<0$ we obtain with \eqref{eq:BesselKAsymptAt0} that also $\widetilde{K}_\alpha(t)t^m\in L^2((0,1),t^{\mu+\nu+1}\td t)$ for any $m\geq0$ (use Lemma \ref{lem:MuNuProperties}: $\mu+\nu+1\geq0$). A similar argument settles the case $\alpha=0$. Hence, we may restrict ourselves to the case $\alpha>0$. In this case (see \eqref{eq:BesselKAsymptAt0})
\begin{align*}
 \widetilde{K}_\alpha(t) &\sim \const\cdot t^{-2\alpha} & \mbox{as }t\rightarrow0.
\end{align*}
Therefore, $\widetilde{K}_\alpha(t)t^m\in L^2((0,1),t^{\mu+\nu+1}\td t)$ if and only if
\begin{align*}
 \mu+\nu+2m-4\alpha+1 &> -1.
\end{align*}
\begin{enumerate}
\item[\textup{(a)}] For\ \ $V$\ \ simple\ \ of\ \ rank\ \ $r\geq3$\ \ we\ \ have\ \ to\ \ show\ \ that $\widetilde{K}_{\frac{\nu}{2}+\ell}(t)t^{2\ell}\in L^2((0,1),t^{\mu+\nu+1}\td t)$ for any $\ell\in\NN_0$. With $\alpha=\frac{\nu}{2}+\ell$ and $m=2\ell$ we obtain
\begin{align*}
 \mu+\nu+2m-4\alpha+1 &= \mu-\nu+1 \geq 1 > -1
\end{align*}
by Lemma \ref{lem:MuNuProperties}.
\item[\textup{(b)}] In\ \ the\ \ case\ \ where\ \ $V=\RR^{p,q}$\ \ with\ \ $p+q$\ \ even\ \ we\ \ need\ \ to\ \ prove\ \ that $\widetilde{K}_{\frac{\nu}{2}+k+\ell}(t)t^{k+2\ell}\in L^2((0,1),t^{\mu+\nu+1}\td t)$ for $k=0,\ldots,|\frac{p-q}{2}|$ and $\ell\in\NN_0$. In this case $\mu=\max(p,q)-2$ and $\nu=\min(p,q)-2$. With $\alpha=\frac{\nu}{2}+k+\ell$ and $m=k+2\ell$ we obtain
\begin{align*}
 \mu+\nu+2m-4\alpha+1 &= |p-q|-2k+1 \geq 1 > -1.
\end{align*}
\item[\textup{(c)}] Now assume that $V=\RR^{p,q}$ with $p+q$ odd. By \eqref{eq:MinKtypeRk2infinite} the $\frakk$-module $W_0$ contains functions $\widetilde{K}_{\frac{\nu}{2}+k}\otimes\varphi$ with $\varphi\in\CC[V]$ homogeneous of degree $k$ where $k$ is arbitrary. For large $k$ we have $\alpha=\frac{\nu}{2}+k>0$ and with $m=k$ we obtain
\begin{align*}
 \mu+\nu+2m-4\alpha+1 &= \mu-\nu-2k+1
\end{align*}
which is $\leq-1$ for $k$ large enough. Hence, $W_0\nsubseteq L^2(\calO,\td\mu)$ in this case.
\end{enumerate}
This proves the $L^2$-statements and completes the proof.
\end{proof}

To make sure we obtain a unitary representation we show the following proposition:

\begin{proposition}\label{prop:InfUnitary}
The $(\frakg,\frakk)$-module $W$ is infinitesimally unitary with respect to the inner product of $L^2(\calO,\td\mu)$.
\end{proposition}

\begin{proof}
Since the group $P$ acts unitarily on $L^2(\calO,\td\mu)$ by $\rho_{\lambda_1}$, its infinitesimal action $\td\rho_{\lambda_1}=\td\pi|_{\frakp^{\textup{max}}}$ is infinitesimally unitary with respect to the $L^2$ inner product. This shows that the action of $\frakn$ and $\frakl$ on $W$ is infinitesimally unitary. (In fact, one can easily prove unitarity of these actions directly with \eqref{eq:L2DerRep1} and \eqref{eq:L2DerRep2}.) Since the action of $\overline{\frakn}$ is by \eqref{eq:L2DerRep3} given in terms of the Bessel operator $\calB$, it follows from Proposition \ref{prop:BnuSA} that $\overline{\frakn}$ acts by skew-symmetric operators on $L^2(\calO,\td\mu)$. In view of the decomposition $\frakg=\frakn+\frakl+\overline{\frakn}$ this finishes the proof.
\end{proof}

The last ingredient to integrate the $(\frakg,\frakk)$-module $W$ is admissibility.

\begin{proposition}\label{prop:admissible}
Assume that $V\neq\RR^{p,q}$, $p+q$ odd, $p,q\geq2$. Then $W$ is admissible.
\end{proposition}

\begin{proof}
We use the criterion \cite[Corollary 3.4.7]{Wal88}. Therefore, it suffices to show that every $X\in\calZ(\frakg)$\index{notation}{Zg@$\calZ(\frakg)$}, the center of $\calU(\frakg)$, acts as a scalar on $W$. To show this we apply \cite[Proposition 1.2.2]{Wal88}. Recall that $(\rho_{\lambda_1},L^2(\calO,\td\mu))$ is a unitary irreducible representation of $P_0$ by Proposition \ref{prop:MackeyRep}. The space $D:=\rho_{\lambda_1}(P_0)W$ is contained in $L^2(\calO,\td\mu)$ since $W\subseteq L^2(\calO,\td\mu)$ by Proposition \ref{prop:WinL2} and $\rho_{\lambda_1}$ acts on $L^2(\calO,\td\mu)$. Since $D$ is $P_0$-invariant, it has to be dense in $L^2(\calO,\td\mu)$. Further, a closer look at Proposition \ref{prop:WinL2} shows that $W\subseteq C^\infty(\calO)$ and hence $D\subseteq L^2(\calO,\td\mu)\cap C^\infty(\calO)$. Therefore, $T:=\td\pi(X)$ is also defined on $D$. Since
\begin{align*}
 \td\pi(Y)\rho_{\lambda_1}(p)w &= \rho_{\lambda_1}(p)\td\pi(\Ad(p^{-1})Y)w & \forall\,Y\in\calU(\frakg),p\in P_0,w\in W,
\end{align*}
by \eqref{eq:CompatibilityRhoPi}, it follows that $T:D\rightarrow D$. Further, using that $T\in Z(\frakg)$, we have $T\rho_{\lambda_1}(p)=\rho_{\lambda_1}(p)T$ for all $p\in P_0$. The adjoint $S:=T^*$ of $T$ with respect to the $L^2$ inner product is by Proposition \ref{prop:InfUnitary} also given by the action of an element of $\calU(\frakg)$. Therefore, $S$ also acts on $D$. Applying \cite[Proposition 1.2.2]{Wal88} to this situation yields that $T$ is a scalar multiple of the identity on $D$. Since $D$ contains $W$, the claim follows.
\end{proof}

\subsection{Integration of the $(\frakg,\frakk)$-module}\label{sec:IntgkModule}

Now we can finally integrate the $(\frakg,\frakk)$-module $W$ to a unitary representation of a finite cover of $G$. For this we first construct the finite cover of $G$ on which we will define the representation.\\

In Section \ref{sec:Guniversalcover} we constructed the universal covering group $\widetilde{G}$ of $G$. Note that in the euclidean case, the covering $\widetilde{G}\rightarrow G$ is not finite. Thus, we have to factor out a discrete central subgroup.

Let $k\in\NN$\index{notation}{k@$k$} be the smallest positive integer such that
\begin{align*}
 k\frac{r_0}{2}\left(d_0-\frac{d}{2}\right)_+ &\in \ZZ.
\end{align*}
Here we use the notation
\begin{align*}
 x_+ &:= \frac{1}{2}(x+|x|) = \begin{cases}x & \mbox{if }x\geq0,\\0 & \mbox{if }x<0.\end{cases}\index{notation}{xplus@$x_+$}
\end{align*}
for the positive part of a real number $x\in\RR$. Then the following lemma holds:

\begin{lemma}\label{lem:FiniteCover}
The discrete subgroup
\begin{align*}
 \Gamma &:= \exp_{\widetilde{G}}(k\pi\ZZ(e,0,-e)) \subseteq \widetilde{G}\index{notation}{Gamma@$\Gamma$}
\end{align*}
is central in $\widetilde{G}$ and the group $\check{G}:=\widetilde{G}/\Gamma$\index{notation}{Gcheck@$\check{G}$} is a finite cover of $G$ with covering map
\begin{align*}
 \check{\pr}:\check{G}\rightarrow G,\,g\Gamma\mapsto\widetilde{\pr}(g).\index{notation}{prcheck@$\check{\pr}$}
\end{align*}
Moreover, $\check{K}:=\widetilde{K}/\Gamma\subseteq\check{G}$\index{notation}{Kcheck@$\check{K}$} is a maximal compact subgroup of $\check{G}$.
\end{lemma}

\begin{proof}
By \eqref{eq:1asexp} we have $\exp_{G_0}(\pi(e,0,-e))=\1$. Therefore, $\Gamma$ has to be central in $\widetilde{G_0}$ since it projects onto $\1\in G_0$. We claim that $\Gamma$ is also central in $\widetilde{G}$. Since $\widetilde{G}$ is generated by $\widetilde{G_0}$ and $\widetilde{\alpha}$, it suffices to prove that
\begin{align*}
 \Ad(\widetilde{\alpha})\exp_{\widetilde{G}_0}(k\pi(e,0,-e)) &= \exp_{\widetilde{G}_0}(k\pi(e,0,-e)).
\end{align*}
But since $\Ad_\frakg(\widetilde{\alpha})=\Ad_\frakg(\alpha)$ on $\frakg$ by \eqref{eq:Adgtildealpha}, this follows from \eqref{eq:AdLieg} and the fact that $\alpha e=e$. Thus, $\Gamma\subseteq\widetilde{G}$ is a discrete central subgroup and the quotient $\check{G}:=\widetilde{G}/\Gamma$ is again a group. By \eqref{eq:1asexp} we have $\exp_{G_0}(X)=\1$ for any $X\in k\pi\ZZ(e,0,-e)$ and thus the covering map $\widetilde{\pr}:\widetilde{G}\rightarrow G$ factors through $\widetilde{G}/\Gamma$ and hence defines a covering map $\check{\pr}:\check{G}\rightarrow G$. It remains to show that the cover $\check{G}\rightarrow G$ is finite and $\check{K}=\widetilde{K}/\Gamma$ is a maximal compact subgroup of $\check{G}$.
\begin{enumerate}
\item[\textup{(a)}] If $V$ is non-euclidean, then by Lemma \ref{lem:CenterK} the maximal compact subgroup $K$ of $G$ is semisimple. Hence, its universal covering $\widetilde{K}$ is a finite covering and therefore $\widetilde{K}$ is a maximal compact subgroup of the universal covering $\widetilde{G}$ of $G$. Passing to quotients modulo the discrete central subgroup $\Gamma$ preserves this property.
\item[\textup{(b)}] Now let $V$ be a euclidean Jordan algebra. Then
\begin{align*}
 \frakk &= Z(\frakk)+[\frakk,\frakk],
\end{align*}
where $Z(\frakk)=\RR(e,0,-e)$ by Lemma \ref{lem:CenterK} and $[\frakk,\frakk]$ is semisimple. Therefore, the universal covering $\widetilde{K}\subseteq\widetilde{G}$ of $K$ is given by
\begin{align*}
 \widetilde{K} &= \RR\times\widetilde{K_{ss}},
\end{align*}
where the first factor $\RR$ is the exponential image of the center $Z(\frakk)$ and $K_{ss}$ denotes the analytic subgroup of $K$ with Lie algebra $[\frakk,\frakk]$. The group $\widetilde{K_{ss}}$ is compact since $K_{ss}$ is semisimple and compact. Therefore, factoring out $\Gamma$ yields a compact group
\begin{align*}
 \check{K} &= \widetilde{K}/\Gamma\cong\SS^1\times\widetilde{K_{ss}}.
\end{align*}
Hence, the covering $\check{K}\rightarrow K$ is finite. $\check{K}$ is a maximal compact subgroup of $\check{G}$, because $K$ is a maximal compact subgroup of $G$. Thus, the covering $\check{G}\rightarrow G$ is also finite in this case.\qedhere
\end{enumerate}
\end{proof}

We denote by $\check{\alpha}$\index{notation}{alphacheck@$\check{\alpha}$} the projection of $\widetilde{\alpha}$ under the covering map $\widetilde{G}\rightarrow\check{G}$. (For the choice of $\widetilde{\alpha}$ see Section \ref{sec:Guniversalcover}.) Further, let $\check{P}:=\check{\pr}^{-1}(P)$\index{notation}{Pcheck@$\check{P}$}.

\begin{example}\label{ex:Gcheck}
\begin{enumerate}
\item[\textup{(1)}] Let $V=\Sym(n,\RR)$. We claim that
\begin{align*}
 \check{G} &= \begin{cases}\Mp(n,\RR) & \mbox{for }n\equiv1,3\ \ \ (\mod\,4),\\\Sp(n,\RR) & \mbox{for }n\equiv2\ \ \ (\mod\,4),\\\Sp(n,\RR)/\{\pm1\} & \mbox{for }n\equiv0\ \ \ (\mod\,4).\end{cases}
\end{align*}
First note that the fundamental group $\pi_1(\Sp(n,\RR))=\ZZ$ is generated by the closed curve (see e.g. \cite[Proposition 4.8]{Fol89})
\begin{align*}
 [0,2\pi] &\rightarrow \Sp(n,\RR),\,\theta \mapsto \left(\begin{array}{cc}\cos\theta\cdot\1_n&\sin\theta\cdot\1_n\\-\sin\theta\cdot\1_n&\cos\theta\cdot\1_n\end{array}\right).
\end{align*}
Since $\frac{r_0}{2}(d_0-\frac{d}{2})_+=\frac{rd}{4}=\frac{n}{4}$, we obtain for the integer $k$:
\begin{align*}
 k &= \begin{cases}4 & \mbox{for }n\equiv1,3\ \ \ (\mod\,4),\\2 & \mbox{for }n\equiv2\ \ \ (\mod\,4),\\1 & \mbox{for }n\equiv0\ \ \ (\mod\,4).\end{cases}
\end{align*}
and the claim follows.
\item[\textup{(2)}] For $V=\RR^{p,q}$, $p+q$ even, we have that $\frac{r_0}{2}(d_0-\frac{d}{2})_+=(\frac{q-p}{2})_+\in\ZZ$ and hence, $k=1$. The Lie algebra element $k\pi(e,0,-e)=\pi(e,0,-e)$ corresponds to
\begin{align*}
 \left(\begin{array}{ccccc}0&-2\pi&&&\\2\pi&0&&&\\&&0&&\\&&&\ddots&\\&&&&0\end{array}\right)
\end{align*}
via the isomorphism $\frakg\cong\so(p+1,q+1)$ (see Example \ref{ex:ConfGrp}~(2)). Since the exponential function of $\SO(p+1,q+1)_0$ applied to this Lie algebra element is equal to $\1$, the group $\check{G}_0$ is a finite cover of the group $\SO(p+1,q+1)_0$.
\end{enumerate}
\end{example}

Now, in a first step, we integrate the $(\frakg,\frakk)$-module to a unitary representation of $\check{G}_0$.

\begin{theorem}\label{thm:IntgkModule}
Assume that $V\ncong\RR^{p,q}$ with $p+q$ odd, $p,q\geq2$. Then there is a unique unitary irreducible representation $\pi_0$\index{notation}{pi0@$\pi_0$} of $\check{G}_0$ on $L^2(\calO,\td\mu)$ with underlying $(\frakg,\frakk)$-module $W$. The representation $\pi_0$ has the additional property that
\begin{align*}
 \pi_0(p) &= \rho_{\lambda_1}(\check{\pr}(p)) & \forall\, p\in\check{P}_0.
\end{align*}
\end{theorem}

\begin{proof}
The $(\frakg,\frakk)$-module $W$ clearly integrates to a $(\frakg,\widetilde{K}_0)$-module since $\widetilde{K}_0$ is connected and simply-connected. (Note that in the euclidean case $\widetilde{K}_0$ is not compact.) By Proposition \ref{prop:Kfinite} the element $k\pi(e,0,-e)$ acts on the highest weight vector $\psi^+$ of $W_0$ by the scalar
\begin{align*}
 2\pi i\cdot k\frac{r_0}{2}\left(d_0-\frac{d}{2}\right)_+
\end{align*}
which\ \ is\ \ in\ \ $2\pi i\ZZ$\ \ by\ \ construction.\ \ Therefore,\ \ the\ \ central\ \ element $\gamma:=\exp(k\pi(e,0,-e))$ of $\widetilde{K}_0$ acts trivially on $\psi^+$ and hence, by Schur's Lemma, also on $W_0$. We claim that $\gamma$ also acts trivially on $W$. By the Poincar\'e--Birkhoff--Witt Theorem we have $W=\calU(\frakp)\calU(\frakk)\psi_0=\calU(\frakp)W_0$. Let
\begin{align*}
 \calU(\frakp) &= \bigcup_{n=0}^\infty{\calU_n(\frakp)}
\end{align*}
be the natural filtration of $\calU(\frakp)$. Then clearly
\begin{align*}
 W &= \bigcup_{n=0}^\infty{\calU_n(\frakp)W_0}
\end{align*}
and it suffices to show that $\gamma$ acts trivially on every $\calU_n(\frakp)W_0$. We show this by induction on $n\in\NN_0$. For $n=0$ we have $\calU_0(\frakp)W_0=W_0$ on which $\gamma$ acts trivially by the previous considerations. Now suppose $\gamma$ acts trivially on $\calU_n(\frakp)W_0$. By the results of Section \ref{subsec:ConfGrpRoots} the space $\frakp$ is the direct sum of eigenspaces of $\ad(e,0,-e)$ for the eigenvalues $-2i$, $0$ and $+2i$. Hence
\begin{align*}
 \Ad(\gamma)X &= e^{k\pi\ad(e,0,-e)}X = X & \forall\,X\in\frakp.
\end{align*}
For $w\in\calU_n(\frakp)W_0$ and every $X\in\frakp$ we obtain
\begin{align*}
 \gamma\cdot Xw &= \Ad(\gamma)X\cdot\gamma w = Xw
\end{align*}
by the induction hypothesis. It follows that $\gamma$ acts trivially on $\calU_{n+1}(\frakp)W_0$ and the induction is complete.\\
Since the subgroup $\Gamma$ is generated by $\gamma$ which acts trivially on $W$, the whole discrete central subgroup $\Gamma$ acts trivially on $W$. It follows that the representation factorizes to a $(\frakg,\check{K}_0)$-module. (In contrast to $\widetilde{K}_0$, the group $\check{K}_0$ is compact in both the euclidean and the non-euclidean case by Lemma \ref{lem:FiniteCover}.)\\
So far, we have constructed a $(\frakg,\check{K}_0)$-module, where $\check{K}_0\subseteq\check{G}_0$ is a maximal compact subgroup of the semisimple Lie group $\check{G}_0$ and $\frakg$ is the Lie algebra of $\check{G}_0$. Since $W$ is admissible by Proposition \ref{prop:admissible}, it integrates to a representation $(\pi_0,\calH)$ of $\widetilde{G}_0$ by a standard theorem of Harish--Chandra (see e.g. \cite[Theorem 6.A.4.2]{Wal88}). Now, $W$ is already infinitesimally unitary with respect to the $L^2$ inner product (see Proposition \ref{prop:InfUnitary}). Thus, $\calH\subseteq L^2(\calO,\td\mu)$ and $\pi_0$ is unitary with respect to the $L^2$-inner product. Further, $\pi_0$ and $\rho_{\lambda_1}\circ\check{\pr}$ agree on $\check{P}_0$ since they have the same Lie algebra action. Now, $\rho_{\lambda_1}|_{P_0}$ is irreducible on $L^2(\calO,\td\mu)$ by Proposition \ref{prop:MackeyRep}. Hence $\calH=L^2(\calO,\td\mu)$ and $(\pi_0,\calH)$ is also irreducible as $\widetilde{G}_0$-representation. This shows the claim.
\end{proof}

Now it is only a technical matter to extend the representation $\pi_0$ from $\check{G}_0$ to $\check{G}$.

\begin{proposition}
The representation $\pi_0$ extends uniquely to a representation $\pi$\index{notation}{pi@$\pi$} of $\check{G}$ such that
\begin{align*}
 \pi(p) &= \rho_{\lambda_1}(\check{\pr}(p)) & \forall\,p\in\check{P}.
\end{align*}
\end{proposition}

\begin{proof}
The group $\check{G}$ has at most $2$ connected components: $\check{G}_0$ and $\check{\alpha}\check{G}_0$. The same is true for $\check{P}$. We then have the following commutative diagram with exact lines in the category of groups:
\begin{equation*}
 \xymatrix{
  \{\1\} \ar[r] & {\check{P}_0}\ar_{\subseteq}[d]\ar^{\subseteq}[r] & \check{P}\ar[r]\ar_{\subseteq}[d]\ar@/^/[ddr]^/-1.5em/{\rho_{\lambda_1}\circ\check{\pr}} & \pi_0(\check{P})\ar[r]\ar[d] & \{\1\}\\ 
  \{\1\} \ar[r] & {\check{G}_0}\ar^{\subseteq}[r]\ar@/_/[drr]_{\pi_0} & \check{G}\ar[r]\ar@{-->}[dr]|{\pi} & \pi_0(\check{G})\ar[r] & \{\1\}\\
  &&&{\calU(\calH)}
 }
\end{equation*}
Here $\calU(\calH)$ denotes the group of unitary operators on the Hilbert space $\calH=L^2(\calO,\td\mu)$. To extend $\pi_0$ to $\check{G}$ such that it agrees with $\rho_{\lambda_1}\circ\check{\pr}$ on $\check{P}$, we first have to show that $\pi_0$ coincides with $\rho_{\lambda_1}\circ\check{\pr}$ on the intersection $\check{G}_0\cap\check{P}$. By Theorem \ref{thm:IntgkModule}, they already agree on $\check{P}_0$. Therefore, we have to deal with other possible connected components of $\check{G}_0\cap\check{P}$.\\
We first claim that
\begin{align*}
 \pi_0(\check{P})\rightarrow\pi_0(\check{G})
\end{align*}
is surjective. In fact, since $G$ is generated by $P$ and $j$, the finite cover $\check{G}$ is generated by $\check{P}$ and $\check{\pr}^{-1}(j)$. Since $\check{\pr}^{-1}(j)$ is contained in the identity component $\check{G}_0$, we have $\check{G}=\check{P}\check{G}_0$ and the claim follows. Note that $\pi_0(\check{P})$ and $\pi_0(\check{G})$ are of order at most $2$.\\
Now we show that
\begin{align*}
 \pi_0 &= \rho_{\lambda_1}\circ\check{\pr} & \mbox{on }\check{G}_0\cap\check{P}.
\end{align*}
\begin{enumerate}
\item[\textup{(a)}] If $V$ is euclidean, then $\check{P}=\check{P}_0$ and $\check{G}_0\cap\check{P}=\check{P}_0$. As previously remarked, $\pi_0=\rho_{\lambda_1}\circ\check{\pr}$ holds on $\check{P}_0$.
\item[\textup{(b)}] Next, if $V=\RR^{p,q}$, then $\pi_0(\check{P})\rightarrow\pi_0(\check{G})$ is always an isomorphism (see Examples \ref{ex:ConfGrp}~(2) and \ref{ex:Gcheck}~(2)) and again it follows that $\check{G}_0\cap\check{P}=\check{P}_0$.
\item[\textup{(c)}] Finally, let $V$ be non-euclidean with $d=2d_0$. If $\check{\alpha}\in\check{G}\setminus\check{G}_0$, then the order of $\pi_0(\check{G})$ is $2$ and $\pi_0(\check{P})\rightarrow\pi_0(\check{G})$ is even an isomorphism. Hence, $\check{G}_0\cap\check{P}=\check{P}_0$ as in (a) and (b) and we are done.\\
It remains to check the case where $\check{\alpha}\in\check{K}_0\subseteq\check{G}_0$. Since $\check{P}$ is generated by $\check{P}_0$ and $\check{\alpha}$ it suffices to show $\pi_0(\check{\alpha})=\rho_{\lambda_1}(\check{\pr}(\check{\alpha}))=\rho_{\lambda_1}(\alpha)$. Let $A:=\rho_{\lambda_1}(\alpha)\circ\pi_0(\check{\alpha}^{-1})$. Since $\Ad(\check{\alpha}):\check{P}_0\rightarrow\check{P}_0$, we have for $p\in\check{P}_0$:
\begin{align*}
 A\circ\rho_{\lambda_1}(\check{\pr}(p)) &= \rho_{\lambda_1}(\alpha)\circ\pi_0(\check{\alpha}^{-1})\circ\pi_0(p)\\
 &= \rho_{\lambda_1}(\alpha)\circ\pi_0(\Ad(\check{\alpha}^{-1})p)\circ\pi_0(\check{\alpha}^{-1})\\
 &= \rho_{\lambda_1}(\check{\pr}(\check{\alpha}))\circ\rho_{\lambda_1}(\check{\pr}(\Ad(\check{\alpha}^{-1})p))\circ\pi_0(\check{\alpha}^{-1})\\
 &= \rho_{\lambda_1}(\check{\pr}(p))\circ\rho_{\lambda_1}(\check{\pr}(\check{\alpha}))\circ\pi_0(\check{\alpha}^{-1})\\
 &= \rho_{\lambda_1}(\check{\pr}(p))\circ A.
\end{align*}
Hence, $A$ commutes with the representation $\rho_{\lambda_1}|_{P_0}$. But $\rho_{\lambda_1}|_{P_0}$ is irreducible and therefore $A=z\cdot\id_{\calH}$ for some $z\in\CC$. We claim that $A\psi_0=\psi_0$ and hence $z=1$. In fact, on the one hand we have $\rho_{\lambda_1}(\alpha)\psi_0=\psi_0$ by \eqref{eq:L2Rep2}. On the other hand, $\check{\alpha}\in\check{K}_0$ and $\psi_0$ is killed by $\td\pi(\frakk)$ (see Proposition \ref{prop:Kfinite}), hence $\pi_0(\check{\alpha})\psi_0=\psi_0$. This shows the claim.
\end{enumerate}
Finally, we can show that and $\pi_0$ extends uniquely to a representation $\pi:\check{G}\rightarrow\calU(\calH)$ with $\pi=\rho_{\lambda_1}\circ\check{\pr}$ on $\check{P}$. First, uniqueness of $\pi$ is clear, since $\check{G}=\check{P}\check{G}_0$. For the existence, let $g\in\check{G}$. Then $g\check{G}_0=p\check{G}_0$ for some $p\in\check{P}$ by our previous considerations. Hence, $g=pg_0$ with $g_0\in\check{G}_0$. We then define $\pi$ by
\begin{align*}
 \pi(g) &:= \rho_{\lambda_1}(\check{\pr}(p))\pi_0(g_0).
\end{align*}
It remains to show that this gives a well-defined homomorphism $\pi:\check{G}\rightarrow\calU(\calH)$.\\
Well-definedness is again obvious: Suppose $pg_0=p'g_0'$ with $p,p'\in\check{P}$, $g_0,g_0'\in\check{G}_0$. Then $p'^{-1}p=g_0'g_0^{-1}\in\check{P}\cap\check{G}_0$. Since $\rho_{\lambda_1}\circ\check{\pr}$ and $\pi_0$ agree on $\check{P}\cap\check{G}_0$, we obtain
\begin{align*}
 \rho_{\lambda_1}(\check{\pr}(p'^{-1}))\rho_{\lambda_1}(\check{\pr}(p)) &= \pi_0(g_0')\pi_0(g_0^{-1})\\
\intertext{and hence}
 \rho_{\lambda_1}(\check{\pr}(p))\pi_0(g_0) &= \rho_{\lambda_1}(\check{\pr}(p'))\pi_0(g_0').
\end{align*}
Thus, $\pi$ is well-defined. We now prove that $\pi$ is indeed a group homomorphism.\\
For this suppose that
\begin{align*}
 p'g_0'p''g_0'' &= pg_0
\end{align*}
with $p,p',p''\in\check{P}$ and $g_0,g_0',g_0''\in\check{G}_0$. We have to show that
\begin{align*}
 \rho_{\lambda_1}(\check{\pr}(p'))\circ\pi_0(g_0')\circ\rho_{\lambda_1}(\check{\pr}(p''))\circ\pi_0(g_0'') &= \rho_{\lambda_1}(\check{\pr}(p))\circ\pi_0(g_0).
\end{align*}
Rearrangement gives $g_0'p''=p'^{-1}pg_0g_0''^{-1}$. Therefore, $p''$ and $p'^{-1}p$ lie in the same connected component of $\check{G}$. Hence, one can find $h_0\in\check{G}_0$ such that $p'^{-1}p=p''h_0$. Together we have $g_0'=p''h_0g_0g_0''^{-1}p''^{-1}=\Ad(p'')(h_0g_0g_0''^{-1})$. Thus we are finished if we show that
\begin{align*}
 \pi_0(\Ad(p)g_0) &= \Ad(\rho_{\lambda_1}(\check{\pr}(p)))\pi_0(g_0) & \forall\, p\in\check{P}, g_0\in\check{G}_0.
\end{align*}
Since $\check{G}_0$ is connected, it is generated by $\exp_{\check{G}_0}(\frakg)$. Therefore, it suffices to show
\begin{align}
 \td\pi(\Ad(p)X) &= \Ad(\rho_{\lambda_1}(p))\td\pi(X) & \forall\, p\in P,X\in\frakg.\label{eq:PhiPsiAdEq}
\end{align}
But this was already shown in Proposition \ref{prop:LieAlgRep} and the proof is complete.
\end{proof}

\begin{corollary}
Let $V$ be a simple Jordan algebra which is not of rank $r=2$ with odd dimension. Then all coefficients $(\calB|a)$, $a\in V$, of the Bessel operator extend to self-adjoint operators on $L^2(\calO,\td\mu)$.
\end{corollary}

\begin{proof}
It is a general result that for a unitary representation the Lie algebra acts by skew-adjoint operators. Then the claim follows from \eqref{eq:L2DerRep3}.
\end{proof}

\begin{remark}\label{rem:MinRep}
We do not give a proof here that $\pi$ is in fact a minimal representation in the sense of \cite{GS05}. For this we refer to \cite[Section 3.4, Remark 2]{Sah92} for the euclidean case, \cite[Remark after Theorem 0.1]{DS99} for the case of a non-euclidean Jordan algebra of rank $\geq3$, and \cite[Remark 3.7.3~(1)]{KO03a} for the case $V=\RR^{p,q}$.
\end{remark}

\subsection{Two prominent examples}\label{sec:ExReps}

We show that the representation $\pi$ of $\check{G}$ is for $V=\Sym(n,\RR)$ isomorphic to the metaplectic representation (see \cite[Chapter 4]{Fol89}) and for $V=\RR^{p,q}$ isomorphic to the minimal representation of $\upO(p+1,q+1)$ as studied by T. Kobayashi, B. {\O}rsted and G. Mano in \cite{KO03a,KO03b,KO03c,KM07a,KM08}.

\subsubsection{The metaplectic representation}\label{sec:ExMetRep}

The \textit{metaplectic representation}\index{subject}{metaplectic representation} $\mu$ as constructed in \cite[Chapter 4]{Fol89} is a unitary representation of the metaplectic group $\Mp(n,\RR)$, the double cover of the symplectic group $\Sp(n,\RR)$, on $L^2(\RR^n)$. We do not want to give a construction here, but we later state the Lie algebra action which uniquely determines the representation $\mu$. The metaplectic representation splits into two irreducible components (see \cite[Theorem 4.56]{Fol89}):
\begin{align*}
 L^2(\RR^n) &= L^2_{\textup{even}}(\RR^n) \oplus L^2_{\textup{odd}}(\RR^n),
\end{align*}
where $L^2_{\textup{even}}(\RR^n)$\index{notation}{L2evenRn@$L^2_{\textup{even}}(\RR^n)$} and $L^2_{\textup{odd}}(\RR^n)$\index{notation}{L2odeeRn@$L^2_{\textup{odd}}(\RR^n)$} denote the spaces of even and odd $L^2$-functions, respectively. We show that for $V=\Sym(n,\RR)$ the representation $\pi$ as constructed in the previous section is isomorphic to the even component $L^2_{\textup{even}}(\RR^n)$. A detailed analysis of the metaplectic representation can e.g. be found in \cite[Chapter 4]{Fol89}.\\

Denote by $\td\mu$\index{notation}{dmu@$\td\mu$} the infinitesimal version of the metaplectic representation. $\td\mu$ is a representation of $\sp(n,\RR)$ on $L^2_{\textup{even}}(\RR^n)$ by skew-adjoint operators. By \cite[Theorem 4.45]{Fol89} we have
\begin{align*}
 \td\mu\left(\begin{array}{cc}0&0\\C&0\end{array}\right) &= -\pi i\sum_{i,j=1}^n{C_{ij}y_iy_j} && \mbox{for $C\in\Sym(n,\RR)$,}\\
 \td\mu\left(\begin{array}{cc}A&0\\0&-A^t\end{array}\right) &= -\sum_{i,j=1}^n{A_{ij}y_j\frac{\partial}{\partial y_i}}-\frac{1}{2}\Tr(A) && \mbox{for $A\in M(n,\RR)$,}\\
 \td\mu\left(\begin{array}{cc}0&B\\0&0\end{array}\right) &= \frac{1}{4\pi i}\sum_{i,j=1}^n{B_{ij}\frac{\partial^2}{\partial y_i\partial y_j}} && \mbox{for $B\in\Sym(n,\RR)$.}
\end{align*}
On the other hand, $\td\pi$ is a representation of $\frakg$ on $L^2(\calO)$ by skew-adjoint operators. Now $\frakg\cong\sp(n,\RR)$ by Example \ref{ex:ConfGrp}~(1) and $L^2(\calO)\cong L^2_{\textup{even}}(\RR^n)$ by Example \ref{ex:EquivMeasures}~(1). We show that under these identifications combined with the automorphism of $\sp(n,\RR)$ given by
\begin{align*}
 \left(\begin{array}{cc}A&C\\B&-A^t\end{array}\right)\mapsto&\left(\begin{array}{cc}0&\sqrt{\pi}\\-\frac{1}{\sqrt{\pi}}&0\end{array}\right)\left(\begin{array}{cc}A&C\\B&-A^t\end{array}\right)\left(\begin{array}{cc}0&\sqrt{\pi}\\-\frac{1}{\sqrt{\pi}}&0\end{array}\right)^{-1}\\
 &=\left(\begin{array}{cc}-A^t&-\pi B\\-\frac{1}{\pi}C&A\end{array}\right),
\end{align*}
the representations $\td\mu$ and $\td\pi$ agree. More precisely, we prove the following equality of skew-adjoint operators on $L^2(\calO)$:

\begin{proposition}
For $A\in M(n,\RR)$ and $B,C\in\Sym(n,\RR)$ we have
\begin{align}
 \calU\circ\td\pi(C,A,B) &= \td\mu\left(\begin{array}{cc}-A^t&-\pi B\\-\frac{1}{\pi}C&A\end{array}\right)\circ\calU\label{eq:MetaIntertwiningFormulaDer}
\end{align}
with $\calU:L^2(\calO)\rightarrow L^2_{\textup{even}}(\RR^n)$ as in \eqref{eq:DefCalU}.
\end{proposition}

\begin{proof}
Choose an orthonormal basis of $V=\Sym(n,\RR)$ with respect to the inner product $(x|y)=\Tr(xy)$. Then for $1\leq i\leq n$:
\begin{align*}
 \frac{\partial\calU\psi}{\partial y_i}(y) &= \frac{\partial}{\partial y_i}\psi(yy^t) = \sum_\alpha{\frac{\partial\psi}{\partial x_\alpha}(yy^t)\frac{\partial(yy^t)_\alpha}{\partial y_i}} = \sum_\alpha{\frac{\partial\psi}{\partial x_\alpha}(yy^t)\frac{\partial}{\partial y_i}\Tr(yy^te_\alpha)}\\
 &= 2\sum_\alpha{\frac{\partial\psi}{\partial x_\alpha}(yy^t)(e_\alpha y)_i} = 2\left(\frac{\partial\psi}{\partial x}(yy^t)y\right)_i.
\end{align*}
\begin{enumerate}
\item[\textup{(a)}] Let $(C,0,0)\in\frakg$, $C\in\frakn=\Sym(n,\RR)$. Then
 \begin{align*}
  \left(\td\mu\left(\begin{array}{cc}0&0\\-\frac{1}{\pi}C&0\end{array}\right)\circ\calU\right)\psi(y) &= i\sum_{i,j=1}^n{C_{ij}y_iy_j}\,\calU\psi(y) = i\Tr(yy^tC)\calU\psi(y)\\  
  &= i(yy^t\psi(yy^t)|C) = \left(\calU\circ\td\pi(C,0,0)\right)\psi(y).
 \end{align*}
\item[\textup{(b)}] Let $(0,A,0)\in\frakg$, $A\in\frakl=\gl(n,\RR)$. $A$ acts on $V$ by $A\cdot x=Ax+xA^t$ (see Example \ref{ex:StrGrp} (1)). Then
 \begin{align*}
  \left(\td\mu\left(\begin{array}{cc}-A^t&0\\0&A\end{array}\right)\circ\calU\right)\psi(y) &= \sum_{i,j=1}^n{A_{ji}y_j\frac{\partial}{\partial y_i}\psi(yy^t)}+\frac{1}{2}\Tr(A)\psi(yy^t)\\
  &= 2\sum_{i=1}^n{(A^ty)_i\left(\frac{\partial\psi}{\partial x}(yy^t)y\right)_i}+\frac{1}{2}\Tr(A)\psi(yy^t)\\
  &= \left(A^t(yy^t)+(yy^t)A\left|\frac{\partial\psi}{\partial x}(yy^t)\right.\right)+\frac{1}{2}\Tr(A)\psi(yy^t)\\
  &= \left(\calU\circ\td\pi(0,A,0)\right)\psi(y),
 \end{align*}
 since
 \begin{align*}
  \Tr(V\rightarrow V,x\mapsto A\cdot x) &= (n+1)\Tr(A).
 \end{align*}
\item[\textup{(c)}] Let $(0,0,B)\in\frakg$, $B\in\overline{\frakn}=\Sym(n,\RR)$. Then
 \begin{align*}
  & \left(\td\mu\left(\begin{array}{cc}0&\pi B\\0&0\end{array}\right)\circ\calU\right)\psi(y)\\
  ={}& \frac{1}{4i}\sum_{i,j=1}^n{B_{ij}\frac{\partial^2}{\partial y_i\partial y_j}\psi(yy^t)} = \frac{1}{2i}\sum_{i,j=1}^n{B_{ij}\frac{\partial}{\partial y_i}\left[\sum_\alpha{\frac{\partial\psi}{\partial x_\alpha}(yy^t)(e_\alpha y)_i}\right]}\\
  ={}& \frac{1}{i}\sum_{i,j=1}^n{B_{ij}\sum_{\alpha,\beta}{\frac{\partial^2\psi}{\partial x_\alpha\partial x_\beta}(yy^t)(e_\alpha y)_i(e_\beta y)_j}}+\frac{1}{2i}\sum_{i,j=1}^n{B_{ij}\sum_\alpha{\frac{\partial\psi}{\partial x_\alpha}(yy^t)(e_\alpha)_{ij}}}\\
  ={}& \frac{1}{i}\sum_{\alpha,\beta}{\frac{\partial^2\psi}{\partial x_\alpha\partial x_\beta}(yy^t)\sum_{i,j=1}^n{B_{ij}(e_\alpha y)_i(e_\beta y)_j}}+\frac{1}{2i}\sum_{i,j=1}^n{B_{ij}\left(\frac{\partial\psi}{\partial x}(yy^t)\right)_{ij}}\\
  ={}& \frac{1}{i}\sum_{\alpha,\beta}{\frac{\partial^2\psi}{\partial x_\alpha\partial x_\beta}(yy^t)\left(\left.P(e_\alpha,e_\beta)(yy^t)\right|B\right)}+\frac{1}{2i}\left(\left.\frac{\partial\psi}{\partial x}(yy^t)\right|B\right)\\
  ={}& \frac{1}{i}\left(\left.\calB_{\frac{1}{2}}\psi(yy^t)\right|B\right) = \left(\calU\circ\td\pi(0,0,-B)\right)\psi(y).\qedhere
 \end{align*}
\end{enumerate}
\end{proof}

To obtain an intertwining operator between the group representations $\pi$ and $\mu$ note that the group $\check{G}$ is by Example \ref{ex:Gcheck}~(1) always a quotient of the metaplectic group $\Mp(n,\RR)$. Therefore, we can lift $\pi$ to a representation of $\Mp(n,\RR)$ which we also denote by $\pi$. Then we have the following intertwining formula:

\begin{corollary}\label{cor:IntertwinerMetaGr}
For $g\in\Mp(n,\RR)$ we have
\begin{align*}
 \calU\circ\pi(g) &= \mu\left(\Ad\left(\begin{array}{cc}0&\sqrt{\pi}\\-\frac{1}{\sqrt{\pi}}&0\end{array}\right)g\right)\circ\calU.
\end{align*}
Hence,
\begin{align*}
 \mu\left(\begin{array}{cc}0&-\sqrt{\pi}\\\frac{1}{\sqrt{\pi}}&0\end{array}\right)\circ\calU:L^2(\calO)\rightarrow L^2_{\textup{even}}(\RR^n)
\end{align*}
is an intertwining operator between $\pi$ and $\mu$.
\end{corollary}

\begin{proof}
This now follows immediately from \eqref{eq:MetaIntertwiningFormulaDer}.
\end{proof}

\begin{remark}
Together with Example \ref{ex:Gcheck}~(1) the previous proposition shows that the even part of the metaplectic representation descends to a representation of $\Sp(n,\RR)$ if $n$ is an even integer, and even to a representation of $\Sp(n,\RR)/\{\pm1\}$ if $n\in4\ZZ$. This can also seen from the explicit calculation of the cocycle of the metaplectic representation in \cite[Section 1.6]{LV80}.
\end{remark}

\subsubsection{The minimal representation of $\upO(p+1,q+1)$}

Let $V=\RR^{p,q}$, $p,q\geq2$. Then by Example \ref{ex:MinimalOrbit} the minimal orbit $\calO=\calO_1$ is the isotropic cone
\begin{align*}
 \calO &= \{x\in\RR^{p+q}:x_1^2+\ldots+x_p^2-x_{p+1}^2-\ldots-x_{p+q}^2=0\}\setminus\{0\},
\end{align*}
and the group $\check{G}_0$ is a finite cover of $\SO(p+1,q+1)_0$ by Example \ref{ex:Gcheck}~(2). In \cite{KO03c} T. Kobayashi and B. {\O}rsted construct a realization of the minimal representation of $\upO(p+1,q+1)$ on $L^2(\calO)$. We use the notation of \cite{KM08} and denote by $\omega$\index{notation}{omega@$\omega$} the minimal representation of $\upO(p+1,q+1)$ on $L^2(\calO)$. The action $\omega$ of the identity component $\SO(p+1,q+1)_0$ is uniquely determined by the corresponding Lie algebra action $\td\omega$\index{notation}{domega@$\td\omega$}. Let $f:\frakg\rightarrow\so(p+1,q+1)$ be the isomorphism of Example \ref{ex:ConfGrp}~(2). Then by \cite[equations (2.3.9), (2.3.11), (2.3.14) and (2.3.18)]{KM08} we have
\begin{align*}
 \td\omega(f(u,0,0)) &= 2i\sum_{j=1}^n{u_jx_j} = i(x|u) && \mbox{for }u\in V,\\
 \td\omega(f(0,T,0)) &= D_{T^*x} && \mbox{for }T\in\so(p,q),\\
 \td\omega(f(0,s\1,0)) &= \sum_{j=1}^n{x_j\frac{\partial}{\partial x_j}}+\frac{\mu+\nu+2}{2} && \mbox{for }s\in\RR,\\
 \td\omega(f(0,0,-\alpha v)) &= \frac{1}{2}i\sum_{j=1}^n{v_jP_j}, && \mbox{for }v\in V,
\end{align*}
where $P_j$ denotes the second order differential operator
\begin{align*}
 P_j &= \varepsilon_jx_j\Box-(2E+n-2)\frac{\partial}{\partial x_j}\index{notation}{Pj@$P_j$}
\end{align*}
with
\begin{align*}
 \Box &= \sum_{j=1}^n{\varepsilon_j\frac{\partial^2}{\partial x_j^2}},\\
 E &= \sum_{j=1}^n{x_j\frac{\partial}{\partial x_j}},\\
 \varepsilon_j &= \begin{cases}+1 & \mbox{for }1\leq j\leq p,\\-1 & \mbox{for }p+1\leq j\leq n.\end{cases}
\end{align*}

\begin{proposition}
For $X\in\frakg$ we have
\begin{align}
 \td\omega(f(X)) &= \td\pi(X).\label{eq:Intertwiningopq}
\end{align}
\end{proposition}

\begin{proof}
By \eqref{eq:L2DerRep1} and \eqref{eq:L2DerRep2} the formula \eqref{eq:Intertwiningopq} clearly holds for $X\in\frakn$ and $X\in\frakl$. It remains to check the case $X=(0,0,-\alpha v)\in\overline{\frakn}$, $v\in V$. In this case $\td\pi(X)$ is given by (see \eqref{eq:L2DerRep3})
\begin{align*}
 \td\pi(X)\psi(x) &= \frac{1}{i}(\calB\psi(x)|\alpha v).
\end{align*}
We calculate $\calB\psi(x)$ explicitly. Let $(e_j)_j$ be the standard basis of $V=\RR^n$. The dual basis with respect to the trace form $\tau$ is given by $\overline{e}_j:=\frac{1}{2}\epsilon_je_j$, where
\begin{align*}
 \epsilon_j &= \begin{cases}+1 & \mbox{for $j=1$ and }p+1\leq j\leq n,\\-1 & \mbox{for }2\leq j\leq p.\end{cases}
\end{align*}
A short calculation shows that
\begin{align*}
 P(\overline{e}_i,\overline{e}_j)x &= \begin{cases}\frac{1}{2}(x_j\overline{e}_i+x_i\overline{e}_j) & \mbox{for }i\neq j,\\\frac{1}{2}\epsilon_i(x_1e_1+x_ie_i)-\frac{1}{4}\epsilon_ix & \mbox{for }i=j\neq1,\\\frac{1}{4}x & \mbox{for }i=j=1.\end{cases}
\end{align*}
Hence,
\begin{align*}
 (\calB\psi(x)|\alpha e_k) &= \sum_{i,j=1}^n{\frac{\partial^2\psi}{\partial x_i\partial x_j}(x)\tau(P(\overline{e}_i,\overline{e}_j)x,e_k)} + \frac{p+q-2}{2}\sum_{i=1}^n{\frac{\partial\psi}{\partial x_i}(x)\tau(\overline{e}_i},e_k)\\
 &= \frac{\partial^2\psi}{\partial x_1^2}(x)\tau\left(\frac{1}{4}x,e_k\right) +\sum_{i=2}^{p+q}{\frac{\partial^2\psi}{\partial x_i^2}(x)\tau\left(\frac{1}{2}\epsilon_i(x_1e_1+x_ie_i)-\frac{1}{4}\epsilon_ix,e_k\right)}\\
 & \ \ \ \ \ \ \ \ \ \ \ +\sum_{\substack{i,j=1\\i\neq j}}^{p+q}{\frac{\partial^2\psi}{\partial x_i\partial x_j}(x)\tau\left(\frac{1}{2}(x_i\overline{e}_j+x_j\overline{e}_i),e_k\right)} + \frac{p+q-2}{2}\frac{\partial\psi}{\partial x_k}(x)\\
 &= \frac{1}{2}\epsilon_kx_k\frac{\partial^2\psi}{\partial x_1^2}(x)
 +\sum_{i=2}^{p+q}{\frac{\partial^2\psi}{\partial x_i^2}(x)
 \left(\epsilon_ix_1\delta_{1k}+x_i\delta_{ik}-\frac{1}{2}\epsilon_i\epsilon_kx_k\right)}\\
 & \ \ \ \ \ \ \ \ \ \ \ \ \ \ \ \ \ \ \ \ \ \ \ \ \ \ \ \ \ \ \ \ \ \ \ \ \ \
 +\sum_{\substack{i=1\\i\neq k}}^{p+q}{x_i\frac{\partial^2\psi}{\partial x_i\partial x_k}(x)}
 + \frac{p+q-2}{2}\frac{\partial\psi}{\partial x_k}(x)\\
 &= -\frac{1}{2}\varepsilon_kx_k\Box\psi+E\frac{\partial\psi}{\partial x_k}
 +\frac{p+q-2}{2}\frac{\partial\psi}{\partial x_k} = -\frac{1}{2}P_k.
\end{align*}
This shows \eqref{eq:Intertwiningopq} and finishes the proof.
\end{proof}

The previous proposition now implies the following result for the group representations:

\begin{corollary}
The representation $\pi$ of $\check{G}_0$ descends to the group $\SO(p+1,q+1)_0$ on which it agrees with $\omega$.
\end{corollary}

\begin{remark}
Second order differential operators similar to the operators $P_j$ appear also in \cite[Section 2]{LS99}. (In \cite{LS99} they are denoted by $\Phi_j$ and $\Theta_j$.)
\end{remark}

\section{Generalized principal series representations}\label{sec:PrincipalSeries}

In this section we show that the definition of the action $\td\pi_\lambda$ is motivated by the study of certain principal series representations $\omega_s$ of $G$. More precisely, the action of every $\td\pi_\lambda$, $\lambda\in\calW$, is obtained by taking the Fourier transform of the non-compact picture of some principal series representation $\td\omega_s$.\\

Recall that $P^{\textup{max}}$ denotes the maximal parabolic subgroup of $G$ corresponding to the maximal parabolic subalgebra $\frakp^{\textup{max}}$ (see Section \ref{sec:KKT}). $P^{\textup{max}}$ has a Langlands decomposition $P^{\textup{max}}=L^{\textup{max}}\ltimes N$ with $L^{\textup{max}}\subseteq\Str(V)$. For $s\in\CC$ we introduce the character
\begin{align*}
 \chi_s(g) &:= |\chi(g)|^{s+\frac{n}{2r}}, & g\in L^{\textup{max}},\index{notation}{chis@$\chi_s$}
\end{align*}
of $L^{\textup{max}}$ and extend it trivially to the opposite parabolic $\overline{P^{\textup{max}}}:=L^{\textup{max}}\ltimes\overline{N}$\index{notation}{Pmaxbar@$\overline{P^{\textup{max}}}$}. Consider the induced representation $(\widetilde{I}_s,\widetilde{\omega}_s):=\Ind_{\overline{P^{\textup{max}}}}^{G}(\chi_s)$\index{notation}{Istilde@$\widetilde{I}_s$}\index{notation}{omegastilde@$\widetilde{\omega}_s$} with
\begin{align*}
 \widetilde{I}_s = \{f\in C^\infty(G): f(g\overline{p})=\chi_s(\overline{p})f(g)\ \forall\, g\in G,\overline{p}\in\overline{P^{\textup{max}}}\}
\end{align*}
and $G$ acting by the left-regular representation. By the Gelfand-Naimark decomposition $N\overline{P^{\textup{max}}}\subseteq G$ is open and dense. Therefore, a function $f\in\widetilde{I}_s$ is already determined by its restriction $f_V(x):=f(n_x)$\index{notation}{fV@$f_V$} ($x\in V$) to $N\cong V$. Let $I_s$\index{notation}{Is@$I_s$} be the subspace of $C^\infty(V)$ consisting of all functions $f_V$ with $f\in\widetilde{I}_s$. Let $\omega_s$\index{notation}{omegas@$\omega_s$} be the action of $G$ on $I_s$ given by
\begin{align*}
 \omega_s(g)f_V &:= (\widetilde{\omega}_s(g)f)_V, & g\in G,\ f\in \widetilde{I}_s.
\end{align*}
This action can be written as (cf. \cite[Section 2]{Pev02})
\begin{align*}
 \omega_s(g)\eta(x) &= \chi_s(Dg^{-1}(x))f(g^{-1}x), & x\in V,
\end{align*}
for $g\in G$ and $\eta\in I_s$, where $Dg^{-1}(x)$ denotes the differential of the conformal transformation $g^{-1}$ at $x$, whenever it is defined. Calculating the differential explicitly yields (see \cite[Section 2]{Pev02})
\begin{align*}
 \omega_s(n_a)\eta(x) &= \eta(x-a), & n_a &\in N,\\
 \omega_s(g)\eta(x) &= \chi_s(g^{-1})\eta(g^{-1}x), & g &\in L^{\textup{max}},\\
 \omega_s(j)\eta(x) &= |\det(x)|^{-2s-\frac{n}{r}}\eta(-x^{-1}).
\end{align*}
Let us describe the infinitesimal version $\td\omega_s$\index{notation}{domegas@$\td\omega_s$} of $\omega_s$ (cf. \cite[Lemma 2.6]{Pev02}):
\begin{align*}
 \td\omega_s(X)\eta(x) &= -D_u\eta(x) & \mbox{for }X &= (u,0,0),\\
 \td\omega_s(X)\eta(x) &= -\left(\frac{rs}{n}+\frac{1}{2}\right)\Tr(T)\eta(x)-D_{Tx}\eta(x) & \mbox{for }X &= (0,T,0),\\
 \td\omega_s(X)\eta(x) &= -\left(2s+\frac{n}{r}\right)\tau(x,v)\eta(x)-D_{P(x)v}\eta(x) & \mbox{for }X &= (0,0,-v).
\end{align*}

Now for $\lambda\in\calW$ consider the Fourier transform $\calF_\lambda:L^2(\calO_\lambda,\td\mu_\lambda)\rightarrow\calS'(V)$\index{notation}{Flambda@$\calF_\lambda$} given by
\begin{align*}
 \calF_\lambda\psi(x) &= \int_{\calO_\lambda}{e^{-i(x|y)}\psi(y)\td\mu_\lambda(y)}, & x\in V.
\end{align*}
We show that $\calF_\lambda$ intertwines the actions of $\td\pi_\lambda$ and $\td\omega_s$ for a certain $s$:

\begin{proposition}\label{prop:IntertwinerPrincipalSeries}
Let $\lambda\in\calW$ and $s:=\frac{1}{2}\left(\lambda-\frac{n}{r}\right)$. Then for $X\in\frakg$ we have
\begin{align*}
 \calF_\lambda\circ\td\pi_\lambda(X) &= \td\omega_s(X)\circ\calF_\lambda.
\end{align*}
\end{proposition}

\begin{proof}
\begin{enumerate}
\item[\textup{(1)}] Let $X=(u,0,0)$, $u\in V$. Then
\begin{align*}
 \calF_\lambda\circ\td\pi_\lambda(X)\psi(x) &= i\int_{\calO_\lambda}{e^{-i(x|y)}\left(\left.y\psi(y)\right|u\right)\td\mu_\lambda(y)}\\
 &= \int_{\calO_\lambda}{(iy|u)e^{-i(x|y)}\psi(y)\td\mu_\lambda(y)}\\
 &= -D_u\int_{\calO_\lambda}{e^{-i(x|y)}\psi(y)\td\mu_\lambda(y)}\\
 &= \td\omega_s(X)\circ\calF_\lambda\psi(x).
\end{align*}
\item[\textup{(2)}] For $X=(0,T,0)$, $T\in\frakl$, the intertwining formula can more easily be checked on the group level. Let $g\in L$. Then by \eqref{eq:dmulambdaEquivariance}:
\begin{align*}
 \chi_s(g^{-1})\int_{\calO_\lambda}{e^{i(g^{-1}x|y)}\psi(y)\td\mu_\lambda(y)} &= \chi(g^*)^{-s-\frac{n}{2r}}\int_{\calO_\lambda}{e^{i(x|g^{-*}y)}\psi(y)\td\mu_\lambda(y)}\\
 &= \chi(g^*)^{\frac{\lambda}{2}}\int_{\calO_\lambda}{e^{i(x|y)}\psi(g^*y)\td\mu_\lambda(y)}
\end{align*}
Now the intertwining formula for the derived action follows by putting $g:=e^{tX}$ and differentiating with respect to $t$ at $t=0$.
\item[\textup{(3)}] Let $X=(0,0,-v)$, $u\in V$. By Theorem \ref{thm:BlambdaTangential} the operator $\calB_\lambda$ is symmetric on $L^2(\calO_\lambda,\td\mu_\lambda)$ and hence
\begin{align*}
 \calF_\lambda\circ\td\pi_\lambda(X)\psi(x) &= \frac{1}{i}\int_{\calO_\lambda}{e^{-i(x|y)}\left(\left.\calB_\lambda\psi(y)\right|v\right)\td\mu_\lambda(y)}\\
 &= \frac{1}{i}\int_{\calO_\lambda}{\left(\left.\calB_\lambda e^{-i(x|y)}\right|v\right)\psi(y)\td\mu_\lambda(y)}\\
 &= \frac{1}{i}\int_{\calO_\lambda}{\left(\left.\left(P\left(\frac{\partial}{\partial y}\right)y+\lambda\frac{\partial}{\partial y}\right)e^{-i(x|y)}\right|v\right)\psi(y)\td\mu_\lambda(y)}\\
 &= \frac{1}{i}\int_{\calO_\lambda}{\left(\left.P\left(-i\alpha x\right)ye^{-i(x|y)}-i\lambda\alpha xe^{-i(x|y)}\right|v\right)\psi(y)\td\mu_\lambda(y)}\\
 &= \int_{\calO_\lambda}{\left(\left(iye^{-i(x|y)}\left|P(x)v\right.\right)-\lambda\tau(x,v)e^{-i(x|y)}\right)\psi(y)\td\mu_\lambda(y)}\\
 &= \left(-D_{P(x)v}-\left(2s+\frac{n}{r}\right)\tau(x,v)\right)\int_{\calO_\lambda}{e^{-i(x|y)}\psi(y)\td\mu_\lambda(y)}\\
 &= \td\omega_s(X)\circ\calF_\lambda\psi(x).\qedhere
\end{align*}
\end{enumerate}
\end{proof}

\begin{remark}
We do not claim that for the minimal Wallach point $\lambda=\lambda_1=\frac{r_0d}{2r}$ the representation $\pi$ is a subrepresentation of $\omega_s$, $s=\frac{1}{2}\left(\lambda-\frac{n}{r}\right)$. In general this is not the case. For instance, for a euclidean Jordan algebra one has to consider principal series representations of some covering of $G$ (see \cite{Sah93}). And for $V=\RR^{p,q}$ the representation $\pi$ is for $p-q\equiv2\ (\mod\ 4)$ not a subrepresentation of the spherical principal series representation $\omega_s$, but of some non-spherical principal series (see \cite[Remark after Theorem 5.A]{Sah95}). With Proposition \ref{prop:IntertwinerPrincipalSeries} we merely want to motivate the definition of the differential action $\td\pi_\lambda$.
\end{remark}

\begin{remark}
Principal series representations as constructed above have been studied thoroughly by S. Sahi and G. Zhang. In \cite{Sah93}, \cite{Sah95} and \cite{Zha95} they determine the irreducible and unitarizable constituents of the principal series representations associated to conformal groups of euclidean and non-euclidean Jordan algebras. The proofs are of an algebraic nature. Using these results, A. Dvorsky and S. Sahi as well as L. Barchini, M. Sepanski and R. Zierau constructed unitary representations of the corresponding groups on $L^2$-spaces of orbits of the structure group. In \cite{Sah92} the case of a euclidean Jordan algebra is treated and the non-euclidean case is studied in \cite{DS99}, \cite{DS03} and \cite[Section 8]{BSZ06}. However, they all exclude the case $V=\RR^{p,q}$ with $p\neq q$, $p,q\geq2$. In this case the $L^2$-model of the minimal representation was first constructed by T. Kobayashi and B. {\O}rsted in \cite{KO03c}. Their construction does not use principal series representations. The relation to the principal series representations in this case is given in \cite[Lemma 2.9]{KO03c}.\\
In contrast to the methods of \cite{Sah93}, \cite{Sah95} and \cite{Zha95}, our construction is only carried out for the orbit of minimal rank. On the other hand, the advantage of our construction is that it includes all cases for which the minimal representation exists. Hence, it gives a unified construction of the minimal representations.
\end{remark}

\section{The $\frakk$-Casimir}\label{sec:CasimirAction}

We now compute the action of the Casimir operator of $\frakk$ on radial functions. It turns out that the Casimir acts as a certain ordinary differential operator $\cal\calD_{\mu,\nu}$ of order four. (The parameters $\mu$ and $\nu$ were defined in \eqref{eq:DefMuNu} and depend on the Jordan algebra $V$.) $\cal\calD_{\mu,\nu}$ extends to a self-adjoint operator on $L^2(\RR_+,t^{\mu+\nu+1}\td t)$. By using the $\frakk$-type decomposition of the minimal representation $\pi$, we compute the spectrum of $\cal\calD_{\mu,\nu}$ and show that its $L^2$-eigenspaces are one-dimensional.

\subsection{$\frakk$-type decomposition}

In this section we give the $\frakk$-type decomposition of the minimal representation $\pi$. Up to this point we have not used any previous results about the minimal representation $\pi$. For the proof of the $\frakk$-type decomposition we use results on principal series representations from \cite{Sah93} and \cite{Sah95} as well as the results of \cite{KO03c} for the case $V=\RR^{p,q}$. However, for the construction of the minimal representation in Section \ref{sec:MinRepConstruction}, we did not need these results.

\begin{theorem}\label{thm:KtypeDecomp}
The $K$-type decomposition of $W$ is given by
\begin{align*}
 W &\cong \bigoplus_{j=0}^\infty{W^j},
\end{align*}
where we put
\begin{align}
 W^j &:= E^{\alpha_0+j\gamma_1}.\label{eq:DefWj}\index{notation}{Wj@$W^j$}
\end{align}
\end{theorem}

\begin{proof}
Comparing the Lie algebra action (see Proposition \ref{prop:IntertwinerPrincipalSeries}), we find that the representation $\pi$ is isomorphic to the corresponding unitary irreducible representation on $L^2(\calO,\td\mu)$ constructed in \cite{Sah92} for the euclidean case, in \cite{DS99}, \cite{DS03} and \cite[Section 8]{BSZ06} for the non-euclidean case $\ncong\RR^{p,q}$ and in \cite{KO03c} for the case $V=\RR^{p,q}$. In \cite{Sah92} the algebraic results of \cite{Sah93} are used, and the constructions in \cite{DS99}, \cite{DS03} and \cite[Section 8]{BSZ06} use the results of \cite{Sah95}. Hence, for these cases the $\frakk$-type decomposition follows from \cite[Equation (7)]{Sah93} for the euclidean case and \cite[Theorem 4.B]{Sah95} for the non-euclidean case $\ncong\RR^{p,q}$. In the remaining case $V=\RR^{p,q}$ the $\frakk$-type decomposition is given in \cite[Lemma 2.6~(2)]{KO03c}. This finishes the proof.
\end{proof}



\subsection{The $\frakk$-Casimir}

Let $(X_j)_j$ be any basis of $\frakk$ and $(X_j')_j$ its dual basis with respect to the $\Ad$-invariant inner product $\langle-,-\rangle$ (see Section \ref{sec:KKT} for the definition of $\langle-,-\rangle$). We call
\begin{align*}
 C_\frakk = \sum_j{X_jX_j'}.\index{notation}{Ck@$C_\frakk$}
\end{align*}
the \textit{Casimir element}\index{subject}{Casimir element} of $\frakk$. This definition is clearly independent of the chosen basis. $C_\frakk$ is an element of the center $\calZ(\frakk)$\index{notation}{Zk@$\calZ(\frakk)$} of the universal enveloping algebra $\calU(\frakk)$ of $\frakk$ and hence it acts as a scalar on each irreducible $\frakk$-representation. In fact, one can show that $C_\frakk$ acts on the irreducible $\frakk$-module with highest weight $\alpha$ by (see e.g. \cite[Proposition 5.28]{Kna02})
\begin{align*}
 \langle\alpha,\alpha+2\rho\rangle,
\end{align*}
where $\rho$ is the half sum of all positive roots (counted with multiplicities see \eqref{eq:DefRho}). Thus we have the following action of $C_\frakk$ on the $K$-types in $W$:

\begin{proposition}\label{prop:CasimirScalar}
The Casimir operator $\td\pi(C_\frakk)$ acts on every $K$-type $W^j$ of $W$ by the scalar
\begin{align*}
 -\frac{r_0}{8n}\left(4j(j+\mu+1)+\frac{r_0d}{2}\left|d_0-\frac{d}{2}\right|\right).
\end{align*}
\end{proposition}

\begin{proof}
By \eqref{eq:DefWj} the $K$-type $W^j$ has highest weight $\alpha=\alpha_0+j\gamma_1$. Hence, we just have to calculate the inner product $\langle\alpha_0+j\gamma_1,\alpha_0+j\gamma_1+2\rho\rangle$. With \eqref{eq:Rhoi}, \eqref{eq:DefGamma0} and \eqref{eq:KappaGamma} we obtain
\begin{align*}
 \langle\alpha_0+j\gamma_1,\alpha_0+j\gamma_1+2\rho\rangle &= j^2\langle\gamma_1,\gamma_1\rangle+2j\langle\gamma_1,\alpha_0+\rho\rangle+\langle\alpha_0,\alpha_0+2\rho\rangle\\
 &= -\frac{r_0}{8n}\left(4j(j+\mu+1)+\frac{r_0d}{2}\left|d_0-\frac{d}{2}\right|\right).\qedhere
\end{align*}
\end{proof}

We now compute the Casimir action on the subspace $L^2(\calO)_{\textup{rad}}$ of radial functions.

\begin{theorem}\label{thm:CasimirAction}
Let $\psi(x)=f(|x|)$ ($x\in\calO$) be a radial function for some $f\in C^\infty(\RR_+)$. Then
\begin{align*}
 \td\pi(C_\frakk)\psi(x) &= -\frac{r_0}{8n}\left(\calD_{\mu,\nu}+\frac{r_0d}{2}\left|d_0-\frac{d}{2}\right|\right)f(|x|),
\end{align*}
where $\calD_{\mu,\nu}$ is the fourth order differential operator in one variable given by
\begin{align*}
 \calD_{\mu,\nu} &= \frac{1}{t^2}\left((\theta+\mu+\nu)(\theta+\mu)-t^2\right)\left(\theta(\theta+\nu)-t^2\right)\index{notation}{Dmunu@$\calD_{\mu,\nu}$}
\end{align*}
and $\theta=t\frac{\td}{\td t}$ denotes the one-dimensional Euler operator.
\end{theorem}

\begin{proof}
The operator $\calD_{\mu,\nu}$ can alternatively be written as
\begin{align*}
 \calD_{\mu,\nu} ={}& t^2\frac{\td^4}{\td t^4} + 2(\mu+\nu+3)t\frac{\td^3}{\td t^3} + (\mu^2+3\mu\nu+\nu^2+6(\mu+\nu)+7-2t^2)\frac{\td^2}{\td t^2}\\
 & \ \ \ \ + \left(\mu\nu(\mu+\nu)+\mu^2+3\mu\nu+\nu^2+2(\mu+\nu)+1-2(\mu+\nu+3)t^2\right)\frac{1}{t}\frac{\td}{\td t}\\
 & \ \ \ \ \ \ \ \ \ \ \ \ \ \ \ \ \ \ \ \ \ \ \ \ \ \ \ \ \ \ \ \ \ \ \ \ \ \ \ \ \ \ \ \ \ \ \ \ \ \ \ \ \ \ \ \ \ \ \ \ + (t^2-(\mu+2)(\mu+\nu+2)).
\end{align*}
Let $(e_j)_j\subseteq V$ be an orthonormal basis of $V$ with respect to the inner product $(-|-)$. Then by \eqref{eq:frakk} and \eqref{eq:KillingForm}
\begin{align*}
 C_\frakk &\equiv \frac{r}{8n}\sum_{j=1}^n{(e_j,0,-\alpha e_j)^2} \ \ \ (\mod\ \calU(\frakk_\frakl))
\end{align*}
Since $\frakk_\frakl$ annihilates radial ($=$ $K_L$-invariant) functions, the action of the Casimir element $C_\frakk$ on radial functions is already given by
\begin{align}
 \td\pi(C_\frakk) &= \frac{r}{8n}\sum_{j=1}^n{\td\pi(e_j,0,-\alpha e_j)^2}.\label{eq:CasimirActionInTermsOfej}
\end{align}
\begin{enumerate}
\item[\textup{(a)}] Let us first assume that $V$ is non-euclidean of rank $r\geq3$. Then $d=2d_0$ and hence $\mu=\frac{n}{r_0}-2$ and $\nu=\frac{d}{2}-e-1$. By \eqref{eq:KActionInTermsOfBesselOp} and Corollary \eqref{cor:BnuRadial}~(2):
\begin{align*}
 \td\pi(u,0,-\alpha u)\psi(x) &= \frac{1}{i}(B_\nu f)(|x|)(x|u),
\end{align*}
with $B_\alpha$ as in \eqref{eq:OrdinaryBesselOp}. Hence, using again \eqref{eq:KActionInTermsOfBesselOp}:
\setmyalign{\eqref{eq:BnuProdRule}}
\begin{align*}
\al\ -\sum_{j=1}^n{\td\pi(e_j,0,-\alpha e_j)^2\psi(x)}\\
 \al[\eqref{eq:BnuProdRule}]= \sum_{j=1}^n{\left[(B_\nu^2f)(|x|)(x|e_j)^2+2\tau\left(\left.P\left(\frac{\partial(B_\nu f)}{\partial x}\right|\frac{\partial(x|e_j)}{\partial x}\right)x,e_j\right)\right.}\\
 \al\ \ \ \ \ \ \ \ \ \ \ \ \ \ \ \ \ \ \ \ \ \ \ \ \ \ \ \ \ \ \ \ \ \ \ \ \ \ \ \ \ \ \ \ \ \ \ \ \ \ \ \ \ \ +\left.(B_\nu f)(|x|)\tau(\calB_\lambda(x|e_j),e_j)\right]\\
 \al[\eqref{eq:ddxradial}]= (B_\nu^2f)(|x|)\|x\|^2+\frac{2r}{r_0}\frac{1}{|x|}(B_\nu f)'(|x|)\sum_{j=1}^n{\tau\left(P\left(\left.\alpha x\right|\alpha e_j\right)x,e_j\right)}\\
 \al\ \ \ \ \ \ \ \ \ \ \ \ \ \ \ \ \ \ \ \ \ \ \ \ \ \ \ \ \ \ \ \ \ \ \ \ \ \ \ \ \ \ \ \ \ \ \ \ \ \ \ \ \ \ \ \ +\frac{r_0d}{2r}(B_\nu f)(|x|)\sum_{j=1}^n{(e_j|e_j)}\\
 \al[\eqref{eq:SumQuadRep}]= \frac{r_0}{r}\left[|x|^2(B_\nu^2f)(|x|)+\frac{2n}{r_0}|x|(B_\nu f)'(|x|)+\frac{nd}{2}(B_\nu f)(|x|)\right].
\end{align*}
Now a short calculation shows that this is equal to $\frac{r_0}{r}\calD_{\mu,\nu}$.
\item[\textup{(b)}] Now suppose $V$ is a euclidean Jordan algebra. Then $\mu=\frac{rd}{2}-1$ and $\nu=-1$. By \eqref{eq:KActionInTermsOfBesselOp} and Corollary \eqref{cor:BnuRadial}~(1)
\begin{align*}
 \td\pi(u,0,-\alpha u)\psi(x) &= \frac{1}{i}(B_\nu f)(|x|)(x|u)+\frac{1}{i}\frac{d}{2}(e|u)f'(|x|).
\end{align*}
Using the calculations of (a) we find that
\begin{align*}
 & -\sum_{j=1}^n{\td\pi(e_j,0,-\alpha e_j)^2\psi(x)}\\
 ={}& \left[|x|^2(B_\nu^2f)(|x|)+\frac{2n}{r_0}|x|(B_\nu f)'(|x|)+\frac{nd}{2}(B_\nu f)(|x|)\right]\\
 & \ \ +\sum_{j=1}^n{\left[\frac{d}{2}(x|e_j)(e|e_j)\left[(B_\nu f)'(|x|)+B_\nu f'(|x|)\right]+\left(\frac{d}{2}\right)^2(e|e_j)^2f''(|x|)\right]}\\
 ={}& \left[|x|^2(B_\nu^2f)(|x|)+\frac{2n}{r_0}|x|(B_\nu f)'(|x|)+\frac{nd}{2}(B_\nu f)(|x|)\right]\\
 & \ \ \ \ \ \ \ \ \ \ \ \ \ \ \ \ \ \ \ \ \ \ \ \ +\frac{d}{2}(x|e)\left[(B_\nu f)'(|x|)+B_\nu f'(|x|)\right]+r\left(\frac{d}{2}\right)^2f''(|x|).
\end{align*}
Note that for $x=ktc_1$, $k\in K_L=\Aut(V)_0$, $t>0$, we have
\begin{align*}
 (x|e) &= \tr(ktc_1) = t = |ktc_1| = |x|.
\end{align*}
Using this, a short calculation gives
\begin{align*}
 \td\pi(C_\frakk)\psi(x) &= -\frac{r_0}{8n}\left(\calD_{\mu,\nu}+r_0\left(\frac{d}{2}\right)^2\right)f(|x|).
\end{align*}
\item[\textup{(c)}] By Proposition \ref{prop:ClassificationEuclSph} the remaining case is $V=\RR^{p,q}$ which is treated in Appendix \ref{sec:Rank2Casimir}.\qedhere
\end{enumerate}
\end{proof}

Recall that $\Xi$ denotes the set of possible pairs $(\mu,\nu)=(\mu(V),\nu(V))$ of parameters that appear for some simple Jordan algebra $V$ for which the minimal representation $\pi$ exists. Then we can draw the following corollary:

\begin{corollary}\label{cor:DmunuSASpectrum}
For $(\mu,\nu)\in\Xi$ the operator $\calD_{\mu,\nu}$ extends to a self-adjoint operator on $L^2(\RR_+,t^{\mu+\nu+1}\td t)$ with only discrete spectrum. The spectrum is given by $\{j(j+\mu+1):j\in\NN_0\}$ and the $L^2$-eigenspaces are one-dimensional.
\end{corollary}

\begin{proof}
Recall that by Theorem \ref{thm:CasimirAction} the operator $\td\pi(C_\frakk)$ acts on the subspace $L^2(\calO,\td\mu)_{\textup{rad}}\cong L^2(\RR_+,t^{\mu+\nu+1}\td t)$ of radial functions by
\begin{align*}
 -\frac{r_0}{8n}\left(\calD_{\mu,\nu}+\frac{r_0d}{2}\left|d_0-\frac{d}{2}\right|\right).
\end{align*}
Further, $\td\pi(C_\frakk)$ acts on each $\frakk$-type $W^j$ by the scalar
\begin{align*}
 -\frac{r_0}{8n}\left(4j(j+\mu+1)+\frac{r_0d}{2}\left|d_0-\frac{d}{2}\right|\right).
\end{align*}
Now, the Casimir element $C_\frakk$ is elliptic in $\calU(\frakk)$ and hence, by \cite[Theorem 4.4.4.3]{War72}, $\td\pi(C_\frakk)$ extends to a self-adjoint operator on $L^2(\calO,\td\mu)$. Restricting to radial functions then shows that $\calD_{\mu,\nu}$ is self-adjoint on $L^2(\RR_+,t^{\mu+\nu+1}\td t)$. Further, the space $W$ of $\frakk$-finite vectors is dense in $L^2(\calO,\td\mu)$ and decomposes discretely into $\frakk$-types $W=\bigoplus_{j=0}^\infty{W^j}$. In every $\frakk$-type $W^j$ the space
\begin{align*}
 W^j_{\textup{rad}} &= W^j\cap L^2(\calO,\td\mu)_{\textup{rad}}
\end{align*}
of radial ($=$ $K_L$-invariant) functions is one-dimensional by the remark at the end of Section \ref{subsec:ConfGrpRoots}. Together it follows that $\bigoplus_{j=0}^\infty{W^j_{\textup{rad}}}$ is dense in $L^2(\RR_+,t^{\mu+\nu+1}\td t)$ and $\calD_{\mu,\nu}$ acts on each summand $W^j_{\textup{rad}}$ by the scalar $4j(j+\mu+1)$. This finishes the proof.
\end{proof}

\section{The unitary inversion operator $\calF_\calO$}\label{sec:UnitInvOp}

In this section we define the unitary inversion operator $\calF_\calO$ and prove various properties of it. The action of $\calF_\calO$ together with the action of the parabolic subgroup $\check{P}$ determine the whole representation. Therefore, one is interested in a closed formula for the operator $\calF_\calO$. As a first step in this direction we give a closed formula for the action of $\calF_\calO$ on radial functions.\\

Let $\check{w_0}$\index{notation}{w0check@$\check{w_0}$} be the projection of $\widetilde{w_0}\in\widetilde{G}$ (see Section \ref{sec:Guniversalcover}) under the covering map $\widetilde{G}\rightarrow\check{G}$. Then $\check{w_0}=\check{\alpha}\check{j}=\check{j}\check{\alpha}$, where $\check{\alpha}$ is as in Section \ref{sec:IntgkModule} and
\begin{align*}
 \check{j} &= \exp_{\check{G}}\left(\frac{\pi}{2}(e,0,-e)\right).\index{notation}{jcheck@$\check{j}$}
\end{align*}
The parabolic $\check{P}$ and the element $\check{w_0}$ generate the whole group $\check{G}$ (since $\check{\alpha}\in\check{P}$). Therefore, the representation $\pi$ of $\check{G}$ is determined by the action of $\check{P}$, which is given by the representation $\rho_{\lambda_1}$ (see \eqref{eq:L2Rep1} and \eqref{eq:L2Rep2}), and the action of $\check{w_0}$. We call the operator
\begin{align*}
 \calF_\calO &:= e^{-i\pi\frac{r_0}{2}(d_0-\frac{d}{2})_+}\pi(\check{w_0})\index{notation}{FO@$\calF_\calO$}
\end{align*}
the \textit{unitary inversion operator}\index{subject}{unitary inversion operator} on the minimal orbit $\calO=\calO_1$. We will later see that $\calF_\calO$ is an operator of order two (see Corollary \ref{cor:ActionFkTypes}) which justifies the name. Since the action of $\pi(\check{\alpha})$ is given by $\rho_{\lambda_1}(\alpha)$ and any two Cartan involutions are conjugate, the operator $\calF_\calO$ does (up to unitary equivalence) neither depend on the choice of the Cartan involution $\alpha$, nor on the choice of $\widetilde{\alpha}\in\widetilde{G}$. We collect a few properties of $\calF_\calO$.

\begin{theorem}\label{prop:FOProperties}
\begin{enumerate}
\item[\textup{(1)}] $\calF_\calO$ is a unitary operator on $L^2(\calO)$ of order at most $2k$ with $k$ as in Theorem \ref{thm:IntgkModule} (1).
\item[\textup{(2)}] The operator $\calF_\calO$ is an automorphism of the following topological vectorspaces:
\begin{align*}
 L^2(\calO)^\infty \subseteq L^2(\calO) \subseteq L^2(\calO)^{-\infty},
\end{align*}
where $L^2(\calO)^\infty$\index{notation}{L2Oinfty@$L^2(\calO)^\infty$} denotes the space of smooth vectors of the representation $\pi$ and $L^2(\calO)^{-\infty}$\index{notation}{L2Ominusinfty@$L^2(\calO)^{-\infty}$} its dual.
\item[\textup{(3)}] $\calF_\calO$ intertwines the Bessel operator $\calB$ and multiplication by $-\alpha x$:
\begin{align}
 \calF_\calO\circ\alpha x &= -\calB\circ\calF_\calO,\label{eq:CommFOx}\\
 \calF_\calO\circ\calB &= -\alpha x\circ\calF_\calO.\label{eq:CommFOB}
\end{align}
Moreover, any other unitary operator on $L^2(\calO)$ with these properties is a scalar multiple of $\calF_\calO$.
\item[\textup{(4)}] We have the following commutation relation for the Euler operator $E:=\sum_{j=1}^n{x_j\frac{\partial}{\partial x_j}}$:
\begin{align}
 \calF_\calO\circ E &= -\left(E+\frac{r_0d}{2}\right)\circ\calF_\calO.\label{eq:CommFOE}
\end{align}
\item[\textup{(5)}] On every $\frakk$-type $W^j$ the unitary inversion operator $\calF_\calO$ acts as a scalar.
\item[\textup{(6)}] $\calF_\calO$ leaves the space $L^2(\calO)_{\rad}$ of radial functions invariant and therefore restricts to a unitary operator
\begin{align*}
 \calF_{\calO,\rad}:L^2(\calO)_{\rad}\rightarrow L^2(\calO)_{\rad}.\index{notation}{FOrad@$\calF_{\calO,\rad}$}
\end{align*}
\item[\textup{(7)}] $\calF_\calO=e^{-i\pi\frac{r_0}{2}(d_0-\frac{d}{2})_+}\rho_{\lambda_1}(\alpha)e^{i\frac{\pi}{2}(e|x-\calB)}=e^{-i\pi\frac{r_0}{2}(d_0-\frac{d}{2})_+}e^{i\frac{\pi}{2}(e|x-\calB)}\rho_{\lambda_1}(\alpha)$.
\end{enumerate}
\end{theorem}

\begin{proof}
\begin{enumerate}
\item[\textup{(1)}] Clearly $\calF_\calO$ is unitary since $\pi$ is a unitary representation. To show the second statement, observe that
\begin{align*}
 \calF_\calO &= e^{-i\pi\frac{r_0}{2}(d_0-\frac{d}{2})_+}\rho_{\lambda_1}(\alpha)\pi(\check{j}).
\end{align*}
First, by the definition of $k$ (see Lemma \ref{lem:FiniteCover}), $e^{-i\pi\frac{r_0}{2}(d_0-\frac{d}{2})_+}$ is of order at most $2k$. Further, the operator $\rho_{\lambda_1}(\alpha)$ is of order $2$ since $\alpha^2=\1$. Finally, by the construction of $\check{G}$ (see Lemma \ref{lem:FiniteCover}), the element $\check{j}$ is of order at most $2k$ in $\check{G}$. Since all three factors commute, this shows the claim.
\item[\textup{(2)}] The whole group $G$ acts by automorphisms on the space $L^2(\calO)^\infty$ of smooth vectors and hence also on its dual.
\item[\textup{(3)}] The adjoint action of $\check{w_0}$ on $\frakg$ is given by (see \eqref{eq:ThetaOng})
\begin{align*}
 \Ad_\frakg(\check{w_0})(u,T,v) &= \Ad_\frakg(w_0)(u,T,v) = \theta(u,T,v) = (-\alpha v,-T^*,-\alpha u).
\end{align*}
and for any $X\in\frakg$ the identity
\begin{align*}
 \pi(\check{w_0})\circ\td\pi(X) &= \td\pi(\Ad(\check{w_0})X)\circ\pi(\check{w_0})
\end{align*}
holds. Now, for $X=(u,0,0)$ we have $\Ad(w_0)X=(0,0,-\alpha u)$ and for $Y=(0,0,u)$ we have $\Ad(w_0)Y=(-\alpha u,0,0)$. Therefore the commutation relations follow with \eqref{eq:L2DerRep1} and \eqref{eq:L2DerRep3}. Conversely, let $A$ be another unitary operator on $L^2(\calO)$ with these properties. Then the operator $\calF_\calO\circ A^{-1}$ commutes with the $\frakn$- and $\overline{\frakn}$-action. Since $\frakn$ and $\overline{\frakn}$ generate the whole Lie algebra $\frakg$ (see Section \ref{sec:KKT}), it follows that $A$ leaves the $\frakg$-module $W$ invariant and the operator $\calF_\calO\circ A^{-1}$ commutes with the action of $\frakg$. Since $W$ is irreducible as $\frakg$-module, it follows from Schur's Lemma that $\calF_\calO\circ A^{-1}$ is a scalar multiple of the identity. Hence, $A$ is a scalar multiple of $\calF_\calO$.
\item[\textup{(4)}] Similar to (3) with $X=(0,\id,0)$.
\item[\textup{(5)}] By Lemma \ref{lem:Tildew0Central} the element $\widetilde{w_0}$ is central in $\widetilde{K}$. Hence, $\check{w_0}$ is central in $\check{K}$. By Schur's Lemma, $\calF_\calO$ acts on every $\frakk$-type $W^j$ as a scalar.
\item[\textup{(6)}] In particular, $\check{w_0}$ commutes with every $k\in\check{K_L}=\check{\pr}^{-1}(K_L)$. Hence, $\calF_\calO$ leaves the space $L^2(\calO)^{K_L}=L^2(\calO)_{\rad}$ of $K_L$-invariant functions invariant.
\item[\textup{(7)}] We have
\begin{align*}
 \pi(\check{w_0}) &= \pi(\check{\alpha})\pi(e^{\frac{\pi}{2}(e,0,-e)}) = \rho_{\lambda_1}(\alpha)e^{\frac{\pi}{2}\td\pi(e,0,-e)}
\end{align*}
and the claim follows from \eqref{eq:L2DerRep1} and \eqref{eq:L2DerRep3}.\qedhere
\end{enumerate}
\end{proof}

\begin{example}\label{ex:UnitInvOp}
\begin{enumerate}
\item[\textup{(1)}] Let $V=\Sym(n,\RR)$. In the notation of Section \ref{sec:ExReps} we have by Corollary \ref{cor:IntertwinerMetaGr}:
\begin{align*}
 \calF_\calO &= e^{-i\pi\frac{n}{4}}\calU^{-1}\mu\left(\begin{array}{cc}0&\pi\\-\frac{1}{\pi}&0\end{array}\right)\calU.
\end{align*}
From \cite[Equation (4.26)]{Fol89} we know that $\mu\left(\begin{array}{cc}0&\pi\\-\frac{1}{\pi}&0\end{array}\right)$ is essentially the inverse euclidean Fourier transform. More precisely,
\begin{align*}
 \mu\left(\begin{array}{cc}0&\pi\\-\frac{1}{\pi}&0\end{array}\right)\psi(x) &= e^{i\pi\frac{n}{4}}2^{\frac{n}{2}}\calF_{\RR^n}^{-1}\psi(2y),
\end{align*}
where
\begin{align*}
 \calF_{\RR^n}\psi(x) &= (2\pi)^{-\frac{n}{2}}\int_{\RR^n}{e^{-ix\cdot y}\psi(y)\td y},\index{notation}{FRn@$\calF_{\RR^n}$}\\
 \calF_{\RR^n}^{-1}\psi(y) &= (2\pi)^{-\frac{n}{2}}\int_{\RR^n}{e^{ix\cdot y}\psi(x)\td x}.
\end{align*}
Note that if one views the Fourier transform $\calF_{\RR^n}$ as operator on $L^2_{\textup{even}}(\RR^n)$, then it is of order two since
\begin{align*}
 \calF_{\RR^n}^2\psi(x) &= \psi(-x) = \psi(x)
\end{align*}
for $\psi\in L^2_{\textup{even}}(\RR^n)$. Therefore, it follows that $\calF_\calO$ is also of order $2$. Further, by \cite[Theorem 4.45]{Fol89}:
\begin{align*}
 \calU L^2(\calO)^\infty &= L^2_{\textup{even}}(\RR^n)^\infty = \calS_{\textup{even}}(\RR^n)
\end{align*}
is the Schwartz space of even functions on which the Fourier transform acts as isomorphism. This corresponds to Theorem \ref{prop:FOProperties}~(2). Moreover, the commutation relations \eqref{eq:CommFOx} and \eqref{eq:CommFOB} follow from to the well-known identities
\begin{align*}
 \calF_{\RR^n}\circ x_j &= -D_j\circ\calF_{\RR^n},\\
 \calF_{\RR^n}\circ D_j &= x_j\circ\calF_{\RR^n}.
\end{align*}
where $D_j=\frac{1}{i}\frac{\partial}{\partial x_j}$.
\item[\textup{(2)}] Let $V=\RR^{p,q}$. Then $\calF_\calO$ is the unitary inversion operator on $L^2(\calO)$ which was studied in detail by T. Kobayashi and G. Mano in \cite{KM07a,KM08}. Most results of Theorem \ref{prop:FOProperties} can be found in \cite[Theorem 2.5.2]{KM08}.
\end{enumerate}
\end{example}

By Theorem \ref{prop:FOProperties}~(6) the operator $\calF_\calO$ restricts to an operator on $L^2(\calO)_{\rad}$. Since the map $\calO\rightarrow\RR_+,\,x\mapsto|x|,$ induces an isomorphism $L^2(\calO)_{\rad}\cong L^2(\RR_+,t^{\mu+\nu+1}\td t)$, we obtain a unitary operator $\calT$ on $L^2(\RR_+,t^{\mu+\nu+1}\td t)$ which makes the following diagram commutative:\index{notation}{T@$\calT$}
\begin{equation*}
\begin{xy}\xymatrix{
L^2(\RR_+,t^{\mu+\nu+1}\td t) \ar[d]_{\sim} \ar[r]^{\calT} & L^2(\RR_+,t^{\mu+\nu+1}\td t) \ar[d]^{\sim}\\
L^2(\calO)_{\rad} \ar[r]^{\calF_{\calO,\rad}}  & L^2(\calO)_{\rad}
}\end{xy}
\end{equation*}

The main result of this section is an explicit expression of the integral kernel of $\calT$ in terms of Meijer's $G$-function. The idea of proof is due to T. Kobayashi and G. Mano who proved the result for the case $V=\RR^{p,q}$ (see \cite[Theorem 4.1.1]{KM08}).

\begin{theorem}\label{thm:RadialUnitaryInversion}
The operator $\calT$ is the $G$-transform $\calT^{\mu,\nu}$ which is defined by
\begin{align*}
 \calT^{\mu,\nu}u(s) &= \int_0^\infty{K^{\mu,\nu}(tt')u(t')t'^{\mu+\nu+1}\td t'} & \forall\,u\in C_c^\infty(\RR_+),\index{notation}{Tmunu@$\calT^{\mu,\nu}$}
\end{align*}
with integral kernel\index{notation}{Kmunux@$K^{\mu,\nu}(x)$}
\begin{align*}
 K^{\mu,\nu}(t) &:= \frac{1}{2^{\mu+\nu+1}}G^{20}_{04}\left(\left(\frac{t}{4}\right)^2\left|0,-\frac{\nu}{2},-\frac{\mu}{2},-\frac{\mu+\nu}{2}\right.\right).
\end{align*}
\end{theorem}

Here $G^{20}_{04}(z|b_1,b_2,b_3,b_4)$ denotes Meijer's $G$-function as defined in Appendix \ref{app:GFct}.

\begin{remark}
For the case $V=\RR^{p,q}$ T. Kobayashi and G. Mano computed the action of $\calF_\calO$ on every $(K_L)_0$-isotypic component of $L^2(\calO)$, not only on radial functions (see \cite[Theorem 4.1.1]{KM08}). As integral kernels they obtained $G$-functions with more general parameters. They further use this result to compute the full integral kernel $K(x,y)\in\calD'(\calO\times\calO)$ of the operator $\calF_\calO$ (see \cite[Theorem 5.1.1]{KM08}). Maybe a similar strategy can be applied in the general case. A further step in this direction would be to compute the action of $\calF_\calO$ of other $(K_L)_0$-isotypic components of $L^2(\calO)$.
\end{remark}

\begin{example}
For $\nu=\pm1$ the integral kernel simplifies by \eqref{eq:GFctKBessel} to a $J$-Bessel function
\begin{align*}
 K^{\mu,\pm1}(t) &= t^{-\frac{\mu+\nu+1}{2}}J_\mu(2t^{\frac{1}{2}})
\end{align*}
such that $\calT^{\mu,\pm1}$ becomes a Hankel type transform. Then, in view of Example \ref{ex:UnitInvOp}~(1), Theorem \ref{thm:RadialUnitaryInversion} for $V=\Sym(n,\RR)$ corresponds to the fact, that the euclidean Fourier transform preserves the space of radial functions and acts on $\psi(x)=f(|x|)$ by (see \cite[\S 3, Theorem 3.3]{SW71})
\begin{align*}
 \calF_{\RR^n}\psi(x) &= H_nf(|x|),
\end{align*}
where
\begin{align*}
 H_nf(t) &= t^{-\frac{n-2}{2}}\int_0^\infty{J_{\frac{n-2}{2}}(tt')f(t')t'^{\frac{n}{2}}\td t'}\index{notation}{Hn@$H_n$}
\end{align*}
is the \textit{Hankel transform}\index{subject}{Hankel transform}. The same happens for the euclidean Jordan algebra $V=\RR^{1,n}$. This case was studied thoroughly in \cite{KM07a}.
\end{example}

The rest of this section is devoted to the proof of Theorem \ref{thm:RadialUnitaryInversion}. For this we transfer the situation from $\RR_+$ to $\RR$ in order to use classical Fourier analysis. We introduce two unitary isomorphisms
\begin{align*}
 \sigma_+: L^2(\RR_+,t^{\mu+\nu+1}\td t) &\rightarrow L^2(\RR), & \sigma_+f(y) &:= e^{\frac{\mu+\nu+2}{2}y}f(e^y),\index{notation}{sigma1@$\sigma_+$}\\
 \sigma_-: L^2(\RR_+,t^{\mu+\nu+1}\td t) &\rightarrow L^2(\RR), & \sigma_-f(y) &:= e^{-\frac{\mu+\nu+2}{2}y}f(e^{-y}).\index{notation}{sigma2@$\sigma_-$}
\end{align*}
Define the subspace $\calS\subseteq L^2(\RR_+,t^{\mu+\nu+1}\td t)$ by
\begin{align*}
 \calS := \sigma_+^{-1}(\calS(\RR)) = \sigma_-^{-1}(\calS(\RR)),\index{notation}{S@$\calS$}
\end{align*}
where $\calS(\RR)$ denotes the space of Schwartz functions on $\RR$. We endow $\calS$ with the locally convex topology such that $\sigma_+$ and $\sigma_-$ become isomorphisms of topological vectorspaces. By $\calS'$ we denote the dual space of $\calS$. Via duality $\sigma_+$ and $\sigma_-$ then extend to isomorphisms of $\calS'$. For any $\kappa\in\calS'$ one can define an operator $A_\kappa:\calS\rightarrow\calS'$ by 
\begin{align*}
 A_\kappa f(t) &:= \int_0^\infty{\kappa(tt')f(t')t'^{\mu+\nu+1}\td t'},\index{notation}{Akappa@$A_\kappa$}
\end{align*}
meant in the distribution sense. It is easily seen that
\begin{align}
 A_\kappa f &= \sigma_-^{-1}(\sigma_-\kappa*\sigma_+f) & \forall\, f\in\calS.\label{eq:Akappasigma}
\end{align}
This shows that $A_\kappa$ indeed defines a continuous linear operator $\calS\rightarrow\calS'$. Now, to prove Theorem \ref{thm:RadialUnitaryInversion} we have to show that $\calT=A_K$ with $K=K^{\mu,\nu}$. Our strategy of proof is due to T. Kobayashi and G. Mano (cf. \cite[Section 4]{KM08}) and can be described as follows:
\begin{enumerate}
\item[\textup{(1)}] We first show that $\calT=A_\kappa$ for some $\kappa\in\calS'$.
\item[\textup{(2)}] Then we prove that $A_\kappa f_0=A_Kf_0$ for a specific function $f_0$.
\item[\textup{(3)}] Finally, (2) will imply that $\kappa=K$.
\end{enumerate}
These claims are proved in the following three subsections.

\subsection{Translation invariant operators on $\RR$}

We recall the following well-known fact which can e.g. be found in \cite[Theorem I.3.18]{SW71}:

\begin{fact}\label{fct:TranslationInvariantOperators}
Every bounded translation invariant operator $B$ on $L^2(\RR)$ is a convolution operator, i.e. $Bf=u*f$ $(f\in L^2(\RR))$ for some tempered distribution $u\in\calS'(\RR)$ whose Fourier transform $\widehat{u}$ is in $L^\infty(\RR)$.
\end{fact}

Here translation invariant means that
\begin{align*}
 B\circ\ell(x) &= \ell(x)\circ B & \forall\, x\in\RR,
\end{align*}
where $\ell(x):L^2(\RR)\rightarrow L^2(\RR)$, $(\ell(x)f)(y):=f(y-x)$\index{notation}{lx@$\ell(x)$} denotes the translation operator.

To transfer our situation from $\RR_+$ to $\RR$ we define an operator $\widetilde{\calT}$\index{notation}{Ttilde@$\widetilde{\calT}$} on $L^2(\RR)$ by the following diagram:
\begin{equation*}
\begin{xy}\xymatrix{
L^2(\RR_+,t^{\mu+\nu+1}\td t) \ar[d]_{\sigma_+} \ar[r]^{\calT} & L^2(\RR_+,t^{\mu+\nu+1}\td t) \ar[d]^{\sigma_-}\\
L^2(\RR) \ar[r]^{\widetilde{\calT}}  & L^2(\RR)
}\end{xy}
\end{equation*}

\begin{lemma}
$\widetilde{\calT}$ is translation invariant.
\end{lemma}

\begin{proof}
We put $H:=L(e)\in\frakl$ and consider $\exp(sH)=e^s\1\in L$ for $s\in\RR$. Since $\Ad(w_0)\exp(sH)=\exp(-sH)$ we have
\begin{align*}
 \calF_\calO\circ\rho_{\lambda_1}(\exp(-sH)) &= \rho_{\lambda_1}(\exp(sH))\circ\calF_\calO.
\end{align*}
Restricting this identity to radial functions yields
\begin{align*}
 \calT\circ\varrho(s) &= \varrho(-s)\circ\calT,
\end{align*}
where $\varrho(s)$ denotes the unitary operator on $L^2(\RR_+,t^{\mu+\nu+1}\td t)$ given by
\begin{align*}
 \varrho(s)f(t)=e^{-\frac{\mu+\nu+2}{2}s}f(e^{-s}t).
\end{align*}
Multiplying with $\sigma_-$ from the left and $\sigma_+^{-1}$ from the right gives
\begin{align*}
 \widetilde{\calT}\circ(\sigma_+\circ\varrho(s)\circ\sigma_+^{-1}) &= (\sigma_-\circ\varrho(-s)\circ\sigma_-^{-1})\circ\widetilde{\calT}.
\end{align*}
Now the claim follows from the identities
\begin{align*}
 \sigma_+\circ\varrho(s)\circ\sigma_+^{-1} &= \sigma_-\circ\varrho(-s)\circ\sigma_-^{-1}=\ell(s).\qedhere
\end{align*}
\end{proof}

By Fact \ref{fct:TranslationInvariantOperators}, $\widetilde{\calT}$ is a convolution operator, i.e. for some tempered distribution $u\in\calS'(\RR)$ we have $\widetilde{\calT}f = u*f$ for every $f\in L^2(\RR)$. Put $\kappa:=\sigma_-^{-1}u\in\calS'$\index{notation}{kappa@$\kappa$}. We then obtain
\begin{align*}
 (\sigma_-\circ\calT)f &= (\widetilde{\calT}\circ\sigma_+)f\\
 &= \sigma_-\kappa*\sigma_+f\\
 &= (\sigma_-\circ A_\kappa)f
\end{align*}
by \eqref{eq:Akappasigma}. Since $\sigma_-$ is an isomorphism this implies $\calT=A_\kappa$.

\subsection{Action on $\psi_0$}

Let $f_0$ be the function on $\RR_+$ defined by
\begin{align*}
 f_0(t) := \widetilde{K}_{\frac{\nu}{2}}(t).\index{notation}{f0@$f_0$}
\end{align*}
$f_0$ is exactly the radial part of the $\frakk$-finite vector $\psi_0$ introduced in Section \ref{sec:ConstructiongkModule}. For the equation $A_\kappa f_0=A_Kf_0$ to make sense we first have to show that $f_0\in\calS$ and $K=K^{\mu,\nu}\in\calS'$.

\begin{lemma}
$f_0\in\calS$.
\end{lemma}

\begin{proof}
We show that $\sigma_+f_0\in\calS(\RR)$. For this define functions $f_k$ ($k\in\NN_0$) on $\RR_+$ by
\begin{align*}
 f_k(t) &:= t^{2k}\widetilde{K}_{\frac{\nu}{2}+k}(t).
\end{align*}
Then $f_0$ is as above. We have
\begin{align*}
 \sigma_+f_k(y) &= e^{(\frac{\mu+\nu+2}{2}+2k)y}\widetilde{K}_{\frac{\nu}{2}+k}(e^y).
\end{align*}
The asymptotic behavior of the $K$-Bessel function near $0$ and $\infty$ is given by \eqref{eq:BesselKAsymptAt0} and \eqref{eq:BesselAsymptAtInfty}. From this together with Lemma \ref{lem:MuNuProperties} it easily follows that the functions $\sigma_+f_k$ are rapidly decreasing, i.e. $y^\ell\sigma_+f_k(y)$ is bounded on $\RR$ for all $k,\ell\in\NN_0$. Finally, from the differential recurrence relation \eqref{eq:BesselDiffFormulas} for the $K$-Bessel function one deduces the following recurrence identity for the derivatives of the functions $\sigma_+f_k$:
\begin{align*}
 \frac{\td}{\td y}(\sigma_+f_k) &= \left(\frac{\mu+\nu+2}{2}+2k\right)(\sigma_+f_k) - \frac{1}{2}(\sigma_+f_{k+1}).
\end{align*}
Hence, higher derivatives of $\sigma_+f_0$ are linear combinations of the functions $\sigma_+f_k$ which are rapidly decreasing by the above considerations. Therefore, $\sigma_+f_0\in\calS(\RR)$.
\end{proof}

To show that $K^{\mu,\nu}\in\calS'$ we first prove a precise statement for the asymptotic behavior of $K^{\mu,\nu}(t)$.

\begin{lemma}\label{lem:KmunuAsymptotics}
The function $K^{\mu,\nu}(t)$ has the following asymptotic behavior:
\begin{enumerate}
\item[\textup{(1)}] As $t\rightarrow0$:
 \begin{align*}
  K^{\mu,\nu}(t) &= \frac{1}{2^{\mu+1}\Gamma(\frac{\mu-|\nu|+2}{2})\Gamma(\frac{\mu+2}{2})}\times\begin{cases}2^\nu\Gamma\left(\frac{\nu}{2}\right)t^{-\nu}+o(t^{-\nu}) & \mbox{for }\nu>0,\\-2\ln(t)+o(\ln(t)) & \mbox{for }\nu=0,\\2^{-\nu}\Gamma\left(-\frac{\nu}{2}\right)+o(1) & \mbox{for }\nu<0.\end{cases}
 \end{align*}
\item[\textup{(2)}] As $t\rightarrow\infty$:
 \begin{align*}
  K^{\mu,\nu}(t) &= -\frac{1}{\sqrt{\pi}}t^{-\frac{2\mu+2\nu+3}{4}}\cos\left(2t^{\frac{1}{2}}-\frac{2\mu-3}{4}\pi\right)\left(1+\calO(t^{-\frac{1}{2}})\right).
 \end{align*}
\end{enumerate}
\end{lemma}

\begin{proof}
This follows directly from \eqref{eq:GFctAsymptotics0} and \eqref{eq:GFctAsymptoticsInfty}.
\end{proof}

The proof of the following lemma does not follow \cite{KM08}. The corresponding proof of \cite[Claim 4.6.4]{KM08} seems more complicated than necessary.

\begin{lemma}
$K^{\mu,\nu}\in\calS'$.
\end{lemma}

\begin{proof}
Using Lemma \ref{lem:KmunuAsymptotics}~(1) (and Lemma \ref{lem:MuNuProperties}) it is easily verified that there is some constant $C>0$ such that $|\sigma_+K^{\mu,\nu}(y)|\leq Ce^{\frac{1}{2}y}$ as $y\rightarrow-\infty$. Hence, $\chi_{(-\infty,0]}\sigma_+K^{\mu,\nu}\in L^1(\RR)\subseteq\calS'(\RR)$, where $\chi_A$ denotes the characteristic function of $A\subseteq\RR$. It remains to show that also $\chi_{[0,\infty)}\sigma_+K^{\mu,\nu}\in\calS'(\RR)$. By Lemma \ref{lem:KmunuAsymptotics}~(2), the asymptotic behavior of $(\sigma_+K^{\mu,\nu})(y)$ as $y\rightarrow\infty$ is given by
\begin{align*}
 (\sigma_+K^{\mu,\nu})(y) &= -\frac{1}{\sqrt{\pi}}e^{\frac{1}{4}y}\cos\left(2e^{\frac{1}{2}y}-\frac{2\mu-3}{4}\pi\right)\left(1+\calO(e^{-\frac{1}{2}y})\right).
\end{align*}
Therefore $\chi_{[0,\infty)}\sigma_+K^{\mu,\nu}-cg\in L^1(\RR)\subseteq\calS'(\RR)$, where $c=-\frac{1}{\sqrt{\pi}}$ and
\begin{align*}
 g(y) &:= \chi_{[0,\infty)}(y)e^{\frac{1}{4}y}\cos\left(2e^{\frac{1}{2}y}-\frac{2\mu-3}{4}\pi\right).
\end{align*}
Hence, it remains to show that $g\in\calS'(\RR)$. (Since clearly $g\notin L^1(\RR)$, this is the crucial part.) For $\varphi\in C_c^\infty(\RR)$ we integrate by parts:
\begin{align*}
 & \int_\RR{g(y)\varphi(y)\td y} = \int_0^\infty{e^{-\frac{1}{4}y}\varphi(y)\cdot\frac{\td}{\td y}\left[\cos\left(2e^{\frac{1}{2}y}-\frac{2\mu-3}{4}\pi\right)\right]\td y}\\
 ={}& -\cos\left(2-\frac{2\mu-3}{4}\pi\right)\varphi(0)-\int_0^\infty{\frac{\td}{\td y}\left[e^{-\frac{1}{4}y}\varphi(y)\right]\cdot\cos\left(2e^{\frac{1}{2}y}-\frac{2\mu-3}{4}\pi\right)\td y}.
\end{align*}
The right hand side clearly makes sense also for $\varphi\in\calS(\RR)$ and defines a tempered distribution. Since $C_c^\infty(\RR)$ is dense in $\calS(\RR)$, $g$ extends to a continuous linear functional on $\calS(\RR)$ and is therefore a tempered distribution which finishes the proof.
\end{proof}

Now we can calculate the actions of $A_\kappa$ and $A_K$ on $f_0$.

\begin{proposition}\label{prop:AkkappaPsi0}
\begin{enumerate}
\item[\textup{(a)}] $A_\kappa f_0=f_0$.
\item[\textup{(b)}] $A_Kf_0=f_0$.
\end{enumerate}
\end{proposition}

\begin{proof}
\begin{enumerate}
\item[\textup{(a)}] By Theorem \ref{prop:FOProperties}~(7) we have $\calF_\calO=e^{-i\pi\frac{r_0}{2}(d_0-\frac{d}{2})_+}e^{i\frac{\pi}{2}(e|x-\calB)}\rho_{\lambda_1}(\alpha)$. Applying $\rho_{\lambda_1}(\alpha)$ to $\psi_0$ gives $\rho_{\lambda_1}(\alpha)\psi_0=\psi_0$ since $\alpha\in K_L$ and $\psi_0$ is $K_L$-invariant. Therefore, $\calF_\calO\psi_0=e^{-i\pi\frac{r_0}{2}(d_0-\frac{d}{2})_+}e^{i\frac{\pi}{2}(e|x-\calB)}\psi_0$.
\begin{enumerate}
\item[\textup{(1)}] If $V$ is euclidean then by \eqref{eq:KactionEucl} we have $(e|x-\calB)\psi_0=\frac{rd}{2}\psi_0$ and hence
\begin{align*}
 e^{i\frac{\pi}{2}(e|x-\calB)}\psi &= e^{i\pi\frac{rd}{4}}\psi_0.
\end{align*}
\item[\textup{(2)}] If $V$ is non-euclidean of rank $r\geq3$, then by \eqref{eq:KactionNonEuclSph} the function $\psi_0$ is annihilated by $(e|x-\calB)$ and therefore
\begin{align*}
 e^{i\frac{\pi}{2}(e|x-\calB)}\psi_0 &= \psi_0.
\end{align*}
\item[\textup{(3)}] The remaining case $V=\RR^{p,q}$ is treated in Appendix \ref{app:Rank2MinKtype}. By Lemma \ref{lem:Rk2ActionFpsi0} we have
\begin{align*}
 e^{i\frac{\pi}{2}(e|x-\calB)}\psi_0 &= e^{i\pi(\frac{q-p}{2})_+}\psi_0.
\end{align*}
\end{enumerate}
Together we obtain
\begin{align*}
 e^{i\frac{\pi}{2}(e|x-\calB)}\psi_0 &= e^{i\pi\frac{r_0}{2}(d_0-\frac{d}{2})_+}\psi_0
\end{align*}
and hence
\begin{align*}
 \calF_\calO\psi_0 &= \psi_0.
\end{align*}
Since $\psi_0(x)=f_0(|x|)$ this gives the result for $\calT f_0=A_\kappa f_0$.
\item[\textup{(b)}] The assumptions for \eqref{eq:GFctIntKBessel} are satisfied if we put $a=-\frac{2\mu+\nu}{4}$, $\alpha=\frac{\nu}{2}$, $\omega=1$, $\eta=(\frac{s}{2})^2$ and $(b_1,b_2,b_3,b_4)=(0,-\frac{\nu}{2},-\frac{\mu}{2},-\frac{\mu+\nu}{2})$ (use Lemma \ref{lem:MuNuProperties}). Then we obtain
\begin{align*}
 & A_Kf_0(t)\\
 ={}& \frac{1}{2^{\mu+\nu+1}}\int_0^\infty{G^{20}_{04}\left(\left(\frac{tt'}{4}\right)^2\left|0,-\frac{\nu}{2},-\frac{\mu}{2},-\frac{\mu+\nu}{2}\right.\right)\widetilde{K}_{\frac{\nu}{2}}(t')t'^{\mu+\nu+1}\td t'}\\
 ={}& \int_0^\infty{G^{20}_{04}\left(\left(\frac{t}{2}\right)^2x\left|0,-\frac{\nu}{2},-\frac{\mu}{2},-\frac{\mu+\nu}{2}\right.\right)K_{\frac{\nu}{2}}(2x^{\frac{1}{2}})x^{\frac{2\mu+\nu}{4}}\td x}\\
 ={}& \frac{1}{2}G^{22}_{24}\left(\left(\frac{t}{2}\right)^2\left|\begin{array}{c}-\frac{\mu}{2},-\frac{\mu+\nu}{2}\\0,-\frac{\nu}{2},-\frac{\mu}{2},-\frac{\mu+\nu}{2}\end{array}\right.\right)\\
 ={}& \frac{1}{2}G^{20}_{02}\left(\left(\frac{t}{2}\right)^2\left|0,-\frac{\nu}{2}\right.\right).
\end{align*}
where we have used the reduction formula \eqref{eq:GFctRed} for the last step. Eventually, the claim follows from the simplification formula \eqref{eq:GFctKBessel}.\qedhere
\end{enumerate}
\end{proof}

\subsection{A uniqueness property}

Now we finally prove that $\kappa=K$. The main point is the following lemma:

\begin{lemma}
Let $\kappa_1,\kappa_2\in\calS'$ If there exists a function $f\in\calS$ such that $\widehat{\sigma_+f}$ vanishes nowhere on $\RR$ and $A_{\kappa_1}f=A_{\kappa_2}f$, then $\kappa_1=\kappa_2$.
\end{lemma}

\begin{proof}
With \eqref{eq:Akappasigma} we obtain
\begin{align*}
 \sigma_-\kappa_1*\sigma_+f &= \sigma_-\kappa_2*\sigma_+f.
\end{align*}
Taking the Fourier transform on both sides yields
\begin{align*}
 \widehat{\sigma_-\kappa_1}\cdot\widehat{\sigma_+f} &= \widehat{\sigma_-\kappa_2}\cdot\widehat{\sigma_+f}.
\end{align*}
Since $\widehat{\sigma_+f}$ vanishes nowhere, this implies $\widehat{\sigma_-\kappa_1}=\widehat{\sigma_-\kappa_2}$ and hence $\kappa_1=\kappa_2$.
\end{proof}

Now, we already know that $A_\kappa$ and $A_K$ agree on $f_0$. To apply the previous lemma and finish the proof of Theorem \ref{thm:RadialUnitaryInversion}, it remains to show that $\widehat{\sigma_+f_0}$ vanishes nowhere on $\RR$. This follows from the next lemma.

\begin{lemma}
For the Fourier transform of $\sigma_+f_0$ we have the following formula:
\begin{align*}
 \widehat{\sigma_+f_0}(\xi) &= 2^{\frac{\mu+\nu-2}{2}-i\xi}\Gamma\left(\frac{\mu+\nu+2}{4}-\frac{1}{2}i\xi\right)\Gamma\left(\frac{\mu-\nu+2}{4}-\frac{1}{2}i\xi\right).
\end{align*}
In particular, $\widehat{\sigma_+f_0}$ vanishes nowhere on $\RR$.
\end{lemma}

\begin{proof}
With the formula \eqref{eq:KBesselMellinTransform} for the Mellin transform of the $K$-Bessel function we calculate
\begin{align*}
 \widehat{\sigma_+f_0}(\xi) &= \int_{-\infty}^\infty{e^{-ix\xi}(\sigma_+f_0)(x)\td x}\\
 &= \int_{-\infty}^\infty{e^{(\frac{\mu+\nu+2}{2}-i\xi)x}\widetilde{K}_{\frac{\nu}{2}}(e^x)\td x}\\
 &= \int_0^\infty{s^{\frac{\mu+\nu}{2}-i\xi}\widetilde{K}_{\frac{\nu}{2}}(s)\td s}\\
 &= 2^{\frac{\mu+\nu-2}{2}-i\xi}\Gamma\left(\frac{\mu+\nu+2}{4}-\frac{1}{2}i\xi\right)\Gamma\left(\frac{\mu-\nu+2}{4}-\frac{1}{2}i\xi\right).
\end{align*}
Since $\mu+\nu\geq-1$ and $\mu-\nu\geq0$ by Lemma \ref{lem:MuNuProperties} this defines a function on the whole real axis $\RR$ which vanishes nowhere.
\end{proof}
\chapter{Generalized Laguerre functions}\label{ch:GenLagFct}

We consider the ordinary fourth order differential operator
\begin{align*}
 \calD_{\mu,\nu} &= \frac{1}{x^2}\left((\theta+\mu+\nu)(\theta+\mu)-x^2\right)\left(\theta(\theta+\nu)-x^2\right),\index{notation}{Dmunu@$\calD_{\mu,\nu}$}
\end{align*}
depending on two complex parameters $\mu,\nu\in\CC$. Here $\theta=x\frac{\td}{\td x}$. In Theorem \ref{thm:CasimirAction} it was proved that for $(\mu,\nu)\in\Xi$, the operator $\calD_{\mu,\nu}$ is the radial part of the action of the $\frakk$-Casimir on the minimal representation of a simple Lie group $\check{G}$ (see Section \ref{sec:MuNu} for the definition of $\Xi$). In this case, the operator extends to a self-adjoint operator on $L^2(\RR_+,x^{\mu+\nu+1}\td x)$ with discrete spectrum given by $\{4j(j+\mu+1):j\in\NN_0\}$ (see Corollary \ref{cor:DmunuSASpectrum}).

In this chapter we explicitly construct the $L^2$-eigenfunctions of $\calD_{\mu,\nu}$. Furthermore, for an odd integer $\mu>0$ and generic $\nu\in\CC$, we find a fundamental system $\Lambda_{i,j}^{\mu,\nu}(x)$, $i=1,2,3,4$, of solutions to the fourth order equation
\begin{align*}
 \calD_{\mu,\nu}u &= 4j(j+\mu+1)u
\end{align*}
for every $j\in\NN_0$. We prove various properties for the functions $\Lambda_{i,j}^{\mu,\nu}(x)$ such as asymptotics, integral formulas and recurrence relations. Finally, we relate the $L^2$-eigenfunctions $\Lambda_{2,j}^{\mu,\nu}(x)$ to the minimal representation of $\check{G}$ as constructed in Section \ref{sec:MinRepConstruction}. This gives explicit expressions of $\frakk$-finite vectors in the representation. On the other hand, results from representation theory also provide simple proofs for statements on the $L^2$-eigenfunctions $\Lambda_{2,j}^{\mu,\nu}(x)$, such as orthogonality relations, completeness in $L^2(\RR_+,x^{\mu+\nu+1}\td x)$ or integral formulas.

Most results of this chapter are published in \cite{HKMM09a} and \cite{HKMM09b}. There only the minimal representation of $O(p+1,q+1)$ is used and hence the class of parameters is more restrictive than in this chapter.

\section{The fourth order differential operator $\calD_{\mu,\nu}$}\label{sec:DiffOp}

In this section we collect basic properties of the fourth order differential operator $\calD_{\mu,\nu}$.

\begin{proposition}\label{prop:DiffOpProperties}
\begin{enumerate}
 \item[\textup{(1)}] $\calD_{\nu,\mu}=\calD_{\mu,\nu}+(\mu-\nu)(\mu+\nu+2)$.
 \item[\textup{(2)}] $\calD_{\mu,\nu}u=\lambda u$ is a differential equation with regular singularity at $x=0$. The characteristic exponents are $0,-\mu,-\nu,-\mu-\nu$.
 \item[\textup{(3)}] If $\mu,\nu\in\RR$, then $\calD_{\mu,\nu}$ is a symmetric unbounded operator on the Hilbert space $L^2(\mathbb{R}_+,x^{\mu+\nu+1}\td x)$.
 \item[\textup{(4)}] If $(\mu,\nu)\in\Xi$, then $\calD_{\mu,\nu}$ extends to a self-adjoint operator on $L^2(\mathbb{R}_+,x^{\mu+\nu+1}\td x)$ with discrete spectrum given by $\{4j(j+\mu+1):j\in\NN_0\}$. Furthermore, every $L^2$-eigenspace is one-dimensional.
 \item[\textup{(5)}] In the special cases where $\nu=\pm1$ the differential operator $\calD_{\mu,\nu}$ collapses to
  \begin{equation*}
   \calD_{\mu,\pm1} = \mathcal{S}_{\mu,\pm1}^2 - (\mu+1)^2,
  \end{equation*}
  where\index{notation}{Smupm1@$\calS_{\mu,\pm1}$}
  \begin{align*}
   \mathcal{S}_{\mu,-1} &:= \frac{1}{x}\left(\theta(\theta+\mu)-x^2\right),\\
   \mathcal{S}_{\mu,+1} &:= \frac{1}{x}\left(\theta(\theta+\mu+2)+\mu+1-x^2\right).
  \end{align*}
\end{enumerate}
\end{proposition}

\begin{proof}
\begin{enumerate}
\item[\textup{(1)}] A simple computation shows that
 \begin{multline}
  \calD_{\mu,\nu} = \frac{1}{x^2}\theta(\theta+\mu)(\theta+\nu)(\theta+\mu+\nu)+x^2\\
  -2\left(\theta^2+(\mu+\nu+2)\theta+\frac{(\mu+2)(\mu+\nu+2)}{2}\right),\label{eq:DiffOp2}
 \end{multline}
 whence $\calD_{\nu,\mu}=\calD_{\mu,\nu}+(\mu-\nu)(\mu+\nu+2)$.
\item[\textup{(2)}] It follows from \eqref{eq:DiffOp2} that
 \begin{equation*}
  x^2(\calD_{\mu,\nu}-\lambda)\equiv\theta(\theta+\mu)(\theta+\nu)(\theta+\mu+\nu)\ \ \ \ \ (\textup{mod }x\cdot\mathbb{C}[x,\theta]),
 \end{equation*}
 where $\mathbb{C}[x,\theta]$ denotes the left $\mathbb{C}[x]$-module generated by $1,\theta,\theta^2,\ldots$ in the Weyl algebra $\mathbb{C}[x,\frac{\td}{\td x}]$. Therefore, the differential equation $\calD_{\mu,\nu}u=\lambda u$ has a regular singularity at $x=0$, and its characteristic equation is given by
 \begin{equation*}
  s(s+\mu)(s+\nu)(s+\mu+\nu)=0.
 \end{equation*}
 Hence the second statement is proved.
\item[\textup{(3)}] The formal adjoint of $\theta$ on $L^2(\mathbb{R}_+,x^{\mu+\nu+1}\td x)$ is given by
 \begin{equation*}
  \theta^*=-\theta-(\mu+\nu+2).
 \end{equation*}
With this it is easily seen from the expression \eqref{eq:DiffOp2} that $\calD_{\mu,\nu}$ is a symmetric operator on the same Hilbert space.
\item[\textup{(4)}] This  statement is simply Corollary \ref{cor:DmunuSASpectrum}.
\item[\textup{(5)}] A simple computation.\qedhere
\end{enumerate}
\end{proof}

\begin{remark}
It is likely that $\calD_{\mu,\nu}$ is still self-adjoint on $L^2(\RR_+,x^{\mu+\nu+1}\td x)$ without assuming that $(\mu,\nu)\in\Xi$. For example, for $\nu=\pm1$ and arbitrary $\mu>-1$ we construct $L^2$-eigenfunctions $\Lambda_{2,j}^{\mu,\pm1}$ of $\calD_{\mu,\pm1}$ which are essentially Laguerre polynomials (see Corollary \ref{cor:SpecialValue} and Remark \ref{rem:SpecialValue2}). Hence, they form a basis of the corresponding $L^2$-space and it follows that $\calD_{\mu,\pm1}$ is self-adjoint with discrete spectrum. However, our proof of self-adjointness uses unitary representation theory and involves the condition $(\mu,\nu)\in\Xi$ in a crucial way.
\end{remark}

\section{The generating functions $G_i^{\mu,\nu}(t,x)$}\label{sec:GenFct}

To determine eigenfunctions of the operator $\calD_{\mu,\nu}$ we define the following generating functions $G_i^{\mu,\nu}(t,x)$, $i=1,2,3,4$:\index{notation}{Gimunutx@$G_i^{\mu,\nu}(t,x)$}
\begin{align}
 G_1^{\mu,\nu}(t,x) &:= \frac{1}{(1-t)^{\frac{\mu+\nu+2}2}}\widetilde{I}_{\frac{\mu}{2}}\left(\frac{tx}{1-t}\right)\widetilde{I}_{\frac{\nu}{2}}\left(\frac{x}{1-t}\right),\label{eq:G1}\\
 G_2^{\mu,\nu}(t,x) &:= \frac{1}{(1-t)^{\frac{\mu+\nu+2}2}}\widetilde{I}_{\frac{\mu}{2}}\left(\frac{tx}{1-t}\right)\widetilde{K}_{\frac{\nu}{2}}\left(\frac{x}{1-t}\right),\label{eq:G2}\\
 G_3^{\mu,\nu}(t,x) &:= \frac{1}{(1-t)^{\frac{\mu+\nu+2}2}}\widetilde{K}_{\frac{\mu}{2}}\left(\frac{tx}{1-t}\right)\widetilde{I}_{\frac{\nu}{2}}\left(\frac{x}{1-t}\right),\label{eq:G3}\\
 G_4^{\mu,\nu}(t,x) &:= \frac{1}{(1-t)^{\frac{\mu+\nu+2}2}}\widetilde{K}_{\frac{\mu}{2}}\left(\frac{tx}{1-t}\right)\widetilde{K}_{\frac{\nu}{2}}\left(\frac{x}{1-t}\right).\label{eq:G4}
\end{align}
Here $\widetilde{I}_\alpha(z)$ and $\widetilde{K}_\alpha(z)$ denote the normalized $I$- and $K$-Bessel functions (see Appendix \ref{app:BesselFunctions} for the definition).

Let us state the differential equations for the generating functions which we will make use of later.

\begin{lemma}[Differential equations for the generating functions]\label{lem:GenFctPDEs}
The functions $G_i^{\mu,\nu}(t,x)$, $i=1,2,3,4$, satisfy the following three differential equations:
\begin{enumerate}
 \item[\textup{(1)}] The fourth order partial differential equation
  \begin{equation*}
   \left(\calD_{\mu,\nu}\right)_xu(t,x)=4\theta_t(\theta_t+\mu+1)u(t,x).
  \end{equation*}
 \item[\textup{(2)}] The second order partial differential equation
  \begin{multline*}
   (2\theta_t+\mu+1)\left(\theta_x+\frac{\mu+\nu+2}{2}\right)u(t,x)\\
   =\left(\frac{1}{t}\theta_t(\theta_t+\mu) - t\left(\theta_t+\frac{\mu+\nu+2}{2}\right)\left(\theta_t+\frac{\mu-\nu+2}{2}\right)\right)u(t,x).
  \end{multline*}
 \item[\textup{(3)}] The fifth order ordinary differential equation in $t$
  \begin{align*}
   & 8x^2\left(\theta_t+\frac{\mu-1}{2}\right)\left(\theta_t+\frac{\mu+1}{2}\right)\left(\theta_t+\frac{\mu+3}{2}\right)u(t,x)\\
   &\ \ \ \ \ =\left[\frac{2}{t^2}\theta_t\left(\theta_t-1\right)\left(\theta_t+\mu-1\right)\left(\theta_t+\mu\right)\left(\theta_t+\frac{\mu-5}{2}\right)\right.\\
   &\ \ \ \ \ \ \ \ \ \ - \frac{8}{t}\theta_t\left(\theta_t+\mu\right)\left(\theta_t+\frac{\mu-3}{2}\right)\left(\theta_t+\frac{\mu}{2}\right)\left(\theta_t+\frac{\mu+1}{2}\right)\\
   &\ \ \ \ \ \ \ \ \ \ + 2\left(\theta_t+\frac{\mu+1}{2}\right)\left(a\theta_t^4+b\theta_t^3+c\theta_t^2+d\theta_t+e\right)
  \end{align*}
  \begin{align*}
   &\ \ \ \ \ \ \ \ \ \ - 8t\left(\theta_t+\frac{\mu+1}{2}\right)\left(\theta_t+\frac{\mu+2}{2}\right)\left(\theta_t+\frac{\mu+5}{2}\right)\\
   &\ \ \ \ \ \ \ \ \ \ \ \ \ \ \ \ \ \ \ \ \ \ \ \ \ \ \ \ \ \ \times\left(\theta_t+\frac{\mu+\nu+2}{2}\right)\left(\theta_t+\frac{\mu-\nu+2}{2}\right)\\
   &\ \ \ \ \ \ \ \ \ \ + 2t^2\left(\theta_t+\frac{\mu+7}{2}\right)\left(\theta_t+\frac{\mu+\nu+2}{2}\right)\left(\theta_t+\frac{\mu-\nu+2}{2}\right)\\
   &\ \ \ \ \ \ \ \ \ \ \ \ \ \ \ \ \ \ \ \ \ \ \ \ \ \ \ \ \ \times\left.\left(\theta_t+\frac{\mu+\nu+4}{2}\right)\left(\theta_t+\frac{\mu-\nu+4}{2}\right)\right]u(t,x),
  \end{align*}
  where we set
  \begin{align*}
   a &= 6,\\
   b &= 12(\mu+1),\\
   c &= \frac{1}{2}(17\mu^2-\nu^2+36\mu+8),\\
   d &= \frac{1}{2}(\mu+1)(5\mu^2-\nu^2+12\mu-4),\\
   e &= \frac{1}{4}(\mu-1)(\mu+2)(\mu+\nu+2)(\mu-\nu+2).
  \end{align*}
\end{enumerate}
\end{lemma}

\begin{proof}
The proof consists of straightforward verifications using the definition of $G_i^{\mu,\nu}(t,x)$ and the differential equation \eqref{eq:DiffEqModBesselTheta} for the $I$- and $K$-Bessel functions $\widetilde{I}_\alpha(z)$, $\widetilde{K}_\alpha(z)$.
\end{proof}

We also need three recurrence relations for the functions $G_i^{\mu,\nu}(t,x)$. To state the formulas in a uniform way we put\index{notation}{deltaofi@$\delta(i)$}\index{notation}{epsilonofi@$\varepsilon(i)$}
\begin{align}
 \delta(i) &= \left\{\begin{array}{ll}+1&\mbox{for $i=1,2$,}\\-1&\mbox{for $i=3,4$,}\end{array}\right. & \varepsilon(i) &= \left\{\begin{array}{ll}+1&\mbox{for $i=1,3$,}\\-1&\mbox{for $i=2,4$.}\end{array}\right.
 \label{eq:DelEps}
\end{align}

\begin{lemma}[Recurrence relations for the generating functions]\label{lem:GenFctRecRels}
The functions $G_i^{\mu,\nu}(t,x)$, $i=1,2,3,4$, satisfy the following three recurrence relations:
\begin{enumerate}
 \item[\textup{(1)}] The recurrence relation in $\mu$
  \begin{equation*}
   \mu(1-t)G_i^{\mu,\nu}(t,x) = 2\delta(i)\left(G_i^{\mu-2,\nu}(t,x)-\left(\frac{tx}{2}\right)^2 G_i^{\mu+2,\nu}(t,x)\right).
  \end{equation*}
 \item[\textup{(2)}] The recurrence relation in $\nu$
  \begin{equation*}
   \nu(1-t)G_i^{\mu,\nu}(t,x) = 2\varepsilon(i)\left(G_i^{\mu,\nu-2}(t,x)-\left(\frac{x}{2}\right)^2 G_i^{\mu,\nu+2}(t,x)\right).
  \end{equation*}
 \item[\textup{(3)}] The recurrence relation in $\mu$ and $\nu$
  \begin{equation*}
   (1-t)\frac{\td}{\td x}G_i^{\mu,\nu}(t,x) = \delta(i)\frac{t^2x}{2}G_i^{\mu+2,\nu} + \varepsilon(i)\frac{x}{2}G_i^{\mu,\nu+2}.
  \end{equation*}
\end{enumerate}
\end{lemma}

\begin{proof}
\begin{description}
 \item[\normalfont(1) and (2):] Use the recurrence relations \eqref{eq:IBesselRecRel} and \eqref{eq:KBesselRecRel} for the $I$- and $K$-Bessel functions .
 \item[\normalfont(3)] In view of \eqref{eq:BesselDiffFormulas} the equation is evident.\qedhere
 \end{description}
\end{proof}

\begin{lemma}[Local monodromy of the generating functions]\label{lem:GenFctParity}
We have the following formula for the functions $G_i^{\mu,\nu}(t,e^{-i\pi}x)$ ($i=1,2,3,4$):
\begin{equation*}
 \left(\begin{array}{c}G_1^{\mu,\nu}\\G_2^{\mu,\nu}\\G_3^{\mu,\nu}\\G_4^{\mu,\nu}\end{array}\right)(t,e^{i\pi}x) = \left(\begin{array}{cccc}1 & 0 & 0 & 0\\b_\nu & a_\nu & 0 & 0\\b_\mu & 0 & a_\mu & 0\\b_\mu b_\nu & a_\nu b_\mu & a_\mu b_\nu & a_\mu a_\nu\end{array}\right)\left(\begin{array}{c}G_1^{\mu,\nu}\\G_2^{\mu,\nu}\\G_3^{\mu,\nu}\\G_4^{\mu,\nu}\end{array}\right)(t,x)
\end{equation*}
where \index{notation}{aalpha@$a_\alpha$}\index{notation}{balpha@$b_\alpha$}
\begin{align*}
 a_\alpha &:= e^{-\alpha\pi i}, & b_\alpha &:= \frac{\Gamma(1-\frac{\alpha}{2})\Gamma(\frac{\alpha}{2})}{2}\left(e^{-\alpha\pi i}-1\right).
\end{align*}
\end{lemma}

\begin{proof}
This follows immediately from the parity formulas \eqref{eq:ParityIBessel} and \eqref{eq:ParityKBessel} for the Bessel functions.
\end{proof}

\begin{remark}[Algebraic symmetries for the generating functions]
It is also easy to see that the generating functions satisfy the following algebraic symmetries
\begin{align*}
 G_i^{\mu,\nu}(t,x) &= G_i^{\nu,\mu}\left(\frac{1}{t},-x\right) & \mbox{($i=1,4$),}\\
 G_2^{\mu,\nu}(t,x) &= G_3^{\nu,\mu}\left(\frac{1}{t},-x\right).
\end{align*}
\end{remark}

\section{The eigenfunctions $\Lambda_{i,j}^{\mu,\nu}(x)$}\label{sec:EigFct}

The function $\widetilde{K}_{\frac{\mu}{2}}(\frac{tx}{1-t})$ is meromorphic near $t=0$ for a fixed $x>0$ if and only if $\mu$ is an odd integer. Therefore, we will henceforth assume the following integrality condition:
\begin{equation}
 \mbox{$\mu$ is an odd integer $\geq1$ for $i=3,4$.}\tag{IC}\label{IntCond2}
\end{equation}
Then the generating functions $G_i^{\mu,\nu}$ are meromorphic near $t=0$ and give rise to sequences $(\Lambda_{i,j}^{\mu,\nu})_{j\in\mathbb{Z}}$\index{notation}{Lambdaijmunux@$\Lambda_{i,j}^{\mu,\nu}(x)$} of functions on $\mathbb{R}_+$ as coefficients of the Laurent expansions
\begin{align}
 G_i^{\mu,\nu}(t,x) &= \sum_{j=-\infty}^\infty{\Lambda_{i,j}^{\mu,\nu}(x)t^j}, & i=1,2,3,4.\label{eq:FctDef}
\end{align}
Since $\widetilde{I}_{\frac{\alpha}{2}}(z)$ is an entire function and $\widetilde{K}_{\frac{\alpha}{2}}(z)$ has a pole of order $\alpha$ at $z=0$ if $\alpha\geq1$ is an odd integer, we immediately obtain
\begin{align*}
 \Lambda_{1,j}^{\mu,\nu} = \Lambda_{2,j}^{\mu,\nu} &= 0 && \mbox{for $j<0$,}\\
 \Lambda_{3,j}^{\mu,\nu} = \Lambda_{4,j}^{\mu,\nu} &= 0 && \mbox{for $j<-\mu$.}
\end{align*}
This allows us to calculate the functions $\Lambda_{i,j}^{\mu,\nu}$ as follows:
\begin{equation}
 \Lambda_{i,j}^{\mu,\nu}(x) = \left\{\begin{array}{ll}\displaystyle\frac{1}{j!}\left.\frac{\partial^j}{\partial t^j}\right|_{t=0}G_i^{\mu,\nu}(t,x) &\mbox{if $i=1,2$, $j\geq0$,}\\\displaystyle\frac{1}{(j+\mu)!}\left.\frac{\partial^{j+\mu}}{\partial t^{j+\mu}}\right|_{t=0}t^{\mu}G_i^{\mu,\nu}(t,x) &\mbox{if $i=3,4$, $j\geq-\mu$.}\end{array}\right.\label{eq:DefAsDerivative}
\end{equation}

\begin{example}\label{ex:3Fcts}
The functions $\Lambda_{2,j}^{\mu,\nu}(x)$ (in the $i=2$ case) will turn out to be $L^2$-eigenfunctions of $\calD_{\mu,\nu}$ and are therefore of special interest. Here are the first three functions of this series:
\begin{align*}
 \Lambda_{2,0}^{\mu,\nu}(x) ={}& \frac{1}{\Gamma(\frac{\mu+2}{2})} \widetilde{K}_{\frac{\nu}{2}}(x),\\
 \Lambda_{2,1}^{\mu,\nu}(x) ={}& \frac{1}{\Gamma(\frac{\mu+2}{2})} \left(\frac{\mu+\nu+2}{2}\widetilde{K}_{\frac{\nu}{2}}(x)+\theta\widetilde{K}_{\frac{\nu}{2}}(x)\right),\\
 \Lambda_{2,2}^{\mu,\nu}(x) ={}& \frac{1}{2\Gamma(\frac{\mu+2}{2})} \left(\frac{(\mu+\nu+2)(\mu+\nu+4)}{4}\widetilde{K}_{\frac{\nu}{2}}(x)\right.\\
 &\ \ \ \ \ \ \ \ \ \ \ \ \ \ \ \ \ \ \ \ \left.+ \frac{(\mu+3)(\mu+\nu+2)}{\mu+2}\theta\widetilde{K}_{\frac{\nu}{2}}(x) + \frac{\mu+3}{\mu+2}\theta^2\widetilde{K}_{\frac{\nu}{2}}(x)\right).
\end{align*}
\end{example}

To formulate the asymptotic behavior of the functions $\Lambda_{i,j}^{\mu,\nu}(x)$ we use the Landau symbols $\mathcal{O}$ and $o$.

\begin{theorem}\label{lem:Asymptotics}
Let $\mu\in\mathbb{C}$, $\mu\neq-1,-2,-3,\ldots$ and $\nu\in\mathbb{R}$. Assume further that $j\geq0$ if $i=1,2$ and $j\geq-\mu$ if $i=3,4$.
\begin{enumerate}
\item[\textup{(1)}] The asymptotic behavior of the functions $\Lambda_{i,j}^{\mu,\nu}(x)$ as $x\rightarrow0$ is given by
\begin{align*}
 \Lambda_{1,j}^{\mu,\nu}(x) &= \frac{(\frac{\mu+\nu+2}{2})_j}{j!\Gamma(\frac{\mu+2}{2})\Gamma(\frac{\nu+2}{2})}+o(1)\\
 \Lambda_{2,j}^{\mu,\nu}(x) &= \frac{(\frac{\mu-|\nu|+2}{2})_j}{j!\Gamma(\frac{\mu+2}{2})}\times \left\{\begin{array}{ll}\displaystyle2^{\nu-1}\Gamma\left(\frac{\nu}{2}\right)x^{-\nu}+o(x^{-\nu})&\mbox{if $\nu>0$,}\\\displaystyle-\log\left(\frac{x}{2}\right)+o\left(\log\left(\frac{x}{2}\right)\right)&\mbox{if $\nu=0$,}\\\displaystyle\frac{1}{2}\Gamma\left(-\frac{\nu}{2}\right)+o(1)&\mbox{if $\nu<0$,}\end{array}\right.\\
 \Lambda_{3,j}^{\mu,\nu}(x) &= \frac{2^{\mu-1}\Gamma(\frac{\mu}{2})(\frac{-\mu+\nu+2}{2})_{j+\mu}}{(j+\mu)!\Gamma(\frac{\nu+2}{2})}x^{-\mu}+o(x^{-\mu}),\\
 \Lambda_{4,j}^{\mu,\nu}(x) &= \frac{\Gamma(\frac{\mu}{2})(\frac{-\mu-|\nu|+2}{2})_{j+\mu}}{(j+\mu)!}\times \left\{\begin{array}{l}\displaystyle2^{\mu+\nu-2}\Gamma\left(\frac{\nu}{2}\right)x^{-\mu-\nu}+o(x^{-\mu-\nu})\\\ \ \ \ \ \ \ \ \ \ \ \ \ \ \ \ \ \ \ \ \ \ \ \ \ \ \ \ \ \ \ \ \ \ \ \ \ \ \mbox{if $\nu>0$,}\\\displaystyle-2^{\mu-1}x^{-\mu}\log\left(\frac{x}{2}\right)+o\left(x^{-\mu}\log\left(\frac{x}{2}\right)\right)\\\ \ \ \ \ \ \ \ \ \ \ \ \ \ \ \ \ \ \ \ \ \ \ \ \ \ \ \ \ \ \ \ \ \ \ \ \ \ \mbox{if $\nu=0$,}\\\displaystyle2^{\mu-2}\Gamma\left(-\frac{\nu}{2}\right)x^{-\mu}+o(x^{-\mu})\\\ \ \ \ \ \ \ \ \ \ \ \ \ \ \ \ \ \ \ \ \ \ \ \ \ \ \ \ \ \ \ \ \ \ \ \ \ \ \mbox{if $\nu<0$,}\end{array}\right.
\end{align*}
where $(a)_n=a(a+1)\cdots(a+n-1)$\index{notation}{1an@$(a)_n$} is the Pochhammer symbol.
\item[\textup{(2)}] As $x\rightarrow\infty$ we have
\begin{align*}
 \Lambda_{1,j}^{\mu,\nu}(x) &= C_{1,j}^{\mu,\nu}x^{j-\frac{\nu+1}{2}}e^x\left(1+\mathcal{O}\left(\frac{1}{x}\right)\right), & \Lambda_{3,j}^{\mu,\nu}(x) &= \mathcal{O}\left(x^{j-\frac{\nu+1}{2}}e^x\right),\\
 \Lambda_{2,j}^{\mu,\nu}(x) &= C_{2,j}^{\mu,\nu}x^{j-\frac{\nu+1}{2}}e^{-x}\left(1+\mathcal{O}\left(\frac{1}{x}\right)\right), & \Lambda_{4,j}^{\mu,\nu}(x) &= \mathcal{O}\left(x^{j-\frac{\nu+1}{2}}e^{-x}\right)
\end{align*}
with constants $C_{1,j}^{\mu,\nu},C_{2,j}^{\mu,\nu}\neq0$.
\end{enumerate}
\end{theorem}

\begin{proof}
The basic ingredient for the proof is the asymptotic behavior of the Bessel functions which is for $x\rightarrow0$ given in \eqref{eq:BesselIAsymptAt0} and \eqref{eq:BesselKAsymptAt0} and for $x\rightarrow\infty$ in \eqref{eq:BesselAsymptAtInfty}. We also make use of the well-known expansion
\begin{equation}
 (1-t)^{-\alpha} = \sum_{j=0}^\infty{\frac{(\alpha)_j}{j!}t^j}.\label{eq:BinomExpansion}
\end{equation}
\begin{enumerate}
\item[\textup{(1)}] We show how to calculate the asymptotic behavior at $x=0$ for the functions $\Lambda_{2,j}^{\mu,\nu}(x)$ with $\nu>0$. The same method applies to the other cases.\\
Using the asymptotics \eqref{eq:BesselIAsymptAt0} and \eqref{eq:BesselKAsymptAt0} and the binomial expansion \eqref{eq:BinomExpansion} we find that
\begin{align*}
 \left.x^\nu G_2^{\mu,\nu}(t,x)\right|_{x=0} &= \frac{1}{(1-t)^{\frac{\mu+\nu+2}{2}}}\frac{1}{\Gamma(\frac{\mu+2}{2})}(2(1-t))^\nu\frac{\Gamma(\frac{\nu}{2})}{2}\\
 &= \sum_{j=0}^\infty{\frac{2^{\nu-1}\Gamma(\frac{\nu}{2})(\frac{\mu-\nu+2}{2})_j}{j!\Gamma(\frac{\mu+2}{2})}t^j}.
\end{align*}
In view of \eqref{eq:FctDef} this yields
\begin{equation*}
 \left.x^\nu\Lambda_{2,j}^{\mu,\nu}(x)\right|_{x=0} = \frac{2^{\nu-1}\Gamma(\frac{\nu}{2})(\frac{\mu-\nu+2}{2})_j}{j!\Gamma(\frac{\mu+2}{2})}.
\end{equation*}
\item[\textup{(2)}] Let us first treat the case $i=1,2$. With equation \eqref{eq:DefAsDerivative} it is easy to see that $\Lambda_{i,j}^{\mu,\nu}$ is a linear combination of terms of the form
\begin{equation*}
 \left\{\begin{array}{ll}\left(\theta^k\widetilde{I}_{\frac{\nu}{2}}\right)(x) & \mbox{for $i=1$,}\\\left(\theta^k\widetilde{K}_{\frac{\nu}{2}}\right)(x) & \mbox{for $i=2$}\end{array}\right.
\end{equation*}
with $0\leq k\leq j$ such that the coefficient for $k=j$ are non-zero. (In fact this can be seen in a more direct way from the recurrence relation in Proposition \ref{prop:RecRelH} and Example \ref{ex:3Fcts}.) Using \eqref{eq:BesselDiffFormulas} this simplifies to terms of the form
\begin{equation*}
 \left\{\begin{array}{ll}x^{2k}\widetilde{I}_{\frac{\nu}{2}+k}(x) & \mbox{for $i=1$,}\\x^{2k}\widetilde{K}_{\frac{\nu}{2}+k}(x) & \mbox{for $i=2$}\end{array}\right.
\end{equation*}
with $0\leq k\leq j$ and non-zero coefficient for $k=j$. Using \eqref{eq:BesselAsymptAtInfty} the leading term appears for $k=j$ and the asymptotics follow.\\
For $i=3,4$ equation \eqref{eq:DefAsDerivative} implies that $\Lambda_{i,j}^{\mu,\nu}$ is a linear combination of terms of the form
\begin{equation*}
 \left\{\begin{array}{ll}\displaystyle x^{k-\mu}\left(\theta^\ell\widetilde{I}_{\frac{\nu}{2}}\right)(x) & \mbox{for $i=3$,}\\\displaystyle x^{k-\mu}\left(\theta^\ell\widetilde{K}_{\frac{\nu}{2}}\right)(x) & \mbox{for $i=4$}\end{array}\right.
\end{equation*}
with $0\leq k+\ell\leq j+\mu$. Using \eqref{eq:BesselDiffFormulas} this simplifies to terms of the form
\begin{equation*}
 \left\{\begin{array}{ll}\displaystyle x^{k+2\ell-\mu}\widetilde{I}_{\frac{\nu}{2}+\ell}(x) & \mbox{for $i=3$,}\\\displaystyle x^{k+2\ell-\mu}\widetilde{K}_{\frac{\nu}{2}+\ell}(x) & \mbox{for $i=4$}\end{array}\right.
\end{equation*}
with $0\leq k+\ell\leq j+\mu$. Then again the claim follows from \eqref{eq:BesselAsymptAtInfty}.\qedhere
\end{enumerate}
\end{proof}

As an immediate consequence of Theorem \ref{lem:Asymptotics} we obtain:

\begin{corollary}\label{prop:L2functions}
If $\mu+\nu,\mu-\nu>-2$, then $\Lambda_{2,j}^{\mu,\nu}\in L^2(\mathbb{R}_+,x^{\mu+\nu+1}\td x)$.
\end{corollary}

From the explicit formulas for the leading terms of the functions $\Lambda_{i,j}^{\mu,\nu}(x)$ at $x=0$ we can draw two more important corollaries.

\begin{corollary}\label{cor:Nonzeroness}
The function $\Lambda_{i,j}^{\mu,\nu}$ is non-zero if one of the following conditions is satisfied:
\begin{itemize}
\item $i=1$ and $\mu,\nu,\mu+\nu>-2$.
\item $i=2$ and $\mu+\nu,\mu-\nu>-2$.
\item $i=3,4$, $\mu$ is a positive odd integer and $\nu>-1$ such that $\mu-\nu\notin2\mathbb{Z}$.
\end{itemize}
\end{corollary}

\begin{proof}
In each case the assumption implies that the leading coefficient at $x=0$ in Theorem \ref{lem:Asymptotics} is non-zero, so that the function itself is non-zero as well.
\end{proof}

\begin{corollary}\label{cor:LinIndependence}
Suppose $\mu$ is a positive odd integer and $\nu>0$ such that $\mu-\nu\notin2\mathbb{Z}$, then for fixed $j\in\mathbb{N}_0$ the four functions $\Lambda_{i,j}^{\mu,\nu}$, $i=1,2,3,4$, are linearly independent.
\end{corollary}

\begin{proof}
The assumptions imply that the leading coefficients at $x=0$ of the functions $\Lambda_{i,j}^{\mu,\nu}(x)$ in Theorem \ref{lem:Asymptotics} never vanish and that the leading terms are distinct. Hence, the asymptotic behavior near $x=0$ is different and the functions have to be linear independent.
\end{proof}

Now we can prove the main theorem of this section.

\begin{theorem}[Differential equation]\label{thm:EigFct}
For $i=1,2,3,4$, $j\in\mathbb{Z}$, the function $\Lambda_{i,j}^{\mu,\nu}$ is an eigenfunction of the fourth order differential operator $\calD_{\mu,\nu}$ for the eigenvalue $4j(j+\mu+1)$. If, in addition, $\mu$ is a positive odd integer and $\nu>0$ such that $\mu-\nu\notin2\mathbb{Z}$, then for fixed $j\in\mathbb{N}_0$ the four functions $\Lambda_{i,j}^{\mu,\nu}$, $i=1,2,3,4$, form a fundamental system of the fourth order differential equation
\begin{equation}
 \calD_{\mu,\nu}u = 4j(j+\mu+1)u.\label{eq:DiffEq}
\end{equation}
\end{theorem}

\begin{proof}
In view of Corollary \ref{cor:LinIndependence} it only remains to show the first statement. We deduce
\begin{align}
 \calD_{\mu,\nu}\Lambda_{i,j}^{\mu,\nu}(x) &= 4j(j+\mu+1)\Lambda_{i,j}^{\mu,\nu}(x) \qquad \forall\, j\in\mathbb{Z}\label{eq:EigFctEq}
\end{align}
from the corresponding partial differential equation for the generating function $G_i^{\mu,\nu}(t,x)$. For this we take generating functions of both sides of \eqref{eq:EigFctEq}. Clearly, $(\calD_{\mu,\nu})_xG_i^{\mu,\nu}(t,x)$ is the generating function for the left hand side of \eqref{eq:EigFctEq}. The generating function for the right hand side is calculated as follows
\begin{align*}
 & \sum_{j=-\infty}^\infty{t^j\cdot4j(j+\mu+1)\Lambda_{i,j}^{\mu,\nu}(x)}\\
 ={}& 4\sum_{j=-\infty}^\infty{\left(\theta_t^2+(\mu+1)\theta_t\right)t^j\Lambda_{i,j}^{\mu,\nu}(x)}\\
 ={}& 4\theta_t\left(\theta_t+\mu+1\right)G_i^{\mu,\nu}(t,x),
\end{align*}
where $\theta_t:=t\frac{\partial}{\partial t}$. The resulting partial differential equation is
\begin{equation*}
 \left(\calD_{\mu,\nu}\right)_xG_i^{\mu,\nu}(t,x)=4\theta_t\left(\theta_t+\mu+1\right)G_i^{\mu,\nu}(t,x)
\end{equation*}
which was verified in Lemma \ref{lem:GenFctPDEs}~(1).
\end{proof}

\begin{remark}
Since $\calD_{\mu,\nu}=\calD_{\nu,\mu}-(\mu-\nu)(\mu+\nu+2)$ by Proposition \ref{prop:DiffOpProperties}~(1) and
\begin{equation*}
 4\left(j+\frac{\mu-\nu}{2}\right)\left(\left(j+\frac{\mu-\nu}{2}\right)+\nu+1\right) = 4j(j+\mu+1)+(\mu-\nu)(\mu+\nu+2),
\end{equation*}
Theorem \ref{thm:EigFct} implies that for $\mu-\nu\in2\mathbb{Z}$ also $\Lambda_{i,j+\frac{\mu-\nu}{2}}^{\nu,\mu}(x)$ is an eigenfunction of $\calD_{\mu,\nu}$ for the eigenvalue $4j(j+\mu+1)$.
\end{remark}

\begin{corollary}\label{cor:completeness}
If $(\mu,\nu)\in\Xi$, then the system $(\Lambda_{2,j}^{\mu,\nu})_{j\in\mathbb{N}_0}$ forms an orthogonal basis of $L^2(\mathbb{R}_+,x^{\mu+\nu+1}\td x)$.
\end{corollary}

\begin{proof}
Lemma \ref{lem:MuNuProperties} implies that $\mu+\nu,\mu-\nu>-2$. Hence, by Corollary \ref{prop:L2functions}, the functions $\Lambda_{2,j}^{\mu,\nu}$ are contained in $L^2(\mathbb{R}_+,x^{\mu+\nu+1}\td x)$ and by Theorem \ref{thm:EigFct} each function $\Lambda_{2,j}^{\mu,\nu}$ is an eigenfunction of $\calD_{\mu,\nu}$ for the eigenvalue $4j(j+\mu+1)$. Therefore the claim follows from Proposition \ref{prop:DiffOpProperties}~(4).
\end{proof}

Corollary \ref{cor:completeness} provides a completeness statement for Bessel functions we could not trace in the literature:

\begin{corollary}\label{cor:BesselCompleteness}
For $(\mu,\nu)\in\Xi$ the sequence $(\theta^j\widetilde{K}_{\frac{\nu}{2}})_{j\in\mathbb{N}_0}$ (resp. $(x^{2j}\widetilde{K}_{\frac{\nu}{2}+j})_{j\in\mathbb{N}_0}$) is a basis for $L^2(\mathbb{R}_+,x^{\mu+\nu+1}\td x)$. The Gram--Schmidt process applied to this sequence yields the orthogonal basis $(\Lambda_{2,j}^{\mu,\nu})_{j\in\mathbb{N}_0}$ (up to scalar factors).
\end{corollary}

\begin{proof}
It is an easy consequence of the definitions that $\Lambda_{2,j}^{\mu,\nu}$ can be written as a linear combination of the functions $\theta^k\widetilde{K}_{\frac{\nu}{2}}$ for $0\leq k\leq j$. (In fact this follows more directly from the recurrence relation in Proposition \ref{prop:RecRelH}.) Then the sequence $(\theta^j\widetilde{K}_{\frac{\nu}{2}})_j$ clearly arises from the complete sequence $(\Lambda_{2,j}^{\mu,\nu})_j$ by a base change and hence is complete. Using \eqref{eq:BesselDiffFormulas} it is also easy to see that the second series $(x^{2j}\widetilde{K}_{\frac{\nu}{2}+j})_j$ arises by a base change from the sequence $(\theta^j\widetilde{K}_{\frac{\nu}{2}})_j$. Finally, we note that both base change matrices considered are upper triangular. Thus the Gram--Schmidt process in both cases yields the orthogonal basis $\Lambda_{2,j}^{\mu,\nu}$.
\end{proof}

We end this section with a formula for the local monodromy of the functions $\Lambda_{i,j}^{\mu,\nu}(x)$ at $x=0$. This implies a parity formula with respect to $x\mapsto-x$ which can be used to determine also the asymptotic behavior as $x\rightarrow-\infty$. The monodromy formula itself is an immediate consequence of Lemma \ref{lem:GenFctParity}:

\begin{proposition}[Local monodromy at $x=0$]\label{prop:ParityLambda}
We have the following local monodromy to the differential equation \eqref{eq:DiffEq}:
\begin{equation*}
 \left(\begin{array}{c}\Lambda_{1,j}^{\mu,\nu}\\\Lambda_{2,j}^{\mu,\nu}\\\Lambda_{3,j}^{\mu,\nu}\\\Lambda_{4,j}^{\mu,\nu}\end{array}\right)(e^{i\pi}x) = \left(\begin{array}{cccc}1 & 0 & 0 & 0\\b_\nu & a_\nu & 0 & 0\\b_\mu & 0 & a_\mu & 0\\b_\mu b_\nu & a_\nu b_\mu & a_\mu b_\nu & a_\mu a_\nu\end{array}\right)\left(\begin{array}{c}\Lambda_{1,j}^{\mu,\nu}\\\Lambda_{2,j}^{\mu,\nu}\\\Lambda_{3,j}^{\mu,\nu}\\\Lambda_{4,j}^{\mu,\nu}\end{array}\right)(x)
\end{equation*}
with coefficients $a_\alpha,b_\alpha$ as in Lemma \ref{lem:GenFctParity}.
\end{proposition}

\begin{remark}\label{rem:asym-infty}
If $\nu$ is an odd integer, the functions $\Lambda_{i,j}^{\mu,\nu}(x)$ extend holomorphically to $\mathbb{C}\backslash\{0\}$, not only to its universal covering. In this case, Proposition \ref{prop:ParityLambda} expresses $\Lambda_{i,j}^{\mu,\nu}(-x)$ as linear combination of the functions $\Lambda_{k,j}^{\mu,\nu}(x)$ ($k=1,2,3,4$). The coefficients contain $a_\alpha$ and $b_\alpha$ with $\alpha=2n+1$ an odd integer. In this case they simplify significantly:
\begin{align*}
 a_{2n+1} &= -1, & b_{2n+1} &= (-1)^{n+1}\pi.
\end{align*}
\end{remark}

\section{Integral representations}

In this section we show that for $i=1,2$ the functions $\Lambda_{i,j}^{\mu,\nu}(x)$ have integral representations in terms of Laguerre polynomials. For the definition of the Laguerre polynomials $L_n^\alpha(z)$ see Appendix \ref{app:Laguerre}.

\begin{theorem}[Integral representations]\label{prop:IntFormulae}
\begin{enumerate}
\item[\textup{(1)}] For $j\in\mathbb{N}_0$, $\Re(\mu),\Re(\nu)>-1$ we have the following double integral
representations
\begin{align}
 \Lambda_{1,j}^{\mu,\nu}(x) ={}& c_{1, j}^{\mu, \nu} \int_0^\pi{\int_0^\pi{e^{-x\cos \phi} L_j^{\frac{\mu+\nu}{2}}(x(\cos \theta+\cos \phi))\sin^\mu\theta\sin^\nu\phi\td\phi}\td\theta},\label{IntFormulaGen1}\\
 \Lambda_{2,j}^{\mu,\nu}(x) ={}& c_{2, j}^{\mu, \nu} \int_0^\pi{\int_0^\infty{e^{-x\cosh \phi} L_j^{\frac{\mu+\nu}{2}}(x(\cos \theta+\cosh \phi))\sin^\mu\theta\sinh^\nu\phi\td\phi}\td\theta},\label{IntFormulaGen2}
\end{align}
with constants $c_{1,j}^{\mu,\nu}$ and $c_{2,j}^{\mu,\nu}$ given by
\begin{align*}
 c_{1,j}^{\mu,\nu} &:= \frac{1}{\pi\Gamma(\frac{\mu+1}{2})\Gamma(\frac{\nu+1}{2})}\quad\text{and}\quad
 c_{2,j}^{\mu,\nu} := \frac{1}{\Gamma(\frac{\mu+1}{2})\Gamma(\frac{\nu+1}{2})}.
\end{align*}
\item[\textup{(2)}] For $\nu=-1$ and $\Re(\mu)>-1$ we have
\begin{align}
 \Lambda_{1,j}^{\mu,-1}(x) &= c_{1, j}^{\mu, -1} \sum_{i=0}^1{\int_0^\pi{e^{-(-1)^ix} L_j^{\frac{\mu-1}{2}}(x(\cos\theta+(-1)^i))\sin^\mu\theta\td\theta}},\label{IntFormulaSpecialCase1}\\
 \Lambda_{2,j}^{\mu,-1}(x) &= c_{2, j}^{\mu, -1} \int_0^\pi{e^{-x} L_j^{\frac{\mu-1}{2}}(x(\cos\theta+1))\sin^\mu\theta\td\theta}\label{IntFormulaSpecialCase2}
\end{align}
with constants $c_{1, j}^{\mu, -1}$ and $c_{2, j}^{\mu, -1}$ given by
\begin{align*}
 c_{1,j}^{\mu,-1} &= \frac{1}{2\pi\Gamma(\frac{\mu+1}2)}\quad\text{and}\quad
 c_{2,j}^{\mu,-1} = \frac{1}{2\Gamma(\frac{\mu+1}2)}.
\end{align*}
\end{enumerate}
\end{theorem}

\begin{proof}
We make use of the formula \eqref{eq:DefAsDerivative} for $\Lambda_{i,j}^{\mu,\nu}$ and the generating function \eqref{eq:LaguerreGenFct} of the Laguerre polynomials. Further, we need the integral representations \eqref{eq:IntFormulaIBessel} and \eqref{eq:IntFormulaKBessel} for the $I$- and $K$-Bessel functions.
\begin{enumerate}
\item[\textup{(1)}] Interchanging differentiation and integration we obtain the desired integral representations for $\Lambda_{1,j}^{\mu,-1}$:
\begin{align*}
 & \pi\Gamma\left(\frac{\mu+1}{2}\right)\Gamma\left(\frac{\nu+1}{2}\right)\Lambda_{1,j}^{\mu,\nu}(x)\\
 ={}& \frac{\pi\Gamma(\frac{\mu+1}{2})\Gamma(\frac{\nu+1}{2})}{j!}\left.\frac{\partial^j}{\partial t^j}\right|_{t=0} G_1^{\mu,\nu}(t,x)\\
 ={}& \frac{1}{j!}\left.\frac{\partial^j}{\partial t^j}\right|_{t=0} \frac{1}{(1-t)^{\frac{\mu+\nu+2}{2}}} \int_0^\pi{\int_0^\pi{e^{-\frac{tx}{1-t}\cos\theta}e^{-\frac{x}{1-t}\cos\phi}\sin^\mu\theta\sin^\nu\phi\td\phi}\td\theta}\\
 ={}& \int_0^\pi{\int_0^\pi{e^{-x\cos\phi}\frac{1}{j!}\left.\frac{\partial^j}{\partial t^j}\right|_{t=0}\left[\frac{1}{(1-t)^{\frac{\mu+\nu+2}{2}}}e^{-\frac{tx}{1-t}(\cos\theta+\cos\phi)}\right]}}\\
 &\ \ \ \ \ \ \ \ \ \ \ \ \ \ \ \ \ \ \ \ \ \ \ \ \ \ \ \ \ \ \ \ \ \ \ \ \ \ \ \ \ \ \ \ \ \ \ \ \ \ \ \ \ \ \ \ \ \ \ \ \ \ \ \ \ \ \ \ \ \sin^\mu\theta\sin^\nu\phi\td\phi\td\theta\\
 ={}& \int_0^\pi{\int_0^\pi{e^{-x\cos\phi} L_j^{\frac{\mu+\nu}{2}}(x(\cos\theta+\cos\phi))\sin^\mu\theta\sin^\nu\phi\td\phi}\td\theta}.
\end{align*}
For the functions $\Lambda_{2,j}^{\mu,\nu}$ we do a similar calculation:
\begin{align*}
 & \Gamma\left(\frac{\mu+1}{2}\right)\Gamma\left(\frac{\nu+1}{2}\right)\Lambda_{2,j}^{\mu,\nu}(x)\\
 ={}& \frac{\Gamma(\frac{\mu+1}{2})\Gamma(\frac{\nu+1}{2})}{j!}\left.\frac{\partial^j}{\partial t^j}\right|_{t=0} G_2^{\mu,\nu}(t,x)\\
 ={}& \frac{1}{j!} \left.\frac{\partial^j}{\partial t^j}\right|_{t=0} \frac{1}{(1-t)^{\frac{\mu+\nu+2}{2}}} \int_0^\pi{\int_0^\infty{e^{-\frac{tx}{1-t}\cos\theta}e^{-\frac{x}{1-t}\cosh\phi}}}\\
  &\ \ \ \ \ \ \ \ \ \ \ \ \ \ \ \ \ \ \ \ \ \ \ \ \ \ \ \ \ \ \ \ \ \ \ \ \ \ \ \ \ \ \ \ \ \ \ \ \ \ \ \ \ \ \ \ \ \ \ \ \ \ \ \ \ \ \sin^\mu\theta\sinh^\nu\phi\td\phi\td\theta\\
 ={}& \int_0^\pi{\int_0^\infty{e^{-x\cosh\phi}\frac{1}{j!}\left.\frac{\partial^j}{\partial t^j}\right|_{t=0}\left[\frac{1}{(1-t)^{\frac{\mu+\nu+2}{2}}}e^{-\frac{tx}{1-t}(\cos\theta+\cosh\phi)}\right]}}\\
 &\ \ \ \ \ \ \ \ \ \ \ \ \ \ \ \ \ \ \ \ \ \ \ \ \ \ \ \ \ \ \ \ \ \ \ \ \ \ \ \ \ \ \ \ \ \ \ \ \ \ \ \ \ \ \ \ \ \ \ \ \ \ \ \ \ \ \sin^\mu\theta\sinh^\nu\phi\td\phi\td\theta\\
 ={}& \int_0^\pi{\int_0^\infty{e^{-x\cosh\phi} L_j^{\frac{\mu+\nu}{2}}(x(\cos\theta+\cosh\phi))\sin^\mu\theta\sinh^\nu\phi\td\phi}\td\theta}.
\end{align*}
\item[\textup{(2)}] Using \eqref{eq:IKBesselMinusHalf}, similar calculations as in (1) give the second part.\qedhere
\end{enumerate}
\end{proof}

\begin{remark}
The integral representations in Theorem \ref{prop:IntFormulae}~(2) for the special case $\nu=-1$ can also be obtained from the integral representations in part (1) for $\nu>-1$ by taking the limit $\nu\rightarrow-1$. For example, to obtain the integral representation for $\Lambda_{2,j}^{\mu,-1}$ we have to verify the limit formula
\begin{equation}
 \lim_{\nu\rightarrow-1}{\frac{1}{\Gamma(\frac{\nu+1}{2})}\int_0^\infty{e^{-x\cosh\phi}\cosh^k\phi\sinh^\nu\phi\td\phi}} = \frac{1}{2}e^{-x}\label{eq:IntLimit}
\end{equation}
for $0\leq k\leq j$. For $k=0$ the identity \eqref{eq:IntFormulaKBessel} turns the left hand side into
\begin{equation*}
 \frac{1}{\sqrt{\pi}}\lim_{\nu\rightarrow-1}{\widetilde{K}_{\frac{\nu}{2}}(x)}.
\end{equation*}
The map $\alpha\mapsto\widetilde{K}_\alpha(x)$ is continuous so \eqref{eq:IntLimit} follows from \eqref{eq:IKBesselMinusHalf}. For $k>0$ and $\phi\geq0$ we have
\begin{align*}
 \cosh^k\phi-\cosh^0\phi &= \cosh^k\phi-1 \leq \sinh\phi \cdot p(\sinh\phi)
\end{align*}
with some polynomial $p$. Then one has to show that
\begin{equation*}
 \lim_{\nu\rightarrow-1}{\frac{1}{\Gamma(\frac{\nu+1}{2})}\int_0^\infty{e^{-x\cosh\phi}\sinh^{\nu+\ell+1}\phi\td\phi}} = \frac{1}{2}e^{-x}.
\end{equation*}
for $\ell\geq0$. But this is easily seen using the integral representation \eqref{eq:IntFormulaKBessel} and the continuity of the map $\alpha\mapsto\widetilde{K}_\alpha(x)$.
\end{remark}

As an easy application of the integral representations we give explicit expressions for the functions $\Lambda_{i,j}^{\mu,\nu}(x)$, $i=1,2$, in the case where $\nu=-1$.

\begin{corollary}\label{cor:SpecialValue}
For $\nu=-1$ and $\mu\in\mathbb{C}$ arbitrary we have the following identity of meromorphic functions
\begin{align}
 \Lambda_{1,j}^{\mu,-1}(x) &= \frac{2^{\mu-1}\Gamma(j+\frac{\mu+1}{2})}{\pi\Gamma(j+\mu+1)}\left(e^{-x}L_j^\mu(2x)+e^{x}L_j^\mu(-2x)\right),\label{eq:Lambda1SpecialValue}\\
 \Lambda_{2,j}^{\mu,-1}(x) &= \frac{2^{\mu-1}\Gamma(j+\frac{\mu+1}{2})}{\Gamma(j+\mu+1)} e^{-x} L_j^{\mu}(2x).\label{eq:Lambda2SpecialValue}
\end{align}
\end{corollary}

\begin{proof}
For the proof we may assume that $\Re(\mu)>-1$. The general case $\mu\in\mathbb{C}$ then follows by meromorphic continuation. With the integral formula \eqref{eq:LaguerreIntFormula} the substitution $y=\frac{1}{2}(1\pm\cos\theta)$ yields
\begin{align}
 \int_0^\pi{L_j^{\frac{\mu-1}{2}}(x(\cos\theta\pm1))\sin^\mu\theta\td\theta}
 ={}& 2^\mu \int_0^1{(1-y)^{\frac{\mu-1}{2}}y^{\frac{\mu-1}{2}}L_j^{\frac{\mu-1}{2}}(\pm2x\cdot y)\td y}\notag\\
 ={}& \frac{2^\mu\Gamma(j+\frac{\mu+1}{2})\Gamma(\frac{\mu+1}{2})}{\Gamma(j+\mu+1)}L_j^\mu(\pm2x).\label{eq:LegendreIntegral2}
\end{align}
Inserting this into the integral representations \eqref{IntFormulaSpecialCase1} and \eqref{IntFormulaSpecialCase2} gives \eqref{eq:Lambda1SpecialValue} and \eqref{eq:Lambda2SpecialValue}.
\end{proof}

\begin{remark}\label{rem:SpecialValue2}
The symmetry property \eqref{eq:KBesselSymmetry} for the $K$-Bessel functions implies that $G_2^{\mu,-1}(x)=\frac{x}{2}G_2^{\mu,1}(x)$ and hence
\begin{align}
 \Lambda_{2,j}^{\mu,1}(x) &= \frac{2}{x}\Lambda_{2,j}^{\mu,-1}(x).\label{eq:RelLambdapm1}
\end{align}
Therefore, Corollary \ref{cor:SpecialValue} also allows us to compute $\Lambda_{2,j}^{\mu,1}$ explicitly:
\begin{equation}
 \Lambda_{2,j}^{\mu,1}(x) = \frac{2^\mu\Gamma(j+\frac{\mu+1}{2})}{\Gamma(j+\mu+1)}x^{-1}e^{-x}L_j^{\mu}(2x).\label{eq:Lambda2SpecialValue2}
\end{equation}
\end{remark}

\begin{remark}\label{rem:DiffEqSpecialCase}
Corollary \ref{cor:SpecialValue} and Remark \ref{rem:SpecialValue2} suggest a relation between the fourth order differential equation $\mathcal{D}_{\mu,\nu}u=4j(j+\mu+1)u$ in the cases where $\nu=\pm1$ and the second order differential equation \eqref{eq:LagDiffEq} for the Laguerre polynomials $L_n^\alpha(x)$. In fact, by Proposition \ref{prop:DiffOpProperties}~(5) the fourth order differential operator $\mathcal{D}_{\mu,\pm1}$ collapses to the simpler form
\begin{align*}
 \mathcal{D}_{\mu,\pm1} &= \mathcal{S}_{\mu,\pm1}^2-(\mu+1)^2
\end{align*}
with second order differential operators $\mathcal{S}_{\mu,\pm1}$ (for their definition see Proposition \ref{prop:DiffOpProperties}~(5)). For $\mu>-1$ the operator $\mathcal{S}_{\mu,-1}$ itself is self-adjoint on $L^2(\mathbb{R}_+,x^\mu\td x)$. It has discrete spectrum given by $(-(2j+\mu+1))_{j\in\mathbb{N}_0}$ and an easy calculation involving \eqref{eq:LagDiffEq} shows that $\Lambda_{2,j}^{\mu,-1}$ is an eigenfunction of $\mathcal{S}_{\mu,-1}$ for the eigenvalue $-(2j+\mu+1)$. Applying $\mathcal{S}_{\mu,-1}$ twice yields the fourth order differential equation of Theorem \ref{thm:EigFct} for $\Lambda_{2,j}^{\mu,-1}$.
The same considerations apply for $\mathcal{S}_{\mu,+1}$ and $\Lambda_{2,j}^{\mu,+1}(x)$ since we have the relation \eqref{eq:RelLambdapm1} and
\begin{equation}
 \mathcal{S}_{\mu,+1}x^{-1} = x^{-1}\mathcal{S}_{\mu,-1}.\label{eq:RelSmupm1}
\end{equation}
\end{remark}

\section{Orthogonal polynomials}

In the previous section we have shown that for $\nu=\pm1$ the functions $\Lambda_{2,j}^{\mu,\nu}(x)$ basically reduce to Laguerre polynomials. Now we prove that for any odd integer $\nu\geq-1$ the functions $\Lambda_{2,j}^{\mu,\nu}(x)$ reduce to polynomials. 

\begin{theorem}\label{thm:SpecialPolynomials}
Suppose $\mu\notin-\NN$ and $\nu\geq1$ is an odd integer. Then
\begin{equation}
 \Lambda_{2,j}^{\mu,\nu}(x) = \frac{2^\mu\Gamma(j+\frac{\mu+1}{2})}{\Gamma(j+\mu+1)}x^{-\nu}e^{-x}M_j^{\mu,\nu}(2x),\label{eq:DefPoly}
\end{equation}
where $M_j^{\mu,\nu}(x)$ is a polynomial of degree $j+\frac{\nu-1}{2}$ ($j\in\NN_0$). The polynomial $M_j^{\mu,\nu}(x)$\index{notation}{Mjmunux@$M_j^{\mu,\nu}(x)$} is given by
\begin{align}
 M_j^{\mu,\nu}(x) &= \frac{\Gamma(j+\mu+1)}{\Gamma(j+\frac{\mu+1}{2})}\sum_{k=0}^j{\sum_{i=0}^{\frac{\nu-1}{2}-k}{\!\!(-1)^k\frac{\Gamma(j-k+\frac{\mu+1}{2})(\nu-i-1)!} {k!\Gamma(j-k+\mu+1)(\frac{\nu-1}{2}-i-k)!i!}L_{j-k}^\mu(x)x^i}}\notag\\
 &= \sum_{k=0}^{j+\frac{\nu-1}{2}}{\beta^{\mu,\nu}_{j,k}x^k},\label{eq:ExplPoly}
\end{align}
where
\begin{multline*}
 \beta^{\mu,\nu}_{j,k} = \frac{\Gamma(j+\mu+1)}{\Gamma(j+\frac{\mu+1}{2})}\sum_{(m,n)\in S_{j,k}^{\mu,\nu}}{(-1)^{m+n}\frac{\Gamma(j-m+\frac{\mu+1}{2})}{\Gamma(n+\mu+1)}}\\
 \times{\frac{(\nu+n-k-1)!}{m!n!(k-n)!(j-m-n)!(\frac{\nu-1}{2}+n-k-m)!}}
\end{multline*}
with
\begin{equation}
 S_{j,k}^{\mu,\nu} = \left\{(m,n)\in\mathbb{N}_0^2:\begin{array}{c}0\leq n\leq j-m\\0\leq k-n\leq\frac{\nu-1}{2}-m\end{array}\right\}.
 \label{eq:Sjk}
\end{equation}
\end{theorem}

\begin{proof}
Let us first assume $\Re(\mu)>-1$. With the explicit expression \eqref{eq:ExplKBessel} for the $K$-Bessel functions with half-integer parameter and the integral representation \eqref{eq:IntFormulaIBessel} for the $I$-Bessel function we obtain
\begin{align*}
 G_2^{\mu,\nu}(t,x) ={}& \frac{1}{(1-t)^{\frac{\mu+\nu+2}{2}}}\widetilde{I}_{\frac{\mu}{2}}\left(\frac{tx}{1-t}\right)\widetilde{K}_{\frac{\nu}{2}}\left(\frac{x}{1-t}\right)\\
 ={}& \frac{1}{\Gamma(\frac{\mu+1}{2})}x^{-\nu}e^{-x} \int_0^\pi{\frac{1}{(1-t)^{\frac{\mu-1}{2}+1}}e^{-\frac{tx}{1-t}(\cos\theta+1)}\sin^\mu\theta\td\theta}\\
 & \ \ \ \ \ \ \ \ \ \ \ \ \ \ \ \ \ \ \ \ \ \ \ \ \ \ \ \ \ \ \ \ \ \ \ \ \ \ \ \ \times\sum_{i=0}^{\frac{\nu-1}{2}}{\frac{(\nu-i-1)!}{(\frac{\nu-1}{2}-i)!\cdot i!}(2x)^i(1-t)^{\frac{\nu-1}{2}-i}}.
\end{align*}
Next, we compute the derivatives of the first factor with respect to $t$ at $t=0$. Using the formula \eqref{eq:LaguerreGenFct} for the generating function of the Laguerre polynomials we find that
\begin{align*}
 & \left.\frac{\partial^j}{\partial t^j}\right|_{t=0}\left[\int_0^\pi{\frac{1}{(1-t)^{\frac{\mu-1}{2}+1}}e^{-\frac{tx}{1-t}(\cos\theta+1)}\sin^\mu\theta\td\phi}\right]\\
 ={}& j!\int_0^\pi{L_j^{\frac{\mu-1}{2}}\left(x(\cos\phi+1)\right)\sin^\mu\phi\td\phi}\\
 \intertext{which is by \eqref{eq:LegendreIntegral2} equal to}
 ={}& \frac{j!2^\mu\Gamma(j+\frac{\mu+1}{2})\Gamma(\frac{\mu+1}{2})}{\Gamma(j+\mu+1)}L_j^\mu(2x).
\end{align*}
Now we can compute the Taylor coefficients of $G_2^{\mu,\nu}(t,x)$ at $t=0$ explicitly as follows
\begin{align*}
 & \left.\frac{\partial^j}{\partial t^j}\right|_{t=0}G_2^{\mu,\nu}(t,x)\\
 ={}& \frac{1}{\Gamma(\frac{\mu+1}{2})}x^{-\nu}e^{-x} \sum_{k=0}^j{{j\choose k}\left.\frac{\partial^{j-k}}{\partial t^{j-k}}\right|_{t=0}\left[\int_0^\pi{\frac{1}{(1-t)^{\frac{\mu-1}{2}+1}}e^{-\frac{tx}{1-t}(\cos\theta+1)}\sin^\mu\phi\td\theta}\right]}\\
 & \ \ \ \ \ \ \ \ \ \ \ \ \ \ \ \ \ \ \ \ \ \ \ \ \ \ \ \ \ \ \ \ \ \ \ \ \ \ \times\left.\frac{\partial^k}{\partial t^k}\right|_{t=0}\left[\sum_{i=0}^{\frac{\nu-1}{2}}{\frac{(\nu-i-1)!}{(\frac{\nu-1}{2}-i)!\cdot i!}(2x)^i(1-t)^{\frac{\nu-1}{2}-i}}\right]\\
 ={}& x^{-\nu}e^{-x}\sum_{k=0}^j{\frac{j!2^\mu\Gamma(j-k+\frac{\mu+1}{2})}{k!\Gamma(j-k+\mu+1)}L_{j-k}^\mu(2x)}\\
 & \ \ \ \ \ \ \ \ \ \ \times\sum_{i=0}^{\frac{\nu-1}{2}}{\frac{(\nu-i-1)!}{(\frac{\nu-1}{2}-i)!\cdot i!}(2x)^i(-1)^k\left(\frac{\nu-1}{2}-i\right)\cdots\left(\frac{\nu-1}{2}-i-k+1\right)}\\
 ={}& x^{-\nu}e^{-x}\sum_{k=0}^j{\sum_{i=0}^{\frac{\nu-1}{2}-k}{(-1)^k\frac{j!2^\mu\Gamma(j-k+\frac{\mu+1}{2})(\nu-i-1)!}{k!\Gamma(j-k+\mu+1)(\frac{\nu-1}{2}-i-k)!i!}L_{j-k}^\mu(2x)(2x)^i}}.
\end{align*}
This gives the first expression for $M_j^{\mu,\nu}(x)$. Inserting the explicit formula \eqref{eq:DefLagFct} for the Laguerre functions one obtains the expression for the coefficients $\beta^{\mu,\nu}_{j,k}$. Since these clearly extend meromorphically to $\mu\in\mathbb{C}$ with poles at most at $\mu=-1,-2,-3,\ldots$, the claim follows.
\end{proof}

\begin{proposition}\label{prop:TopConstantTerm}
\begin{enumerate}
\item[\textup{(1)}] (Top term) $$M_j^{\mu,\nu}(x)=\frac{(-1)^j}{j!}x^{j+\frac{\nu-1}{2}}+\textup{lower order terms}.$$
\item[\textup{(2)}] (Constant term)
$$M_j^{\mu,\nu}(0)=\frac{2^{\nu-\mu-1}\Gamma(\frac{\nu}{2})\Gamma(j+\mu+1)\left(\frac{\mu-\nu+2}{2}\right)_j}{j!\Gamma(\frac{\mu+2}{2})\Gamma(j+\frac{\mu+1}{2})}.$$
\end{enumerate}
\end{proposition}

\begin{proof}
For $k=j+\frac{\nu-1}{2}$ the set $S_{j,k}^{\mu,\nu}$ defined in \eqref{eq:Sjk} only contains the tuple $(0,j)$ and we
obtain the top term
\begin{equation*}
 \beta^{\mu,\nu}_{j,j+\frac{\nu-1}{2}}=\frac{(-1)^j}{j!}.
\end{equation*}
To calculate the bottom term $M_j^{\mu,\nu}(0)$ we use the asymptotic behavior of $\Lambda_{2,j}^{\mu,\nu}(x)$ as $x\rightarrow0$ (see Theorem \ref{lem:Asymptotics}~(1)). Together with \eqref{eq:DefPoly} this gives the bottom term $M_j^{\mu,\nu}(0)$.
\end{proof}

\begin{remark}
As proved in the previous section, we have
\begin{align*}
 M_j^{\mu,1}(x) = L_j^\mu(x).
\end{align*}
However, for $\nu\geq3$ the special polynomials $M_j^{\mu,\nu}(x)$ do not appear in the standard literature. Properties for these polynomials such as differential equations, orthogonality relations, completeness, recurrence relations and integral representations simply translate from the corresponding properties for the functions $\Lambda_{2,j}^{\mu,\nu}(x)$. The corresponding statements can be found in \cite{HKMM09b}.
\end{remark}

\section{Recurrence relations}\label{sec:RecRel}

In this section we give three types of recurrence relations for the functions $\Lambda_{i,j}^{\mu,\nu}(x)$.\\

Our first recurrence relation involves the first order differential operator $\mathcal{H}_\alpha$\index{notation}{Halpha@$\mathcal{H}_\alpha$} ($\alpha\in\mathbb{C}$) on $\mathbb{R}_+$, given by
\begin{equation*}
 \mathcal{H}_\alpha := \theta+\frac{\alpha+2}{2}.
\end{equation*}
If $\alpha\in\mathbb{R}$, then $\mathcal{H}_\alpha$ is a skew-symmetric operator on $L^2(\mathbb{R}_+,x^{\alpha+1}\td x)$. This allows us to compute the $L^2$-norms for $\Lambda_{2,j}^{\mu,\nu}$ explicitly if $(\mu,\nu)\in\Xi$.

\begin{proposition}\label{prop:RecRelH}
For $\mu,\nu\in\mathbb{C}$, $i=1,2,3,4$ we have the following recurrence relation in $j\in\mathbb{Z}$
\begin{multline}
 (2j+\mu+1)\mathcal{H}_{\mu+\nu}\Lambda_{i,j}^{\mu,\nu}(x) = (j+1)(j+\mu+1)\Lambda_{i,j+1}^{\mu,\nu}(x)\\
 - \left(j+\frac{\mu+\nu}{2}\right)\left(j+\frac{\mu-\nu}{2}\right)\Lambda_{i,j-1}^{\mu,\nu}(x).\label{eq:RecRelH}
\end{multline}
\end{proposition}

\begin{proof}
As in the proof of Theorem \ref{thm:EigFct} we verify \eqref{eq:RecRelH} via a partial differential equation for the generating function $G_i^{\mu,\nu}$. A short calculation (similar to the one in the proof of Theorem \ref{thm:EigFct}) shows that the recurrence relation \eqref{eq:RecRelH} is equivalent to the partial differential equation
\begin{multline*}
 (2\theta_t+\mu+1)\left(\theta_x+\frac{\mu+\nu+2}{2}\right)G_i^{\mu,\nu}(t,x)\\
 = \left(\frac{1}{t}\theta_t(\theta_t+\mu) - t\left(\theta_t+\frac{\mu+\nu+2}{2}\right)\left(\theta_t+\frac{\mu-\nu+2}{2}\right)\right)G_i^{\mu,\nu}(t,x),
\end{multline*}
which holds by Lemma \ref{lem:GenFctPDEs}~(2).
\end{proof}

\begin{corollary}\label{cor:Norms}
If $(\mu,\nu)\in\Xi$, then
\begin{equation}
 \|\Lambda_{2,j}^{\mu,\nu}\|_{L^2(\mathbb{R}_+,x^{\mu+\nu+1}\td x)}^2 = \frac{2^{\mu+\nu-1}\Gamma(j+\frac{\mu+\nu+2}{2})\Gamma(j+\frac{\mu-\nu+2}{2})}{j!(2j+\mu+1)\Gamma(j+\mu+1)}.\label{eq:Norms}
\end{equation}
\end{corollary}

\begin{proof}
We prove this by induction on $j$. For $j=0$, in view of Example \ref{ex:3Fcts} and the integral formula \eqref{eq:KBesselL2Norms}, we can calculate
\begin{align*}
 \|\Lambda_{2,0}^{\mu,\nu}\|^2 &= \int_0^\infty{|\Lambda_{2,0}^{\mu,\nu}(x)|^2x^{\mu+\nu+1}\td x}\\
 &= \frac{1}{\Gamma(\frac{\mu+2}{2})^2}\int_0^\infty{|\widetilde{K}_{\frac{\nu}{2}}(x)|^2x^{\mu+\nu+1}\td x}\\
 &= \frac{2^{\mu+\nu-1}\Gamma(\frac{\mu+\nu+2}{2})\Gamma(\frac{\mu-\nu+2}{2})}{(\mu+1)\Gamma(\mu+1)}.
\end{align*}
For the induction step we reformulate \eqref{eq:RecRelH} as
\begin{equation*}
 \mathcal{H}_{\mu+\nu}\Lambda_{2,j}^{\mu,\nu} = \frac{(j+1)(j+\mu+1)}{2j+\mu+1}\Lambda_{2,j+1}^{\mu,\nu} - \frac{(2j+\mu+\nu)(2j+\mu-\nu)}{4(2j+\mu+1)}\Lambda_{2,j-1}^{\mu,\nu}.
\end{equation*}
Using the skew-symmetry of $\mathcal{H}_{\mu+\nu}$ on $L^2(\mathbb{R}_+,x^{\mu+\nu+1}\td x)$ together with the pairwise orthogonality of the functions $\Lambda_{2,j}^{\mu,\nu}$ (cf. Corollary \ref{cor:completeness}) we calculate
\begin{align*}
 {}& \|\mathcal{H}_{\mu+\nu}\Lambda_{2,j}^{\mu,\nu}\|^2
 = \left(\mathcal{H}_{\mu+\nu}\Lambda_{2,j}^{\mu,\nu}|\mathcal{H}_{\mu+\nu}\Lambda_{2,j}^{\mu,\nu}\right)
 = - \left(\Lambda_{2,j}^{\mu,\nu}|\mathcal{H}_{\mu+\nu}^2\Lambda_{2,j}^{\mu,\nu}\right)\\
 ={}& - \left(\Lambda_{2,j}^{\mu,\nu}\left|\frac{(j+1)(j+\mu+1)}{2j+\mu+1}\mathcal{H}_{\mu+\nu}
  \Lambda_{2,j+1}^{\mu,\nu}\right.\right)\\
 &\ \ \ \ \ \ \ \ \ \ \ \ \ \ \ \ \ \ \ \ \ \ \ \ \ \ \ \ \ \ \ \ \ \ \ \ \ \ \ \ \ \ \ \ \ \ \left.-\frac{(2j+\mu+\nu)(2j+\mu-\nu)}{4(2j+\mu+1)}\mathcal{H}_{\mu+\nu}\Lambda_{2,j-1}^{\mu,\nu}\right)\\
 ={}& \left(\frac{(j+1)(j+\mu+1)}{2j+\mu+1}\cdot\frac{(2(j+1)+\mu+\nu)(2(j+1)+\mu-\nu)}{4(2(j+1)+\mu+1)}
   \right.\\
 &\ \ \ \ \ \ \ \ \ \ \ \ \ \ \ \ \ \ \ \ \ \ \
 +\left.\frac{(2j+\mu+\nu)(2j+\mu-\nu)}{4(2j+\mu+1)}\cdot\frac{j((j-1)+\mu+1)}{2(j-1)+\mu+1}\right)\|
 \Lambda_{2,j}^{\mu,\nu}\|^2.
\end{align*}
On the other hand, orthogonality and the recurrence relation yield
\begin{multline*}
 \|\mathcal{H}_{\mu+\nu}\Lambda_{2,j}^{\mu,\nu}\|^2 = \left(\frac{(j+1)(j+\mu+1)}{2j+\mu+1}\right)^2\|\Lambda_{2,j+1}^{\mu,\nu}\|^2\\
 +\left(\frac{(2j+\mu+\nu)(2j+\mu-\nu)}{4(2j+\mu+1)}\right)^2\|\Lambda_{2,j-1}^{\mu,\nu}\|^2.
\end{multline*}
Both identities together complete the induction.
\end{proof}

The second type of recurrence relations expresses $x^2\Lambda_{i,j}^{\mu,\nu}$ as linear combination in $\Lambda_{i,k}^{\mu,\nu}$ for $k=j-2,\ldots,j+2$. These recurrence relations are an immediate consequence of the fifth order differential equation for the generating functions $G_i^{\mu,\nu}(t,x)$ given in Lemma \ref{lem:GenFctPDEs}~(3).

\begin{proposition}\label{prop:RecRelxSq}
For $\mu,\nu\in\mathbb{C}$ we have
\begin{align*}
 &8\left(j+\frac{\mu-1}{2}\right)\left(j+\frac{\mu+1}{2}\right)\left(j+\frac{\mu+3}{2}\right)x^2\Lambda_{i,j}^{\mu,\nu}(x)\notag\\
 &\ \ \ \ \ =2(j+1)(j+2)(j+\mu+1)(j+\mu+2)\left(j+\frac{\mu-1}{2}\right)\Lambda_{i,j+2}^{\mu,\nu}(x)\notag\\
 &\ \ \ \ \ \ \ \ - 8(j+1)(j+\mu+1)\left(j+\frac{\mu-1}{2}\right)\left(j+\frac{\mu+2}{2}\right)\left(j+\frac{\mu+3}{2}\right)\Lambda_{i,j+1}^{\mu,\nu}(x)\notag\\
 &\ \ \ \ \ \ \ \ + 2\left(j+\frac{\mu+1}{2}\right)(aj^4+bj^3+cj^2+dj+e)\Lambda_{i,j}^{\mu,\nu}(x)\notag
\end{align*}
\begin{align*}
 &\ \ \ \ \ \ \ \ - 8\left(j+\frac{\mu-1}{2}\right)\left(j+\frac{\mu}{2}\right)\left(j+\frac{\mu+3}{2}\right)\left(j+\frac{\mu+\nu}{2}\right)\notag\\
 &\ \ \ \ \ \ \ \ \ \ \ \ \ \ \ \ \ \ \ \ \ \ \ \ \ \ \ \ \ \ \ \ \ \ \ \ \ \ \ \ \ \ \ \ \ \ \ \ \ \ \ \ \ \ \ \ \ \ \ \ \ \ \ \ \ \ \ \left(j+\frac{\mu-\nu}{2}\right)\Lambda_{i,j-1}^{\mu,\nu}(x)\notag\\
 &\ \ \ \ \ \ \ \ + 2\left(j+\frac{\mu+3}{2}\right)\left(j+\frac{\mu+\nu-2}{2}\right)\left(j+\frac{\mu-\nu-2}{2}\right)\notag\\
 &\ \ \ \ \ \ \ \ \ \ \ \ \ \ \ \ \ \ \ \ \ \ \ \ \ \ \ \ \ \ \ \ \ \ \ \ \ \ \ \ \ \ \ \ \ \ \ \ \times\left(j+\frac{\mu+\nu}{2}\right)\left(j+\frac{\mu-\nu}{2}\right)\Lambda_{i,j-2}^{\mu,\nu}(x)
\end{align*}
with $a,b,c,d,e$ as in Lemma \ref{lem:GenFctPDEs}~(3).
\end{proposition}

\begin{remark}
For $j\neq-\frac{\mu-1}{2},-\frac{\mu+1}{2},-\frac{\mu+3}{2}$ the recurrence relation of Proposition \ref{prop:RecRelxSq} can be rewritten as
\begin{equation*}
 x^2\Lambda_{i,j}^{\mu,\nu}(x) = \sum_{k=-2}^{2}{a_{i,j}^{\mu,\nu}(k)\Lambda_{i,j+k}^{\mu,\nu}(x)}
\end{equation*}
with constants $a_{i,j}^{\mu,\nu}(k)$.
\end{remark}

The last set of recurrence relations in the parameters $\mu$ and $\nu$ are again immediate with the corresponding differential equations for the generating functions which have already been stated in Lemma \ref{lem:GenFctRecRels}:

\begin{proposition}\label{prop:Formula1}
Let $\mu,\nu\in\mathbb{C}$. With $\delta(i)$, $\varepsilon(i)$ as in \eqref{eq:DelEps} we have three different recurrence relations:
\begin{enumerate}
 \item[\textup{(1)}] The recurrence relation in $\mu$
  \begin{equation*}
   \mu\left(\Lambda_{i,j}^{\mu,\nu}(x)-\Lambda_{i,j-1}^{\mu,\nu}(x)\right) = 2\delta(i)\left(\Lambda_{i,j}^{\mu-2,\nu}(x)-\left(\frac{x}{2}\right)^2\Lambda_{i,j-2}^{\mu+2,\nu}(x)\right).
  \end{equation*}
 \item[\textup{(2)}] The recurrence relation in $\nu$
  \begin{equation*}
   \nu\left(\Lambda_{i,j}^{\mu,\nu}(x)-\Lambda_{i,j-1}^{\mu,\nu}(x)\right) = 2\varepsilon(i)\left(\Lambda_{i,j}^{\mu,\nu-2}(x)-\left(\frac{x}{2}\right)^2\Lambda_{i,j}^{\mu,\nu+2}(x)\right).
  \end{equation*}
 \item[\textup{(3)}] The recurrence relation in $\mu$ and $\nu$
  \begin{equation*}
   \frac{\td}{\td x}\left(\Lambda_{i,j}^{\mu,\nu}(x)-\Lambda_{i,j-1}^{\mu,\nu}(x)\right) = \delta(i)\frac{x}{2}\Lambda_{i,j-2}^{\mu+2,\nu}(x) + \varepsilon(i)\frac{x}{2}\Lambda_{i,j}^{\mu,\nu+2}(x).
  \end{equation*}
\end{enumerate}
\end{proposition}

\section{Meijer's $G$-transform}\label{sec:GTrafo}

The main result of this section is that for $(\mu,\nu)\in\Xi$ the functions $\Lambda_{2,j}^{\mu,\nu}(x)$ are eigenfunctions of a special type $\calT^{\mu,\nu}$ of Meijer's $G$-transform. This $G$-transform appears as the radial part of the unitary inversion operator $\calF_\calO$ (see Theorem \ref{thm:RadialUnitaryInversion}).

The integral transform $\mathcal{T}^{\mu,\nu}$ is defined by
\begin{align*}
 \mathcal{T}^{\mu,\nu}f(x) &:= \int_0^\infty{K^{\mu,\nu}(xy)f(y)y^{\mu+\nu+1}\td y} & \forall\, f\in\mathcal{C}_c^\infty(\mathbb{R}_+),\index{notation}{Tmunu@$\calT^{\mu,\nu}$}
\end{align*}
where the kernel function $K^{\mu,\nu}(x)$ is given by\index{notation}{Kmunux@$K^{\mu,\nu}(x)$}
\begin{equation*}
 K^{\mu,\nu}(x) := \frac{1}{2^{\mu+\nu+1}}G^{20}_{04}\left(\left(\frac{x}{4}\right)^2\left|0,-\frac{\nu}{2},-\frac{\mu}{2},-\frac{\mu+\nu}{2}\right.\right).
\end{equation*}
Here $G^{20}_{04}(x|b_1,b_2,b_3,b_4)$ denotes Meijer's $G$-function (see Appendix \ref{app:GFct}). Using the differential equation \eqref{eq:GFctDiffEqSpecialized} for the $G$-function it is easy to see that $K^{\mu,\nu}(x)$ satisfies the fourth order differential equation
\begin{align}
 \theta(\theta+\mu)(\theta+\nu)(\theta+\mu+\nu)u(x) &= x^2u(x)\label{eq:GDiffEq}
\end{align}

The operator $\calT^{\mu,\nu}$ is a special case of the more general $G$-transform which was first systematically investigated by C. Fox. In \cite[Theorem 1]{Fox61} he shows that for certain parameters the $G$-transform defines a unitary involutive operator on a certain $L^2$-space. This result is used to prove the first statement of the following proposition. Note that we do not yet assume that $(\mu,\nu)\in\Xi$. However, if $(\mu,\nu)\in\Xi$, then the following statement can also be obtained from representation theory (see Section \ref{sec:RepTh} for a proof).

\begin{proposition}\label{prop:GTrafoProperties}
Suppose $\mu+\nu,\mu-\nu>-2$.
\begin{enumerate}
 \item[\textup{(1)}] $\mathcal{T}^{\mu,\nu}$ extends to a unitary involutive operator on $L^2(\mathbb{R}_+,x^{\mu+\nu+1}dx)$.
 \item[\textup{(2)}] The $G$-transform $\mathcal{T}^{\mu,\nu}$ commutes with the fourth order differential operator $\calD_{\mu,\nu}$.
\end{enumerate}
\end{proposition}

\begin{proof}
\begin{enumerate}
 \item[\textup{(1)}] It is proved in \cite[Theorem 1]{Fox61} that
 \begin{equation*}
  Tf(r) := \frac{1}{c}\int_0^\infty{G^{20}_{04}\left((rr')^{\frac{1}{c}}\left|0,-\frac{\nu}{2},-\frac{\mu}{2},-\frac{\mu+\nu}{2}\right.\right)f(r')\td r'}
 \end{equation*}
 defines a unitary involutive operator $T:L^2(\mathbb{R}_+)\longrightarrow L^2(\mathbb{R}_+)$ if $c>0$ and $c>\nu$. By assumption $c=\frac{\mu+\nu+2}{2}$ satisfies this condition. Then the coordinate change $r=\left(\frac{x}{2}\right)^{2c}$, $r'=\left(\frac{y}{2}\right)^{2c}$ gives the claim.
 \item[\textup{(2)}] A\ \ short\ \ calculation,\ \ using\ \ that\ \ $\calD_{\mu,\nu}$\ \ is\ \ a\ \ symmetric\ \ operator\ \ in $L^2(\mathbb{R}_+,x^{\mu+\nu+1}\td x)$, gives the desired statement if one knows that the kernel function $K^{\mu,\nu}(x)$ satisfies the following differential equation
 \begin{equation*}
  \left(\calD_{\mu,\nu}\right)_xK^{\mu,\nu}(xy) = \left(\calD_{\mu,\nu}\right)_yK^{\mu,\nu}(xy).
 \end{equation*}
 But this is easily derived from the expression \eqref{eq:DiffOp2} for $\calD_{\mu,\nu}$ using the identity
 \begin{align*}
  \theta_xK^{\mu,\nu}(xy) &= (\theta K^{\mu,\nu})(xy) = \theta_yK^{\mu,\nu}(xy)
 \end{align*}
 and the differential equation \eqref{eq:GDiffEq} for $K^{\mu,\nu}(x)$.\qedhere
\end{enumerate}
\end{proof}

\begin{theorem}[Meijer's $G$-transform]\label{thm:EigFctGTrafo}
Suppose that $(\mu,\nu)\in\Xi$. Then for each $j\in\mathbb{N}_0$ the function $\Lambda_{2,j}^{\mu,\nu}(x)$ is an eigenfunction of Meijer's $G$-transform $\mathcal{T}^{\mu,\nu}$ for the eigenvalue $(-1)^j$.
\end{theorem}

\begin{proof}
Since $\calD_{\mu,\nu}$ and $\calT^{\mu,\nu}$ commute, the function $\calT^{\mu,\nu}\Lambda_{2,j}^{\mu,\nu}$ is also an eigenfunction of $\calD_{\mu,\nu}$ for the eigenvalue $4j(j+\mu+1)$. But by Proposition \ref{prop:DiffOpProperties}~(4) and Theorem \ref{thm:EigFct} the function $\Lambda_{2,j}^{\mu,\nu}$ spans the $L^2$-eigenspace of $\calD_{\mu,\nu}$ with eigenvalue $4j(j+\mu+1)$. Hence, there exists $\varepsilon_j\in\RR$ such that
\begin{align}
 \calT^{\mu,\nu}\Lambda_{2,j}^{\mu,\nu}(x) &= \varepsilon_j\Lambda_{2,j}^{\mu,\nu}(x).\label{eq:TmunuLambdaEigFct}
\end{align}
We calculate $\varepsilon_j$ by specializing \eqref{eq:TmunuLambdaEigFct} to $x=0$.\\
Let us just treat the case $\nu>0$ here. The other cases $\nu=0$ and $\nu<0$ can be treated similarly. Multiplying \eqref{eq:TmunuLambdaEigFct} with $x^\nu$ and taking the limit $x\rightarrow0$ yields
\begin{align}
 \lim_{x\rightarrow0}{\int_0^\infty{(xy)^\nu K^{\mu,\nu}(xy)\Lambda_{2,j}^{\mu,\nu}(y)y^{\mu+1}\td y}} &= \varepsilon_j\lim_{x\rightarrow0}{x^\nu\Lambda_{2,j}^{\mu,\nu}(x)}.\label{eq:LambdaEigenfctMultiplied}
\end{align}
The right hand side is by Theorem \ref{lem:Asymptotics} equal to
\begin{align*}
 \varepsilon_j\frac{(\frac{\mu-\nu+2}{2})_j2^{\nu-1}\Gamma(\frac{\nu}{2})}{j!\Gamma(\frac{\mu+2}{2})}.
\end{align*}
To justify interchanging limit and integral on the left hand side we apply the dominated convergence theorem. By the asymptotic behavior of the $G$-function at $x=0$ and $x=\infty$ (see \eqref{eq:GFctAsymptotics0} and \eqref{eq:GFctAsymptoticsInfty}), the function $x^\nu K^{\mu,\nu}(x)$ is bounded on $\RR_+$ and hence
\begin{align*}
 \left|(xy)^\nu K^{\mu,\nu}(xy)\Lambda_{2,j}^{\mu,\nu}(y)y^{\mu+1}\right| &\leq C\cdot\left|\Lambda_{2,j}^{\mu,\nu}(y)y^{\mu+1}\right|
\end{align*}
for some constant $C>0$. Therefore, the integrand in \eqref{eq:LambdaEigenfctMultiplied} is dominated by the function $C\cdot\left|\Lambda_{2,j}^{\mu,\nu}(y)y^{\mu+1}\right|$ which is integrable by Theorem \ref{lem:Asymptotics}. Hence, the assumptions of the dominated convergence theorem are satisfied and with the asymptotic behavior of the function $K^{\mu,\nu}(x)$ at $x=0$ (see Lemma \ref{lem:KmunuAsymptotics}) we obtain
\begin{align*}
 \lim_{x\rightarrow0}{\int_0^\infty{(xy)^\nu K^{\mu,\nu}(xy)\Lambda_{2,j}^{\mu,\nu}(y)y^{\mu+1}\td y}} &= \frac{2^{\nu-1}\Gamma(\frac{\nu}{2})}{2^\mu\Gamma(\frac{\mu+2}{2})\Gamma(\frac{\mu-\nu+2}{2})}\int_0^\infty{\Lambda_{2,j}^{\mu,\nu}(y)y^{\mu+1}\td y}
\end{align*}
Together with the following lemma this shows that $\varepsilon_j=(-1)^j$ which finishes the proof. (Part (2) of the lemma is needed for the case $\nu<0$.)
\end{proof}


\begin{lemma}\label{lem:Lambda2L1}
\begin{enumerate}
\item[\textup{(1)}] $\displaystyle\int_0^\infty{\Lambda_{2,j}^{\mu,\nu}(x)x^{\mu+1}\td x} = (-1)^j\frac{2^\mu\Gamma(\frac{\mu-\nu+2}{2}+j)}{j!}$,
\item[\textup{(2)}] $\displaystyle\int_0^\infty{\Lambda_{2,j}^{\mu,\nu}(x)x^{\mu+\nu+1}\td x} = (-1)^j\frac{2^{\mu+\nu}\Gamma(\frac{\mu+\nu+2}{2}+j)}{j!}$.
\end{enumerate}
\end{lemma}

\begin{proof}
We use the integral formula \eqref{eq:IntegralIKBessel}. Together with \eqref{eq:G2} we obtain
\begin{align*}
 \int_0^\infty{G_2^{\mu,\nu}(t,x)x^{\mu+1}\td x} &= 2^\mu\Gamma\left(\frac{\mu-\nu+2}{2}\right)(1-t)^{\frac{\mu-\nu+2}{2}}\\
 & \ \ \ \ \ \ \ \ \ \ \ \ \ \ \ \ \ \ \ \ \ \ \ \ \ \times{_2F_1}\left(\frac{\mu+2}{2},\frac{\mu-\nu+2}{2};\frac{\mu+2}{2};t^2\right)\\
 &= 2^\mu\Gamma\left(\frac{\mu-\nu+2}{2}\right)(1+t)^{-\frac{\mu-\nu+2}{2}}\\
 &= \sum_{j=0}^\infty{\frac{2^\mu\Gamma(\frac{\mu-\nu+2}{2}+j)}{j!}(-t)^j}
\end{align*}
and
\begin{align*}
 \int_0^\infty{G_2^{\mu,\nu}(t,x)x^{\mu+\nu+1}\td x} &= 2^{\mu+\nu}\Gamma\left(\frac{\mu+\nu+2}{2}\right)(1-t)^{\frac{\mu+\nu+2}{2}}\\
 & \ \ \ \ \ \ \ \ \ \ \ \ \ \ \ \ \ \ \ \ \ \ \ \times{_2F_1}\left(\frac{\mu+\nu+2}{2},\frac{\mu+2}{2};\frac{\mu+2}{2};t^2\right)\\
 &= 2^{\mu+\nu}\Gamma\left(\frac{\mu+\nu+2}{2}\right)(1+t)^{-\frac{\mu+\nu+2}{2}}\\
 &= \sum_{j=0}^\infty{\frac{2^{\mu+\nu}\Gamma(j+\frac{\mu+\nu+2}{2})}{j!}(-t)^j}.
\end{align*}
In view of \eqref{eq:FctDef} the claim follows.
\end{proof}

\begin{remark}
For $\nu=\pm1$ the functions $\Lambda_{2,j}^{\mu,\nu}(x)$ are Laguerre functions by \eqref{eq:Lambda2SpecialValue} and \eqref{eq:Lambda2SpecialValue2}. In this case the reduction formula \eqref{eq:GFctJBessel} implies that the kernel function $K^{\mu,\nu}(x)$ simplifies to a $J$-Bessel function:
\begin{align*}
 K^{\mu,\pm1}(x) &= x^{-\frac{\mu+\nu+1}{2}}J_\mu(2x^{\frac{1}{2}}).
\end{align*}
Then $\calT^{\mu,\nu}$ is a Hankel type transform and Theorem \ref{thm:EigFctGTrafo} is a reformulation of \cite[8.9~(3)]{EMOT54a}. Note that the integral formula in \cite[8.9~(3)]{EMOT54a} holds for a more general set of parameters.
\end{remark}

One can use Theorem \ref{thm:EigFctGTrafo} to obtain an integral formula for the generating function $G_2^{\mu,\nu}(t,x)$ and hence for the Bessel functions involved:

\begin{corollary}\label{cor:BesselIntFormula2}
Let $\frac{1}{2}<\alpha<\infty$, $\beta=\frac{\alpha}{2\alpha-1}$ and assume that $(\mu,\nu)\in\Xi$. Then for $x>0$
\begin{multline*}
 \int_0^\infty{K^{\mu,\nu}(xy)\widetilde{I}_{\frac{\mu}{2}}\left((\alpha-1)y\right)\widetilde{K}_{\frac{\nu}{2}}\left(\alpha y\right)y^{\mu+\nu+1}\td y}\\
 = \left(\frac{\beta}{\alpha}\right)^{\frac{\mu+\nu+2}{2}}\widetilde{I}_{\frac{\mu}{2}}((\beta-1)x)\widetilde{K}_{\frac{\nu}{2}}(\beta x).
\end{multline*}
\end{corollary}

\begin{proof}
By Theorem \ref{thm:EigFctGTrafo} we have
\begin{align*}
 \mathcal{T}^{\mu,\nu}\Lambda_{2,j}^{\mu,\nu} &= (-1)^j\Lambda_{2,j}^{\mu,\nu} \quad \text{ for every }j\in\mathbb{N}_0.
\end{align*}
Taking generating functions of both sides yields
\begin{equation*}
 \left(\mathcal{T}^{\mu,\nu}\right)_xG_2^{\mu,\nu}(t,x) = G_2^{\mu,\nu}(-t,x),
\end{equation*}
and this gives the desired formula for $\alpha=\frac{1}{1-t}$.
\end{proof}

\section{Applications to minimal representations}\label{sec:RepTh}

In this section we relate the $L^2$-eigenfunctions $\Lambda_{2,j}^{\mu,\nu}(x)$ of $\calD_{\mu,\nu}$ to representation theory in the case where $(\mu,\nu)\in\Xi$. If $(\mu,\nu)\in\Xi$, then they are the parameters introduced in Section \ref{sec:MuNu} corresponding to a simple real Jordan algebra $V$ of split rank $r_0\geq2$ with simple euclidean subalgebra $V^+$ which is not isomorphic to $\RR^{p,q}$ with $p+q$ odd. For these Jordan algebras we have constructed a unitary irreducible representation $\pi$ of the simple group $\check{G}$ on $L^2(\calO,\td\mu)$. We now show that for $j\in\NN_0$ the functions $\Lambda_{2,j}^{\mu,\nu}(x)$ give rise to $\frakk_\frakl$-spherical vectors in the $\frakk$-type $W^j$. We further explain the representation theoretic meaning of several properties for the special functions $\Lambda_{2,j}^{\mu,\nu}(x)$ which were derived in this chapter.

\subsubsection{$\frakk$-finite vectors}

\begin{theorem}\label{thm:KTypes}
In each $\frakk$-type $W^j$ the space of $\frakk_\frakl$-invariant vectors is one-dimen- sional and spanned by the functions
\begin{align*}
 \psi_j(x) &:= \Lambda_{2,j}^{\mu,\nu}(|x|), & x\in\calO.\index{notation}{psij@$\psi_j$}
\end{align*}
\end{theorem}

\begin{proof}
By Proposition \ref{prop:klSphericalsReps} the space of $\frakk_\frakl$-invariant vectors in each irreducible $\frakk$-representation is at most one-dimensional. The functions $\psi_j$ are clearly $\frakk_\frakl$-invariant since they are $K_L$-invariant. By Proposition \ref{prop:CasimirScalar} the Casimir operator $\td\pi(C)$ acts on $W^j$ by the scalar
\begin{align*}
 -\frac{r_0}{8n}\left(4j(j+\mu+1)+\frac{r_0d}{2}\left|d_0-\frac{d}{2}\right|\right),
\end{align*}
and by Theorems \ref{thm:CasimirAction} and \ref{thm:EigFct} it acts on $\psi_j$ by the same scalar. Since all these scalars are distinct, the claim follows.
\end{proof}

\begin{remark}
For the euclidean case it is (indirectly) shown in \cite[Section XV.4]{FK94} that the subspace of $K_L$-invariant vectors in $W^j$ is spanned by the so-called generalized Laguerre function $\ell_{\bf m}^\lambda(x)$ with ${\bf m}=(j,0\ldots,0)$ and $\lambda=\lambda_1=\frac{r_0d}{2r}=\frac{d}{2}$. We show that these functions agree with the functions $\psi_j$ on the orbit $\calO$.

The generalized Laguerre functions are defined in purely Jordan algebraic terms. Let us recall their construction from \cite[Section XV.4]{FK94}: For ${\bf m}\in\NN_0^r$, ${\bf m}\geq0$ (i.e. $m_1\geq\ldots\geq m_r\geq0$), we define the \textit{generalized power function}\index{subject}{generalized power functions} $\Delta_{\bf m}$ on $V$ by
\begin{align*}
 \Delta_{\bf m}(x) &:= \Delta_1(x)^{m_1-m_2}\cdots\Delta_{r-1}(x)^{m_{r-1}-m_r}\Delta_r(x)^{m_r}.\index{notation}{Deltamx@$\Delta_{\bf m}(x)$}
\end{align*}
Here $\Delta_j(x)$ denote the \textit{principal minors}\index{subject}{principal minors} of $V$ (see \cite[Section VI.3]{FK94} for their definition). Then the corresponding spherical polynomials are obtained by integrating over $K_L$:
\begin{align*}
 \Phi_{\bf m}(x) &:= \int_{K_L}{\Delta_{\bf m}(kx)\td k}.\index{notation}{Phimx@$\Phi_{\bf m}(x)$}
\end{align*}
The polynomials $\Phi_{\bf m}$ constitute a basis for the $K_L$-invariant polynomials on $V$. Since $K_L$ stabilizes the identity element $e$ in the euclidean case, the polynomial $\Phi_{\bf m}(e+x)$ is again $K_L$-invariant and hence a linear combination
\begin{align*}
 \Phi_{\bf m}(e+x) &= \sum_{\bf n}{{{\bf m}\choose{\bf n}}\Phi_{\bf n}(x)}\index{notation}{1mchoosen@${\bf m\choose\bf n}$}
\end{align*}
with certain coefficients ${{\bf m}\choose{\bf n}}$ which are called \textit{generalized binomial coefficients}\index{subject}{generalized binomial coefficients}. We then define the \textit{generalized Laguerre polynomials}\index{subject}{generalized Laguerre polynomials} by
\begin{align*}
 L_{\bf m}^\lambda(x) &:= (\lambda)_{\bf m}\sum_{\bf n}{{{\bf m}\choose{\bf n}}\frac{1}{(\lambda)_{\bf n}}\Phi_{\bf n}(-x)},\index{notation}{Lmlambdax@$L_{\bf m}^\lambda(x)$}
\intertext{where}
 (\lambda)_{\bf m} &:= \prod_{i=1}^r{\left(\lambda-(i-1)\frac{d}{2}\right)_{m_i}}\index{notation}{lambdam@$(\lambda)_{\bf m}$}
\end{align*}
and $(a)_n=a(a+1)\cdots(a+n-1)$ denotes the Pochhammer symbol. Finally, the \textit{generalized Laguerre functions}\index{subject}{generalized Laguerre functions} are
\begin{align*}
 \ell_{\bf m}^\lambda(x) &:= e^{-\tr(x)}L_{\bf m}^\lambda(2x).\index{notation}{lmlambdax@$\ell_{\bf m}^\lambda(x)$}
\end{align*}
Now let us calculate $\ell_{\bf m}^\lambda(x)$ for ${\bf m}=(m_1,0,\ldots,0)$, $\lambda=\lambda_1=\frac{d}{2}$ and $x\in\calO$. Since $\Delta_1(x)=(x|c_1)$, we obtain:
\begin{align*}
 \Phi_{\bf m}(e+x) &= \int_{K_L}{(k(e+x)|c_1)^{m_1}\td k} = \int_{K_L}{\left(1+(kx|c_1)\right)^{m_1}\td k}\\
 &= \sum_{n_1=0}^{m_1}{{m_1\choose n_1}\int_{K_L}{(kx|c_1)^{m_1}\td k}} = \sum_{n_1=0}^{m_1}{{m_1\choose n_1}\Phi_{(n_1,0,\ldots,0)}(x)}.
\end{align*}
Therefore,
\begin{align*}
 {{\bf m}\choose{\bf n}} &= \left\{\begin{array}{cl}\displaystyle{m_1\choose n_1} & \mbox{if $n_2=\ldots=n_r=0$,}\\0 & \mbox{else.}\end{array}\right.
\end{align*}
For the generalized Laguerre functions we thus obtain
\begin{align*}
 \ell_{\bf m}^\lambda(x) &= e^{-\tr(x)}\cdot\left(\frac{d}{2}\right)_{m_1}\sum_{n_1=0}^{m_1}{{m_1\choose n_1}\frac{1}{(\frac{d}{2})_{n_1}}\Phi_{(n_1,0,\ldots,0)}(-2x)}.
\end{align*}
Now suppose $x=ktc_1\in\calO$ with $k\in K_L$ and $t>0$. Since $\ell_{\bf m}^\lambda$ is $K_L$-invariant, it only depends on $t=|x|$ and we obtain
\begin{align}
 \ell_{\bf m}^\lambda(x) &= e^{-|x|}\cdot\left(\frac{d}{2}\right)_{m_1}\sum_{n_1=0}^{m_1}{{m_1\choose n_1}\frac{1}{(\frac{d}{2})_{n_1}}(-2|x|)^{n_1}\Phi_{(n_1,0,\ldots,0)}(c_1)}.\label{eq:GenLaguerreOnO1}
\end{align}
To calculate $\Phi_{(n_1,0,\ldots,0)}(c_1)$ we use the following expansion (see \cite[Section XI.5]{FK94}):
\begin{align*}
 (\tr(y))^k &= k!\sum_{|{\bf n}|=k}{\frac{d_{\bf n}}{(\frac{n}{r})_{\bf n}}\Phi_{\bf n}(y)},
\end{align*}
where $d_{\bf m}$\index{notation}{dm@$d_{\bf m}$} is defined in \cite[Proposition XI.4.1~(i)]{FK94}. For $y=c_1$ we have $\tr(c_1)=1$ and $\Phi_{\bf n}(c_1)=0$ if one of the $n_2,\ldots,n_r$ is non-zero. Therefore
\begin{align*}
 \Phi_{(n_1,\ldots,0)}(c_1) &= \frac{(\frac{n}{r})_{n_1}}{n_1!\,d_{(n_1,0,\ldots,0)}}.
\end{align*}
For $d_{(n_1,0,\ldots,0)}$ one obtains, using the results of \cite[Section XIV.5]{FK94}:
\begin{align*}
 d_{(n_1,0,\ldots,0)} &= \frac{(\frac{n}{r})_{n_1}(\frac{rd}{2})_{n_1}}{n_1!\,(\frac{d}{2})_{n_1}}.
\end{align*}
Inserting all this into \eqref{eq:GenLaguerreOnO1} finally yields
\begin{align*}
 \ell_{\bf m}^\lambda(x) &= \frac{m_1!\,(\frac{d}{2})_{m_1}}{(\frac{rd}{2})_{m_1}}\cdot e^{-|x|}L_{m_1}^\mu(2|x|),
\end{align*}
where $L_n^\alpha(z)$ denote the classical Laguerre polynomials as defined in Appendix \ref{app:Laguerre}. In view of Corollary \ref{cor:SpecialValue}~(2) we have
\begin{align*}
 \ell_{(j,0,\ldots,0)}^\lambda(x) &= \const\cdot e^{-|x|}L_j^\mu(2|x|) = \const\cdot\Lambda_{2,j}^{\mu,\nu}(|x|) = \const\cdot\psi_j(x)& \forall\, x\in\calO.
\end{align*}
\end{remark}

\begin{example}
For the metaplectic representation $\mu$ of $\Mp(n,\RR)$ on $L^2_{\textup{even}}(\RR^n)$ as introduced in Section \ref{sec:ExMetRep} the $K_L$-invariant vectors in each $\frakk$-type $W^j$ are spanned by
\begin{align*}
 \calU\psi_j(y) &= \Lambda_j^{\mu,-1}(|yy^t|) = \const\cdot e^{-|y|^2}L_j^\mu(2|y|^2).
\end{align*}
\end{example}

\subsubsection{The unitary inversion operator}

In Section \ref{sec:UnitInvOp} we introduced the unitary inversion operator $\calF_\calO$. On radial functions it acts by the $G$-transform $\calT^{\mu,\nu}$ (see Theorem \ref{thm:RadialUnitaryInversion}). For the special case $(\mu,\nu)\in\Xi$ the results of Proposition \ref{prop:GTrafoProperties} can also be obtained using representation theory. In fact, $\calF_\calO$ is (up to a scalar) given by the action of the element $\check{w_0}$.
\begin{enumerate}
\item[\textup{(1)}] Since $\pi$ is a unitary representation, $\calF_\calO$ is unitary on $L^2(\calO,\td\mu)$. The operator $\calT^{\mu,\nu}$, being the radial part of $\calF_\calO$, also has to be unitary on the subspace
\begin{align*}
 L^2(\calO,\td\mu)_{\textup{\rad}}\cong L^2(\RR_+,x^{\mu+\nu+1}\td x)
\end{align*}
of radial functions. This proves part (1) of Proposition \ref{prop:GTrafoProperties}.
\item[\textup{(2)}] For part (2) observe that $\check{w_0}$ is central in $\check{K}$. Therefore, it particularly commutes with the Casimir element $C_\frakk$ of $\frakk$ as introduced in Section \ref{sec:CasimirAction}. It follows that the actions of $\check{w_0}$ and $C_\frakk$ have to commute as well. The action of $\check{w_0}$ on radial functions gives $\calT^{\mu,\nu}$ (up to a scalar) and the action of the Casimir $C_\frakk$ is on radial functions (up to scalars) given by the differential operator $\calD_{\mu,\nu}$ (see Theorem \ref{thm:CasimirAction}). Hence, $\calT^{\mu,\nu}$ and $\calD_{\mu,\nu}$ commute.
\end{enumerate}
This proves Proposition \ref{prop:GTrafoProperties} for the case where $(\mu,\nu)\in\Xi$ using representation theory.

In Section \ref{sec:UnitInvOp} we further showed that the unitary inversion operator $\calF_\calO$ acts as a scalar on each $\frakk$-type $W^j$. For the minimal $\frakk$-type $W^0$ we showed by direct computation that this scalar is $1$. Using the results of Section \ref{sec:GTrafo} we can now give the action on all $\frakk$-types.

\begin{corollary}\label{cor:ActionFkTypes}
The unitary inversion operator $\calF_\calO$ acts on the $\frakk$-type $W^j$ by the scalar $(-1)^j$. In particular, $\calF_\calO$ is of order $2$.
\end{corollary}

\begin{proof}
By Theorem \ref{thm:RadialUnitaryInversion} the operator $\calF_\calO$ acts on radial functions by the $G$-transform $\calT^{\mu,\nu}$. The $G$-transform $\calT^{\mu,\nu}$ acts on $\Lambda_{2,j}^{\mu,\nu}(x)$ by the scalar $(-1)^j$ (see Theorem \ref{thm:EigFctGTrafo}). Hence, $\calF_\calO$ acts on the radial function $\psi_j(x)=\Lambda_{2,j}^{\mu,\nu}(|x|)$ by the scalar $(-1)^j$. By Theorem \ref{thm:KTypes} the function $\psi_j$ is in the $\frakk$-type $W^j$. Since $\calF_\calO$ acts on $W^j$ by a scalar, this scalar has to be $(-1)^j$.
\end{proof}

\begin{example}
For the euclidean case the analogous statement for the continuous part of the Wallach set is proved in \cite[Corollary XV.4.3]{FK94}.
\end{example}

\subsubsection{Recurrence relations via the $\frakg$-action}

Finally, we can also give a representation theoretic explanation for the recurrence relations in Propositions \ref{prop:RecRelH} and \ref{prop:RecRelxSq}. For this we consider the Lie algebra action $\td\pi$. For $H:=(0,\id,0)\in\frakl$ the action is given by
\begin{align*}
 \td\pi(H) &= E+\frac{r_0d}{4},
\end{align*}
where $E=\sum_{i=1}^n{x_i\frac{\partial}{\partial x_i}}$ denotes the Euler operator on $\calO$. Hence, $H=(0,\id,0)$ leaves the space $L^2(\calO)_{\textup{rad}}$ of radial function invariant and acts on it by the differential operator $\mathcal{H}_{\mu+\nu}$ (see Section \ref{sec:RecRel}). Further, let $(e_k)_k$ be any orthonormal basis of $V$ with respect to the inner product $(-|-)$ and put $N_k:=(e_k,0,0)\in\frakn$. Then
\begin{align*}
 \td\pi(N_k) &= i(x|e_k).
\end{align*}
In particular, the sum of squares
\begin{align*}
 \sum_{k=1}^n{\td\pi(N_k)^2} &= -\|x\|^2
\end{align*}
leaves the space $L^2(\calO)_{\textup{rad}}$ of radial function invariant.

The key to an understanding of the underlying algebraic structure of the recurrence relations in Propositions \ref{prop:RecRelH} and \ref{prop:RecRelxSq} is the action of $H$ and $N_k$ on the $\frakk$-types $W^j$. For convenience put $W^{-1}:=0$.

\begin{lemma}
The Lie algebra action $\td\pi(X):\bigoplus_{j=0}^\infty{W^j}\longrightarrow\bigoplus_{j=0}^\infty{W^j}$ ($X\in\mathfrak{g}$) induces the following linear maps for each $j\in\mathbb{N}_0$:
\begin{align*}
 \td\pi(H):W^j &\longrightarrow W^{j+1}\oplus W^{j-1},\\
 \td\pi(\overline{N}_k):W^j &\longrightarrow W^{j+1}\oplus W^j\oplus W^{j-1}, & 1\leq k\leq n.
\end{align*}
\end{lemma}

\begin{proof}
We have
\begin{align}
 \td\pi(X):W^j &\longrightarrow W^j & \forall\,X\in\mathfrak{k}\label{eq:MapPropK}
\end{align}
since $W^j$ is a $\frakk$-module. For the action of $\mathfrak{p}$ recall that the $\frakk$-weights of $\frakp$ are by Section \ref{subsec:ConfGrpRoots} contained in
\begin{align*}
 \left\{\frac{\pm\gamma_i\pm\gamma_j}{2}:1\leq i,j\leq r_0\right\}.
\end{align*}
To determine which $\frakk$-types may appear in $\td\pi(\frakp)W^j$ one has to add the weights of $\frakp$ to the highest weight of $W^j$. By the $\frakk$-type decomposition (see Theorem \ref{thm:KtypeDecomp}) the only possible $\frakk$-types that also appear in $W$ are $W^{j-1}$ and $W^{j+1}$. Hence, we have
\begin{align}
 \td\pi(X):W^j &\longrightarrow W^{j+1}\oplus W^{j-1} & \forall\,X\in\mathfrak{p}\label{eq:MapPropP}
\end{align}
Putting \eqref{eq:MapPropK} and \eqref{eq:MapPropP} together proves the claim since $H\in\mathfrak{a}\subseteq\mathfrak{p}$.
\end{proof}

By our previous considerations both $\td\pi(H)$ and $\sum_{i=1}^n{\td\pi(N_i)^2}$ leave $L^2(\calO)_{\textup{rad}}$ invariant. Since $W^j_{\textup{rad}}=W^j\cap L^2(\calO)_{\textup{rad}}$ is one-dimensional and spanned by $\psi_j(x)=\Lambda_{2,j}^{\mu,\nu}(|x|)$ for every $j\in\mathbb{N}_0$, we obtain
\begin{align*}
 \mathcal{H}_{\mu+\nu}\Lambda_{2,j}^{\mu,\nu} &\in\textup{span}\{\Lambda_{2,k}^{\mu,\nu}: k=j-1,j+1\},\\
 x^2\Lambda_{2,j}^{\mu,\nu} &\in\textup{span}\{\Lambda_{2,k}^{\mu,\nu}:k=j-2,j-1,j,j+1,j+2\},
\end{align*}
which can be viewed as a qualitative version of Propositions \ref{prop:RecRelH} and \ref{prop:RecRelxSq}.

\appendix

\chapter{Tables of simple real Jordan algebras}

\section{Structure constants of $V$}

\begin{table}[ht]
\begin{center}
\begin{tabular}{|r|c|c|c|c|c|c|c|c|c|}
  \hline
  & $V$ & $n$ & $r$ & $d$ & $e$\\
  \hline
  I$.1$ & $\Sym(n,\RR)$ & $\frac{n}{2}(n+1)$ & $n$ & $1$ & $0$\\
  I$.2$ & $\Herm(n,\CC)$ & $n^2$ & $n$ & $2$ & $0$\\
  I$.3$ & $\Herm(n,\HH)$ & $n(2n-1)$ & $n$ & $4$ & $0$\\
  I$.4$ & $\RR^{1,n-1}$ ($n\geq3$) & $n$ & $2$ & $n-2$ & $0$\\
  I$.5$ & $\Herm(3,\OO)$ & $27$ & $3$ & $8$ & $0$\\
  \hline
  II$.1$ & $\times\times\times$ & $\times$ & $\times$ & $\times$ & $\times$\\
  II$.2$ & $M(n,\RR)$ & $n^2$ & $n$ & $2$ & $0$\\
  II$.3$ & $\Skew(2n,\RR)$ & $n(2n-1)$ & $n$ & $4$ & $0$\\
  II$.4$ & $\RR^{p,q}$ ($p,q\geq2$) & $p+q$ & $2$ & $p+q-2$ & $0$\\
  II$.5$ & $\Herm(3,\OO_s)$ & $27$ & $3$ & $8$ & $0$\\
  \hline
  III$.1$ & $\Sym(n,\CC)$ & $n(n+1)$ & $2n$ & $2$ & $1$\\
  III$.2$ & $M(n,\CC)$ & $2n^2$ & $2n$ & $4$ & $1$\\
  III$.3$ & $\Skew(2n,\CC)$ & $2n(2n-1)$ & $2n$& $8$ & $1$\\
  III$.4$ & $\CC^n$ ($n\geq3$) & $2n$ & $4$ & $2(n-2)$ & $1$\\
  III$.5$ & $\Herm(3,\OO)_\CC$ & $54$ & $6$ & $16$ & $1$\\
  \hline
  IV$.1$ & $\Sym(2n,\CC)\cap M(n,\HH)$ & $n(2n+1)$ & $2n$ & $4$ & $2$\\
  IV$.2$ & $M(n,\HH)$ & $4n^2$ & $2n$ & $8$ & $3$\\
  IV$.3$ & $\times\times\times$ & $\times$ & $\times$ & $\times$ & $\times$\\
  IV$.4$ & $\RR^{n,0}$ ($n\geq2)$ & $n$ & $2$ & $0$ & $n-1$\\
  IV$.5$ & $\times\times\times$ & $\times$ & $\times$ & $\times$ & $\times$\\
  \hline
\end{tabular}
\caption{Structure constants of $V$}
\label{tb:ConstantsV}
\end{center}
\end{table}

\newpage

\section{Structure constants of $V^+$}

\begin{table}[ht]
\begin{center}
\begin{tabular}{|r|c|c|c|c|c|c|c|c|c|}
  \hline
  & $V$ & $V^+$ & $n_0$ & $r_0$ & $d_0$\\
  \hline
  I$.1$ & $\Sym(n,\RR)$ & I.$1$ & $\frac{n}{2}(n+1)$ & $n$ & $1$\\
  I$.2$ & $\Herm(n,\CC)$ & I.$2$ & $n^2$ & $n$ & $2$\\
  I$.3$ & $\Herm(n,\HH)$ & I.$3$ & $n(2n-1)$ & $n$ & $4$ \\
  I$.4$ & $\RR^{1,n-1}$ ($n\geq3$) & I.$4$ & $n$ & $2$ & $n-2$\\
  I$.5$ & $\Herm(3,\OO)$ & I.$5$ & $27$ & $3$ & $8$\\
  \hline
  II$.1$ & $\times\times\times$ & $\times$ & $\times$ & $\times$ & $\times$\\
  II$.2$ & $M(n,\RR)$ & I.$1$ & $\frac{n}{2}(n+1)$ & $n$ & $1$\\
  II$.3$ & $\Skew(2n,\RR)$ & I.$2$ & $n^2$ & $n$ & $2$\\
  II$.4$ & $\RR^{p,q}$ ($p,q\geq2$) & I.$4$ & $q+1$ & $2$ & $q-1$\\
  II$.5$ & $\Herm(3,\OO_s)$ & I.$3$ & $15$ & $3$ & $4$\\
  \hline
  III$.1$ & $\Sym(n,\CC)$ & I.$1$ & $\frac{n}{2}(n+1)$ & $n$ & $1$\\
  III$.2$ & $M(n,\CC)$ & I.$2$ & $n^2$ & $n$ & $2$\\
  III$.3$ & $\Skew(2n,\CC)$ & I.$3$ & $n(2n-1)$ & $n$ & $4$\\
  III$.4$ & $\CC^n$ ($n\geq3$) & I.$4$ & $n$ & $2$ & $n-2$\\
  III$.5$ & $\Herm(3,\OO)_\CC$ & I.$5$ & $27$ & $3$ & $8$\\
  \hline
  IV$.1$ & $\Sym(2n,\CC)\cap M(n,\HH)$ & I.$2$ & $n^2$ & $n$ & $2$\\
  IV$.2$ & $M(n,\HH)$ & I.$3$ & $n(2n-1)$ & $n$ & $4$\\
  IV$.3$ & $\times\times\times$ & $\times$ & $\times$ & $\times$ & $\times$\\
  IV$.4$ & $\RR^{n,0}$ ($n\geq2)$ & I.$4$ & $1$ & $1$ & $0$\\
  IV$.5$ & $\times\times\times$ & $\times$ & $\times$ & $\times$ & $\times$\\
  \hline
\end{tabular}
\caption{Structure constants of $V^+$}
\label{tb:ConstantsVplus}
\end{center}
\end{table}

\newpage

\section{The constants $\mu$ and $\nu$}

\begin{table}[ht]
\begin{center}
\begin{tabular}{|r|c|c|c|}
  \hline
  & $V$ & $\mu$ & $\nu$\\
  \hline
  I$.1$ & $\Sym(n,\RR)$ & $(n-2)/2$ & $-1$\\
  I$.2$ & $\Herm(n,\CC)$ & $n-1$ & $-1$\\
  I$.3$ & $\Herm(n,\HH)$ & $2n-1$ & $-1$\\
  I$.4$ & $\RR^{1,n-1}$ ($n\geq3$) & $n-3$ & $-1$\\
  I$.5$ & $\Herm(3,\OO)$ & $11$ & $-1$\\
  \hline
  II$.1$ & $\times\times\times$ & $\times$ & $\times$\\
  II$.2$ & $M(n,\RR)$ & $n-2$ & $0$\\
  II$.3$ & $\Skew(2n,\RR)$ & $2n-3$ & $1$\\
  II$.4$ & $\RR^{p,q}$ ($p,q\geq2$) & $\max(p,q)-2$ & $\min(p,q)-2$\\
  II$.5$ & $\Herm(3,\OO_s)$ & $7$ & $3$\\
  \hline
  III$.1$ & $\Sym(n,\CC)$ & $n-1$ & $-1$\\
  III$.2$ & $M(n,\CC)$ & $2(n-1)$ & $0$\\
  III$.3$ & $\Skew(2n,\CC)$ & $2(2n-2)$ & $2$\\
  III$.4$ & $\CC^n$ ($n\geq3$) & $n-2$ & $n-4$\\
  III$.5$ & $\Herm(3,\OO)_\CC$ & $16$ & $6$\\
  \hline
  IV$.1$ & $\Sym(2n,\CC)\cap M(n,\HH)$ & $2n-1$ & $-1$\\
  IV$.2$ & $M(n,\HH)$ & $4n-2$ & $0$\\
  IV$.3$ & $\times\times\times$ & $\times$ & $\times$\\
  IV$.4$ & $\RR^{n,0}$ ($n\geq2$) & $n-2$ & $-n$\\
  IV$.5$ & $\times\times\times$ & $\times$ & $\times$\\
  \hline
\end{tabular}
\caption{The constants $\mu$ and $\nu$}
\label{tb:Parameters}
\end{center}
\end{table}

\newpage

\section{Conformal algebra and structure algebra}

\begin{table}[ht]
\begin{center}
\begin{tabular}{|r|c|c|c|}
  \hline
  & $V$ & $\frakg=\co(V)$ & $\frakl=\str(V)$\\
  \hline
  I$.1$ & $\Sym(n,\RR)$ & $\sp(n,\RR)$ & $\sl(n,\RR)\times\RR$\\
  I$.2$ & $\Herm(n,\CC)$ & $\su(n,n)$ & $\sl(n,\CC)\times\RR$\\
  I$.3$ & $\Herm(n,\HH)$ & $\so^*(4n)$ & $\sl(n,\HH)\times\RR$\\
  I$.4$ & $\RR^{1,n-1}$ ($n\geq3$) & $\so(2,n)$ & $\so(1,n-1)\times\RR$\\
  I$.5$ & $\Herm(3,\OO)$ & $\mathfrak{e}_{7(-25)}$ & $\mathfrak{e}_{6(-26)}\times\RR$\\
  \hline
  II$.1$ & $\times\times\times$ & $\times\times\times$ & $\times\times\times$\\
  II$.2$ & $M(n,\RR)$ & $\sl(2n,\RR)$ & $\sl(n,\RR)\times\sl(n,\RR)\times\RR$\\
  II$.3$ & $\Skew(2n,\RR)$ & $\so(2n,2n)$ & $\sl(2n,\RR)\times\RR$\\
  II$.4$ & $\RR^{p,q}$ ($p,q\geq2$) & $\so(p+1,q+1)$ & $\so(p,q)\times\RR$\\
  II$.5$ & $\Herm(3,\OO_s)$ & $\mathfrak{e}_{7(7)}$ & $\mathfrak{e}_{6(6)}\times\RR$\\
  \hline
  III$.1$ & $\Sym(n,\CC)$ & $\sp(n,\CC)$ & $\sl(n,\CC)\times\CC$\\
  III$.2$ & $M(n,\CC)$ & $\sl(2n,\CC)$ & $\sl(n,\CC)\times\sl(n,\CC)\times\CC$\\
  III$.3$ & $\Skew(2n,\CC)$ & $\so(4n,\CC)$ & $\sl(2n,\CC)\times\CC$\\
  III$.4$ & $\CC^n$ ($n\geq3$) & $\so(n+2,\CC)$ & $\so(n,\CC)\times\CC$\\
  III$.5$ & $\Herm(3,\OO)_\CC$ & $\mathfrak{e}_7(\CC)$ & $\mathfrak{e}_6(\CC)\times\CC$\\
  \hline
  IV$.1$ & $\Sym(2n,\CC)\cap M(n,\HH)$ & $\sp(n,n)$ & $\sl(n,\HH)\times\HH$\\
  IV$.2$ & $M(n,\HH)$ & $\sl(2n,\HH)$ & $\sl(n,\HH)\times\sl(n,\HH)\times\HH$\\
  IV$.3$ & $\times\times\times$ & $\times\times\times$ & $\times\times\times$\\
  IV$.4$ & $\RR^{n,0}$ ($n\geq2$) & $\so(1,n+1)$ & $\so(n)\times\RR$\\
  IV$.5$ & $\times\times\times$ & $\times\times\times$ & $\times\times\times$\\
  \hline
\end{tabular}
\caption{Conformal algebra and structure algebra}
\label{tb:Groups}
\end{center}
\end{table}
\chapter{Calculations in rank $2$}\label{app:Rank2}

Let $V=\RR^{p,q}$ with $p,q\geq2$ (see Example \ref{ex:JordanAlgebras}~(2)). The structure constants of $V$ and $V^+$ can be found in Tables \ref{tb:ConstantsV} and \ref{tb:ConstantsVplus} where $V$ corresponds to the case II.4. In particular, $V$ has dimension $n=p+q$ and rank $r=r_0=2$. The parameter $\lambda$ of the zeta function and the Bessel operator corresponding to the minimal orbit of $\Str(V)_0$ is $\lambda=\frac{r_0d}{2r}=\frac{p+q-2}{2}$. Let us for convenience assume that $p\leq q$. (The case $p\geq q$ can be treated similarly.) Then $\mu=q-2$ and $\nu=p-2$.

Denote by $(e_j)_{j=1,\ldots,n}$ the standard basis of $V=\RR^n$. We use coordinates $x_1,\ldots,x_n$ for $x=\sum_{j=1}^n{x_je_j}$. In these coordinates the trace form and the determinant are given by
\begin{align*}
 \tau(x,y) &= 2(x_1y_1-x_2y_2-\ldots-x_py_p+x_{p+1}y_{p+1}+\ldots+x_ny_n),\\
 \Delta(x) &= x_1^2+\ldots+x_p^2-x_{p+1}^2-\ldots-x_n^2.
\end{align*}
Hence, $\alpha e_j:=\epsilon_je_j$ defines a Cartan involution of $V$, where
\begin{align*}
 \epsilon_j := \begin{cases}+1 & \mbox{for $j=1$ or $p+1\leq j\leq n$,}\\-1 & \mbox{for $2\leq j\leq p$.}\end{cases}
\end{align*}
The basis dual to $(e_j)$ with respect to the trace form $\tau$ is therefore given by $\overline{e}_j=\frac{1}{2}\epsilon_je_j$. The corresponding inner product and norm are
\begin{align*}
 (x|y) &= \tau(x,\alpha y) = 2(x_1y_1+\ldots+x_ny_n),\\
 |x|^2 &= \|x\|^2 = 2(x_1^2+\ldots+x_n^2).
\end{align*}
An orthonormal basis of $V$ with respect to the inner product $(-|-)$ is hence given by the vectors $\frac{1}{\sqrt{2}}e_j$, $1\leq j\leq n$. We also fix the Jordan frame $c_1:=\frac{1}{2}(e_1+e_{p+q})$, $c_2:=\frac{1}{2}(e_1-e_{p+q})$.

The gradient $\frac{\partial}{\partial x}$ with respect to the trace form writes
\begin{align*}
 \frac{\partial}{\partial x} &= \left(\frac{\epsilon_j}{2}\frac{\partial}{\partial x_j}\right)_j = \frac{1}{2}\left(\frac{\partial}{\partial x_1},-\frac{\partial}{\partial x_2},\ldots,-\frac{\partial}{\partial x_p},\frac{\partial}{\partial x_{p+1}},\ldots,\frac{\partial}{\partial x_n}\right).
\end{align*}

Let us use the notation $x=(x',x'')\in\RR^p\times\RR^q$ for $x\in V$. Abusing notation, we also write $x'$ and $x''$ for the vectors $(x',0)\in V$ and $(0,x'')\in V$, respectively. In this notation the minimal orbit $\calO=\calO_1$ can be written as
\begin{align*}
 \calO &= \{x\in V\setminus\{0\}: \Delta(x)=0\}\\
 &= \{(x',x'')\in V\setminus\{0\}: |x'|=|x''|\}.
\end{align*}
In particular, for $x=(x',x'')\in\calO$ we have $|x'|^2=|x''|^2=\frac{1}{2}|x|^2$.

Now let us calculate the action of the Bessel operator $\calB_\lambda$ ($\lambda=\lambda_1$) in this case. To apply Proposition \ref{prop:BnuRadial} we have to calculate $ke$ for $k\in(K_L)_0=\SO(p)\times\SO(q)$. One finds that for $x=ktc_1\in\calO$ with $t>0$, $k\in(K_L)_0$ we have $kte=(2x',0)$. Hence we obtain for $\psi(x)=f(|x|)$:
\begin{align*}
 \calB_\lambda\psi(x) &= \left(f''(|x|)+d_0\frac{1}{|x|}f'(|x|)\right)\alpha(x')+\left(f''(|x|)+(d-d_0)\frac{1}{|x|}f'(|x|)\right)\alpha(x'').
\end{align*}
For the action of the Lie algebra elements $(e_j,0,-\alpha e_j)\in\frakk$ this yields
\begin{multline}
 \td\pi(e_j,0,-\alpha e_j)\psi(x)\\
 = \begin{cases}\frac{1}{i}\left(f''(|x|)+(q-1)\frac{1}{|x|}f'(|x|)-f(|x|)\right)(x|e_j) & \mbox{for $1\leq j\leq p$,}\\\frac{1}{i}\left(f''(|x|)+(p-1)\frac{1}{|x|}f'(|x|)-f(|x|)\right)(x|e_j) & \mbox{for $p+1\leq j\leq n$.}\end{cases}\label{eq:kActionRank2}
\end{multline}

\section{The minimal $K$-type}\label{app:Rank2MinKtype}

In this section we prove the remaining parts (c) and (d) of Proposition \ref{prop:Kfinite} and calculate the action of $\calF_\calO$ on $\psi_0$.

To prove \eqref{eq:MinKtypeRk2finite} and \eqref{eq:MinKtypeRk2infinite} we need to calculate the $\frakk$-action on the spaces $\widetilde{K}_{\frac{\nu}{2}+k}\otimes\calH^k(\RR^p)$, $0\leq k\leq\frac{q-p}{2}$. For this we introduce operators $(-)_j^\pm$ on $\calH^k(\RR^p)$ for $j=1,\ldots,p$ by:\index{notation}{1bracketxpmj@$(-)^\pm_j$}
\begin{align*}
 (-)^+_j:\calH^k(\RR^p) &\rightarrow \calH^{k+1}(\RR^p), & \varphi^+_j(x) &:= x_j\varphi(x)-\frac{x_1^2+\ldots+x_p^2}{p+2k-2}\frac{\partial\varphi}{\partial x_j}(x),\\
 (-)^-_j:\calH^k(\RR^p) &\rightarrow \calH^{k-1}(\RR^p), & \varphi^-_j(x) &:= \frac{1}{p+2k-2}\frac{\partial\varphi}{\partial x_j}(x).
\end{align*}
That $\varphi^+_j$ and $\varphi^-_j$ are (for $\varphi\in\calH^k(\RR^p)$) indeed homogeneous harmonic polynomials of degree $k+1$ and $k-1$, respectively, can easily be checked by direct computation. Then for $\varphi\in\calH^k(\RR^p)$ one clearly has the following decomposition of $x_j\varphi(x)$ into spherical harmonics:
\begin{align*}
 x_j\varphi(x) &= \varphi^+_j(x)+(x_1^2+\ldots+x_p^2)\varphi_j^-(x).
\end{align*}
For convenience we also put $(-)^+_j:=(-)^-_j:=0$ for $j=p+1,\ldots,n$. Using the operators $(-)^+_j$ and $(-)^-_j$ we prove the following lemma:

\begin{lemma}\label{lem:Rank2KAction}
For $j=1,\ldots,n$ the action of $(e_j,0,-\alpha e_j)\in\frakk$ on $\widetilde{K}_{\frac{\nu}{2}+k}\otimes\varphi\in\widetilde{K}_{\frac{\nu}{2}+k}\otimes\calH^k(\RR^p)$ is given by
\begin{multline*}
 \td\pi(e_j,0,-\alpha e_j)(\widetilde{K}_{\frac{\nu}{2}+k}\otimes\varphi)\\
 = \frac{1}{i}\left[(2k+p-q)\widetilde{K}_{\frac{\nu}{2}+k+1}\otimes\varphi_j^+-(2k+p+q-4)\widetilde{K}_{\frac{\nu}{2}+k-1}\otimes\varphi_j^-\right].
\end{multline*}
\end{lemma}

\begin{proof}
Let $\varphi\in\calH^k(\RR^p)$ and $1\leq j\leq n$. With \eqref{eq:KActionInTermsOfBesselOp}, the product rule \eqref{eq:BnuProdRule} and \eqref{eq:ddxradial} we obtain
\begin{multline*}
 \td\pi(e_j,0,-\alpha e_j)(\widetilde{K}_{\frac{\nu}{2}+k}\otimes\varphi)(x) = \td\pi(e_j,0,-\alpha e_j)\widetilde{K}_{\frac{\nu}{2}+k}(x)\cdot\varphi(x)\\
 +\frac{2}{i}\frac{\widetilde{K}_{\frac{\nu}{2}+k}'(|x|)}{|x|}\tau\left(P\left(\alpha x,\frac{\partial\varphi}{\partial x}\right)x,e_j\right)+\frac{1}{i}\widetilde{K}_{\frac{\nu}{2}+k}(|x|)\tau(\calB_\lambda\varphi(x),e_j).
\end{multline*}
Using the two identities
\begin{align*}
 \sum_{j=1}^p{x_j\frac{\partial\varphi}{\partial x_j}(x)} &= k\varphi(x) & \mbox{and} && \sum_{j=1}^p{\frac{\partial^2\varphi}{\partial x_j^2}}=0,
\end{align*}
which hold since $\varphi\in\calH^k(\RR^p)$, a short calculation gives
\begin{align*}
 P\left(\alpha x,\frac{\partial\varphi}{\partial x}\right)x &= k\varphi(x')x''+\frac{1}{2}|x|^2\frac{\partial\varphi}{\partial x}(x'),\\
 \calB_\lambda\varphi &= (\lambda+k-1)\frac{\partial\varphi}{\partial x}.
\end{align*}
\begin{enumerate}
\item[\textup{(a)}] $1\leq j\leq p$. With \eqref{eq:kActionRank2} and the differential equation \eqref{eq:DiffEqModBessel} we obtain
\begin{multline*}
 \td\pi(e_j,0,-\alpha e_j)(\widetilde{K}_{\frac{\nu}{2}+k}\otimes\varphi)(x) = \frac{1}{i}\left[-2(2k+p-q)\frac{1}{|x|}\widetilde{K}_{\frac{\nu}{2}+k}'(|x|)\cdot x_j\varphi(x)\right.\\
 \left.+ |x|\widetilde{K}_{\frac{\nu}{2}+k}'(|x|)\cdot\frac{\partial\varphi}{\partial x_j}(x) + \left(k+\frac{p+q-4}{2}\right)\widetilde{K}_{\frac{\nu}{2}+k}(|x|)\cdot\frac{\partial\varphi}{\partial x_j}(x)\right].
\end{multline*}
Applying \eqref{eq:BesselDiffFormulas} and \eqref{eq:KBesselRecRel} yields the stated formula.
\item[\textup{(b)}] $p+1\leq j\leq n$. In this case $\tau(\frac{\partial\varphi}{\partial x},e_j)=0$ and $\tau(x'',e_j)=(x|e_j)$. Using \eqref{eq:kActionRank2} again we find that
\begin{multline*}
 \td\pi(e_j,0,-\alpha e_j)(\widetilde{K}_{\frac{\nu}{2}+k}\otimes\varphi)(x)\\
 = \frac{1}{i}\left[\widetilde{K}_{\frac{\nu}{2}+k}''(|x|)+(2k+p-1)\frac{1}{|x|}\widetilde{K}_{\frac{\nu}{2}+k}'(|x|)-\widetilde{K}_{\frac{\nu}{2}+k}(|x|)\right](x|e_j)\varphi(x).
\end{multline*}
But this is $=0$ by \eqref{eq:DiffEqModBessel} which finishes the proof.\qedhere
\end{enumerate}
\end{proof}

Now
\begin{align*}
 \frakk = \frakk_\frakl \oplus \{(u,0,-\alpha u):u\in V\}
\end{align*}
and $\frakk_\frakl=\so(p)\oplus\so(q)$ acts irreducibly on $\calH^k(\RR^p)$ for every $k\geq0$. Therefore, the previous lemma implies that
\begin{align*}
 W_0 &= \begin{cases}\displaystyle\bigoplus_{k=0}^{\frac{q-p}{2}}{\widetilde{K}_{\frac{\nu}{2}+k}\otimes\calH^k(\RR^p)} & \mbox{if $q-p\in2\ZZ$,}\\\displaystyle\bigoplus_{k=0}^\infty{\widetilde{K}_{\frac{\nu}{2}+k}\otimes\calH^k(\RR^p)} & \mbox{else,}\end{cases}
\end{align*}
which proves parts (c) and (d) of Proposition \ref{prop:Kfinite}. It remains to calculate the highest weight of $W_0$.

The isomorphism $\frakg\stackrel{\sim}{\rightarrow}\so(p+1,q+1)$ given in Example \ref{ex:ConfGrp}~(2) restricts to an isomorphism $\frakk\stackrel{\sim}{\rightarrow}\so(p+1)\oplus\so(q+1)\subseteq\so(p+1,q+1)$ given by
\begin{align*}
 (0,T,0) &\mapsto \left(\begin{array}{c|c|c}&&\\\hline&T&\\\hline&&\end{array}\right), && T\in\frakk_\frakl=\so(p)\oplus\so(q),\\
 (u,0,-\alpha u) &\mapsto \left(\begin{array}{c|cc|c}&2(u')^t&&\\\hline-2u'&&&\\&&&-2u''\\\hline&&2(u'')^t&\end{array}\right), && u\in V.
\end{align*}

\begin{lemma}\label{lem:IdentLowestKtypeRk2}
Under the identification $\frakk\cong\so(p+1)\oplus\so(q+1)$ the map
\begin{align*}
 \Phi:\bigoplus_{k=0}^{\frac{q-p}{2}}{\widetilde{K}_{\frac{\nu}{2}+k}\otimes\calH^k(\RR^p)}&\rightarrow\calH^{\frac{q-p}{2}}(\RR^{p+1}),\\
 \Phi\left(\widetilde{K}_{\frac{\nu}{2}+k}\otimes\varphi\right)(x_0,x')&:=\frac{(-2i)^k}{(\frac{p-q}{2})_k}\widetilde{C}_{\frac{q-p}{2}-k}^{\frac{p-1}{2}+k}(x_0)\varphi(x'), & (x_0,x')\in\SS^p\subseteq\RR^{p+1}.
\end{align*}
becomes an isomorphism of $\frakk$-modules.
\end{lemma}

Here $\widetilde{C}_n^\lambda(z)$ denote the normalized Gegenbauer polynomials as defined in Appendix \ref{app:Gegenbauer} and $(a)_k=a(a+1)\cdots(a+k-1)$ is the Pochhammer symbol.

\begin{proof}
Since $\so(q+1)$ acts trivially on both sides, we only have to check the $\so(p+1)$-action. The action of $A\in\so(p+1)$ on $\psi\in\calH^{\frac{q-p}{2}}(\RR^{p+1})$ is given by
\begin{align*}
 (A\cdot\psi)(x) &:= D_{A^*x}\psi(x) = -D_{Ax}\psi(x).
\end{align*}
Then it is clear that $\Phi$ intertwines the actions of $\so(p)\subseteq\frakk_\frakl$. It remains to check that it also intertwines the actions of $(e_j,0,-\alpha e_j)$ for $1\leq j\leq p$. To prove this, we make use of the two formulas \eqref{eq:Gegenbauer1} and \eqref{eq:Gegenbauer2} for the normalized Gegenbauer polynomials. We then have for $(x_0,x')\in\SS^p$:
\setmyalign{\eqref{eq:Gegenbauer1}}
\begin{align*}
 \al\ \left((e_j,0,-\alpha e_j)\cdot\Phi(\widetilde{K}_{\frac{\nu}{2}+k}\otimes\varphi)\right)(x_0,x')\\
 \al= \left(-2x_j\frac{\partial}{\partial x_0}+2x_0\frac{\partial}{\partial x_j}\right)\Phi(\widetilde{K}_{\frac{\nu}{2}+k}\otimes\varphi)(x_0,x')\\
 \al[\eqref{eq:Gegenbauer1}]= -4\frac{(-2i)^k}{(\frac{p-q}{2})_k}\widetilde{C}_{\frac{q-p}{2}-(k+1)}^{\frac{p-1}{2}+(k+1)}(x_0)x_j\varphi(x')+2\frac{(-2i)^k}{(\frac{p-q}{2})_k}x_0\widetilde{C}_{\frac{q-p}{2}-k}^{\frac{p-1}{2}+k}(x_0)\frac{\partial\varphi}{\partial x_j}(x')\\
 \al[\eqref{eq:Gegenbauer2}]= \frac{1}{i}\left[(2k+p-q)\Phi(\widetilde{K}_{\frac{\nu}{2}+(k+1)}\otimes\varphi_j^+)\right.\\
 \al\ \ \ \ \ \ \ \ \ \ \ \ \ \ \ \ \ \ \ \ \ \ \ \ \ \ \ \ \ \ \ \ \ \ \ -\left.(2k+p+q-4)\Phi(\widetilde{K}_{\frac{\nu}{2}+(k-1)}\otimes\varphi_j^-)\right](x_0,x')\\
 \al= \Phi\left((e_j,0,-\alpha e_j)\cdot(\widetilde{K}_{\frac{\nu}{2}+k}\otimes\varphi)\right)(x_0,x')
\end{align*}
by Lemma \ref{lem:Rank2KAction}.
\end{proof}

We use the identification $\frakk\cong\so(p+1)\oplus\so(q+1)$ to transfer the torus $\frakt\subseteq\frakk$, the roots $\{\pm\frac{\gamma_i\pm\gamma_j}{2}\}$ and our choice of a positive system from $\frakk$ to $\so(p+1)\oplus\so(q+1)$. Then by \cite[Chapter IV.7, Examples (1) \& (2)]{Kna86} the function $\zeta(x)=(x_0+ix_1)^{\frac{q-p}{2}}$ is a highest weight vector in $\calH^{\frac{q-p}{2}}(\RR^{p+1})$ and the corresponding highest weight is given by
\begin{align*}
 \varepsilon\left(\begin{array}{cc|ccc}0&a&&&\\-a&0&&&\\\hline&&*&&\\&&&\ddots&\\&&&&*\end{array}\right) &= \frac{q-p}{2}ia.
\end{align*}
Under the above identification $\varepsilon$ corresponds to $\frac{q-p}{4}(\gamma_1+\gamma_2)$ which is in turn the highest weight of the $\frakk$-module $W_0$. This proves the last part of Proposition \ref{prop:Kfinite}.

Finally we calculate the action of $\calF_\calO=e^{-i\pi\frac{r_0}{2}(d_0-\frac{d}{2})_+}\pi(\check{w_0})$ on the function $\psi_0$ to finish the proof of Proposition \ref{prop:AkkappaPsi0}. The missing part in the proof is the following lemma:

\begin{lemma}\label{lem:Rk2ActionFpsi0}
$e^{i\frac{\pi}{2}(e|x-\calB)}\psi_0=e^{i\pi(\frac{q-p}{2})_+}\psi_0$.
\end{lemma}

\begin{proof}
By the definition of $\td\pi$ we have
\begin{align*}
 e^{i\frac{\pi}{2}(e|x-\calB)} &= e^{\td\pi(\frac{\pi}{2}(e,0,-e))}.
\end{align*}
Under the identification $\frakk\cong\so(p+1)\times\so(q+1)$ the element $(e,0,-e)\in\frakk$ corresponds to the matrix
\begin{align*}
 \left(\begin{array}{ccccc}0&2&&&\\-2&0&&&\\&&0&&\\&&&\ddots&\\&&&&0\end{array}\right).
\end{align*}
Applying the exponential function of $\SO(p+1)\times\SO(q+1)$ to $\frac{\pi}{2}$ times this matrix gives
\begin{align}
 \left(\begin{array}{cc}-\1_2&\\&\1_{p+q}\end{array}\right).\label{eq:Rk2MatrixJ}
\end{align}
By Lemma \ref{lem:IdentLowestKtypeRk2} the function $\psi_0\in W_0$ corresponds to the function
\begin{align}
 \SS^p\rightarrow\CC,\,(x_0,x')\mapsto\widetilde{C}_{\frac{q-p}{2}}^{\frac{p-1}{2}}(x_0).\label{eq:Rk2GegenbauerFct}
\end{align}
In view of the parity formula \eqref{eq:GegenbauerParity} for the Gegenbauer polynomials we see that the matrix \eqref{eq:Rk2MatrixJ} acts on the function \eqref{eq:Rk2GegenbauerFct} by the scalar $(-1)^{\frac{q-p}{2}}=e^{i\pi\frac{q-p}{2}}$. Similarly one shows that for $p\geq q$ the scalar is $1$. Therefore, the claim follows.
\end{proof}

\section{The Casimir action}\label{sec:Rank2Casimir}

In this section we calculate the Casimir action in the rank $2$ case. This completes the proof of Theorem \ref{thm:CasimirAction}.

The action of $(e_j,0,-\alpha e_j)$ on radial functions is given in \eqref{eq:kActionRank2}. Note that $(e_j)_j$ denotes the standard basis of $\RR^n$ which is not orthonormal with respect to the inner product $(-|-)$. With the notation of \eqref{eq:OrdinaryBesselOp} we find that for $x\in\calO$:
\setmyalign{\eqref{eq:BnuProdRule}}
\begin{align*}
 \al\ -\sum_{j=1}^p{\td\pi(e_j,0,-\alpha e_j)^2\psi(x)}\\
 \al[\eqref{eq:BnuProdRule}]= \sum_{j=1}^p{\left[(B_{q-2}^2f)(|x|)(x|e_j)^2+2\tau\left(\left.P\left(\frac{\partial(B_{q-2}f)}{\partial x}\right|\frac{\partial(x|e_j)}{\partial x}\right)x,e_j\right)\right.}\\
 \al\ \ \ \ \ \ \ \ \ \ \ \ \ \ \ \ \ \ \ \ \ \ \ \ \ \ \ \ \ \ \ \ \ \ \ \ \ \ \ \ \ \ \ \ \ \ \ \ \ \ \ \ \ \ \ \ \ \ \ \ \ +\left.\vphantom{\tau\left(\left.P\left(\frac{\partial(B_{q-2}f)}{\partial x}\right|\frac{\partial(x|e_j)}{\partial x}\right)x,e_j\right)}(B_{q-2}f)(|x|)\tau(\calB_\lambda(x|e_j),e_j)\right]\\
 \al[\eqref{eq:ddxradial}]= 2(B_{q-2}^2f)(|x|)|x'|^2+\frac{2}{|x|}(B_{q-2}f)'(|x|)\sum_{j=1}^p{\tau\left(P\left(\left.\alpha x\right|\alpha e_j\right)x,e_j\right)}\\
 \al\ \ \ \ \ \ \ \ \ \ \ \ \ \ \ \ \ \ \ \ \ \ \ \ \ \ \ \ \ \ \ \ \ \ \ \ \ \ \ \ \ \ \ \ \ \ \ \ \ \ \ \ \ \ \ \ \ \ \ \ \ \ \ \ \ +\frac{d}{2}(B_{q-2}f)(|x|)\sum_{j=1}^p{(e_j|e_j)}\\
 \al= (B_{q-2}^2f)(|x|)|x|^2+\frac{2}{|x|}(B_{q-2}f)'(|x|)\sum_{j=1}^p{\tau\left(P\left(\left.\alpha x\right|\alpha e_j\right)x,e_j\right)}+pd(B_{q-2}f)(|x|)
\end{align*}
and similarly for $\sum_{j=p+1}^n{\td\pi(e_j,0,-\alpha e_j)^2\psi(x)}$. Now a direct computation shows that
\begin{align*}
 \tau(P(\alpha x,\alpha e_j)x,e_j) &= |x|^2 & \forall\, j=1,\ldots,n.
\end{align*}
Therefore, we obtain
\begin{multline*}
 -\sum_{j=1}^n{\td\pi(e_j,0,-\alpha e_j)^2\psi(x)}\\
 = |x|^2(B_{q-2}^2f)(|x|)+2p|x|(B_{q-2}f)'(|x|)+dp(B_{q-2}f)(|x|)\\
 + |x|^2(B_{p-2}^2f)(|x|)+2q|x|(B_{p-2}f)'(|x|)+dq(B_{p-2}f)(|x|),
\end{multline*}
which turns out to be equal to $2\calD_{\mu,\nu}+(q-p)(p+q-2)$. Taking into account that $(\frac{1}{\sqrt{2}}e_j)_j$ forms an orthonormal basis of $V$ with respect to $(-|-)$ we obtain with \eqref{eq:CasimirActionInTermsOfej} that
\begin{align*}
 \td\pi(C_\frakk)\psi(x) &= \frac{r}{8n}\sum_{j=1}^n{\td\pi\left(\frac{1}{\sqrt{2}}e_j,0,-\frac{1}{\sqrt{2}}\alpha e_j\right)^2\psi(x)}\\
 &= -\frac{r}{8n}\left(\calD_{\mu,\nu}+\frac{(q-p)(p+q-2)}{2}\right)
\end{align*}
which finishes the proof of Theorem \ref{thm:CasimirAction}.
\chapter{Parabolic subgroups}\label{app:Parabolics}

Let $G$ be a real reductive group of inner type (i.e. $\Ad(G)\subseteq\Int(\frakg_\CC)$) and $\frakg=\frakk+\frakp$ be a Cartan decomposition of its Lie algebra. Further, let $\fraka\subseteq\frakp$ be any (not necessarily maximal) abelian subalgebra. Assume that the set $\Sigma(\frakg,\fraka)$ is a root system. In this section we construct parabolic subgroups of $G$ just in terms of the root system $\Sigma(\frakg,\fraka)$, not involving a maximal abelian subalgebra.\\

For $\alpha\in\fraka^*$ we consider the weight space
\begin{align*}
 \frakg_\alpha := \{X\in\frakg:[H,X]=\alpha(H)X\ \forall\, H\in\fraka\}.
\end{align*}
Denote by $\Sigma(\frakg,\fraka)$ the set of all $0\neq\alpha\in\fraka^*$ such that $\frakg_\alpha\neq0$. As mentioned in the beginning, we assume that $\Sigma(\frakg,\fraka)$ is a root system. Choose a positive system $\Sigma^+(\frakg,\fraka)\subseteq\Sigma(\frakg,\fraka)$ and denote by $\Pi(\frakg,\fraka)$ the corresponding set of simple roots. For any subset $F\subseteq\Pi(\frakg,\fraka)$ of simple roots we form the Lie algebras
\begin{align*}
 \fraka_F &:= \{H\in\fraka:\alpha(H)=0\ \forall\,\alpha\in F\} \subseteq \fraka,\index{notation}{afrakF@$\fraka_F$}\\
 \frakm_F &:= \{X\in\frakg:[X,\fraka_F]=0\}\index{notation}{mfrakF@$\frakm_F$}
\intertext{with corresponding Lie groups}
 A_F &:= \exp(\fraka_F),\index{notation}{AupF@$A_F$}\\
 M_F &:= \{g\in G:\Ad(g)H=H\ \forall\, H\in\fraka_F\}.\index{notation}{MupF@$M_F$}
\end{align*}
Further put
\begin{align*}
 \Sigma_F^+(\frakg,\fraka) &:= \{\alpha\in\Sigma^+(\frakg,\fraka):\alpha|_{\fraka_F}\neq0\}\index{notation}{SigmaFplus@$\Sigma_F^+$}
\intertext{and}
 \frakn_F &:= \bigoplus_{\alpha\in\Sigma_F^+(\frakg,\fraka)}{\frakg_\alpha},\index{notation}{nfrakF@$\frakn_F$}\\
 N_F &:= \exp(\frakn_F).\index{notation}{NupF@$N_F$}
\end{align*}
Finally, we define
\begin{align*}
 P_F &:= M_FN_F.\index{notation}{PupF@$P_F$}
\end{align*}

\def\thetheorem{\thechapter.\arabic{theorem}} 

\begin{theorem}\label{thm:Parabolics}
\begin{enumerate}
\item[\textup{(1)}] For every subset $F\subseteq\Pi(\frakg,\fraka)$ the group $P_F$ is a parabolic subgroup of $G$.
\item[\textup{(2)}] Each $P_F$ has the Langlands decomposition $P_F=^\circ\!\! M_FA_FN_F$, where
\begin{align*}
 ^\circ\! M_F &= \{g\in M_F:\chi(g)=1\ \forall\,\mbox{characters }\chi:M_F\rightarrow\RR_+\},
\end{align*}
and the maps
\begin{align*}
 M_F\times N_F &\rightarrow P_F, & (m,n) &\mapsto mn,\\
 ^\circ\! M_F\times A_F\times N_F &\rightarrow P_F, & (m,a,n) &\mapsto man,
\end{align*}
are diffeomorphisms.
\item[\textup{(3)}] $G=KP_F$ and we have the following integral formula for $f\in C_c(G)$:
\begin{align*}
 \int_G{f(g)\td g} &= \int_K{\int_{^\circ\! M_F}{\int_{A_F}{\int_{N_F}{f(kman)a^{2\rho_F}\td n}\td a}\td m}\td k},
\end{align*}
where $\td g$, $\td k$, $\td m$, $\td a$ and $\td n$ denote suitably normalized Haar measures on $G$, $K$, $^\circ\! M_F$, $A_F$ and $N_F$, respectively, and
\begin{align*}
 \rho_F(H) &= \frac{1}{2}\Tr(\ad(H)|_{\frakn_F}) = \frac{1}{2}\sum_{\alpha\in\Sigma_F^+(\frakg,\fraka)}{\dim(\frakg_\alpha)\cdot\alpha}\index{notation}{rhoF@$\rho_F$}
\end{align*} 
is the half sum of all positive roots in $\Sigma_F^+(\frakg,\fraka)$.
\end{enumerate}
\end{theorem}

\begin{proof}
Choose a maximal abelian subalgebra $\widetilde{\fraka}$ of $\frakp$ that contains $\fraka$:
\begin{align*}
 \fraka\subseteq\widetilde{\fraka}\subseteq\frakp.
\end{align*}
We denote the weight space with respect to $\alpha\in\widetilde{\fraka}^*$ by $\widetilde{\frakg}_\alpha$. Then clearly
\begin{align*}
 \widetilde{\frakg}_\alpha &\subseteq \frakg_{\widehat{\alpha}},
\end{align*}
where $\widehat{\alpha}=\alpha|_{\fraka}$. Since $\frakg=\bigoplus_{\alpha\in\fraka^*}{\frakg_\alpha}=\bigoplus_{\alpha\in\widetilde{\fraka}^*}{\widetilde{\frakg}_\alpha}$, we conclude that
\begin{align*}
 \Phi:\Sigma(\frakg,\widetilde{\fraka})\cup\{0\} &\rightarrow \Sigma(\frakg,\fraka)\cup\{0\},\,\alpha\mapsto\widehat{\alpha}=\alpha|_\fraka,
\end{align*}
is defined and surjective. One can choose a positive system $\Sigma^+(\frakg,\widetilde{\fraka})$ of $\Sigma(\frakg,\widetilde{\fraka})$ with corresponding simple roots $\Pi(\frakg,\widetilde{\fraka})$ such that $\Phi$ restricts to surjections
\begin{align*}
 \Phi^+:\Sigma^+(\frakg,\widetilde{\fraka})\cup\{0\}&\rightarrow\Sigma^+(\frakg,\fraka)\cup\{0\}, &&\mbox{and}& \Phi_\Pi:\Pi(\frakg,\widetilde{\fraka})\cup\{0\}&\rightarrow\Pi(\frakg,\fraka)\cup\{0\}.
\end{align*}
Then put $\widetilde{F}:=\Phi_\Pi^{-1}(F)\cup(\Pi(\frakg,\widetilde{\fraka})\setminus\Phi_\Pi^{-1}(\Pi(\frakg,\fraka)))\subseteq\Pi(\frakg,\widetilde{\fraka})$. The statements now follow immediately from \cite[Lemmas 2.2.7, 2.2.8 and 2.4.1]{Wal88} applied to $\widetilde{F}$.
\end{proof}
\chapter{Special Functions}\label{app:SpecialFunctions}

In this chapter we give definitions and basic properties of the classical special functions that appear in this paper.

\section{Bessel functions}\label{app:BesselFunctions}

The series
\begin{align}
 I_\alpha(z) &:= \left(\frac{z}{2}\right)^\alpha\sum_{n=0}^\infty{\frac{1}{n!\Gamma(n+\alpha+1)}\left(\frac{z}{2}\right)^{2n}}\label{eq:DefIBessel}
\end{align}
defines a meromorphic function in $z$ and $\alpha$, called the modified Bessel function of the first kind or $I$-Bessel function (see \cite[Section 3.7]{Wat44}). For $\alpha>-1$ and $z>0$ this function takes real values. $I_\alpha(z)$ solves the following second order differential equation:
\begin{align}
 z^2\frac{\td^2u}{\td z^2}+z\frac{\td u}{\td z}-(z^2+\alpha^2)u &= 0.\label{eq:DiffEqBessel}
\end{align}
Another solution of \eqref{eq:DiffEqBessel} which is linearly independent of $I_\alpha(z)$ is given by the modified Bessel function of the third kind or $K$-Bessel function:
\begin{align}
 K_\alpha(z) &:= \frac{\pi}{2\sin\pi\alpha}(I_{-\alpha}(z)-I_\alpha(z)).\label{eq:DefKBessel}
\end{align}
For convenience we use the following renormalizations:\index{notation}{Iuptildealphaz@$\widetilde{I}_\alpha(z)$}\index{notation}{Kuptildealphaz@$\widetilde{K}_\alpha(z)$}
\begin{align*}
 \widetilde{I}_\alpha(z) &:= \left(\frac{z}{2}\right)^{-\alpha}I_\alpha(z), & \widetilde{K}_\alpha(z) &:= \left(\frac{z}{2}\right)^{-\alpha}K_\alpha(z).
\end{align*}
Note that $\widetilde{I}_\alpha(z)$ is an entire function. Further, since $K_{-\alpha}=K_\alpha$ we have
\begin{align}
 \widetilde{K}_{-\alpha}(z) &= \left(\frac{z}{2}\right)^{2\alpha}K_\alpha(z).\label{eq:KBesselSymmetry}
\end{align}
It follows directly from the definitions that
\begin{align}
 \widetilde{I}_\alpha(e^{i\pi}x) &= \widetilde{I}_\alpha(x),\label{eq:ParityIBessel}\\
 \widetilde{K}_\alpha(e^{i\pi}x) &= a_{2\alpha}\widetilde{K}_\alpha(x)+b_{2\alpha}\widetilde{I}_\alpha(x),\label{eq:ParityKBessel}
\end{align}
where\index{notation}{aalpha@$a_\alpha$}\index{notation}{balpha@$b_\alpha$}
\begin{align*}
 a_\alpha &:= e^{-\alpha\pi i}, & b_\alpha &:= \frac{\Gamma(1-\frac{\alpha}{2})\Gamma(\frac{\alpha}{2})}{2}\left(e^{-\alpha\pi i}-1\right).
\end{align*}
For the special value $\alpha=-\frac{1}{2}$ the normalized $I$- and $K$-Bessel functions are
\begin{align}
 \widetilde{I}_{-\frac{1}{2}}(z) &= \frac{1}{\sqrt{\pi}}\cosh(z), & \widetilde{K}_{-\frac{1}{2}}(z) &= \frac{\sqrt{\pi}}{2}e^{-z}.\label{eq:IKBesselMinusHalf}
\end{align}
In the case where the parameter $\alpha\in\NN_0+\frac{1}{2}$ is a half-integer the $K$-Bessel function degenerates to a combination of power and exponential function and polynomial (see e.g. \cite[III.71~(12)]{Wat44}):
\begin{equation}
 \widetilde{K}_\alpha(z) = \sqrt{\pi}z^{-2\alpha}e^{-z}\sum_{i=0}^{\alpha-\frac{1}{2}}{\frac{(2\alpha-i-1)!}{(\alpha-i-\frac{1}{2})!\cdot i!}(2z)^i}.
 \label{eq:ExplKBessel}
\end{equation}
Corresponding to \eqref{eq:DiffEqBessel}, $\widetilde{I}_\alpha(z)$ and $\widetilde{K}_\alpha(z)$ solve the second order equation
\begin{align}
 z^2\frac{\td^2u}{\td z^2}+(2\alpha+1)z\frac{\td u}{\td z}-z^2u &= 0,\label{eq:DiffEqModBessel}
\end{align}
or equivalently
\begin{align}
 \left(\theta^2+2\alpha\theta-z^2\right)u &= 0,\label{eq:DiffEqModBesselTheta}
\end{align}
where $\theta=z\frac{\td}{\td z}$. For the normalized Bessel functions one has the differential recurrence relations (see \cite[III.71~(6)]{Wat44})
\begin{align}
 \frac{\td}{\td z}\widetilde{I}_\alpha(z) &= \frac{z}{2}\widetilde{I}_{\alpha+1}(z), & \frac{\td}{\td z}\widetilde{K}_\alpha(z) &= -\frac{z}{2}\widetilde{K}_{\alpha+1}(z),\label{eq:BesselDiffFormulas}
\end{align}
with which the differential equation \eqref{eq:DiffEqModBessel} can equivalently be written as recurrence relation (see e.g. \cite[III.71~(1)]{Wat44}):
\begin{align}
 \alpha\widetilde{I}_\alpha(z) &= \widetilde{I}_{\alpha-1}(z)-\left(\frac{z}{2}\right)^2 \widetilde{I}_{\alpha+1}(z),\label{eq:IBesselRecRel}\\
 \alpha\widetilde{K}_\alpha(z) &= \left(\frac{z}{2}\right)^2\widetilde{K}_{\alpha+1}(z)-\widetilde{K}_{\alpha-1}(z).\label{eq:KBesselRecRel}
\end{align}
For $\Re(\alpha)>-\frac{1}{2}$ the Bessel functions have the following integral representations in $x>0$ (cf. formulas III.71~(9) and VI.15~(5) in \cite{Wat44}):
\begin{align}
 \widetilde{I}_\alpha(x) &= \frac{1}{\sqrt{\pi}\Gamma(\alpha+\frac{1}{2})}\int_0^\pi{e^{-x\cos\theta}\sin^{2\alpha}\theta\td\theta},\label{eq:IntFormulaIBessel}\\
 \widetilde{K}_\alpha(x) &= \frac{\sqrt{\pi}}{\Gamma(\alpha+\frac{1}{2})}\int_0^\infty{e^{-x\cosh\phi}\sinh^{2\alpha}\phi\td\phi}.\label{eq:IntFormulaKBessel}
\end{align}
The Mellin transform of the $K$-Bessel function is given by the following formula which holds for $\Re(\sigma),\Re(\sigma-2\alpha)>0$ and $\Re(a)>0$ (see e.g. \cite[equation 6.561~(16)]{GR65}):
\begin{align}
 \int_0^\infty{\widetilde{K}_\alpha(ax)x^{\sigma-1}\td x} &= 2^{\sigma-2}a^{-\sigma}\Gamma\left(\frac{\sigma}{2}\right)\Gamma\left(\frac{\sigma-2\alpha}{2}\right).\label{eq:KBesselMellinTransform}
\end{align}
We further have the following two integral formulas involving two Bessel functions:
\begin{itemize}
\item For $\Re(\sigma),\Re(\sigma-2\beta)>0$, $a<b$ the following holds (see e.g. \cite[equation 6.576~(5)]{GR65}):
\begin{multline}
 \int_0^\infty{\widetilde{I}_\alpha(ax)\widetilde{K}_\beta(bx)x^{\sigma-1}\td x} = \frac{2^{\sigma-2}\Gamma(\frac{\sigma}{2})\Gamma(\frac{\sigma-2\beta}{2})}{b^\sigma\Gamma(\alpha+1)}\\
 \times{_2F_1}\left(\frac{\sigma}{2},\frac{\sigma-2\beta}{2};\alpha+1;\left(\frac{a}{b}\right)^2\right).\label{eq:IntegralIKBessel}
\end{multline}
\item For $\Re(\sigma)>2\max(\Re(\alpha),0)+2\max(\Re(\beta),0)$ we have (see formula 10.3~(49) in \cite{EMOT54a})
\begin{multline}
 \int_0^\infty{\widetilde{K}_\alpha(x)\widetilde{K}_\beta(x)x^{\sigma-1}dx} = \frac{2^{\sigma-3}}{\Gamma(\sigma-\alpha-\beta)}\Gamma\left(\frac{\sigma}{2}\right)\Gamma\left(\frac{\sigma-2\alpha}{2}\right)\\
 \times\Gamma\left(\frac{\sigma-2\beta}{2}\right)\Gamma\left(\frac{\sigma-2\alpha-2\beta}{2}\right).\label{eq:KBesselL2Norms}
\end{multline}
\end{itemize}
Finally, on the positive real line $\RR_+$ the normalized $I$- and $K$-Bessel functions have the following asymptotic behavior (see \cite[Chapters III and VII]{Wat44} and \cite[Chapter 4]{AAR99}): as $x\rightarrow0$
\begin{align}
 \widetilde{I}_\alpha(0) &= \frac{1}{\Gamma(\alpha+1)},\label{eq:BesselIAsymptAt0}\\
 \widetilde{K}_\alpha(x) &= \left\{\begin{array}{ll}\frac{\Gamma(\alpha)}{2}\left(\frac{x}{2}\right)^{-2\alpha}+o(x^{-2\alpha}) &\mbox{if $\alpha>0$}\\-\log(\frac{x}{2})+o(\log(\frac{x}{2})) &\mbox{if $\alpha=0$}\\\frac{\Gamma(-\alpha)}{2}+o(1) &\mbox{if $\alpha<0$}\end{array}\right.,\label{eq:BesselKAsymptAt0}
\end{align}
and as $x\rightarrow\infty$
\begin{align}
\begin{split}
 \widetilde{I}_\alpha(x) &= \frac{1}{2\sqrt{\pi}}\left(\frac{x}{2}\right)^{-\alpha-\frac{1}{2}}e^x\left(1+\mathcal{O}\left(\frac{1}{x}\right)\right),\\
 \widetilde{K}_\alpha(x) &= \frac{\sqrt{\pi}}{2}\left(\frac{x}{2}\right)^{-\alpha-\frac{1}{2}}e^{-x}\left(1+\mathcal{O}\left(\frac{1}{x}\right)\right).
\end{split}\label{eq:BesselAsymptAtInfty}
\end{align}

\section{Laguerre polynomials}\label{app:Laguerre}

For $n\in\mathbb{N}_0$ and $\alpha\in\CC$ the Laguerre polynomial $L_n^\alpha(z)$\index{notation}{Lupnalphaz@$L_n^\alpha(z)$} is defined by (cf. \cite[Equation (6.2.2)]{AAR99})
\begin{align}
 L_n^\alpha(z) &= \frac{(\alpha+1)_n}{n!}\sum_{k=0}^n{{n\choose k}\frac{(-1)^k}{(\alpha+1)_k}z^k}.\label{eq:DefLagFct}
\end{align}
$L_n^\alpha(z)$ solves the following second order differential equation (see \cite[Equation (6.2.8)]{AAR99})
\begin{equation}
 \left(z\frac{\td^2}{\td z^2}+(\alpha+1-z)\frac{\td}{\td z}+n\right)u = 0.\label{eq:LagDiffEq}
\end{equation}
The generating function of the Laguerre polynomials is given by (see e.g. formula (6.2.4) \cite{AAR99}):
\begin{equation}
 \sum_{n=0}^\infty{L_n^\alpha(z)t^n} = \frac{1}{(1-t)^{\alpha+1}}e^{-\frac{tz}{1-t}}.\label{eq:LaguerreGenFct}
\end{equation}
Finally, we have the following integral formula for $\Re(\beta)>\Re(\alpha)>-1$ (cf. formula 16.6~(5) in \cite{EMOT54a})
\begin{align}
 \int_0^1{(1-y)^{\beta-\alpha-1}y^\alpha L_n^\alpha(xy)\td y} &= \frac{\Gamma(\alpha+n+1)\Gamma(\beta-\alpha)}{\Gamma(\beta+n+1)}L_n^\beta(x).\label{eq:LaguerreIntFormula}
\end{align}

\section{Gegenbauer polynomials}\label{app:Gegenbauer}

The classical Gegenbauer polynomials $C_n^\lambda(z)$ with parameters $n\in\NN_0$ and $\lambda\in\CC$ are defined by (see \cite[3.15~(2)]{EMOT81})
\begin{align*}
 C_n^\lambda(z) &= \frac{1}{\Gamma(\lambda)}\sum_{k=0}^n{\frac{(-1)^k\Gamma(\lambda+k)\Gamma(n+2\lambda+k)}{k!(n-k)!\Gamma(2\lambda+2k)}\left(\frac{1-z}{2}\right)^k}.
\end{align*}
We rather use the normalized version\index{notation}{Cuptildenlambdaz@$\widetilde{C}_n^\lambda(z)$}
\begin{align*}
 \widetilde{C}_n^\lambda(z) &= \Gamma(\lambda)C_n^\lambda(z).
\end{align*}
$\widetilde{C}_n^\lambda(z)$ is an even function if $n$ is even and an odd function if $n$ is odd (see \cite[3.15~(5)\&(6)]{EMOT81}). This can be stated as the parity formula
\begin{align}
 \widetilde{C}_n^\lambda(-z) &= (-1)^n\widetilde{C}_n^\lambda(z).\label{eq:GegenbauerParity}
\end{align}
The Gegenbauer polynomial $\widetilde{C}_n^\lambda(z)$ solves the second order differential equation (see \cite[3.15~(21)]{EMOT81})
\begin{align}
 (z^2-1)u''+(2\lambda+1)zu'-n(n+2\lambda)u &= 0.\label{eq:GegenbauerDiffEq}
\end{align}
With the differential recurrence relation (see \cite[3.15~(30)]{EMOT81})
\begin{align}
 \frac{\td}{\td z}\widetilde{C}_n^\lambda(z) &= 2\widetilde{C}_{n-1}^{\lambda+1}(z)\label{eq:Gegenbauer1}
\end{align}
the differential equation \eqref{eq:GegenbauerDiffEq} rewrites as the recurrence relation
\begin{align}
 (2\lambda+n-1)(n+1)\widetilde{C}_{n+1}^{\lambda-1}(z)&-2z(2\lambda-1)\widetilde{C}_n^\lambda(z)+4(1-z^2)\widetilde{C}_{n-1}^{\lambda+1}(z) = 0.\label{eq:Gegenbauer2}
\end{align}

\section{Meijer's $G$-function}\label{app:GFct}

The $G$-function can be defined in a very general setting. However, we restrict our definition to the case which is needed in this paper. The results in this section are taken from \cite[Chapter V]{Luk69}.

Let $0\leq m\leq q$, $0\leq n\leq p$, $p<q$ and $a_1,\ldots,a_p,b_1,\ldots,b_q\in\CC$. Assume further that $a_k-b_j$ is not a positive integer for $j=1,\ldots,m$, $k=1,\ldots,n$. For $z\neq0$ in the univeral covering of $\CC^\times$ we define
\begin{align*}
 G^{m,n}_{p,q}\left(z\left|\begin{array}{cc}a_1,\ldots,a_p\\b_1,\ldots,b_q\end{array}\right.\right) &:= \frac{1}{2\pi i}\int_L{\frac{\prod_{j=1}^m{\Gamma(b_j-s)}\prod_{j=1}^n{\Gamma(1-a_j+s)}}{\prod_{j=m+1}^q{\Gamma(1-b_j+s)}\prod_{j=n+1}^p{\Gamma(a_j-s)}}z^s\td s}.\index{notation}{Gupmnpqz@$G^{m,n}_{p,q}(z|^{a_p}_{b_q})$}
\end{align*}
Here $L$ is a loop beginning and ending at $+\infty$ and encircling all poles of $\Gamma(b_j-s)$, $1\leq j\leq m$, once in the negative direction, but none of the poles of $\Gamma(1-a_k+s)$, $1\leq k\leq n$. $G^{m,n}_{p,q}(z|^{a_p}_{b_q})$ is called \textit{Meijer's $G$-function}. From the definition one immediately obtains the reduction formula (cf. \cite[Equation 5.4~(1)]{Luk69})
\begin{align}
 G^{m,n}_{p,q}\left(z\left|\begin{array}{cc}a_1,\ldots,a_p\\b_1,\ldots,b_{q-1},a_1\end{array}\right.\right) &= G^{m,n-1}_{p-1,q-1}\left(z\left|\begin{array}{cc}a_2,\ldots,a_p\\b_1,\ldots,b_{q-1}\end{array}\right.\right).\label{eq:GFctRed}
\end{align}
The\ \ $G$-function\ \ $G^{m,n}_{p,q}(z|^{a_p}_{b_q})$\ \ solves\ \ the\ \ following\ \ differential\ \ equation\ \ of\ \ order $\max(p,q)$ (see \cite[Equation 5.8~(1)]{Luk69}):
\begin{align}
 \left[(-1)^{m+n-p}z\prod_{j=1}^p{(\theta-a_j+1)}-\prod_{k=1}^q{(\theta-b_k)}\right]u &= 0,\label{eq:GFctDiffEq}
\end{align}
where $\theta=z\frac{\td}{\td z}$ and an empty product is treated as $1$. For the special case of $G^{20}_{04}(z|b_1,b_2,b_3,b_4)$ we find
\begin{align}
 \left[\prod_{j=1}^4\left(\theta-b_j\right)\right]u(z)=zu(z).\label{eq:GFctDiffEqSpecialized}
\end{align}
Various special functions can be expressed in terms of the $G$-function. For instance, the $J$- and $K$-Bessel functions are given by (see \cite[Equations 6.4~(8) \& (11)]{Luk69})
\begin{align}
 \left(\frac{z}{4}\right)^\beta J_\alpha(z) &= G^{20}_{04}\left(\left.\left(\frac{z}{4}\right)^4\right|\frac{\beta+\alpha}{4},\frac{\beta+\alpha+2}{4},\frac{\beta-\alpha}{4},\frac{\beta-\alpha+2}{4}\right),\label{eq:GFctJBessel}\index{notation}{Jupalphaz@$J_\alpha(z)$}\\
 2\left(\frac{z}{2}\right)^\beta K_\alpha(z) &= G^{20}_{02}\left(\left.\left(\frac{z}{2}\right)^2\right|\frac{\beta+\alpha}{2},\frac{\beta-\alpha}{2}\right).\label{eq:GFctKBessel}
\end{align}
We also need the following integral formula for the $G$-function $G^{20}_{04}(z|b_1,b_2,b_3,b_4)$ which holds for $\omega,\eta>0$ and $\Re(b_j-a\pm\frac{\alpha}{2})>-1$ ($j=1,2$) (see \cite[Equation 5.6~(21), assumptions as in case 4]{Luk69}):
\begin{multline}
 \int_0^\infty{x^{-a}K_\alpha(2(\omega x)^{\frac{1}{2}}))G^{20}_{04}(\eta x|b_1,b_2,b_3,b_4)\td x}\\
 = \frac{\omega^{a-1}}{2}G^{22}_{24}\left(\frac{\eta}{\omega}\left|\begin{array}{c}a+\frac{\alpha}{2},a-\frac{\alpha}{2}\\b_1,b_2,b_3,b_4\end{array}\right.\right).\label{eq:GFctIntKBessel}
\end{multline}
Finally, we give the asymptotic behavior of the function $G^{20}_{04}(z|b_1,b_2,b_3,b_4)$ as $z\rightarrow0$ and $z\rightarrow\infty$. For the asymptotics as $z\rightarrow0$ we assume without loss of generality that $b_1\leq b_2$. Then it follows from \cite[Equations 5.2~(7) \& (10)]{Luk69} that
\begin{multline}
 G^{20}_{04}(z|b_1,b_2,b_3,b_4)\\
 = \frac{1}{\Gamma(1+b_1-b_3)\Gamma(1+b_1-b_4)}\times\begin{cases}\Gamma(b_2-b_1)z^{b_1}+o(z^{b_1}) & \mbox{if }b_1<b_2,\\-\ln(z)z^{b_1}+o(\ln(z)z^{b_1}) & \mbox{if }b_1=b_2.\end{cases}\label{eq:GFctAsymptotics0}
\end{multline}
For the asymptotic behavior as $x\rightarrow\infty$ we find with \cite[Section 5.10, Theorem 2]{Luk69} that
\begin{align}
 G^{20}_{04}(x|b_1,b_2,b_3,b_4) = -\frac{1}{\sqrt{2\pi}}x^\theta\cos\left(4x^{\frac{1}{4}}+(b_3+b_4-2\theta)\pi\right)\left(1+\calO(x^{-\frac{1}{4}})\right),\label{eq:GFctAsymptoticsInfty}
\end{align}
where $\theta=\frac{1}{4}(b_1+b_2+b_3+b_4-\frac{3}{2})$.

\bibliographystyle{amsalpha}
\bibliography{bibdb}

\providecommand{\bysame}{\leavevmode\hbox to3em{\hrulefill}\thinspace}
\providecommand{\MR}{\relax\ifhmode\unskip\space\fi MR }
\providecommand{\MRhref}[2]{%
  \href{http://www.ams.org/mathscinet-getitem?mr=#1}{#2}
}
\providecommand{\href}[2]{#2}
\begin{thebibliography}{HKMM09b}

\bibitem[AAR99]{AAR99}
G.~E. Andrews, R.~Askey, and R.~Roy, \emph{Special functions}, Encyclopedia of
  Mathematics and its Applications, vol.~71, Cambridge University Press,
  Cambridge, 1999.

\bibitem[AD{\'O}07]{ADO07}
M.~Aristidou, M.~Davidson, and G.~{\'O}lafsson, \emph{Laguerre functions on
  symmetric cones and recursion relations in the real case}, J. Comput. Appl.
  Math. \textbf{199} (2007), no.~1, 95--112.

\bibitem[Ber00]{Ber00}
W.~Bertram, \emph{The geometry of {J}ordan and {L}ie structures}, Lecture Notes
  in Mathematics, vol. 1754, Springer-Verlag, Berlin, 2000.

\bibitem[BK66]{BK66}
H.~Braun and M.~Koecher, \emph{Jordan-{A}lgebren}, Die Grundlehren der
  mathematischen Wissenschaften in Einzeldarstellungen mit besonderer
  Ber\"ucksichtigung der Anwendungsgebiete, Band 128, Springer-Verlag, Berlin,
  1966.

\bibitem[BK94]{BK94}
R.~Brylinski and B.~Kostant, \emph{Minimal representations, geometric
  quantization, and unitarity}, Proc. Nat. Acad. Sci. U.S.A. \textbf{91}
  (1994), no.~13, 6026--6029.

\bibitem[Bry98]{Bry98}
R.~Brylinski, \emph{Geometric quantization of real minimal nilpotent orbits},
  Differential Geom. Appl. \textbf{9} (1998), no.~1-2, 5--58.

\bibitem[BSZ06]{BSZ06}
L.~Barchini, M.~Sepanski, and R.~Zierau, \emph{Positivity of zeta distributions
  and small unitary representations}, The ubiquitous heat kernel, Contemp.
  Math., vol. 398, Amer. Math. Soc., Providence, RI, 2006, pp.~1--46.

\bibitem[BZ91]{BZ91}
B.~Binegar and R.~Zierau, \emph{Unitarization of a singular representation of
  {${\rm SO}(p,q)$}}, Comm. Math. Phys. \textbf{138} (1991), no.~2, 245--258.

\bibitem[DS99]{DS99}
A.~Dvorsky and S.~Sahi, \emph{Explicit {H}ilbert spaces for certain unipotent
  representations. {II}}, Invent. Math. \textbf{138} (1999), no.~1, 203--224.

\bibitem[DS03]{DS03}
\bysame, \emph{Explicit {H}ilbert spaces for certain unipotent representations.
  {III}}, J. Funct. Anal. \textbf{201} (2003), no.~2, 430--456.

\bibitem[Dvo07]{Dvo07}
A.~Dvorsky, \emph{Tensor square of the minimal representation of {${\rm
  O}(p,q)$}}, Canad. Math. Bull. \textbf{50} (2007), no.~1, 48--55.

\bibitem[EMOT53]{EMOT81}
A.~Erd{\'e}lyi, W.~Magnus, F.~Oberhettinger, and F.~G. Tricomi, \emph{Higher
  transcendental functions. {V}ols. {I}, {II}}, McGraw-Hill Book Company, Inc.,
  New York, 1953.

\bibitem[EMOT54]{EMOT54a}
\bysame, \emph{Tables of integral transforms. {V}ol. {II}}, McGraw-Hill Book
  Company, Inc., New York, 1954.

\bibitem[FK94]{FK94}
J.~Faraut and A.~Kor{\'a}nyi, \emph{Analysis on symmetric cones}, The Clarendon
  Press, Oxford University Press, New York, 1994.

\bibitem[Fol89]{Fol89}
G.~B. Folland, \emph{Harmonic analysis in phase space}, Annals of Mathematics
  Studies, vol. 122, Princeton University Press, Princeton, NJ, 1989.

\bibitem[Fox61]{Fox61}
C.~Fox, \emph{The {$G$} and {$H$} functions as symmetrical {F}ourier kernels},
  Trans. Amer. Math. Soc. \textbf{98} (1961), 395--429.

\bibitem[GR65]{GR65}
I.~S. Gradshteyn and I.~M. Ryzhik, \emph{Table of integrals, series, and
  products}, Academic Press, New York, 1965.

\bibitem[GS64]{GS64}
I.~M. Gel'fand and G.~E. Shilov, \emph{Generalized functions. {V}ol. {I}:
  {P}roperties and operations}, Translated by Eugene Saletan, Academic Press,
  New York, 1964.

\bibitem[GS05]{GS05}
W.~T. Gan and G.~Savin, \emph{On minimal representations definitions and
  properties}, Represent. Theory \textbf{9} (2005), 46--93 (electronic).

\bibitem[Hel69]{Hel69}
K.-H. Helwig, \emph{Halbeinfache reelle {J}ordan-{A}lgebren}, Math. Z.
  \textbf{109} (1969), 1--28.

\bibitem[Hel84]{Hel84}
S.~Helgason, \emph{Groups and geometric analysis}, Pure and Applied
  Mathematics, vol. 113, Academic Press Inc., Orlando, FL, 1984.

\bibitem[HKMM09a]{HKMM09b}
J.~Hilgert, T.~Kobayashi, G.~Mano, and J.~M{\" o}llers, \emph{Orthogonal
  polynomials associated to a certain fourth order differential operator},
  preprint, available at
  \href{http://arxiv.org/abs/0907.2612}{arXiv:0907.2612}.

\bibitem[HKMM09b]{HKMM09a}
\bysame, \emph{Special functions associated to a certain fourth order
  differential operator}, preprint, available at
  \href{http://arxiv.org/abs/0907.2608}{arXiv:0907.2608}.

\bibitem[Jac49]{Jac49}
N.~Jacobson, \emph{Derivation algebras and multiplication algebras of
  semi-simple {J}ordan algebras}, Ann. of Math. (2) \textbf{50} (1949),
  866--874.

\bibitem[Jac68]{Jac68}
\bysame, \emph{Structure and representations of {J}ordan algebras}, American
  Mathematical Society Colloquium Publications, Vol. XXXIX, American
  Mathematical Society, Providence, R.I., 1968.

\bibitem[Kan98]{Kan98}
S.~Kaneyuki, \emph{The {S}ylvester's law of inertia in simple graded {L}ie
  algebras}, J. Math. Soc. Japan \textbf{50} (1998), no.~3, 593--614.

\bibitem[KM07a]{KM07a}
T.~Kobayashi and G.~Mano, \emph{The inversion formula and holomorphic extension
  of the minimal representation of the conformal group}, Harmonic analysis,
  group representations, automorphic forms and invariant theory: In honor of
  {R}oger {H}owe, (eds. J. S. {L}i, E. C. {T}an, N. {W}allach and C. B. {Z}hu),
  World Scientific, 2007, pp.~159--223.

\bibitem[KM07b]{KM08}
\bysame, \emph{The {S}chr\"odinger model of the minimal representation of the
  indefinite orthogonal group {$\textup{O}(p,q)$}}, Memoirs of the Amer. Math.
  Soc. (2007), to appear, available at
  \href{http://arxiv.org/abs/0712.1769}{arXiv:0712.1769}.

\bibitem[Kna86]{Kna86}
A.~W. Knapp, \emph{Representation theory of semisimple groups}, Princeton
  Mathematical Series, vol.~36, Princeton University Press, Princeton, NJ,
  1986, An overview based on examples.

\bibitem[Kna02]{Kna02}
\bysame, \emph{Lie groups beyond an introduction}, second ed., Progress in
  Mathematics, vol. 140, Birkh\"auser Boston Inc., Boston, MA, 2002.

\bibitem[K{\O}03a]{KO03a}
T.~Kobayashi and B.~{\O}rsted, \emph{Analysis on the minimal representation of
  {$\textup{O}(p,q)$}. {I}. {R}ealization via conformal geometry}, Adv. Math.
  \textbf{180} (2003), no.~2, 486--512.

\bibitem[K{\O}03b]{KO03b}
\bysame, \emph{Analysis on the minimal representation of {$\textup{O}(p,q)$}.
  {II}. {B}ranching laws}, Adv. Math. \textbf{180} (2003), no.~2, 513--550.

\bibitem[K{\O}03c]{KO03c}
\bysame, \emph{Analysis on the minimal representation of {$\textup{O}(p,q)$}.
  {III}. {U}ltrahyperbolic equations on {${\mathbb{R}}^{p-1,q-1}$}}, Adv. Math.
  \textbf{180} (2003), no.~2, 551--595.

\bibitem[Kos90]{Kos90}
B.~Kostant, \emph{The vanishing of scalar curvature and the minimal
  representation of {${\rm SO}(4,4)$}}, Operator algebras, unitary
  representations, enveloping algebras, and invariant theory ({P}aris, 1989),
  Progr. Math., vol.~92, Birkh\"auser Boston, Boston, MA, 1990, pp.~85--124.

\bibitem[Loo53]{Loo53}
L.~H. Loomis, \emph{An introduction to abstract harmonic analysis}, D. Van
  Nostrand Company, Inc., Toronto, 1953.

\bibitem[Loo77]{Loo77}
O.~Loos, \emph{{Bounded Symmetric Domains and Jordan Pairs}}, 1977, p.~98.

\bibitem[LS99]{LS99}
T.~Levasseur and J.~T. Stafford, \emph{Differential operators on some nilpotent
  orbits}, Represent. Theory \textbf{3} (1999), 457--473 (electronic).

\bibitem[Luk69]{Luk69}
Y.~L. Luke, \emph{The special functions and their approximations, {V}ol. {I}},
  Mathematics in Science and Engineering, Vol. 53, Academic Press, New York,
  1969.

\bibitem[LV80]{LV80}
G.~Lion and M.~Vergne, \emph{The {W}eil representation, {M}aslov index and
  theta series}, Progress in Mathematics, vol.~6, Birkh\"auser Boston, Mass.,
  1980.

\bibitem[MS10]{MS10}
S.~Merigon and H.~Sepp{\"a}nen, \emph{Branching laws for discrete {W}allach
  points}, J. Funct. Anal. \textbf{258} (2010), no.~10, 3241--3265.

\bibitem[Pev02]{Pev02}
M.~Pevzner, \emph{Analyse conforme sur les alg\`ebres de {J}ordan}, J. Aust.
  Math. Soc. \textbf{73} (2002), no.~2, 279--299.

\bibitem[Sah92]{Sah92}
S.~Sahi, \emph{Explicit {H}ilbert spaces for certain unipotent
  representations}, Invent. Math. \textbf{110} (1992), no.~2, 409--418.

\bibitem[Sah93]{Sah93}
\bysame, \emph{Unitary representations on the {S}hilov boundary of a symmetric
  tube domain}, Representation theory of groups and algebras, Contemp. Math.,
  vol. 145, Amer. Math. Soc., Providence, RI, 1993, pp.~275--286.

\bibitem[Sah95]{Sah95}
\bysame, \emph{{Jordan algebras and degenerate principal series}}, Journal
  f\"{u}r die reine und angewandte Mathematik \textbf{462} (1995).

\bibitem[Sat80]{Sat80}
I.~Satake, \emph{Algebraic structures of symmetric domains}, Kan\^o Memorial
  Lectures, vol.~4, Iwanami Shoten, Tokyo, 1980.

\bibitem[Seg63]{Seg63}
I.~E. Segal, \emph{Transforms for operators and symplectic automorphisms over a
  locally compact abelian group}, Math. Scand. \textbf{13} (1963), 31--43.

\bibitem[Sep07a]{Sep07b}
H.~Sepp{\"a}nen, \emph{Branching laws for minimal holomorphic representations},
  J. Funct. Anal. \textbf{251} (2007), no.~1, 174--209.

\bibitem[Sep07b]{Sep07a}
\bysame, \emph{Branching of some holomorphic representations of {${\rm
  SO}(2,n)$}}, J. Lie Theory \textbf{17} (2007), no.~1, 191--227.

\bibitem[Sep08]{Sep08}
\bysame, \emph{Tube domains and restrictions of minimal representations},
  Internat. J. Math. \textbf{19} (2008), no.~10, 1247--1268.

\bibitem[Sha62]{Sha62}
D.~Shale, \emph{Linear symmetries of free boson fields}, Trans. Amer. Math.
  Soc. \textbf{103} (1962), 149--167.

\bibitem[SW71]{SW71}
E.~M. Stein and G.~Weiss, \emph{Introduction to {F}ourier analysis on
  {E}uclidean spaces}, Princeton University Press, Princeton, N.J., 1971.

\bibitem[vdB05]{vdB05}
E.~P. van~den Ban, \emph{The {P}lancherel theorem for a reductive symmetric
  space}, Lie theory, Progr. Math., vol. 230, Birkh\"auser Boston, Boston, MA,
  2005, pp.~1--97.

\bibitem[Vog81]{Vog81}
D.~A. Vogan, Jr., \emph{Singular unitary representations}, Noncommutative
  harmonic analysis and {L}ie groups ({M}arseille, 1980), Lecture Notes in
  Math., vol. 880, Springer, Berlin, 1981, pp.~506--535.

\bibitem[Wal88]{Wal88}
N.~R. Wallach, \emph{Real reductive groups. {I}}, Pure and Applied Mathematics,
  vol. 132, Academic Press Inc., Boston, MA, 1988.

\bibitem[War72]{War72}
G.~Warner, \emph{Harmonic analysis on semi-simple {L}ie groups. {I}},
  Springer-Verlag, New York, 1972, Die Grundlehren der mathematischen
  Wissenschaften, Band 188.

\bibitem[Wat44]{Wat44}
G.~N. Watson, \emph{A {T}reatise on the {T}heory of {B}essel {F}unctions},
  Cambridge University Press, Cambridge, England, 1944.

\bibitem[Wei64]{Wei64}
A.~Weil, \emph{Sur certains groupes d'op\'erateurs unitaires}, Acta Math.
  \textbf{111} (1964), 143--211.

\bibitem[Zha95]{Zha95}
G.~Zhang, \emph{{Jordan algebras and generalized principal series
  representations}}, Mathematische Annalen \textbf{786} (1995), 773--786.

\end{thebibliography}

\cleardoublepage
\afterpage{\fancyhead[LE,RO]{Notation Index}}
\printindex{notation}{Notation Index}
\cleardoublepage

\printindex{subject}{Subject Index}
\afterpage{\fancyhead[LE,RO]{Subject Index}}

\end{document}